\documentclass[11pt]{imsart}
\oddsidemargin
-0.0in \topmargin -0.5in

\RequirePackage{amsthm,amsmath,amsfonts,amssymb}
\RequirePackage[numbers]{natbib}

\startlocaldefs

\usepackage{stackengine}

\usepackage[textsize=footnotesize]{todonotes}
\usepackage{xcolor}
\usepackage{hyperref}

\usepackage{mathrsfs}
\usepackage{graphicx}
\usepackage{amsmath}        
\usepackage{amssymb}
\usepackage{amsthm}
\usepackage{bm}
\usepackage[mathscr]{euscript}
   
        \usepackage[countmax]{subfloat}  
\usepackage{fourier}
\usepackage{xcolor}
\usepackage{mathtools}
\usepackage{enumerate}
\usepackage{enumitem}
\counterwithin{figure}{section}
\usepackage{natbib}
\usepackage{url}
\usepackage{verbatim}
\usepackage{booktabs}
\usepackage{placeins}
\usepackage{multirow}
\usepackage{caption}
\usepackage{subcaption}

\hypersetup
{
    pdfauthor={Anne van Delft},
    pdfkeywords={functional data, time series, spectral analysis,martingales}
}
\newcounter{alphcount}
{\begin{list}{{\upshape(}\alph{alphcount}\/{\upshape)\ }}%
             {\usecounter{alphcount}\labelwidth1.5em%
              \leftmargin2em\labelsep0.5em\topsep0.25em plus 0.5ex%
              \itemsep0.25em plus 0.5ex\parsep0em}}{\end{list}}
{\begin{list}{{\upshape(#1\arabic{alphcount})\hfill}}%
             {\usecounter{alphcount}\labelwidth2.5em%
              \leftmargin2.5em\labelsep0em\topsep0.25em plus 0.5ex%
              \itemsep0.25em plus 0.5ex\parsep0em}}{\end{list}}

\newtheorem{proposition}{Proposition}[section]     
\newtheorem{thm}{Theorem}[section]
\newtheorem{lemma}{Lemma}[section]    
            
\newtheorem{definition}{Definition}[section]
\newtheorem{Remark}{Remark}[section]
\newtheorem{Corollary}{Corollary}[section]
\newtheorem{assumption}{Assumption}[section]

\theoremstyle{definition}

\newcommand{\im}{\mathrm{i}}
\newcommand{\purp}{\textcolor{black}}

\newcommand{\HDB}{\textcolor{blue}}

\newcommand{\znum}{\mathbb{Z}}
\newcommand{\cnum}{\mathbb{C}}
\newcommand{\rnum}{\mathbb{R}}
\newcommand{\nnum}{\mathbb{N}}
\newcommand{\pro}{\mathbb{P}}
\newcommand{\E}{\mathbb{E}}
\newcommand{\Cov}{\text{Cov}}
\newcommand{\Var}{\text{Var}}
\newcommand{\mynegspace}{\hspace{-0.12em}}

\newcommand{\bigsnorm}[1]{\Big\rvert\mynegspace\Big\rvert\mynegspace\Big\rvert\mynegspace {#1} \Big\rvert\mynegspace \Big\rvert\mynegspace \Big\rvert}

\newcommand{\norm}[1]{{\|}{#1}{\|}}
\newcommand{\bignorm}[1]{\Big{\|}{#1}\Big{\|}}
\newcommand{\Bignorm}[1]{\Bigg{\|}{#1}\Bigg{\|}}
\newcommand{\inprod}[2]{\langle #1, #2 \rangle}
\newcommand{\biginprod}[2]{\Big\langle #1, #2 \Big\rangle}
\newcommand\tageq{\addtocounter{equation}{1}\tag{\theequation}}
\DeclareMathOperator{\Tr}{Tr}
\DeclareMathOperator*{\argmin}{arg\,min}

\newcommand{\ldmi}[4]{\mathcal{Z}^{{#1}(#2,\omega_{#3})}_{m,#4}}
\newcommand{\dmpi}[4]{\tilde{D}^{({#1},{#2})}_{m,#3,#4}}
\newcommand{\dm}[3]{D^{({#1},{#2})}_{m,#3}}

\newcommand{\F}{\mathcal{F}}
\newcommand{\C}{\mathcal{C}}

\newcommand{\T}{\mathcal{T}}

\newcommand{\G}{\mathcal{G}}

\newcommand{\Hi}{\mathcal{H}}
\newcommand{\hi}{\mathcal{H}}

\newcommand{\op}{\mathcal{L}}
\newcommand{\opl}{S_{\infty}}
\newcommand{\flo}[1]{\lfloor #1 \rfloor}

\newcommand{\Xu}[1]{X^{(\frac{#1}{T})}_{#1}}
\newcommand{\tu}[2]{\mathfrak{t}_{#1}(#2)} 
\newcommand{\te}[1]{\mathfrak{t}(#1)} 
\newcommand{\Hd}{\mathcal{H}}

\newcommand{\bw}{\tilde{b}}
\renewcommand{\bf}{b_f}
\allowdisplaybreaks

\newcommand{\pr}{\prime}
\newcommand\barbelow[1]{\stackunder[1.2pt]{$#1$}{\rule{.8ex}{.075ex}}}

\newcounter{relctr} 
\everydisplay\expandafter{\the\everydisplay\setcounter{relctr}{0}} 

\AtBeginDocument{}

\endlocaldefs

\begin{document}

\begin{frontmatter}

\title{A general framework to quantify deviations from structural assumptions in the analysis of  nonstationary function-valued processes
}

\begin{aug}
\author[A]{\fnms{Anne} \snm{van Delft}\ead[label=e1,mark]{anne.vandelft@columbia.edu}}
\and
\author[B]{\fnms{Holger} \snm{Dette}\ead[label=e2]{holger.dette@rub.de}}
\address[A]{Department of Statistics, Columbia University, 1255 Amsterdam Avenue, New York, NY 10027, USA. \printead{e1}}

\address[B]{Ruhr-Universit\"at Bochum, Fakult\"at f\"ur Mathematik, 44780 Bochum, Germany \printead{e2}}
\end{aug}

\begin{abstract}
We present a general theory to quantify the  uncertainty from imposing   structural assumptions on the second-order structure of nonstationary Hilbert space-valued processes, which can be measured via   functionals of time-dependent spectral density operators.
 The second-order dynamics are well-known to be elements of the space of trace-class operators, the latter is a Banach space of type 1 and of cotype 2, which  makes the development of  statistical  inference tools
more challenging. 
A part of our contribution is to obtain 
a weak invariance principle as well as concentration inequalities for (functionals of) the sequential time-varying spectral density operator. 
\purp{In addition, we introduce deviation measures in the nonstationary context, and derive corresponding estimators that are asymptotically pivotal.}
 We then apply this framework and propose statistical methodology to 
investigate the validity of  structural assumptions for nonstationary 
response surface data, such as low-rank assumptions in the context of 
time-varying dynamic fPCA and principle separable component analysis, 
deviations from stationarity with respect to the square root distance,
and  deviations from zero functional canonical coherency.

\end{abstract}


\end{frontmatter}


\section{Introduction}

\label{sec1}

Let $\{X_t| t\in \znum\}$ denote a possibly nonstationary stochastic processes with elements taking values in a function space, often a separable Hilbert space $\Hi$.
\purp{Typical examples include functional and  response surface data, and the analysis of data stemming from such processes 
may be challenging from a statistical and computational point of view.  Structural assumptions  on the second-order structure such as  stationarity, low-rank or separability assumptions are quite common, either  to simplify the statistical analysis or to address computational issues, in particular storage problems or running time
\citep[see, for example,][]{PanTav2013,Hormann2015,AstPigTav2017,Masak2020,vDCD18}}. 
These assumptions  are rarely made because one believes in them, but in the hope that the simplified model describes the reality with sufficient accuracy such that a more efficient  analysis
and a better interpretation of the results
is possible {\citep[see][]{MarcGenton2007,Masak2020,DetDieKut21}. }
Disregarding the  computational aspect, this can be viewed as a bias-variance trade-off where a structural assumption yields a decrease of variance at the cost of an increase in bias due to model miss-specification.  In many cases, it is possible to describe deviations from model assumptions by calculating a minimal distance between the ``general'' model and its approximation under a structural assumption. Estimates of these ``optimal'' distances are then used to investigate if a particular structural assumption can be justified {\citep[see][]{BagDette2017,vDCD18,Masak2020}}.  

In this paper, we put forward a comprehensive theory to quantify uncertainty resulting from imposing structural assumptions on the  aforementioned type of processes. The theory developed involves deviation measures that can be represented as functionals of the time-varying spectral density operator, and is readily applicable  (under milder assumptions) to (time-varying) covariance operators, independent functional data or the Euclidean setting. In \autoref{sec2}, we specify our context of interest, introduce the general form of the deviation measures, and explain  the principle idea behind our approach to statistically quantify uncertainty. 
\autoref{sec3} - which provides the theoretical backbone of our work - is devoted to the introduction of  estimators of these functionals (see point 2. below) and the derivation of two main theoretical results.
The first, \autoref{thm:Conv}, provides a weak invariance principle for a sequential estimator 
in the space of $S_1(\Hi)$-valued continuous functions on $[0,1]$. The second main result,
\autoref{thmmain}, provides a nontrivial self-normalization approach in the nonstationary setting.

Our contribution is relevant for the applied researcher, and the theory is novel from a statistical methodological perspective as well as from an applied probabilistic perspective. We elaborate briefly on this below.

\begin{enumerate}[wide, labelwidth=!, labelindent=0pt]
    \item \textbf{Statistical methodology to quantify uncertainty in various applications:} We introduce new measures to statistically quantify; the loss from a low-rank assumption in what we refer to as i) \textit{time-varying dynamic functional principal component analysis} (tvDFPCA) and ii) \textit{time-varying dynamic principle separable  component analysis} (tvDPSCA), iii) deviations from stationarity with respect to the square root distance, iv) deviations from zero (time-varying) functional canonical coherency. We then derive the limiting distributional properties of the corresponding estimators (normalized and centered around the population measures) as  special cases of our general theory introduced in \autoref{sec3}.  
    We emphasize that the measures and subsequent tests proposed are also \purp{new and of importance} in the classical Euclidean  setting. For example, \purp{ in  PCA of multivariate data  one often relies on ad-hoc methods such as visual inspection of scree plots to determine the number of factors to retain in an exploratory factor analysis. 
    Instead, we provide the applied researcher with novel  statistical tools to make such decisions with statistical guarantees.}  More specifically,  \autoref{sec4} provides methodology to construct confidence intervals for the measure of deviation, to test relevant hypotheses, and to address the problem of reducing the model complexity when a given model can be approximated by an increasing sequence of ``simpler'' models.  
    \purp{For instance, we develop a test for the hypotheses that a given number of 
    (time-varying) dynamic functional principal components explains at least $100\% \nu$ 
    of variation and we propose a pivotal confidence interval for a quantity which measures the deviation from stationarity. 
    We illustrate the finite sample performance and implementation of this method in Section \ref{sec5new} by means of a simulation study.
    Other applications will be discussed in Section \ref{sec24} below.}

    \item \textbf{Development of pivotal statistics for inference on  nonstationary  processes:}
 Even in the stationary finite-dimensional setting, the distributions of proposed  deviation measures 
 are complicated because these generally depend on nuisance parameters that are difficult to estimate. 
    In case of nonstationary data - especially in the infinite dimensional setting -   
    robust estimation of the complex nuisance parameters to obtain critical values becomes even more problematic and often is infeasible. \purp{The main goal of this article is to address this important problem and to develop a theoretical framework that circumvents this problem
    and thus still enables statistical analysis in these cases. More precisely, we will develop   asymptotically pivotal statistical inference tools for nonstationary data, which is highly non-trivial,} and - to the best of the author's knowledge - no results are available, albeit self-normalization approaches that induce pivotal tests in case of stationary functional time series have been considered \citep{Zhang2015,detkokvol2020}, as well as for stationary univariate time series \citep[see][for an overview]{shao2015}. 
   To this purpose,
   we proceed by showing that the distributional properties of the functionals can be approximated by a sum of independent martingales, which take values in the space of trace class operators, $S_1(\Hi)$, and that are smooth functions in re-scaled time. We subsequently establish a functional central limit theorem for the latter. The aforementioned approximation theory in $S_1(\Hi)$ is highly nontrivial and  relies upon several new results.
\item \textbf{Inequalities and concentration bounds for $S_1(\Hi)$-valued random processes:}
It is well-known that covariance operators of $\Hi$-valued processes with finite second strong moment are elements of $S_1(\Hi)$. It is therefore the natural space to work with when considering convergence results of corresponding estimators. Yet, most of the literature on functional data considers distributional convergence of estimators of such operators by endowing the space with the Hilbert-Schmidt topology. 
This has - from a theoretical point of view - the advantage that  the space of Hilbert-Schmidt operators, $S_2(\Hi)$, is a separable Hilbert space. 
This does not only considerably simplify derivations, but 
convergence results can also be obtained under weaker assumptions (see Remark \autoref{rem:S1vsS2}).
However, often the $S_2(\Hi)$-topology is insufficient and the $S_1(\Hi)$-topology is  necessary, as is the case for the statistical 
methodology developed in the current paper. For example, if the statistics involve the trace functional $\Tr:S_1(\Hi) \to \cnum$, which is a positive linear functional that is continuous with respect to  $\norm{\cdot}_{S_1}$ but \textbf{not} with respect to $\norm{\cdot}_{S_2}$. 
The space $S_1(\Hi)$ is however a non-commutative Banach space that has type 1 and cotype 2 \citep{tj74};  see also \cite{LT91} for type and cotype of Banach spaces. Consequently, various useful inequalities that one might take for granted when working in 2-smooth Banach spaces are no longer valid. In order to establish our general framework, we derive appropriate alternatives (see  Remark \autoref{rem:sqrt} and \autoref{secA}). It is worth mentioning that  a CLT for nuclear covariance operators was obtained by \cite{Mas2006} in the stationary setting under a more stringent set of conditions on the dependence structure. 

Because of space constraints, all  proofs and underlying theory are deferred to the appendix. 

\end{enumerate}

\textbf{Notation.}
We denote the topological dual space of a separable Banach space, $V$, by $V^\prime$. 
For  complex separable Hilbert spaces $\Hi, \Hi_1, \Hi_2$, we let $\Hi_1\otimes \Hi_2$ denote the Hilbert tensor product space, and we denote the $r$th Schatten class by $S_r(\Hi)$, which is the class of bounded linear operators $A:\Hi \to \Hi$ with bounded $r$th Schatten norm, i.e,  $\|A\|_{S_r}<\infty$, $1\le r \le \infty$.  We denote the respective subspaces of self-adjoint and non-negative definite operators by $(\cdot)^{\dagger} $ and $(\cdot)^+$, respectively. We remark that the cone of non-negative definite elements satisfies   $S_r(\Hi)^+ \subseteq S_r(\Hi)^\dagger$, and that $S_\infty(\Hi)\cong (S_1(\Hi))^\prime$. We furthermore denote the tensor product operator $h\otimes g:\Hi_1\to \Hi_2$, which is the bounded linear operator $(h\otimes g)v=\inprod{v}{g}h$, $v,g \in \Hi_1, h\in \Hi_2$. We emphasize that the meaning of $\cdot \otimes \cdot$ will be clear from the context.
Throughout this paper, we use the notation $I= [0,1]$ for the unit interval in order to refer 
 to a process indexed by a parameter $\eta \in I$, while we use the parameter $u\in [0,1]$ to refer to re-scaled time. The set of all continuous functions $f:I \to V$ is denoted by $C_{V}=C(I, V)$.  \\
For a measurable space $(T, \Sigma)$ with $\sigma$-finite measure $\mu$, we define $L^p_V(T,\mu)$ as the Banach space of all strongly measurable functions $f:T \to V$ with finite norm $\|f\|_{\mathcal{L}^p_V(T,\mu)}=\big(\int \|f(\tau)\|^p_{V} d\mu(\tau)\big)^{1/p}$, $1\le p<\infty$. If $T \subset\rnum^d$ and $\mu$ is the Lebesgue measure on $T$, we simply write $L^p_V(T)$. Given a probability space $(\Omega, \mathcal{A}, \mathbb{P})$, we write $\mathcal{L}^p_V= L^p_V(\Omega, \mathcal{A}, \mathbb{P})$ for the space of all strongly measurable functions  $X: (\Omega, \mathcal{A}, \mathbb{P}) \to (V, \mathcal{B}_V)$ with $\|X\|_{V,p} = \big(\E\|X\|^p_V\big)^{1/p}<\infty$. The (conditional) expectation of $X\in \op^p_V$ thus exist (in the sense of a Bochner integral), and the $k$th moments $k\le p$ can be naturally defined using tensor products.  We use $\stackrel{\mathcal{D}}{\Longrightarrow}$ for process convergence,  $\stackrel{\mathcal{D}}{\longrightarrow} $  for  distributional  convergence of a vector, and  $\stackrel{\mathbb{P}}{\longrightarrow} $ for convergence in probability. Finally,  we let $\flo{\cdot}$ denote the floor function and for two positive sequences, $(a_T)$, $(b_T)$, we write $a_T \sim b_T$ if $\lim_{T \to \infty} a_T/b_T=1$.

\section{Structural assumptions} \label{sec2}
\def\theequation{2.\arabic{equation}}
\setcounter{equation}{0}

\subsection{Locally stationary processes} \label{sec21}

We start by introducing a general definition of locally stationary processes  of rate $\zeta$ with values in a separable Banach space $V$, which can be seen as a natural generalization of the definition provided in \cite{vde16,drw2019}, and lies in the line of literature on locally stationary (Euclidean) time series, as initiated by \cite{dahlhaus1997} (see also \cite{ZW21,yz2021,kp15,jensub2015,nvsk2000}).

\begin{definition}[locally stationary function-valued process of rate $\zeta$] \label{def:locstat}
Let $\{X_{t,T}| t=1,\ldots, T; T \in \nnum\}$ be a sequence of $V$-valued stochastic processes indexed by $T \in \nnum$, and let $\{{X}^{(u)}_{t}|t \in \znum; u \in [0,1]\}$ be a $V$-valued stationary stochastic process such that, 
for some $p  >0$ and constants $C_p >0$, $\zeta>0$,
\[\|X_{t,T}-{X}^{(t/T)}_{t}\|_{V,p} \le C_p T^{-\zeta} \quad \text{ and } \quad \|{X}^{(u)}_{t}-{X}^{(v)}_{t}\|_{V,p} \le   C_p |u-v|^\zeta\]
uniformly in $t=1, \ldots,T$ and $u,v \in [0,1]$. Then we call the $V$-valued process  $\{X_{t,T}| t=1,\ldots, T; T \in \nnum\}$ \textit{locally stationary with approximation rate $\zeta$}.
\end{definition}

  We will mainly focus on the case where $V$ is a separable Hilbert space, i.e., $V=\Hi$. Provided
 $\sup_{t,T}\|X_{t,T}\|_{V,p} < \infty$ and $\sup_{u \in [0,1]}\|X^{(u)}_{t}\|_{V,p} < \infty$, $p\ge 2$,
 the local second-order dynamics of $\{X_{t,T}| t=1,\ldots, T; T \in \nnum\}$  can be approximated by those of the auxiliary process $\{{X}^{(u)}_{t}|t \in \znum; u \in [0,1]\}$.  More specifically, for fixed $T$, the instantaneous $h$th lag  covariance operator at time $t$ of $\{X_{t,T}| t=1, \ldots , T, T\in \nnum\}$ is then an element of $S_1(\Hi)$ and is given by 
$
\C^{(T)}_{t,h} = \E (X_{t+h,T} \otimes X_{t,T})$, $h \in \znum
$, and the $h$th lag covariance operator at a fixed time $u$ of the auxiliary process is given by
\begin{equation}
\label{deq40}
\C_{h}^{(u)} = \E (X^{(u)}_{h} \otimes X^{(u)}_{0}), \quad h \in \znum,
\end{equation}
where $\C^{(\cdot)}_{h} : [0,1] \to S_1(\Hi)$ is Lipschitz continuous with respect to the trace-class norm $\norm{\,\cdot}_{S_1}$. Furthermore, for fixed $h \in \znum$, 
$
\norm{\C^{(T)}_{t,h} -\C_{h}^{(t/T)}}_{S_1} =O(T^{-\zeta})$ uniformly in $t=1, \ldots,T$.

In order to capture the full second-order dynamics of function-valued processes effectively, a time-frequency approach is taken. 
Under regularity conditions given below,  
the collection of time-varying spectral density operators belongs to the space of continuous functions from the time-frequency plane onto the cone of non-negative definite trace-class operators, i.e., 
\begin{equation}
\label{tvSDO}
\F =
\{\F_{u,\omega}\}_{u\in [0,1], \omega \in [-\pi,\pi]} 
\in C([0,1] \times [-\pi,\pi], S_1(\Hi)^+)~,
\end{equation}
where $\F_{u,\omega}=
\frac{1}{2\pi}\sum_{h \in \znum} \C^{(u)}_{h}  e^{-\im \omega h}
\in 
  S_1(\Hi)^+ 
$. 
{Because the second-order dynamics of the triangular array $\{X_{t,T}| t=1,\ldots, T; T \in \nnum\}$ can - under the stated conditions - be completely described by $\F$, we refer to $\F$ as the (time-varying) spectral density operator of the process.} \\
For fixed $(u,\omega) \in  [0,1] \times [-\pi,\pi]$, let $\{ 
\dot{\lambda}^{(u,\omega)}_i : i \ge 1\}$, denote the sequence of eigenvalues in descending order and let  $\{\phi^{(u,\omega)}_i: i \ge 1\}$ denote the corresponding eigenfunctions. We consider the following representation of the  eigendecomposition of  the operator 
\begin{equation}
\label{evd}  
\F_{u,\omega} = 
 \sum_{j=1}^{\infty }  \lambda^{(u,\omega)}_{j} \Pi^{(u,\omega)}_j  \in S_1(\Hi)^+,  
\end{equation}
where
$\lambda^{(u,\omega)}_{1} > \lambda^{(u,\omega)}_{2}
> 
  \ldots > 0$ are the distinct
  eigenvalues   
and where $
\Pi^{(u,\omega)}_j = \sum_{k \in \{i: \dot{\lambda}^{(u,\omega)}_{i} =\lambda^{(u,\omega)}_{j}\}} \phi^{(u,\omega)}_k \otimes \phi^{(u,\omega)}_k  
$ denotes the projection operator onto the $j$th eigenspace.

\subsection{Some measure of model deviation in the second order dynamics} \label{sec24}

\purp{In this section, we introduce 
several measures of model deviation which can be represented as a functional, say $\mathcal{T}$, 
of the  collection of time-varying spectral density operators defined in \eqref{tvSDO}.}

\subsubsection{Time-varying Dynamic Functional  principal component analysis (tvDFPCA)} \label{sec222}

Similar to classical PCA, the goal of tvDFPCA - which we introduce here - is to provide an optimal lower-dimensional representation of the nonstationary process by applying an appropriate ``filter''. Unlike dynamic FPCA (\cite{Hormann2015,PanTav2013,vDE19}),
tvDFCPA  takes into account the changing temporal dynamics. More specifically, it follows from the latter work that, for fixed $u$, the functional Cramer representation of 
 the approximating stationary process $\{X_t^{(u)}|t \in \mathbb{Z};u\in[0,1]\} $ 
 is given by 
 \begin{equation}
X^{(u)}_t= \int_{-\pi}^{\pi} e^{\im \omega t} d Z_{u,\omega} ~~~ a.e. \text{ in } \Hi  ~,  \label{eq:tvCram}
 \end{equation}
where   $\{Z_{u,\omega}: \omega \in (-\pi,\pi] \}$ is an $\Hi$-valued right-continuous orthogonal increment process for each $u \in [0,1]$.  Provided $u \mapsto X_t^{(u)}$ is continuous, it can be shown \citep[similarly to Theorem  B2.2 of][]{vde16} that  $u \mapsto Z_{u,\cdot}$ is continuous as well.
Define the Bochner space $L^2_{S_\infty(\Hi)}([0,1]\times [-\pi,\pi], \mu_{\mathfrak{F}})$ where  $\mu_{\mathfrak{F}}(E)= 
 \int_{E} \norm{\F_{u,\omega}}_{S_1} d\omega du$
for all Borel measurable sets $E\subseteq [-\pi,\pi]\times [0,1]$.
Then, among all rank $d$ reductions of $\{X^{(u)}_t\}$ to a process $\{Y^{(u)}_t\}$ with representation 
$Y^{(u)}_t = \int_{-\pi}^{\pi} e^{\im \omega t} A_{u,\omega} d Z_{u,\omega}$\,,  $\text{rank}(A_{u,\omega} ) \le d$, we have
\[
\E \|X^{(u)}_t - Y^{(u)}_t\|^2_\Hi \ge \E\norm{X^{(u)}_t-\breve{X}^{(u)}_t }^2_\Hi = \int_{-\pi}^{\pi}
 \sum_{j=d+1}^{\infty} \lambda^{(u,\omega)}_{j} d\omega \tageq\label{eq:tvFPCAerror}~,
\]
where $\breve{X}^{(u)}_t $ yields the optimal representation of the process at time $u$ with $d$ degrees of freedom, which is given by 
\[
\breve{X}^{(u)}_t = \int_{-\pi}^{\pi} e^{\im \omega t} \sum_{j=1}^{d} \Pi^{(u,\omega)}_j d Z_{u,\omega},  \tageq\label{eq:tvFPCAproc}
\]
where $\lambda^{(u,\omega)}_j$ and $\Pi^{(u,\omega)}_j$ denote, respectively, the $j$th largest eigenvalue and corresponding projector of $\F_{u,\omega}$ as defined in \eqref{tvSDO}.
From \eqref{eq:tvFPCAerror}, we can define a normalized measure of total variation explained by the principal $d$ directions over 
 $[0,1] \times [a,b]$ by \begin{equation}
  \label{eq20}  
s_d=
\mathcal{T}  ( {\F}) :=  \frac{\int_0^1 \int_{a}^{b}  \sum_{j=1}^{d} \lambda^{(u,\omega)}_{j} d\omega du }{\int_0^1 \int_{a}^{b}
 \sum_{j=1}^{\infty} \lambda^{(u,\omega)}_{j} d\omega du}~, 
\end{equation}
where the last equality defines the functional 
$\mathcal{T} $  of the spectral density operator $ {\F}$ explicitly.
\purp{Statisticians may be interested in  a confidence interval for $s_d$. A further important question is  
if (for    fixed $d=d_0$ and $\nu \in [0,1]$)
the first $d_0$ principal  components 
explain at least $100\cdot \nu \%$  of variation. This problem can be addressed by 
testing the hypotheses 
$H_{0}^{(d_{0})}: s_{d_{0}} \leq \nu $    versus  $H_{1}^{(d_{0})}: s_{d_{0}}  > \nu$.
Moreover, one might be also  interested in  estimating the minimal rank $d^*$ such 
 that $s_{d^*} > \nu$.
 }

 \subsubsection{Time-varying dynamic principal separable component analysis} \label{sec223}
 
Another application of interest is a measure  for the degree of separability. To make this more precise, we introduce the \textit{separable component decomposition} of ${\F}_{u,\omega}$ in spirit of \cite{Masak2020}. Let $\Hi =\Hi_1 \otimes \Hi_2$, then it is known that we have the following isometric isomorphisms
\[
\Hi_1 \otimes \Hi_2 \otimes \Hi_1 \otimes \Hi_2 \cong S_2(\Hi, \Hi) \cong S_2(\Hi_2 \otimes \Hi_2, \Hi_1 \otimes \Hi_1) \cong  S_2(\Hi_1) \otimes S_2(\Hi_2)~.
\]
Furthermore, we recall the definition of the Kronecker tensor product $A\widetilde{\otimes}B$, which satisfies $(A\widetilde{\otimes}B)C=ACB^\dagger$, $A,B,C \in S_r(\Hi)$, and which essentially entails a permutation of the dimensions, i.e.,  for $a,a^\prime \in \Hi_1$, $b,b^\prime \in \Hi_2$, we have $
\big((a \otimes a^\prime) \widetilde{\otimes} (b \otimes b^\prime)\big)= \big((a \otimes b) {\otimes} (a^\prime \otimes b^\prime)\big)$.
When viewed as an element of $S_2(\Hi)^\dagger$, we have the eigendecomposition
\[
{\F}_{u,\omega}= \sum_{j=1}^{\infty} \dot{\lambda}^{(u,\omega)}_{j} \phi^{(u,\omega)}_{j} \otimes  \phi^{(u,\omega)}_{j} \tageq \label{eq:eigdec}~, 
\]
where $\{\phi^{(u,\omega)}_{j}\}_{j \ge 1}$ is an orthonormal basis 
(ONB) of $\Hi$. Alternatively, we may view this object as an element of the product space $S_2(\Hi_2 \otimes \Hi_2,  \Hi_1 \otimes \Hi_1)$, say $F_{u,\omega}$,  in which case we have the spectral decomposition ${F}_{u,\omega}=\sum_{j=1}^{\infty} \delta^{(u,\omega)}_{j} f^{(u,\omega)}_{j} \otimes  g^{(u,\omega)}_{j}$ where $\{f^{(u,\omega)}_{j}\}_{j \ge 1}$ is an ONB of $\Hi_1 \otimes \Hi_1$ and $\{g^{(u,\omega)}_{j}\}_{j \ge 1}$ is an ONB of $\Hi_2 \otimes \Hi_2$. If we view ${F}_{u,\omega}$ as an element of $S_2(\Hi_1) \otimes S_2(\Hi_2)$, we can write the latter as 
${F}_{u,\omega}=\sum_{j=1}^{\infty} \delta^{(u,\omega)}_{j} A^{(u,\omega)}_{j} \otimes B^{(u,\omega)}_{j}$
such that a self-adjoint version ${\F}_{u,\omega} \in S_2(\Hi,\Hi)$ is given by
\[
{\F}_{u,\omega}= \sum_{j=1}^{\infty} \delta^{(u,\omega)}_{j} A^{(u,\omega)}_{j} \widetilde{\otimes} B^{(u,\omega)}_{j}, \tageq \label{eq:sepdec}
\]
We refer to \eqref{eq:sepdec} as the \textit{separable component decomposition} of ${\F}_{u,\omega}$. It is worth emphasizing that \eqref{eq:sepdec} and \eqref{eq:eigdec} are different decompositions of the same operator. The operator is called separable of degree $d$  if all but the first $d$ scores $\delta^{(u,\omega)}_{j}$ are zero. Note that this is a generalization of the notion of a separable covariance operator \citep[see][among others]{AstPigTav2017,BagDette2017,conkokrei2017}, which in our setting would  indicate only the first score $\delta^{(u,\omega)}_{1}$ is nonzero. We therefore  define
\begin{align}
\label{eq:sd_SPCA}
s_d = \mathcal{T}  ( {\F}) & :=  
\frac{\int_{0}^1 \int_{a}^{b} \sum_{j=1}^{d} \big(\delta^{(u,\omega)}_{j}\big)^2 du d\omega}{\int_{0}^1 \int_{a}^{b} \sum_{j=1}^{\infty} \big(\delta^{(u,\omega)}_{j}\big)^2 du d\omega} 
\end{align}
as the (normalized) degree $d$ separability and are interested in investigating how much variation   
can be explained by a $d$-separable approximation.

\subsubsection{Functional canonical coherence} \label{sec224}

In many applications, it is   of interest to measure the degree of ``relationship'' between two functional time series.  To this end, consider the case where the Hilbert space arises as the direct sum of Hilbert spaces, i.e., 
$\Hi = \Hi_1 \oplus \Hi_2$. 
Then, we may write 
$\{X_{t,T}| t=1,\ldots, T, T \in \nnum\}  $ = $ \{(X_{t,T,1}, X_{t,T,2})^\top| t=1,\ldots, T, T \in \nnum\} $, where we refer to $\{X_{t,T,j}| t=1,\ldots, T, T \in \nnum\}$ as the $j$th component process, which takes values in the Hilbert space $\Hi_j$ ($j=1,2$). The spectral density operator can be represented in the form
\begin{align*}
\F = \begin{bmatrix}
\F^{11} & \F^{12}\\
\F^{21} & \F^{22}
\end{bmatrix}
 =\Bigg\{ \begin{bmatrix}
\E \big(Z^{1}_{u,\omega} \otimes Z^{1}_{u,\omega}\big) &\E \big(Z^{1}_{u,\omega}\otimes Z^{2}_{u,\omega}\big)\\
\E \big(Z^{2}_{u,\omega} \otimes  Z^{1}_{u,\omega}\big) & \E \big( Z^{2}_{u,\omega} \otimes  Z^{2}_{u,\omega}\big)
\end{bmatrix}: \omega \in (-\pi, \pi], u \in [0,1] \Bigg\}~,
\end{align*}
where  $ \big \{ ( Z^{1}_{u,\omega} , Z^{2}_{u,\omega} \big ) ^\top \}_{t \in \mathbb{Z}}$
is the right-continuous  orthogonal  increment process in \eqref{eq:tvCram} corresponding to the 
approximating process $ \big  \{  X_t^{(u)}|t\in \znum; u\in [0,1] \big  \}  =
\big \{ ( X_{t,1}^{(u) },  X_{t,2}^{(u)})^\top \big  |t\in \znum; u\in [0,1] \big  \} $.  
The diagonal gives the time-varying component spectra, while the other entries give the time-varying cross-spectra, which satisfy $\F^{ij}_{u,\omega} 
= \E  (Z^{i}_{u,\omega} \otimes Z^{j}_{u,\omega} )
\in S_1(\Hi_j, \Hi_i)$. Let $P^{\Hi}_{i}$, $i=1,2$, denote the orthogonal projection of $\Hi=\Hi_1 \oplus \Hi_2$ onto $\Hi_i$. Then observe that 
 \[
 \F^{ij}_{u,\omega}= P^{\Hi}_i{\F}_{u,\omega} P^{\Hi}_j
 \]
so that $\F^{ij}_{u,\omega}$ is simply the restriction of ${\F}_{u,\omega} $ to $\Hi_i $ and $\Hi_j$, $i,j=1,2$. The cross-spectra capture the co-variation of the component processes, and therefore provide a natural starting point to construct a measure that quantifies the linear association between the two component processes. However, classical formulations of coherency measures (see e.g., \cite{b81}) do not naturally extend to the infinite-dimensional (nonstationary) setting because it then inherently becomes an ill-posed inverse problem. 
We refer to \cite{cup07,Yang2011} for approaches to measure functional canonical correlation in the context of functional data and references therein for related literature in the finite-dimensional setting. 

In what follows, we introduce a notion of time-varying functional canonical coherence. To make this more precise, consider finding functions $g \in \Hi_1$ and $h\in \Hi_2$ such that the co-variation between $ Z^{1}_{u,\omega}$ and 
$ Z^{2}_{u,\omega}$ is maximized, that is,
\[
\sup_{\|g\|_{\Hi_1}=1,\|h\|_{\Hi_2}=1} \Cov\big(\inprod{Z^{1}_{u,\omega}}{g}_{\Hi_1}, \inprod{Z^{2}_{u,\omega}}{h}_{\Hi_2}\big) =
\sup_{\|g\|_{\Hi_1}=1,\|h\|_{\Hi_2}=1} \inprod{\F^{12}_{u,\omega}(h)}{g}_{\Hi_1} = \nu^{u,\omega}_1  \tageq \label{eq:canco1} ~.
\] 
Here  $\nu^{u,\omega}_1$ is the largest singular value of the operator $\F^{12}_{u,\omega}$, since the supremum is attained for the eigenfunctions belonging to the largest eigenvalue of $\F^{11}_{u,\omega}$ and  $\F^{22}_{u,\omega}$, respectively, i.e.,  $g=\phi^{(u,\omega)}_{11,1}$ and $h=\phi^{(u,\omega)}_{22,1}$. This leads to the following notion of \textit{first order canonical coherence at time $u$ and frequency $\omega$}
\[
\mathcal{R}^{u,\omega}_{1} =\frac{\Cov\big(\inprod{Z^{1}_{u,\omega}}{\phi^{(u,\omega)}_{11,1}}_{\Hi_1}, \inprod{Z^{2}_{u,\omega}}{\phi^{(u,\omega)}_{22,1}}_{\Hi_2}\big)}{\sqrt{\Var \big (\inprod{Z^{1}_{u,\omega}}{\phi^{(u,\omega)}_{11,1}}_{\Hi_1} \big ) \Var \big (\inprod{Z^{2}_{u,\omega}}{\phi^{(u,\omega)}_{22,1}}_{\Hi_2} \big )} } = \frac{\nu_1^{u,\omega}}{\sqrt{\lambda^{u,\omega}_{11,1} \lambda^{u,\omega}_{22,1}}}
\]
where $\lambda^{u,\omega}_{ii,1}$ 
is the largest eigenvalue of the operator $\F_{u,\omega}^{ii}$ ($i=1,2$). 

Repeating the maximization argument in \eqref{eq:canco1} on sequences of orthogonal complements, we find that the 
\textit{$d$th order canonical coherence at time $u$ and frequency $\omega$} is given by the variable 
$$
\mathcal{R}^{u,\omega}_{d} = \nu_d^{u,\omega}/\sqrt{\lambda^{u,\omega}_{11,d} \lambda^{u,\omega}_{22,d}}.
$$
We may therefore define as a measure of $d$th order canonical coherence over the  time-frequency interval $[0,1] \times [a,b]$ by  
\[
s_{d} =
\mathcal{T}  ( {\F}) :=  \int_0^1 \int_a^b \mathcal{R}^{u,\omega}_{d} d\omega du~.
\]
\purp{
Typical quantities of interest
 are confidence intervals for $s_d$ or tests  for  \textit{relevant $d$th order coherency on the given set of the time-frequency plane}, that is 
$H_0: s_d \leq \nu $ versus $H_1: s_d > \nu $.
Note that  rejecting $H_0$ allows us to conclude at a controlled type I 
error that the first order canonical coherence exceeds $\nu$.
}

\subsubsection{Measuring deviations from stationarity using  the square root distance}
 \label{sec221}

The hypothesis of second-order stationarity can be described by the property that  the spectral density operator    $
\F =\{\F_{u,\omega}\}_{u\in [0,1], \omega \in [-\pi,\pi]}  $  in \eqref{tvSDO} 
of the approximating process
is independent of $u$. Therefore, it is reasonable to measure deviations from stationarity by the  distance between 
$\F$ and its ``best'' approximation
by a  time-independent  spectral density operator, say $\G =\{\G_{\omega}\}_{\omega \in [-\pi,\pi]}$. For this approach, the choice of 
an appropriate distance is crucial. Several authors have considered  the Hilbert-Schmidt metric for comparing 
covariance structures from two (independent) groups of 
functional data \citep[see for example][]{fremdt2013,Panaretos2010,detkokvol2020},
and   \cite{vDCD18} used this metric to develop a test for stationarity of a functional time series. 
However, as pointed out by \cite{Pigolietal2014}  such an extrinsic approach ignores the geometry
of the space of the space of covariance operators. Therefore, these authors proposed several alternatives metrics
including the largest absolute eigenvalue, the Procrustes metric and  the square root distance.
\cite{Drydenetal2009}
demonstrated that   the latter  distance   produces good results in the finite-dimensional
setting and we will use it here to measure deviations from stationarity.
To be more precise, denote the square root distance between two operators  $A$ and $B$ in ${S_1(\hi)}^+ $ by
\[
d_R(A,B)= \norm{A^{1/2}- B^{1/2}}_{S_2(\hi)}. \tageq \label{eq:DR}
\]
Then we consider  the best approximation of $\F$ in the following sense
\begin{align}
\label{eq:prostat1}
r= r(\F)  =  \min_{\G   } \int_{0}^\pi \int_{0}^1\bignorm{\F^{1/2}_{u,\omega} -\G^{1/2}_{\omega}}^2_{S_2(\Hi)} du d\omega,
\end{align}
 where the minimum
is taken over all functions $\G: [0,\pi] \to S_1(\Hi)^+$.  
The following result (see Section  \ref{sec5}) gives an explicit expression for $r$ in terms of $\F$.

\begin{lemma}\label{optimalprocr}
The minimum in \eqref{eq:prostat1} is attained for 
$\ddot{\F} =\{  \ddot{\F}_{\omega}\}_{\omega \in [0, \pi]} $, where 
$\ddot{\F}_{\omega} :=(\int_0^1 \F^{1/2}_{u,\omega}du)^2$
and
$r$ is given by  
\begin{align}
\nonumber 
r 
= \int_{0}^\pi \int_{0}^1   \Tr(\F_{u,\omega}) du d\omega 
-  \int_{0}^\pi   \Tr \big (  \ddot{\F}_{\omega} \big )   d \omega. 
\label{eq:prostat}
\end{align}
\end{lemma}

\subsection{Deviations from structural assumptions} \label{sec22}

\purp{In the previous section, we expressed  important measures of model deviation involving the second-order dynamics as functionals of the collection of time-varying spectral density operator $\F$ defined in \eqref{tvSDO}. 
In this section and the subsequent section, we discuss several  inference problems related to measures of this type.
}

\purp{To be precise, consider  the  class, say ${\cal P}$, of  locally stationary functional time series models introduced in Section \ref{sec21}. 
Suppose that $X=\{X_{t,T}| t=1,\ldots, T; T \in \nnum\}$ is an element of the class ${\cal P}$, but that 
we want to work under an additional structural assumption. }
We denote by  ${\cal P}^\prime \subset {\cal P}$ the class of  models satisfying this 
 assumption. \purp{A typical example is a finite rank assumption on the operators in \eqref{evd}  made in (dynamic) FPCA, but   several other examples were also given in Section \ref{sec24}. }
We assume that we can measure the deviation between an ``optimal'' model from the class  ${\cal P}^\prime$ and  
the ``true'' model in ${\cal P}$
by a functional of the spectral density operator  defined in \eqref{tvSDO}, say
 \begin{eqnarray} \label{h2} 
  r= 
  \mathcal{T}  ( {\F}) ,~
 \end{eqnarray}
 where $\mathcal{T}$   defines a mapping from  the set of all  time-varying spectral density operators  onto the non-negative real line. Note that $r=0$ if the process $X$ is already an element of the class 
$
{\cal P}^\prime$.

The measure $r$ provides important information about the loss resulting from the simplified model assumption, and we are interested in  quantifying   the (statistical) error if an experimenter decides to work with 
a model from the class  ${\cal P}^\prime$ instead of  from the more general class  ${\cal P} $ (e.g., to reduce computational costs). A  confidence interval for the measure $r$
would be a useful tool to quantify the statistical uncertainty of such a decision.
Alternatively, it  might also be of interest to investigate if the 
model deviation (measured by $r$)  is smaller than a  pre-specified constant, say $\Delta> 0$. In order to make this decision   a test for the hypotheses
\begin{equation} \label{deq14b}
H_{0}: r\leq \Delta    ~~{\rm versus } ~~H_{1}: r  > \Delta     ~
\end{equation}
would be desirable.
Rejection of the  null hypothesis in \eqref{deq14b} allows 
 to decide at a controlled type I error  that 
 the deviation of  models 
 from the   class ${\cal P}^\prime$
 is  larger than the given threshold $\Delta >0  $. This threshold 
thus defines the maximum ``loss'' in precision that one is willing to tolerate if one works with the simplified model from the class
 ${\cal P}^\prime$, and therefore its value depends on the specific application. 
 We also emphasize that we do not discuss the case $\Delta =0$ in this paper. Rather, we consider the situation where the experimenter does not believe that the given process
 is an element of the class 
  ${\cal P}^\prime $, but wants to make this assumption to simplify the statistical analysis without a too large model bias.
 
If an estimator of the spectral density operator, say $\hat {\F} $, is available
(an appropriate estimator will be defined in Section \ref{sec32}), these questions can be addressed estimating 
the measure 
$r$  in \eqref{h2} by  the statistic  $\mathcal{T}  ( \hat {\F}) $.  Uncertainty quantification via this approach requires some 
knowledge of the distribution of $\mathcal{T}  ( \hat {\F}) $, at least in an asymptotic scenario.
It turns out that  in general  the centered statistic $\mathcal{T}  ( \hat {\F}) -\mathcal{T}  ({\F}) $ (standardized appropriately) does not converge weakly.
However, in \autoref{sec3} we derive the asymptotic distribution of the estimator 
$\mathcal{T}_T  ( \hat {\F})  - \mathcal{T}  (  {\F})  $ (standardized appropriately), where $\mathcal{T}_T$ is an approximation of  $\mathcal{T}$, which depends on the sample size.
Although this result provides the first  step in developing  statistical inference methods for the functional $r = \T (\F )$,  the typically 
complicated structure  of the functional $\mathcal T$ (see the examples in Section \ref{sec24}) and 
the nonstationarity of the process 
make a direct application not feasible; the limiting distribution generally 
depends on several nuisance parameters that are impossible to  estimate. In Section \ref{sec3}, we therefore develop statistical theory that  circumvents this issue, and propose the resulting  statistical inference tools in Section \ref{sec4}.

\subsection{How much model complexity is really needed?} \label{sec23}
 
More generally, consider  an increasing  sequence of (possibly simpler) model classes  
\begin{equation} \label{h1}
{\cal P}_{1} \subset {\cal P}_{2}  \subset {\cal P}_{3} \subset  \ldots  \ldots  \subset {\cal P}    
\end{equation}
 approximating the ``true''  model
 from the class ${\cal P}$, and let
   $r_{d} ={\cal T}_d (\F)$, $d \in \mathbb{N}$, denote a  measure  of the form 
 \eqref{h2} for the deviation between the ``optimal'' model  in ${\cal P}_{d} $ 
 and the given locally stationary functional time series  in the class
 ${\cal P}$. 
 In certain applications, we have
 $r_d \in  [0,1]$ such that 
$  s_{d} =1- r_{d}  \in [0,1]$
satisfies
$s_{1} <   s_{2}  < s_{3}  <  \ldots $   and  $\lim_{d \to \infty} s_{d} =1$. Note that $s_{d}$ can be interpreted as a measure for the quality of the approximation of the process
 by models from the class  $ {\cal P}_{d}$, which is $1$ if the process is in fact 
 an  element of the class   ${\cal P}_{d} $
 (in this case the sequence is terminating).
 Typical examples  include dynamic FPCA  \citep{Hormann2015} and principle separable  components \citep{Masak2020}, which {we introduced for the}  nonstationary framework 
 in Section \ref{sec222}  and \ref{sec223}, respectively. 
 Assume that an experimenter decides  to work with a simplified model from  the class
${\cal P}_{d_{0}}$ instead of from ${\cal P}$. As  $s_d$ specifies the quality of approximation of the model from the class  ${\cal P}_d$,  
it might be of interest to investigate if  $s_{d_{0}}$ is larger than
$\nu$, where $\nu \in (0,1)$ is  a pre-specified constant (for example $95\%$). In order to make this decision at a controlled type I error we 
propose to construct a  test for the hypotheses
\begin{equation} \label{deq14}
H_{0}^{(d_{0})}: s_{d_{0}} \leq \nu    ~~{\rm versus } ~~H_{1}^{(d_{0})}: s_{d_{0}}  > \nu    ~.
\end{equation}
For example, in the case of (functional) PCA   the quantity 
$s_{d_{0}}$ could  represent the (relative) amount of variation explained by the first $d_{0}$ eigenvalues, in which case \eqref{deq14} allows to conclude if - at a controlled type I error - the  first $d_{0}$ (functional) principal components 
explain  at least  $100 \cdot \nu \%$  of the total variation.

On the other hand,  it is  often of interest 
to identify  the ``most simple''  class of 
models, say ${\cal P}_{d^{*}}$,  such that 
a model  from  ${\cal P}_{d^{*}}$ is sufficiently  close to the given model
from the class ${\cal P}$.
If the  quality of approximation  of the given model by  models from the class ${\cal P}_{d}$  is quantified by the measure 
$s_{d}$, we can define $d^{*}$ by 
\begin{equation}
\label{dstar}
d^{*} = \min \big \{ d \in \mathbb{N} ~|~s_{d} > v \big  \}  = \min \big  \{ d \in \mathbb{N} ~|~1- s_{d} < 1-  v \big  \}  ~
\end{equation}
as the minimal model complexity such that $s_{d}^* > v$, and an important problem  
 is to  estimate   $d^*$ from the given data. An estimator of $d^*$, say $\hat d$,
 will be discussed in Section \ref{sec42},  and it is of importance 
 to control the probabilities of under- and overestimating $d^*$. Moreover, in the same section 
we also address the problem of investigating if a pre-specified class of models, say ${\cal P}_{d_{0}}$,
satisfies 
$ {\cal P}_{d^{*}} \subseteq  {\cal P}_{d_{0}} $ or  ${\cal P}_{d_{0}} \subsetneqq  {\cal P}_{d^{*}}$
by testing the hypotheses  
 \begin{equation}\label{hx0}
   H_0: d^* \leq d_0 \qquad \qquad H_1: d^* > d_0.
 \end{equation}

\section{Methodology} \label{sec3}

\def\theequation{3.\arabic{equation}}
\setcounter{equation}{0}

 We now explain our general approach and present the main results. In Section \ref{sec31}, we begin by stating 
the technical assumptions on the processes.
\purp{In Section \ref{sec32}, we introduce a sequential estimator of $\F_{u,\omega}$ and - after presenting an illustrative  example in Section \ref{sec32a} - our main results are given in 
Section \ref{sec34}. There, we develop an estimator of  a rather general functional $\T ( \F )$, which has (after an appropriate standardization) an asymptotically
 pivotal distribution. Finally, we demonstrate in Section \ref{sec35} that  
the developed methodology is applicable to  the measures that were introduced in Section \ref{sec24}.}

\subsection{Technical  assumptions}
 \label{sec31}

We consider processes that admit representations of the form 
\[
X_{t,T}=\mathfrak{G}\big(t,T,\mathfrak{f}_t \big),  \tageq\label{eq:bershift}
\]
where $\mathfrak{f}_t =(\epsilon_t, \epsilon_{t-1},\ldots )$ for an i.i.d. sequence $\{\epsilon_t| t\in\mathbb{Z}\}$  of random elements taking values in some measurable space $S$, and where  $\mathfrak{G}:\znum\times \mathbb{N} \times S^{\infty} \to \Hd$ is a measurable function.  The type of mixing conditions we impose are generalizations of those in \cite{vD19} to nonstationary processes, which are of the same nature as the physical dependence measure introduced in \cite{Wu05}. To this end, consider $\mathfrak{f}_{t,\{k\}} =(\epsilon_t, \epsilon_{t-1}, \ldots, \epsilon_{k+1}, \epsilon^{\prime}_{k}, \epsilon_{k-1}, \ldots)$ for some independent copy $\epsilon^\prime_k$ of $\epsilon_k$, and denote the corresponding filtrations by $\G_{t}=\sigma(\mathfrak{f}_t)$ and $\G_{t,\{k\}}=\sigma(\mathfrak{f}_{t,\{k\}})$, respectively.

\begin{assumption} \label{as:depstrucnonstat}
The process $\{{X}_{t,T}\}_{t\in \znum, T\in \mathbb{N}}$ satisfies \autoref{def:locstat}, $\sup_{t,T}\|X_{t,T}\|_{\hi,p} < \infty$ for some $p>0$, and admits the representation \eqref{eq:bershift}. The series' dependence structure is specified through the dependence measure
\begin{align}
\label{eq:as32}
\nu^{X_{\cdot,\cdot}}_{\hi,p}(k,T)=
 \sup_{t \in \znum}\bignorm{X_{t,T}- \E[X_{t,T}|\G_{t,\{k\}}]}_{\Hi,p} 
\end{align}
The dependence structure of the process satisfies 
$\sum_{k=0}^\infty \sup_{T \in \nnum} \nu^{X_{\cdot,\cdot}}_{\hi,p}(k,T)<\infty$
\end{assumption}

If  $X_{t,T} = \mathfrak{G}(\mathfrak{f}_t)$ for some filter $\mathfrak{G}$, the 
process $\{X_{t,T}\}_{t,T} $ is in fact stationary 
and \eqref{eq:as32}  reduces to  $\nu^{X_{\cdot,\cdot}}_{\hi,p}(k)=
\norm{X_{k}- \E[X_{k}|\G_{k,\{0\}}]}_{\Hi,p}$. For the approximating process, we  assume
\begin{assumption} \label{as:depstruc}
{Let $p > 0$}.
The auxiliary process $\{{X}^{(u)}_{t}\}_{t\in \znum, u \in [0,1]}$ is continuous in $u \in [0,1]$ and satisfies $\sup_{u \in [0,1]}\|X^{(u)}_{t}\|_{\Hi ,p} < \infty$. For all $u \in [0,1]$, it admits a representation
 $X^{(u)}_{t}=G\big(u,\mathfrak{f}_t\big)$
 where $G:[0,1] \times S^{\infty} \to \Hd$ is a measurable function. 
The series' dependence structure satisfies 
\begin{itemize}
\item[I)]
$
\sum_{k=0}^{\infty}\nu^{{X}^{\cdot}_{\cdot}}_{\hi,p}(k)< \infty$  where $
\nu^{{X}^{\cdot}_{\cdot}}_{\hi,p}(k)=\sup_{u \in [0,1]}\|{X}^{(u)}_k -\E[{X}^{(u)}_k|\G_{k,\{0\}}] \|_{\hi,p};
$ 
\item[II)]
There exists an orthonormal basis $\{e_\ell\}$ of $\Hi$ such that $\sum_{\ell \in \nnum}\sum_{k=0}^\infty  \nu^{X^{\cdot}_{\cdot}(e_\ell)}_{\cnum,p}(k)<\infty$, 
where 
\begin{align}
   \label{eq:as33} 
\nu^{X^{\cdot}_{\cdot}(e_\ell)}_{\cnum,p}(k) =\sup_{u\in [0,1]}\bignorm{\inprod{e_\ell}{X^{(u)}_{k}}- \E[\inprod{e_\ell}{X^{(u)}_{k}}|\G_{k,\{0\}}]}_{\cnum,p}.
\end{align}
\end{itemize}
\end{assumption}
Note that I) is weaker than II) (see also Remark \ref{rem:S1vsS2}). Once   \autoref{as:depstruc}~I) is satisfied for some $p \ge 2$, then the fact that $\norm{f\otimes g}_{S_1}=\norm{f}_{\Hi}\norm{g}_{\Hi}$, $f,g\in \Hi$ (see \autoref{prop:S1inHprod}), orthogonality of the  projections and Jensen's inequality yield that the lag-covariance operators in \eqref{deq40} satisfy 
\begin{align*}
\sup_{u\in [0,1]}\sum_{h=0}^{\infty} \norm{C^{(u)}_h}_{S_1} 
 \le \big (\sum_{j=0}^{\infty} \nu^{{X}^{\cdot}_{\cdot}}_{\hi,p}(j) \big)^2 <\infty. \tageq \label{eq:tracecov}
\end{align*}
To control the bias of functionals of the partial sum estimator that is introduced below, we furthermore impose the following smoothness conditions.
\begin{assumption}\label{as:smooth}
$\sum_{\ell \in \nnum}\sum_{k=0}^\infty k^{\iota}  \nu^{X^{\cdot}_{\cdot}(e_\ell)}_{\cnum,p}(k)<\infty$, $\iota \ge 2$ and 
for almost every $\omega$, the mapping 
$u\mapsto \F_{u,\omega}$ is twice Fr{\'e}chet differentiable with $\sup_{u}\norm{\frac{\partial^2 }{\partial u^2}\F_{u,\omega}}_{S_1} <\infty$. 
\end{assumption}

\noindent
The first ensures that $\sum_{\ell \in \nnum} \sup_{u\in [0,1]}\sum_{h=0}^{\infty} h^{\iota} \norm{C^{(u)}_h(e_\ell)}_{\Hi} <\infty$, whereas the second is a sufficient condition to control the bias that arises from functionals that involve approximation of a time integral.

\begin{Remark}[{Conditions for weak convergence in $S_1(\Hi)$ versus  $S_2(\Hi)$}]\label{rem:S1vsS2}
{\rm While \autoref{as:depstruc}~I) suffices for weak convergence in the space of Hilbert-Schmidt operators, we require the stronger condition \autoref{as:depstruc}~II) to establish weak convergence in the space of trace class operators.
To grasp the rational behind these assumptions, we iterate back to a result from \cite{Ac70}, who studied CLT's for random variables with values in a Banach space $V$ with 
a Schauder basis of type $p$ \cite[see][Definition 3.1]{Ac70}.  If $\{Z_i\} \in V$ is an iid sequence with $\E\inprod{Z_i}{q}=0$ and $\E|\inprod{Z_i}{q}|^2<\infty$ for each $q\in Q$, where  $Q$ denotes a sequentially $w^\star$-dense subspace of $V^\prime$, then aforementioned author proved that a CLT holds if its covariance form  satisfies \[\sum_{j} \big(\Gamma(q_j,q_j)\big)^{p/2}=\sum_{j}(\E|\inprod{Z}{q_j}|^2)^{p/2}
<\infty
,  \quad \{q_j\} \in Q. \tageq \label{eq:DeAc}\] 
We remark that $S_{\infty}(\Hi)\cong (S_1(\Hi))^\prime$ and that $S_r(\Hi)$ is a Banach space with Schauder basis of type $r=\{1,2\}$, where the Schauder basis can be taken as $\{e_i \otimes e_j\}_{i,j\ge 1}$. We refer to Appendix \ref{prel} for details. For $S_1(\Hi)$,
condition \eqref{eq:DeAc} holds under \autoref{as:depstruc}II) for the covariance form belonging to $\hat{\F}_{u,\omega}$.  Indeed, in this case $\sum_j \norm{\inprod{e_j}{D^{u,\omega}_{0}}}_{\cnum,4}<\infty
$, where $D^{u,\omega}_0=\sum_{l=0}^\infty \big(\E[X_l^{(u)}|\G_{0}]-\E[X_{l}^{(u)}|\G_{-1}]\big) e^{-\im \omega l}$ of which permutations of the 4th order tensors define the limiting  covariance form of $\hat{\F}_{u,\omega}$ (see proof of \autoref{lem:fidis_dis}). It is worth emphasizing that for Gaussian data,  there exists an ONB such that $\E|\inprod{D^{u,\omega}_{0}}{e_j}|^4=O\big ( (\lambda^{u,\omega}_j)^2 \big ) $
, in which case \eqref{eq:DeAc} corresponds to  $\sum_j (\lambda^{u,\omega}_j)^{1/2}<\infty$.
}
\end{Remark}

\subsection{Estimation}
 \label{sec32}

\purp{
As mentioned in Section \ref{sec2}, our approach is based on a sequential 
estimator of the spectral density operator $\F$ defined in \eqref{tvSDO}. 
  To be  precise, 
  for $\eta \in [0,1]$, we 
  define the sequential estimator of  $\F_{u,\omega}$ by
   \begin{align*}
   \tageq \label{eq:Fint}
& \hat{\F}_{u,\omega}(\eta)=\frac{1}{\flo{\eta N}}\Bigg\{\sum_{s=1}^{\flo{\eta N}} \big({X}_{\te{u}+s}  \otimes\Big[ \sum_{t=1}^{\flo{\eta N}}  \tilde{w}^{(\omega)}_{\bf,s,t} {X}_{\te{u}+t}+ 
f(\eta,N) \tilde{w}^{(\omega)}_{\bf,s,\flo{\eta N}+1} {X}_{\te{u}+\flo{\eta N}+1} \Big]  
\\
&+
 f(\eta,N) 
{X}_{\te{u}+\flo{\eta N}+1}\otimes\Big[ \sum_{t=1}^{\flo{\eta N}}  \tilde{w}^{(\omega)}_{\bf,\flo{\eta N}+1,t} {X}_{\te{u}+t}+ \tilde{w}^{(\omega)}_{\bf,\flo{\eta N}+1,\flo{\eta N}+1} {X}_{\te{u}+\flo{\eta N}+1} \Big]  \Bigg\}, 
\end{align*}
where $f(\eta,N)= (\eta N-\flo{\eta N})$,
  $\te{u}=\tu{N}{u}=\flo{u T}-\flo{N/2}$,   $N=N(T)$ defines the neighborhood over which the process is approximately stationary, and where the weights are defined by
\begin{equation}
\label{eq1}
    \tilde{w}^{(\omega)}_{\bf,s,t}=(2\pi)^{-1}w(\bf(s-t)) e^{\im \omega (s-t)}
    \end{equation}
    for some function $w : \mathbb{R} \to \mathbb{R}$ supported on the interval $[-1,1]$, 
    and $\bf=b(N)$ is a bandwidth parameter.
We also remark that if $\flo{\eta N}=0$ then $\hat{\F}_{u,\omega}(\eta) =0 $.
This yields  the process
  \begin{align*} \tageq \label{eq:Fint1}
 \big  \{ \hat{\F} (\eta  ) ~| ~\eta \in I  \big  \} = 
  \big  \{ \hat{\F}_{u,\omega}(\eta ) ~| ~ u \in [0, 1] , \omega \in [0, \pi] , \eta \in I   \big 
  \} ~.
\end{align*} 
Observe that  the last three terms in \eqref{eq:Fint} ensure that  $\{ \hat{\F}_{u,\omega} (\eta) \}_{\eta \in I} \in C_{S_1}  $ for fixed $u \in [0,1]$, $\omega \in [-\pi, \pi]$.
Moreover, for $\eta=1$ we have
\begin{eqnarray}
    \label{hatF}
  \hat{\F}_{u,\omega} &:= &  \hat{\F}_{u,\omega}  (1) = \frac{1}{ N} \sum_{s,t=1}^N   \tilde{w}^{(\omega)}_{\bf,s,t} {X}_{\te{u}+s}  \otimes   {X}_{\te{u}+t}  \\
  \hat \F &:=& \hat \F (1)  =  \label{Fset}
       \{   \hat{\F}_{u,\omega} (1)  ~|~u \in [0,1]~,~
       \omega \in [0,\pi] \} ~.
~,
\end{eqnarray}
and thus these quantities will coincide. Therefore, we will  use both  notations $\hat{\F}_{u,\omega}$, $\hat{\F}_{u,\omega}(1)$ and 
$\hat \F $, $\hat \F (1)  $  simultaneously.\\ {Note that 
  $         \hat{\F} $
is  a Bochner integrable mapping from $[0,1] \times [-\pi,\pi]$ onto $S_1(\Hi)^+$, and recall that 
  $L^1_{S_1} ([0,1] \times [-\pi,\pi])$ denotes the space of all Bochner integrable functions from $[0,1] \times [-\pi,\pi] $ onto $S_1(\Hi)$. Let
$\mathcal{T}: L^1_{S_1} ([0,1] \times [-\pi,\pi]) \to  W$ denote a Fr{\'e}chet differentiable map,  
     where   $W$ is a separable Banach space}. We prove in Section \ref{sec34} that, under appropriate assumptions, the statistic  
   $ \rho_{T}M^{1/2} (\mathcal{T}_T (\hat \F ) - \mathcal{T} (\F) )$ converges weakly to a Gaussian random element, 
   say ${\cal W}$, in $W$. Here,  
   $\mathcal{T}_T (\hat \F )  $ is a discrete approximation 
   of the mapping $\mathcal{T} ( \F ) $, and  
   the normalizing sequence is given by  
   \begin{equation}
 \rho^2_T={ N \bf \over k_f}, 
\label{deq1}
\end{equation}
with $\kappa_f = \int_{-1}^{1}  w^2(x) dx$.
   Thus, if the   distribution of ${\cal W}$ were known, (asymptotically) valid  statistical inference could be based on this 
   distribution. However, in the examples discussed in Section \ref{sec2} (and in most other examples of interest) 
   the covariance operator of the   distribution of ${\cal W}$ has such a complicated structure that it becomes (nearly) impossible to estimate. These difficulties are  caused by the nonstationarity of the functional time series  and by the nature of the deviation measures. \\
   As an alternative, which will yield reliable results for a large class of deviation measures,
   we propose a non-standard self-normalization approach, {which 
   requires the previously introduced sequential version \eqref{eq:Fint} of the estimator \eqref{hatF}.} 
\\
}

\purp{We conclude this section stating the assumptions on the  weight function  and the parameter  $\bf$ in \eqref{eq1} that are required for the development of the asymptotic theory in the subsequent discussion.
\begin{assumption}\label{as:Weights}
$w$ is  an even bounded piecewise continuous function with compact support on $[-1,1]$  with $\lim_{x \to 0} w(x) =1$ and $w(x) - 1=O(x^\iota)$, $\iota\ge2$, as $x \to 0$.
\end{assumption}
\begin{assumption}\label{as:bandwidth}
Let $\zeta=1$. Assume that $N=T^\alpha$,
the bandwidth satisfies $\bf \simeq N^{-\kappa}, 1/(2\iota+1)< \kappa <1$ and the parameters $M$ and $N$ satisfy $N\to \infty$, $M\to \infty$ as $T\to \infty$  such that $M=o(N^{1-\kappa})$ and $N^{1-\kappa}=o(M^3)$. More specifically, we have the following cases
\begin{align*}
\begin{cases}
\alpha <\frac{2}{(2\iota+1)\kappa-1+3} \text{ and } M=o(N^{(2\iota+1)\kappa-1}) &
 \text{ if } \frac{1}{2\iota+1} < \kappa \le \frac{1}{\iota+1}, 
 \\
\alpha <\frac{2}{4-\kappa} & \text{ if } \frac{1}{\iota+1} \le \kappa < 1. 
\end{cases} \label{as:bw1}
\end{align*} 
\end{assumption}
We remark that if one is interested in
a fixed number of time points, then $N=o(T^{2/3})$ and $1/(2\iota+1) <\kappa<1$ suffice. The conditions on $\bf, M$ and $N$ are required to control the bias-variance trade-off under \autoref{as:smooth}, and to control the approximation error of the time integral by  a discrete sum of $M$ terms.  However, these assumptions can be relaxed depending on the involved form of the functionals and if, e.g.,  a higher degree of stationarity is known to hold ($\zeta>1$) or the process has shorter memory ($\iota >2$).  
}

\begin{Remark}[A consistent sequential  local covariance operator] \label{remnew}
{\rm 
Replacing $\tilde{w}^{(\omega)}_{\bf,s,t}$ in \eqref{eq1} with $w_{s,t}=\mathrm{1}_{s=t}$, where $\mathrm{1}_A$ denotes the indicator function of the set
$A$, reduces the statistic \eqref{eq:Fint}  to a sequential estimator of the local covariance operator 
${\C}^{(u)}_{0}$, which is defined in \eqref{deq40}. The method of proof can be straightforwardly adjusted to obtain that the results in this paper hold 
for this estimator under simplified assumptions.
}
\end{Remark}

\subsection{\purp{ Estimating integrals of the largest eigenvalue in time varying functional PCA}}
\label{sec32a}

The main difficulty when analyzing the stochastic properties of an estimator, say  $\T ( \hat \F  ) $,  for the measure of deviation $\T (\F) $ in \eqref{h2} lies in the fact that the estimator
$ \hat \F   = \hat{\F} (1)  = \{   \hat{\F}_{u,\omega} (1) | u\in [0,1], ~
\omega \in [0, \pi]\}$ 
defined in \eqref{eq:Fint} does not converge weakly (after appropriate standardization) as a process
in $(u, \omega)$. This fact is well known, even  
in the case of real-valued and stationary data   \citep[see, for example,][]{LiuWu10}. As a consequence, classical tools of 
asymptotic statistics such as the functional delta method
are not directly applicable to analyze 
the asymptotic properties of the  statistic 
$ 
\T ( \hat \F) - \T (\F) 
$. 
On the other hand,  typical  deviation measures (e.g., those in Section \ref{sec24}) have some more structure, which can be used to derive   convergence results 
for the statistic $ \rho_{T}M^{1/2} (\T_T ( \hat \F) - \T (\F) )$, where the mapping $\T_T$ is a discrete approximation of $\T$, which is  defined in a sophisticated manner. A general mathematical description of this mapping  requires some  notation and assumptions, which will be introduced 
in Section \ref{sec34}. 

In order to illustrate the main idea, we first  explore an example in the context of tvDFPCA (see Section \ref{sec222}) in a little more detail. To be precise, recall the definition of  the measure $s_d$ in \eqref{eq20}. To ease the exposition, we focus on   the numerator of $s_d$ and set $d=1$ (the case $d>1$ is analogous). That is, we focus for a moment on statistical inference
on the quantity 
\begin{equation}
  \label{eq30}  
\int_0^1 \int_{a}^{b}    \lambda^{(u,\omega)}_{1} d\omega du \,,
\end{equation}
where $ \lambda^{(u,\omega)}_{1} $ is the largest eigenvalue in the eigendecomposition \eqref{evd}. As explained in
Section \ref{sec31}, our goal to obtain (asymptotically) pivotal  statistics requires sequential  estimators of the form \eqref{eq:Fint1}. Actually, for technical reasons, we consider the process  $ \{ \eta^x \hat{\F} (\eta  ) ~| ~\eta \in I \}  $ for some $x \ge 1$, which is an estimator of the process $ \{ \eta^x {\F}  ~| ~\eta \in I \}$.
Observing \eqref{evd}, we obtain for \eqref{eq30}
\begin{align} \nonumber
\int_0^1 \int_{a}^{b}  \eta^x  \lambda^{(u,\omega)}_{1} d\omega du  &= 
\int_0^1 \int_{a}^{b}   \Big \langle  \Pi_1^{(u,\omega)} ,  \eta^x \F_{u,\omega} 
 \Big \rangle  d\omega du   \\
 & = {\rm Tr }\Big  ( 
\int_0^1 \int_{a}^{b}    \Pi_1^{(u,\omega)}  (  \eta^x \F_{u,\omega}  ) 
 d\omega du   \Big )  =
 {\rm Tr }\big  ( \mathrm{L}^{a,b}  \circ {\cal G}_{\Pi_1} (\F) (\eta ) \big )~, 
  \label{eq33} 
 \end{align} 
 where $\Pi_1^{(u,\omega)}$ denotes the first eigenprojector of $\F_{(u,\omega)}$,
 $ {\cal G}_{\Pi_1}$ is an operator that maps the (constant) process $\eta \to \F   $
 to a new process $\eta \to  {\cal G}_{\Pi_1}(\F) (\eta)  =\{  {\cal G}_{\Pi_1, u , \omega } (\F) (\eta )  ~|~ u \in  [0,1], \omega  \in [0,\pi] \} $  defined (point-wise) by  
 ${\cal G}_{\Pi_1, u , \omega } (\F) (\eta )  = \Pi_1^{(u,\omega)}  (  \eta^x \F^{(u,\omega)} ) $, and 
 $\mathrm{L}^{a,b} $ maps the process ${\cal G}_{\Pi_1}$ to its integral with respect to time and frequency (the precise definitions  will be given in  
 a more general context below; see specifically equation \eqref{eq7} and \eqref{eq:Lin}).
 
 We can use a similar argument for the process  $\{\eta^{x} \hat \F (\eta) ~|~\eta \in I  \} $ to obtain
 \begin{align*} 
\frac{1}{M}\sum_{u \in \boldsymbol{U}_M}   \int_{a}^{b}  \eta^x  \hat \lambda^{(u,\omega)}_{1}  (\eta ) d\omega  
& = {\rm Tr }\Big  ( 
\frac{1}{M}\sum_{u \in \boldsymbol{U}_M}  \int_{a}^{b}   \hat   \Pi_1^{(u,\omega)}  (  \eta^x \hat  \F_{u,\omega } (\eta ) ) 
 d\omega    \Big )  
 \\&  =
 {\rm Tr }\big  ( \mathrm{L}^{a,b}_{U_M}  \circ {\cal G}_{\Pi_1} (\hat \F) (\eta )  \big )  + o_p\big (\tfrac{1 }{\rho_{T}M^{1/2}} \big ) ~,
 \tageq \label{eq32} 
 \end{align*}
 where
 $\hat \lambda^{(u,\omega)}_{1} (\eta ) $ and   $\hat   \Pi_1^{(u,\omega)} (\eta ) $ denote
 the largest eigenvalue and corresponding  projector of $\hat \F_{u,\omega} (\eta )$,
 respectively, and where 
\[
{U}_M=  \{u_1,\ldots,u_M   \} \tageq \label{eq:midset} 
\]
denotes an increasingly dense set of distinct equidistant points 
$u_1 < u_2 < \ldots <u_M$ in the unit interval $(0,1)$ where $u_1 = \frac{N}{2T}$ and $u_M = 1-\frac{N}{2T}$, and where $M=M(T) \to \infty$ as $T\to \infty$ at an appropriate rate.  $ \mathrm{L}^{a,b}_{U_M} $ defines a Riemann approximation of the mapping $ \mathrm{L}^{a,b}  $ (see equation \eqref{eq:Lin2}).
 Note that  ${\cal G}_{\Pi_1}$ in \eqref{eq32} depends
 on the projector $\Pi_1$ and not on its  estimator  $\hat \Pi_1$.
 We prove later
 that the order of  the error term
 in \eqref{eq32}
 holds   for all  $x\ge 3$ and emphasize  that  \eqref{eq33}  and \eqref{eq32} formally define the mappings
 $$
 \T_T =   {\rm Tr } \circ  \mathrm{L}^{a,b}_{U_M}  \circ {\cal G}_{\Pi_1} ~,~~ \T =   {\rm Tr }\circ   \mathrm{L}^{a,b} \circ {\cal G}_{\Pi_1}
 $$
from the space  $L^1_{C_{S_1}}([0,1] \times [0,\pi] $
 onto the space $C_{\rnum}$.

 In Theorem \ref{thm:Conv}, we prove in a more general framework - which  covers a broad class of applications -  the
  convergence of the process 
 \begin{equation}
    \label{eq32a} 
  \rho_{T} M^{1/2}
  \big\{ 
 \big (\mathrm{L}^{a,b}_{U_M}\circ  {\cal G}_{\Pi_1 }(\hat \F)  \big ) (\eta) - \big (
 \mathrm{L}^{a,b}\circ {\cal G}_{\Pi_1}  ( \F)  \big ) (\eta)  \big\}_{\eta \in I }
  \end{equation}
  to a Gaussian process 
 in the space $C_{S_1}$ of all continuous $S_1(\Hi)$-valued functions on the interval
 $I$. 
 Therefore, observing \eqref{eq33} and \eqref{eq32}
 yields for for $x \ge 3$,
 \begin{equation}
 \label{eq34}
 \Big \{ 
 \rho_T M^{1/2} \Big (
  \frac{1}{M}
\sum_{u \in \boldsymbol{U}_M}  \int_{a}^{b}  \eta^x  \hat \lambda^{(u,\omega)}_{1}  (\eta ) d\omega 
 - 
 \int_0^1 \int_{a}^{b}  \eta^x  \lambda^{(u,\omega)}_{1} d\omega du   \Big ) 
 \Big \}_{\eta \in I }  \stackrel{{\cal D}}{\to} \{ \mathbb{G} (\eta ) \} _{\eta \in I} ~,
 \end{equation}
 where  $ \{ \mathbb{G} (\eta ) \} _{\eta \in I}$ is  a real-valued  Gaussian process. Moreover, it can be shown that the equality 
 (in distribution)  
 \begin{equation}
   \label{eq35}  
 \{ \mathbb{G} (\eta ) \} _{\eta \in I} 
 \stackrel{{\cal D}}{=}\{ \sigma \eta^{x-1}\mathbb{B} (\eta ) \} _{\eta \in I}
  \end{equation}
holds, where $\mathbb{B}$ denotes a standard Brownian motion and $\sigma $ is a non-negative constant. Results of the type \eqref{eq34} will be the main tool to develop pivotal 
 inference tools for the quantity in \eqref{eq30}. 
For example, \eqref{eq34} with $x=3$, equation \eqref{eq35}, and 
an application of the continuous mapping theorem show that 
\begin{equation}
    \label{eq36}
  \frac{
  \frac{1}{M}
\sum\limits_{u \in \boldsymbol{U}_M} \int_{a}^{b}  \hat \lambda^{(u,\omega)}_{1}  (1)d\omega   
 - 
\int_0^1 \int_{a}^{b}   \lambda^{(u,\omega)}_{1} d\omega du }{\Big \{
 \int\limits_I \Big (
  \frac{  \eta^3 }{M}
\sum\limits_{u \in \boldsymbol{U}_M} \int_{a}^{b}    \hat \lambda^{(u,\omega)}_{1}  (\eta ) d\omega   -
   \frac{ \eta^3 }{M} \sum\limits_{u \in \boldsymbol{U}_M} \int_{a}^{b}    \hat \lambda^{(u,\omega)}_{1}  (1 ) d\omega   \Big )^2  d \eta \Big \}^{1/2} }  
  \stackrel{{\cal D}}{\to } \frac{\mathbb{B}(1)}{\big \{ \int_0^1 (\eta^2 ( \mathbb{B} (\eta )  -
\eta\mathbb{B} (1)) ) ^2 d \eta \big \}^{1/2} }~,
\end{equation}
which provides an asymptotically pivotal 
standardization for the estimator of the quantity of interest.
We mention that a similar but more complicated argument holds for the ratio
$s_d$  defined in \eqref{eq20} such that an (asymptotically) pivotal statistic for inference regarding $s_d$ is available  
(see  Section \ref{sec531}).

\subsection{Weak convergence}
 \label{sec34}
 
We now make the statements of the previous section rigorous. 
 We introduce our main theoretical findings in a general context such that a broad class 
 of applications, including those introduced in Section \ref{sec24}, can be treated. The generality of our  approach makes some
 additional notation necessary, which will be carefully introduced now. We refer to the previous subsection for an illustrative example.
  
 Firstly, we require the introduction of the following mappings from the square $[0,1]  \times [0,\pi] $ onto the space $\mathfrak{L}( C_{S_r})$,  the space of bounded linear operators from $C_{S_r}$ onto $C_{S_r}$.{ We denote the corresponding operator norm by $\|\cdot\|_{\infty}$. }

 \begin{assumption}\label{as:mappings}
 Let $1\le r \le \infty$ and consider the mapping $\Upsilon: [0,1]  \times [0,\pi] \to \mathfrak{L}( C_{S_r})$, $
(u, \omega) \mapsto \Upsilon_{u,\omega}$ 
that satisfies
\begin{enumerate}[label=\roman*)]\itemsep1.5ex
\item  $
 \sup_{\eta \in I}\norm{\Upsilon_{u,\omega}(B(\eta))}_{S_r} \le \norm{\Upsilon_{u,\omega}}_{\infty}\sup_{\eta \in I}\norm{B(\eta)}_{S_r}$,  $\forall (u, \omega) \in [0,1] \times [0,\pi]$ and all  $B \in C_{S_r}$;
 \item $\int_{0}^{\pi}\sup_{u \in [0,1]}\norm{\Upsilon_{u,\omega}}^p_{\infty}   d\omega<\infty$;
\item 
$\Upsilon_{\cdot,\omega}: [0,1] \to \mathfrak{L}(S_1)$ is twice Fr{\'e}chet differentiable with $\int_0^\pi\sup_{u\in [0,1]}\norm{\frac{\partial^2}{\partial u^2}\Upsilon_{u,\omega}}_{\mathfrak{L}(S_1)}d\omega<\infty$. If the mapping $\Upsilon $ is a function of $\F$, then we assume that this holds on a set on which the Fr{\'e}chet derivatives of $u  \to \F_{u,\omega}$ are well-defined. 
\end{enumerate}
\end{assumption}
Very roughly speaking,  the  mapping $\Upsilon$ relates the quantity of interest - which is usually integrated
with respect to time and frequency - to the spectral density operator.  In the example of Section \ref{sec32a}, the mapping is given by $\Upsilon_{u, \omega}=(\Pi_1^{(u,\omega)}\widetilde{\otimes} I)$, where 
$\Pi_1^{(u,\omega)}$ is the first eigenprojector of $\F_{u,\omega}$ and $I$ is the identity operator on $S_1(\Hi)$.  Composition with the trace functional yields
${\rm Tr } \big ( (\Pi_1^{(u,\omega)}\widetilde{\otimes} I) (\F_{u,\omega } ) \big ) ={\rm Tr } \big ( \Pi_1^{(u,\omega)} (\F_{u,\omega } ) \big ) = \lambda_1^{(u,\omega)}$ 
(see the first steps in \eqref{eq30}). 
As was done in the aforementioned section, we will use the mapping $\Upsilon$ to define an operator $\G_\Upsilon$, which can be applied to  the process $\hat \F $  of sequential spectral density estimators introduced 
in \eqref{eq:Fint1}. Because  other applications also require functionals of $\hat \F $ such as 
$\hat \F^{1/2} = \{ \hat \F_{u,\omega}^{1/2} | ~u \in [0,1],~\omega \in [0, \pi] \} $, we will
define $\G_\Upsilon$ in a more general form using
   holomorphic functional calculus. 
More specifically, let  $\Omega$ be an open set  in $\cnum$ with smooth boundary  $\partial \Omega$ and let  $A \in  L^1_{S_r(\Hi)^+} ([0,1]\times [0,\pi]\times I)$, where we slightly abuse notation to restrict the codomain to the cone of positive semi-definite elements of the Banach space  $S_r(\Hi)^\dagger$. Now, let $\sigma( A_{u,\omega}(\eta))$ denote the 
spectrum of the  operator  $A_{u,\omega}(\eta) \in  S_r(\Hi)^+$ and 
define 
\begin{equation}
    \label{eq6b}
    \sigma(A) = \cup_ {u,\omega,\eta}\sigma( A_{u,\omega}(\eta))
    \end{equation}
as the ``spectrum'' of $A$. Then we consider the set $\mathbb{D}_r$ of functions of which the spectrum is enclosed in the interior of $\partial \Omega$, that is, 
\begin{equation}
    \label{eq6} 
   \mathbb{D}_r =  \big \{ A \in  L_{S_r(\Hi)^+}([0,1]\times [0,\pi]\times I)
 ~| ~ \sigma(A)  \subset \Omega ~, \text{dist}(\partial \Omega, \sigma(A) )>0 \big \}.
\end{equation}
 Let $D \subset \mathbb{C}$ denote an open set with  $D \supset \overline{\Omega}$
and  $\phi: D \to \cnum$  be a holomorphic function, then the Riesz-Dunford integral of the elements in $\mathbb{D}_r$ is well-defined; we refer to \autoref{HFC} for more details. 
For a mapping $\Upsilon$ satisfying \autoref{as:mappings}, we now define  the mapping 
\begin{equation}
   \label{eq7} 
\G_\Upsilon :
\begin{cases}
~~~~  \mathbb{D}_r 
& \to   {L^1_{C_{S_r}}([0,1]\times [0,\pi])},  \\
~~~~ 
A  & \mapsto \G_\Upsilon  (A ) :
\begin{cases} 
& [0,1]\times [0,\pi] \to  C_{S_r} 
\\
& (u, \omega) \mapsto    \{  {\cal G}_{\Upsilon , u,\omega}( {A} )  (\eta) ~|   ~
\eta \in [0,1] \}  
\end{cases}
\end{cases}
\end{equation}
by 
\begin{equation}
    \label{eq4a}
{\cal G}_{\Upsilon ,u,\omega} ( A ) (\eta)    = \Upsilon_{u,\omega} \big(\eta^{x-1} \phi(\eta{A (\eta )}_{u,\omega})\big)
\end{equation}
where  
$x \ge 1$ is introduced for technical reasons, and where 
 we do not reflect the dependence
of $\G_\Upsilon $ on the function $\phi$  
because it will always be clear from the context.
Note that for  $\phi (z) =z $ and $\Upsilon_{u,\omega} = \Pi_1^{(u,\omega)}$ we obtain the mapping 
considered in Section \ref{sec32a}.

Next,  we formalize the integration and summation in \eqref{eq33} and  \eqref{eq32}, respectively, 
in the general case. 
For this  purpose, recall that  $   L^1_{C_{S_1}}([0,1]\times [0,\pi])$
is the space of Bochner integrable mappings
of the form 
\[
A :
\begin{cases}
[0,1]\times [0,\pi]  & \to C_{S_1} , \\
~~~~~~~  (u,\omega)  & \mapsto 
A_{u,\omega} : 
\begin{cases}
& I \to  S_1(\Hi) \\
& \eta \mapsto 
 A_{u,\omega} (\eta) 
\end{cases}
\end{cases}
\tageq \label{eq2}
\]
We   define  the linear maps
\begin{align*}
& \mathrm{L}^{a,b}_{U_M}:
\begin{cases}
{L^1_{C_{S_1}}([0,1]\times [0,\pi])}
& \to C_{S_1}, \\
~~~~~~~~~~~~~~~~~~
A  & \mapsto 
\mathrm{L}^{a,b}_{U_M}(A): 
\begin{cases}
&  I \to  S_1(\Hi)
\\
& \eta \mapsto \frac{1}{M}\sum_{u \in U_M} \int_a^b  A_{u,\omega}(\eta) d\omega
\end{cases} 
\end{cases}
\tageq \label{eq:Lin}\\
& \mathrm{L}^{a,b}:
\begin{cases}
{L^1_{C_{S_1}}([0,1]\times [0,\pi])}
& \to C_{S_1}, \\
~~~~~~~~~~~~~~~~~~
A  & \mapsto 
\mathrm{L}^{a,b} (A):
\begin{cases}
&  I \to  S_1(\Hi) \\
& \eta \mapsto 
\int_0^1 \int_a^b  A_{u,\omega}(\eta) d\omega du
\end{cases}
\end{cases}
\tageq \label{eq:Lin2}
\end{align*}
In the situation considered in Section \ref{sec32a} ($\phi (z)= z $, $\Upsilon_{u,\omega}=\Pi_1^{(u,\omega)}$) we have 
$$
 \mathrm{L}^{a,b}   \circ {\cal G}_{\Pi_1} (  \F) (\eta )  
= 
\int_0^1  \int_{a}^{b}     \Pi_1^{(u,\omega)}  (  \eta^{x}  \F_{u,\omega} ) 
 d\omega   =
  \int_0^1 \int_{a}^{b}  \eta^x  \lambda^{(u,\omega)}_{1} d\omega du  
  ~, 
 $$
 and a similar representation holds for  $ \mathrm{L}^{a,b}_{U_M}  \circ {\cal G}_{\Pi_1} (\hat \F) (\eta ) $.
  The next theorem, which is the first main result, establishes the  convergence of the 
 processes  of the latter type  in the general context. For the example considered in Section \ref{sec32a}, it justifies the  convergence of the statistic in 
 \eqref{eq32a}.
 
For ease of the exposition, we focus in the following on processes that adhere to \autoref{def:locstat} with  $\zeta = 1$.

\begin{thm} 
\label{thm:Conv} \label{thm:conv_an}
Let $[a,b] \subseteq [0,\pi], a \le b$.  Assume that  \autoref{as:depstrucnonstat}--\ref{as:bandwidth} hold, and consider a mapping $\G_\Upsilon$ of the form \eqref{eq4a} 
that satisfies    \autoref{as:mappings} for some $x\ge 1$. \noindent
\begin{enumerate}
\item[(a)]\textbf{case} $\phi(z)=z$: Assume the above conditions with  $p\ge 6$. Then
\[\Big\{
 \rho_{T} M^{1/2}\Big(
 \big (\mathrm{L}^{a,b}_{U_M}\circ  {\cal G}_\Upsilon (\hat \F)  \big ) (\eta) - \big (
 \mathrm{L}^{a,b}\circ {\cal G}_\Upsilon  ( \F)  \big ) (\eta)\Big)  \Big\}_{\eta \in I } \stackrel{\mathcal{D}}{\Longrightarrow} \Big\{\eta^{x-1}\mathbb{W}_{\mu_{\Upsilon}}(\eta)\Big\}_{\eta \in I}~, 
 \tageq \label{eq:BrlimFu2}
 \]
where $\mathbb{W}_{\mu_{\Upsilon}}$ denotes a Brownian motion on $S_1(\Hi)$ corresponding to a zero-mean Gaussian measure $\mu_\Upsilon$  with covariance operator $\tau_\Upsilon^2 =\int_0^1\Gamma_{\Upsilon}^{(u)} du$,  pseudo-covariance operator $\tilde{\tau}_{\Upsilon}^2=\int_0^1 \Sigma_{\Upsilon}^{(u)} du$, and where  $\Gamma_{\Upsilon}^{(u)}$ and 
$ \Sigma^{(u)}_{\Upsilon} \in S_1(\Hi \otimes \Hi)$ are defined  in \eqref{eq:var}   and  \eqref{eq:psvar}, respectively.
\medskip
\item[(b)]\textbf{case} $\phi(z)\neq z$: assume the above conditions with $p$ such that \eqref{eq:maxdevmain} holds. In addition, assume that  $\F \in \mathbb{D}_1$  
 and that for some $0<\rho<1$,
\begin{align}
\label{zzeq3}
    \sum_{r\in \nnum}\sum_{j=l}^{\infty} \nu^{X_{\cdot}^{\cdot}(e_r)}_{\cnum,p}(j) =O({(l+1)}^{-\rho}).
    \end{align}
Then for analytic functions $\phi(z)$ that satisfy $\phi^\prime(cz) = h(c) \phi^\prime(z)$, where  $c\in [0,1]$ and $h$ is a continuous function with $h(1)=1$,  
\[\Big\{
 \rho_{T} M^{1/2}\Big(
 \big (\mathrm{L}^{a,b}_{U_M}\circ  {\cal G}_\Upsilon (\hat \F)  \big ) (\eta) - \big (
 \mathrm{L}^{a,b}\circ {\cal G}_\Upsilon  ( \F)  \big ) (\eta)\Big)  \Big\}_{\eta \in I } \stackrel{\mathcal{D}}{\Longrightarrow} \Big\{\eta^{x-1} h(\eta) \mathbb{W}_{\mu_{\Upsilon}}(\eta)\Big\}_{\eta \in I}~. \tageq \label{eq:thm31an}
 \]
\end{enumerate}
\end{thm}
\autoref{thm:Conv}(b) relies on a delicate  concentration inequality (\autoref{thm:maxdev}) from which we obtain 
that 
\[
 \sup_{\omega }\sup_{u \in U_M}\sup_{\eta \in I}\bignorm{\eta\big(\hat{\F}_{u,\omega}(\eta)-{\F}_{u,\omega}\big)}_{S_1} =o_p(1),  \quad (T\to \infty).  \tageq\label{eq:maxdevmain}
\]
 To prove \eqref{eq:maxdevmain}, we make use of the mixing condition exhibiting  polynomial decay, where the speed of the decay is controlled by the parameter $\rho$ in condition \eqref{zzeq3}. 
For most reasonable choices of $\kappa$ and $\rho$,  $p = 7$  suffices. We refer to \autoref{thm:maxdevcon} in the supplement, which gives the details for the specific choice of $p$ depending on $\rho$ and $\kappa$.

\begin{Remark}[{Relaxations of conditions in \autoref{thm:Conv}}]
\label{rem:sqrt}~

{\rm  
\begin{itemize}
    \item[(a)] 
  The conditions of   \autoref{thm:Conv} can be relaxed if convergence in $C_{S_2}$ suffices for a given  application. In this case,
\autoref{as:depstruc}II) can be replaced with \autoref{as:depstruc}I) (see also Remark 
 \autoref{rem:S1vsS2}). Moreover,
$p=4+\epsilon$ then suffices in 
\autoref{thm:Conv}(a), whereas 
\autoref{thm:Conv}(b) - which relies on
 \eqref{eq:maxdevmain} -   also requires a larger value for $p$ in this case.
The assumption $p\ge 6$, imposed under (a) to establish the weak invariance principle in $C_{S_1}$, stems from the fact that $S_1(\Hi)$ has type 1 and cotype 2 \citep[][]{tj74}. Consequently, certain  inequalities applicable to sums of iid random variables or martingale differences - such as Burkholder's inequality - do not hold.
However, we obtain an adjusted inequality (\autoref{lem:Burkh}(ii)) that corroborates with the additional assumption  that is required for convergence of functionals  (\autoref{as:depstruc}II)) of the estimator of $\F$ in $S_1$. For technical reasons this inequality relies upon the fact that $p/2$ is an integer. 
We refer to Section \ref{sec:S1ineq} for more details.
\item[(b)] 
In case of compact operators on infinite-dimensional spaces, an open neighborhood $D\subset\cnum$ containing the spectrum, say $\sigma(A)$, of a compact operator $A$ will also contain the origin. Consequently, functions of the form $\phi(z)= z^{x}, x<1, z\in \cnum$ are not analytic on $D$. However, part (b) of \autoref{thm:Conv} then holds for the subspace of finite rank operators.
\end{itemize}
}
\end{Remark}

The following result is then  a consequence of the functional delta method.

 \begin{Corollary} \label{cor:wcFdif2}
Suppose that the conditions of \autoref{thm:Conv}  hold true 
for $k$ mappings 
$\Upsilon_1, \ldots , \Upsilon_k$   that satisfy \autoref{as:mappings}.
Furthermore, let $W$ denote a  separable Banach space and consider $k$  Fr{\'e}chet differentiable mappings ${\Psi}_1, \Psi_2, \ldots, \Psi_k : C_{S_1}  \to C_W$. Then
\begin{align*}\Big\{
&\rho_{T}M^{1/2}\big( \Psi_j  \big ( 
\mathrm{L}^{a,b}_{U_M}\circ \G_{\Upsilon_j} (\hat \F \big ) (\eta)
-  \Psi_j  \big ( 
\mathrm{L}^{a,b}\circ  \G_{\Upsilon_j} (\F \big ) (\eta)
\big)_{j=1,\ldots,k}  \Big\}_{\eta \in I } 
\\&~~~~~~~~~~~~~~~
\stackrel{\mathcal{D}}{\Longrightarrow}
\Big\{\Big(\Psi^{\prime}_{j,(\mathrm{L}^{a,b}\circ  \G_{\Upsilon_j} (\F  ))}(\eta^{x-1} h(\eta)   \mathbb{W}_{\mu_{\Upsilon_j} }(\eta) )\Big)_{j=1,\ldots,k}\Big\}_{\eta \in I }~, 
\end{align*}
where $\Psi^{\prime}_{j,(\mathrm{L}^{a,b}\circ  \G_{\Upsilon_j} (\F  ))}$ denotes the 
Fr{\'e}chet derivatice of \,$\Psi_j$ 
at $\mathrm{L}^{a,b}\circ  \G_{\Upsilon_j}(\F  )$ and 
$\mathbb{W}_{\mu_{\Upsilon_1}} , \ldots , \mathbb{W}_{\mu_{\Upsilon_k}}$ 
are $S_1(\Hi)$-valued Brownian motions corresponding to  a complex Gaussian measure with covariance operator $\tau_{\Upsilon_j}$ and pseudo-covariance operator $\tilde{\tau}_{\Upsilon_j}$ as defined in Theorem \ref{thm:Conv}
$(j=1, \ldots , k)$.
\end{Corollary}

Note that formally 
\begin{align}
    \T_{j,T}  & =  \Psi_j \circ  \mathrm{L}^{a,b}_{U_M}\circ \G_{\Upsilon_j} ~,~~
    \T_j = \Psi_j \circ \mathrm{L}^{a,b} \circ \G_{\Upsilon_j}
    \label{de11}
\end{align}
 define mappings   from {  $L^1_{C_{S_1}}([0,1]\times [0,\pi])$}
 onto $C_W$. In particular  
 $$
  \T_j  (\F) (\eta)   = 
    \Psi_j \Big ( \int_0^1 \int_a^b \Upsilon_{u,\omega} \big(\eta^{x-1}h(\eta) \phi({\F (\eta )}_{u,\omega})\Big)  d \omega d u
    $$
     defines for $\eta=1$ the $j$th measure of deviation, that is
     $ \T_j  (\F) = \T_j  (\F) (1)$, 
     and 
\begin{align}
  \label{zeq11}  
      \T_{j,T}  (\hat \F) (\eta)   = 
    \Psi_j \Big ( \frac{1}{M}\sum_{u \in U_M}  \int_a^b \Upsilon_{u,\omega} \big(\eta^{x-1} h(\eta)\phi({\eta \hat \F (\eta )}_{u,\omega})\Big)  d \omega  
\end{align}
  defines for $\eta=1$  a corresponding  estimator. For example,
 in the situation considered in  Section \ref{sec32a} (with $\Psi_j=$Tr, $\phi (z) = z $, $h(z)=1$ and $\Upsilon_{u,\omega}=
 \Pi_1^{(u, \omega)}$) we obtain 
 $$
   \T_j  (\F) (1)  = \int_0^1\int_a^b\lambda^{(u,\omega)}_{1} d\omega du  
  ~~~ ~\text{ and } ~~
 \T_{j,T}  (\hat \F) (1)  =
\frac{1}{M} \sum_{u \in \boldsymbol{U}_T}   \int_{a}^{b}    \hat \lambda^{(u,\omega)}_{1}   d\omega    + o_p\big (\tfrac{1 }{\rho_{T}M^{1/2}} \big ).
 $$
 Reflecting the notations in \eqref{de11}, Corollary \ref{cor:wcFdif2} reads as
\begin{equation}  
\Big \{ \Big ( M^{1/2} \rho_{T} \big( 
\T_{j,T} (\hat \F   ) (\eta )   
-  \T_j ( \F ) (\eta) \big ) \Big)_{j=1,\ldots,k}  \Big\}_{\eta \in I }  \stackrel{{\cal D}}{\Longrightarrow} \Big\{\Big(\Psi^{\prime}_{j,(\mathrm{L}^{a,b}\circ {\cal G} )}( \eta^{x-1} h( \eta )
\mathbb{W}_{\mu_{\Upsilon_j} }(\eta) )\Big)_{j=1,\ldots,k}\Big\}_{\eta \in I } 
~.
\tageq \label{eq:BrlimFu2a}
\end{equation}
Let $W=\rnum$ or $W=\cnum$. In our applications, there exist functions 
$g_{11}, \ldots , g_{k k}
\in C(I, \rnum_{\ge 0} )$, and a 
matrix  
$\sigma \in \rnum^{k\times k}$ 
such that the limiting process in \eqref{eq:BrlimFu2a} satisfies
\begin{equation}
   \label{eq14}
\Big\{\Big( {\Psi}^{\prime}_{j,(\mathrm{L}^{a,b}\circ {\cal G} )}(h(\eta)\eta^{x-1}\mathbb{W}_{\mu_{\Upsilon_j} }(\eta) )\Big)_{j=1,\ldots,k}\Big\}_{\eta \in I }
\overset{{\cal D}}{=} \Big\{\sigma g(\eta)\mathbb{B}(\eta) \Big\}_{\eta \in I } ~,
\end{equation}
where $  g={\rm diag}  (g_{11}, \ldots , g_{k k} ) $,
the symbol $\overset{{\cal D}}{=} $ denotes equality in distribution, and where
$\mathbb{B}$ is a $k$-dimensional 
vector of independent real- or \textit{proper} complex-valued Brownian motions.
Additionally, in our applications 
the mappings $\mathcal{T}_1, \mathcal{T}_2, \ldots, \mathcal{T}_k$ 
in  \eqref{eq:BrlimFu2a} are ``linear'' in the sense that 
\begin{equation}
    \label{eq15}
  \mathcal{T}_j\big( \{ \eta G(\eta ) | \eta \in I \} \big) =
 \big   \{  f_j(\eta) \mathcal{T}_j 
  (  G (\cdot  ))~ | \eta \in I  \big  \}   ~~,~j=1, \ldots , k~,
\end{equation}
where 
$f_1, \ldots , f_k  \in C(I, \mathbb{R} )$ are known functions (defined by the specific problem) satisfying $f_j(1)=1$  and 
$\{ \eta G (\eta ) | \eta \in I \}$ and $ \{  G(\eta ) | \eta \in I \}$
are processes in 
$L^1_{C_{S_1}}([0,1]\times [0,\pi])$.
Under these conditions, we can construct (asymptotically) pivotal versions of the 
statistics  $\T_{1,T} (\hat \F   ) (1), \ldots , \T_{k,T} (\hat \F   ) (1)$, which serve as estimators of the deviation measures 
$ \T_1 ( \F ) (1) , \ldots , \T_k ( \F ) (1)  $.

More specifically, let $\nu$ denote a measure on the interval $(0,1) $, define the quantities
\begin{equation}
\label{hx6a}
\hat{\mathcal{D}}_j(\eta) := 
\T_{j,T} (\hat \F   ) (\eta )   
-  \T_j ( \F ) (\eta) ~,
\end{equation}
\begin{equation}
\label{hx6}
{V}^2_{i,j}  = \int_0^1
\big( \hat{\mathcal{D}}_i (\eta ) 
- f_i(\eta) \hat{\mathcal{D}}_i (1)   \big )
\big ) \overline{\big(\hat{\mathcal{D}}_j(\eta)-f_j(\eta)\hat{\mathcal{D}}_j(1)
\big )} \nu(d\eta),
\end{equation}
and note that
$$
\hat{\mathcal{D}}_i (\eta ) 
- f_i(\eta) \hat{\mathcal{D}}_i (1)  =
\T_{i,T} (\hat \F   ) (\eta )  - f_i(\eta) 
\T_{i,T} (\hat \F   ) ( 1 ) ~.
$$
 \autoref{cor:wcFdif2} now implies that 
\begin{align}
\label{det301}
\Big \{ \Big ( M^{1/2} \rho_{T} \hat{\mathcal{D}}_j(\eta)  \Big)_{j=1,\ldots,k}  \Big\}_{\eta \in I }  \stackrel{{\cal D}}{\Longrightarrow}
 \Big\{\sigma g(\eta)\mathbb{B}(\eta) \Big\}_{\eta \in I } ~, 
\end{align}
and  an application of 
the continuous mapping theorem gives 
 the following result.
 
\begin{thm} \label{thmmain}
Suppose that the conditions of \autoref{cor:wcFdif2}  hold true and 
that there exists a function $f=(f_1, \ldots , f_k )^\top  \in C(I, \mathbb{R}^{k})$,
a diagonal matrix
$ g={\rm diag}  (g_{11}, \ldots , g_{k k} )
\in C(I, \mathbb{R}^{k \times k}_{> 0} )$,
 and
a positive definite matrix $\sigma \in \rnum^{k \times k}$ such that
\eqref{eq14} and \eqref{eq15} hold.
Assume additionally that the matrix  
$\mathbb{U}  = (\mathbb{U}_{ij}^2)_{i,j=1,\ldots ,k}$ given by 
\begin{align}\label{deq17}
\mathbb{U}_{ij}^2=
\int_0^1 \big(g_{ii}(\eta) \mathbb{B}_i(\eta)- f_i(\eta) g_{ii}(1)\mathbb{B}_i(1)\big) \overline{ \big(  g_{jj}(\eta ) \mathbb{B}_j(\eta)- f_j(\eta) g_{jj}(1)\mathbb{B}_j(1)\big)}\nu(d\eta), 
\end{align}
where $\mathbb{B}_1, \ldots , \mathbb{B}_k$ are the components of the $k$-dimensional Brownian motion in \eqref{eq14}, is non-singular.
Then the random vector $\hat{\mathcal{D}} = (\hat{\mathcal{D}}_1(1) ,\ldots, \hat{\mathcal{D}}_k(1))^\top $ satisfies 
\begin{eqnarray}
  \label{hx6d}
  \hat
{\mathcal{D}}^{\top} \boldsymbol{V}^{-1}_k   \hat {\mathcal{D}}
\stackrel{\mathcal{D}}{\longrightarrow}
 \mathbb{B}(1)^{\top} \purp{ g(1)  }
 \mathbb{U}^{-1}  \purp{ g(1)  }   \mathbb{B}(1)~, 
\end{eqnarray}
where the matrix $\boldsymbol{V}_k = ( {V}_{i,j}^2 )_{i,j=1,\ldots ,k}$ is defined by   \eqref{hx6}.
In particular, for each $j=1, \ldots , k$, 
   $$ 
   \frac{\hat {\cal D}_{j}(1)}{
   {V}_{jj}} = 
   \frac{  \T_{j,T} (\hat \F   ) (1 )   
-  \T_j ( \F ) (1) }{ {V}_{jj}}
\stackrel{\mathcal{D}}{\longrightarrow} {  \purp{ g_{jj}(1)  }
 \mathbb{B}_j(1) \over \mathbb{U}_{jj} }~,
   $$
   where $\mathbb{B}_j$ denote a standard Brownian motion and  ${V}_{jj}$ and  $\mathbb{U}_{jj}$
  are defined by \eqref{hx6} and \eqref{deq17}, respectively.
\end{thm}

\begin{Remark}
\label{rem34}
~~~
\\
{\rm 
(a)  We remark that condition \eqref{eq14} in 
\autoref{thmmain} ensures a self-normalization approach is feasible and will be satisfied if, for any fixed $\eta$, the functional $\mathcal{T}^{\prime}_{j,\eta \F}$ can be expressed as a well-defined element of the topological dual space.  Condition \eqref{eq15} 
ensures that the 
matrix $\boldsymbol{V}_k$ 
converges weakly to
$
\sigma  g(1) \mathbb{U}  g^\top  (1) \sigma^\top
.
$
Because the vector
$  \hat
{\mathcal{D}} $  converges weakly to the vector 
$\sigma g(1)  \mathbb{B} (1)$, the right-hand side of 
\eqref{hx6d} is a pivotal statistic.
 \purp{ We also note that in all our applications we have $g_{jj}(1) =1 $  ($j=1, \ldots , k$). }
\\
(b)  In a concrete application  the functions $f_{ii}$, $g_{ii}$ and the measure $\nu$ are known and the non-singularity of the matrix $\mathbb{U}$ can be checked. 
For illustration, we consider a simple example choosing  $\nu$ as  the Lebesgue measure on
the interval $(0,1) $, $g_{ii} (\eta ) = 1$ and 
$f_{ii}(\eta ) = \eta $ $(i=1,\ldots, k)$.
We can then use the series expansion of the Brownian bridge to write
$$
\mathbb{B}_i (\eta)-\eta \mathbb{B}_i (1) = \sum^\infty_{\ell=1} \frac{\sqrt{2}}{\ell\pi}Z_{i\ell}~~~~~
(i=1,\ldots, k)~, 
$$
where $\{Z_{i\ell}|i=1,\ldots, k, \ell = 1,2,\ldots\}$ are independent standard normally  distributed random variables, from which we  obtain
$$
\mathbb{U}_{ij}^2=\sum^\infty_{\ell=1} \frac{1}{(\ell\pi)^2} Z_{i\ell}Z_{j\ell}
~~~~~(i,j=1, \ldots , k). 
$$
Consequently, the matrix $\mathbb{U}= (\mathbb{U}_{ij}^2)_{i,j=1,\ldots ,k}$ 
defined by \eqref{deq17} is a weighted sum of independent standard normally distributed $k$-dimensional vectors, which is positive definite with probability $1$. 
}
\end{Remark}

\subsection{ \purp{ Pivotal statistics for measures of model deviation}} \label{sec:sec35}

\purp{We now discuss how this general theory applies to the measures of deviations which we introduced in  Section \ref{sec24}.}
We emphasize that casting the stated examples into the form \autoref{thm:Conv}/\autoref{cor:wcFdif2} is by no means straightforward and requires additional theoretical arguments which are provided in
Section \ref{sec5}.

\label{sec35}
\subsubsection{time-varying dynamic FPCA measure}
\label{sec531} 
A sequential estimator for the measure \eqref{eq20} of 
total variation explained by the first $d$ directions is given by\
\[
\hat{s}_d(\eta)=
 \mathcal{T}_T(\hat \F) (\eta)  = 
\frac{\sum_{i=1}^d \frac{1}{M}\sum_{u \in U_M} \int_a^b \hat{\lambda}_i^{(u,\omega)}(\eta)  d\omega}{
\frac{1}{M} \sum_{u \in U_M} \int_a^b \Tr\big(\hat{\F}_{u,\omega}(\eta)) \big)d\omega }~. \tageq \label{eq:pcahat}
\]
In this case,  \autoref{thm:Conv} and \autoref{cor:wcFdif2} are applicable with $k=2$ and $x=3$, and  \eqref{eq14} holds with $h(\eta ) =1 $ and 
$g_{jj} (\eta)  = \eta^2 $.
The mathematical details are worked out in 
Section \ref{secd2}.

\begin{thm} \label{thm:ldFPCA}
Suppose the conditions of \autoref{thm:Conv}(b) hold true and that ${\lambda}_1^{(u,\omega)} > \ldots > {\lambda}_d^{(u,\omega)}>0$ uniformly in $(u,\omega) \in [0,1]\times [0,\pi]$. Then
\begin{align*}\Bigg\{
 M^{1/2}\rho_{T} \eta^3\Big(\frac{\sum_{i=1}^d \frac{1}{M}\sum_{u \in U_M} \int_a^b \hat{\lambda}_i^{(u,\omega)}(\eta)  d\omega}{
\frac{1}{M} \sum_{u \in U_M} \int_a^b \Tr\big(\hat{\F}_{u,\omega}(\eta) \big)d\omega }-
\frac{\sum_{i=1}^d \int_0^1 \int_a^b({\lambda}_i^{(u,\omega)}) d\omega du }{\int_0^1 \int_a^b \Tr\big(\F_{u,\omega}\big)d\omega du } \Big) \Bigg\}_{\eta \in I }
\stackrel{\mathcal{D}}{\Rightarrow} \begin{Bmatrix}\eta^2 \sigma\mathbb{B}(\eta) \end{Bmatrix}_{\eta \in I } \tageq \label{eq:pca_wc}
\end{align*}
for some
$\sigma \geq 0$  and a standard Brownian motion $\mathbb{B}$.
\end{thm}
\autoref{thmmain} applies with $g_{n,n}(\eta) = \eta^2$ and $f_n(\eta) = \eta^3$, yielding an asymptotically pivotal statistic.

 \subsubsection{Time-varying dynamic principal separable component measure}
 \label{sec352}

In order to analyze the 
sequential estimator  
\begin{eqnarray}
  \label{zeq31}
  \hat{s}_d(\eta)=
 \mathcal{T}_T(\hat \F) (\eta)  = 
  \frac{\sum_{j=1}^d \frac{1}{M}\sum_{u \in U_M} \int_a^b(\hat{\delta}_j^{(u,\omega)}(\eta) )^2 d\omega }{\sum_{j=1}^\infty \frac{1}{M}\sum_{u \in U_M} \int_a^b(\hat{\delta}_j^{(u,\omega)}(\eta))2 d\omega }
\end{eqnarray}
for the measure \eqref{eq:sd_SPCA} of total  variation explained by the degree $d$-separable approximation, 
we will use
\autoref{thm:Conv}  and 
\autoref{cor:wcFdif2} with  $k=2$, $g_{n,n}(\eta) =\eta^2$ and $f_n(\eta) = \eta^3$,  and prove the following result in  Section \ref{secd2}.
\begin{thm} \label{thm:lPSCA}
Suppose the conditions of \autoref{thm:Conv}(b) hold true  and that ${\delta}_1^{(u,\omega)} > \ldots > {\delta}_d^{(u,\omega)}>0$ uniformly in $(u,\omega) \in [0,1]\times [0,\pi]$. Then 
\begin{align*}\Bigg\{
 M^{1/2}\rho_{T}\eta^3 \Big(\frac{\sum_{j=1}^d \frac{1}{M}\sum_{u \in U_M} \int_a^b(\hat{\delta}_j^{(u,\omega)}(\eta) )^2 d\omega }{\sum_{j=1}^\infty \frac{1}{M}\sum_{u \in U_M} \int_a^b(\hat{\delta}_j^{(u,\omega)}(\eta) )^2d\omega }-\frac{\sum_{j=1}^d \int_0^1 \int_a^b({\delta}_j^{(u,\omega)}  )^2d\omega du}{\sum_{j=1}^\infty  \int_0^1 \int_a^b({\delta}_j^{(u,\omega)}  )^2d\omega du}\Big) \Bigg\}_{\eta \in I}
\stackrel{\mathcal{D}}{\Longrightarrow} \begin{Bmatrix}{\eta}^2 \sigma\mathbb{B}(\eta) \end{Bmatrix}_{\eta \in I} \tageq \label{eq:sca_wc}
\end{align*}
for some $\sigma \ge 0$  and a standard Brownian motion $\mathbb{B}$.
\end{thm}
If one is interested instead in a non-normalized measure of
deviation from degree $d$ separability, one could take 
\[
\sum_{j=1}^d \int_0^1 \int_a^b({\delta}_j^{(u,\omega)}  )^2d\omega du \quad \text{ and }\quad
\sum_{j=1}^d \frac{1}{M}\sum_{u \in U_M} \int_a^b(\hat{\delta}_j^{(u,\omega)}(\eta) )^2 d\omega ~.
\]
In this  case, \autoref{thm:Conv}  and 
\autoref{cor:wcFdif2} apply
with $k=1$, $g(\eta) = \eta^2$ and $f(\eta) = \eta^3$ to obtain an analogue of \autoref{thm:lPSCA} for the numerator of \eqref{zeq31}. 
Again, pivotal statistics in all considered 
cases are obtained by an application of \autoref{thmmain}. 

\subsubsection{Time-varying  functional canonical coherence measure} 

 As a sequential estimator of a  measure of $d$-th order functional canonical coherence over the  time-frequency interval $[0,1] \times [a,b]$, consider
\[
\hat{s}_{d}(\eta) =
 \mathcal{T}_T(\hat \F) (\eta)  =
\frac{1}{M}\sum_{u \in U_M} \int_a^b \hat{\mathcal{R}}^{u,\omega}_{d}(\eta) d\omega, 
\]
where $(\hat{\mathcal{R}}^{u,\omega}_{d}(\eta))^2= {\big(\hat{\nu}_d^{u,\omega}(\eta))^2} \big /  \big ( {\hat{\lambda}^{(u,\omega)}_{11,d}(\eta) \hat{\lambda}^{(u,\omega)}_{22,d}(\eta)\big )},$
and where $\hat{\nu}_d(\eta) $, $\hat{\lambda}^{(u,\omega)}_{11,d}(\eta)$ $\hat{\lambda}^{(u,\omega)}_{22,d}(\eta)$ 
are the sequential  estimators   of the quantities ${\nu}_d(\eta) $, ${\lambda}^{(u,\omega)}_{11,d}(\eta)$ ${\lambda}^{(u,\omega)}_{22,d}(\eta)$ which were
defined in Section \ref{sec224}.
We prove the following result in Section \ref{secd3}.
\begin{thm} \label{thm:lFC}
Suppose the conditions of \autoref{thm:Conv}(b)  hold true and that $\nu^{u,\omega}_1 >\ldots >\nu^{u,\omega}_d >0 $ and ${\lambda}^{(u,\omega)}_{ii,1}> \ldots, {\lambda}^{(u,\omega)}_{ii,d}>0$, $i \in \{1,2\}$. Then,
\begin{align*}
\Bigg\{
M^{1/2}
 \rho_{T}  \Big(\frac{1}{M}\sum_{u \in U_M} \int_a^b\eta^4 \hat{\mathcal{R}}^{u,\omega}_{d}(\eta) d\omega-  \int_0^1 \int_a^b \eta^4\mathcal{R}^{u,\omega}_{d} d\omega du\Big)\Bigg\}_{\eta \in I}  
\stackrel{\mathcal{D}}{\Longrightarrow} \begin{Bmatrix} \eta^3\sigma\mathbb{B}(\eta) \end{Bmatrix}_{\eta \in I}  \tageq \label{eq:sca_wc}
\end{align*}
for some $\sigma\ge 0$  and a standard Brownian motion $\mathbb{B}$. \end{thm}

\autoref{thmmain}  is applicable to obtain pivotal statistics 
in this case with $g_{n,n}(\eta) = \eta^3$ and $f_n(\eta) = \eta^4$ for all $\eta \in I$.

\begin{Remark}
 {\rm \autoref{thm:ldFPCA}-\autoref{thm:lFC} are given under the assumption of no multiplicity of the eigenvalues. This simplifies the asymptotic derivations but the authors believe that this can be relaxed using a blocking technique as for example given in \cite{Anderson1963}. We leave the details for future work.} 
\end{Remark}
\subsubsection{A measure of deviation from stationarity based on the square root distance}

Distributional convergence relies on a well-defined Fr{\'e}chet derivative, and  Remark \autoref{rem:sqrt} indicates that we therefore need to consider some form of regularization. One option is to impose a small $\gamma >0$, which ensures the regularized operators are invertible.  The other option is to regularize via restriction of the operators to an appropriate subspace. We will focus on the latter as the procedure in \eqref{h1} provides a way to select the appropriate 'truncation parameter' $d$. More specifically, 
let $\F_{u,\omega,d}=\sum_{j=1}^d \Pi^{(u,\omega)}_{j} \F_{u,\omega}$ be the restriction of $\F_{u,\omega}$ to the space spanned by the first $d$ leading functional principal components, and define $\hat{\F}_{u,\omega,
d}(\eta)$ similarly. 
We remark that the results of  Section \ref{sec42} below can furthermore be used for estimating 
 $d^\star$ and testing hypotheses of the form $H_0: d^* \leq d_0 $ or $H_0: d^* > d_0 $, where $d_0$ is a fixed positive integer.

Now, define
\[
r_{d}(\eta) =\int_{0}^\pi \int_{0}^1 d^2_{R}(\ddot{\F}_{\omega,d},\eta \F_{u,\omega,d}) du d\omega 
\quad 
\text{ and } 
\hat{r}_{d}(\eta) = \frac{1}{M}
\sum_{u \in U_M}\int_{0}^\pi d^2_{R}(\hat{\ddot{\F}}_{\omega,d}(\eta), \eta\hat{\F}_{u,\omega,d}(\eta))  d\omega 
\]
where $d_R$ is the square root distance defined in \eqref{eq:DR}, and where
\begin{equation}
\label{de32}
\hat{\ddot{\F}}^{1/2}_{\omega, d}(\eta) =  \frac{1}{M} \sum_{u \in U_M}  (\eta\hat{\F}_{u,\omega,d }(\eta))^{1/2}
\quad \text{ and }  \quad  {\ddot{\F}}^{1/2}_{\omega,d}(\eta)=  \int_0^1 (\eta {\F}_{u,\omega,d})^{1/2} du.
\end{equation}
Then, we show in Section \ref{secsqrmet} that the distributional properties of the test of no relevant deviations from  stationarity can be cast in terms of \autoref{thmmain}. More specifically, we prove that
\begin{thm}\label{thm:procconv}
Under the conditions of \autoref{thm:Conv}(b) 
\begin{align*}
\Big\{M^{1/2} \rho_{T}\eta\Big(\hat{r}_{d}(\eta)-r_{d}(\eta)\Big)\Big\}_{\eta} \stackrel{\mathcal{D}}{\Longrightarrow} \Big\{\sigma \eta\mathbb{B}(\eta)\Big\}_{\eta}~.
\end{align*}
\end{thm}

In this case 
\autoref{thmmain} applies with $k=2$, $g_{n,n}(\eta) = \eta$ and $f_n(\eta) = \eta^2$.

\section{Statistical consequences} \label{sec4}
\def\theequation{4.\arabic{equation}}
\setcounter{equation}{0}

In this section, we discuss several statistical applications  of the developed theory from a general perspective. We focus 
on confidence intervals for the measure of deviation, testing relevant hypotheses and on the  problem of reducing the model complexity when a given model can be approximated by an increasing sequence of ``simpler'' models.  \purp{We emphasize once more that focus is on measures involving the second-order dynamics structure of nonstationary Hilbert space-valued processes, and that
all results of this section are directly applicable to the measures introduced in Section \ref{sec35}. Moreover,  all results  presented in this section also provide  novel  methodology   for the analysis of multivariate nonstationary data.} The proofs of the statements in this section can be found in Section \ref{proofsec4}.

\subsection{Confidence intervals and hypotheses testing}
 \label{sec41}
 
 Recall the scenario discussed in Section \ref{sec22} and assume that one decides to work with the class of models ${\cal P}^\prime $  instead of ${\cal P}$.
 It is of interest to derive a confidence interval  for 
 the quantity $r$  in \eqref{h2} which  measures the  deviation of the class 
$ {\cal P}^\prime $ from 
the given model in ${\cal P}$. Let $\hat r = {\cal T}_{1,T} (\hat \F) (1) $ 
denote the estimator of $r = \T_1 (\F) = \T_1 (\F)  (1) $ defined in Section \ref{sec32}, then it follows from the second part of 
 Theorem \ref{thmmain} 
(we assume throughout this section that its   assumptions are satisfied)
  that 
\begin{equation}  \label{deq12a}
  \frac{\hat{r} - r}{{ { V}}} 
  =
    \frac{ {\cal T}_{1,T} (\hat \F) (1) 
    - \T_1 (\F) 
    }{{{ V}_{11}}} 
  \stackrel{\mathcal{D}}{\longrightarrow}
 \mathbb{T}  = 
   {\purp{ g(1) }\mathbb{B} (1) \over  
\big (   \int_0^1 \big |  g(\eta ) \mathbb{B} (\eta)-
 f(\eta ) g(1) \mathbb{B}(1) \big |^2
\nu(d\eta)
\big )^{1/2}}~,
 \end{equation}
 where $ { V} =  { V}_{11} $
 is defined in \eqref{hx6}, 
$ \mathbb{B} $ denotes a standard Brownian motion and where
 the functions  $f=f_{1}$,  $g=g_{11}$ and the measure $\nu $ are known, and   depend on the concrete application  (see Section \ref{sec35} for some examples).
Consequently, the distribution of the statistic $\mathbb{T}$ on the right-hand side of \eqref{deq12a} 
is pivotal and its quantiles can be readily  simulated.

\begin{thm} \label{thm1}
For $\beta \in (0,1)$ let $q_{\beta}$ denote the $\beta$-quantile of the distribution of the random variable $ \mathbb{T}  $ defined in \eqref{deq12a}.
If  $r = \T (\F) >0 $  and the assumptions of  \autoref{thmmain} are satisfied, then the interval
\begin{align} 
	\hat I_{T}
&=
	\left[ \hat r + q_{\alpha/2} { { V}}  \, ,
		 \hat r + q_{1-\alpha/2} { { V}} \right]~,
\label{deq13}
\end{align}
defines an  asymptotic $(1-\alpha)$-confidence interval for the measure $r $.
\end{thm}

In addition, it might 
be of interest to investigate if the model deviation measure $r$ is smaller than a given threshold
$\Delta$. In order to make this decision at a controlled type I error,  we 
propose to construct a  test for the relevant hypotheses in \eqref{deq14b}. To this end,
we use  the duality between confidence intervals and tests and  \autoref{thm1}.
To be precise,  similar arguments as given in the proof of \autoref{thm1}  show that 
  the interval
 $
{\hat {{I}}_{T} }= [ q_{ \alpha} { { V}}  +  \hat r ~ ,  \infty)
 $
 defines a one-sided asymptotic ($1-\alpha$)-confidence interval for  the measure $r  $. 
 Now, by the duality between confidence intervals and tests, 
 an asymptotic
 level $\alpha$-test for the hypotheses in \eqref{deq14} is obtained by rejecting the null hypothesis  whenever
 $
 [0, \Delta   ] \cap {\hat {{I}}_{T} }  =   [0, \Delta   ] \cap  [ q_{ \alpha} \hat  { { V}} +  \hat r ~ ,  \infty)  = \emptyset .
 $
This  is   equivalent to rejecting  whenever 
	\begin{equation} \label{testrel}
	  \hat r >  	 \Delta    - q_{\alpha} { { V}}  =  \Delta    + q_{1- \alpha} { { V}}     ~,
	\end{equation}
	and the following result describes the asymptotic properties of this decision rule.

\begin{thm} \label{thm2}
Suppose that the assumptions of Theorem \ref{thmmain} hold. The test
\eqref{testrel}   is  a consistent asymptotic level-$\alpha$ test for the hypotheses
\eqref{deq14b} with $\Delta >0$.
\end{thm}

 \purp{ We conclude this section with the important observation that the assumptions $\Delta >0$ and $r>0 $  are crucial for the statement in Theorem \ref{thm2}. 
First note that  $\Delta =0$  implies $r=0$ (but not vice versa). If $r=0$,  the 
 asymptotic distribution (under the null hypothesis) of the numerator and denominator on the left hand side   of \eqref{deq12a} is a Dirac measure at the point $0$. More precisely, in this case we have  
 $ M^{1\over 2} \rho_T \big (
 {\cal T}_{1,T} (\hat \F) (1) 
    - \T_1 (\F) \big )     \stackrel{\mathcal{D}}{\longrightarrow}
    \sigma_{11}  g_{11}(1) \mathbb{B}(1) $, 
and    $M^{1\over 2} \rho_T V_{11} \stackrel{\mathcal{D}}{\longrightarrow} 
\sigma_{11} \big (   \int_0^1 \big |  g(\eta ) \mathbb{B} (\eta)-
 f(\eta ) g(1) \mathbb{B}(1) \big |^2
\nu(d\eta)
\big )^{1/2}$, 
    where $\sigma_{jj}=0$. 
   Consequently,   canceling  the factor $\sigma_{jj}=0$ in \eqref{deq12a} is not possible.  
In many cases  both terms do converge weakly, but with a different normalization. The ratio is then asymptotically not  distribution-free.}

\subsection{Finding the appropriate degree of model complexity}
 \label{sec42}
 
 Consider   an increasing  sequence of  model classes
${\cal P}_{1} \subset {\cal P}_{2}  \subset {\cal P}_{3} \subset  \ldots  \ldots  \subset {\cal P}  $
as introduced in Section \ref{sec23}. We assume that, for each $d \in \mathbb{N}$, there exists a functional ${\cal T}_d$ 
such that $s_d = {\cal T}_d ( \F )$  measures  the quality of the approximation of the
locally stationary functional time series by an ``optimal''  model from the class $ {\cal P}_{d}$   (again, we assume 
$s_{1} <   s_{2}  < s_{3}  <  \ldots $   and  $\lim_{d \to \infty} s_{d} =1$, and the
case $s_d=1$, for some $d$, is interpreted as a terminating sequence, i.e., ${\cal P}_{d} = {\cal P}$).

Assume that $s_d= {\cal T}_{d} (  \F )= {\cal T}_{d} (  \F )(1) $ for some functional 
${\cal T}_{d} $, and 
define by  $ \hat s_d = {\cal T}_{d,T} (  \hat \F ) (1)$ the corresponding estimator introduced in Section  \ref{sec32}. As a consequence of the second part of Theorem \ref{thmmain},
we obtain for $d=1, 2, \ldots  ~$  the  convergence 
 \begin{equation}  \label{deq12}
  \frac{\hat{s}_d - s_d}{{ { V}}_{d,d}}  
  =  \frac{  \T_{d,T} (\hat \F   ) (1 )   
-  \T_d ( \F ) (1) }{ {V}_{dd}}
\stackrel{\mathcal{D}}{\longrightarrow}
 \mathbb{T}_d  = 
   {  \mathbb{B}_{d} (1) \over  
\big (  
\int_0^1 \big | f_d(\eta)   \mathbb{B}_d\eta)- g_{dd} (1)  \mathbb{B}_d(1)\big |^2  \nu(d\eta)
\big )^{1/2}}~, 
 \end{equation}
where $ \mathbb{B}_{d} $ denotes a standard Brownian motion and
${V}_{dd}$ is defined in \eqref{hx6}. We note that the limiting distribution in \eqref{deq12} depends   only  on  $d$ through the functions $f_d$, $g_{dd}$  and the measure $\nu$ (which are known),
and that the random variables $\mathbb{T}_1,\mathbb{T}_2, \ldots $ are
not necessarily independent.

A confidence interval 
 for $s_{d_0}$
can be obtained along the lines of Section \ref{sec41}.
For example, in the context of  manifold  data, computational costs and storage restrictions often forces  the researcher to work with separable ($d_{0} =1$) or
$2$-separable ($d_{0} =2$) approximations. In this case, the asymptotic $(1- \alpha)$-confidence interval 
\[	\left[ \hat s_{d_0} + q_{\alpha/2} {{ V}_{d_0,d_0}}  \, ,
		 \hat s_{d_0}  + q_{1-\alpha/2} {{ V}_{d_0,d_0}} \right]\]
		 for $s_{d_0} >0$
 can be used to make a statement regarding the quality of approximation with a 
statistical guarantee. 

Although there are many applications where a natural  choice of $d_{0}$ is possible such that the inference tools developed in Section \ref{sec41} and in the previous paragraph can be applied, there  also exist numerous situations 
where one is interested in  identifying a  ``simplest'' 
model. We discuss statistical inference for this type of problem in the following two subsections.
For this purpose, recall the definition of $d^{*} = \min \big \{ d \in \mathbb{N} |~s_{d} > v \big  \}$ in \eqref{dstar}, which corresponds to the ``most simple'' model class 
${\cal P}_{d^{*}}$ of which the approximation 
quality  for  the  given  model  in the class ${\cal P}$ is better than  a given threshold 
$\nu \in [0,1]$.

 \subsubsection{Estimating  the quantity  $d^{*}$ in \eqref{dstar}} \label{sec421}
 
  For the construction of an estimator of $d^{*}$ let 
 $$
 \hat{\mathbb{T}}_d = \frac{\hat{s}_d - \nu}{{{ V}}_{d,d}} = \frac{\hat{s}_d - s_d}{{{ V}}_{d,d}} + \frac{s_d - \nu}{{{ V}}_{d,d}}
 $$
denote the test statistic for the hypotheses in \eqref{deq14} with $d_{0}=d$ (note that  $H_{0}^{(d)}: s_{d} \leq \nu$ is rejected if
$\hat{\mathbb{T}}_d > q_{1-\alpha }$; see \eqref{testrel}). 
It follows from \eqref{eq:BrlimFu2a}, 
\eqref{hx6a},  \eqref{hx6}, the continuous mapping theorem and \eqref{deq1}  that
 \begin{equation} \label{deq19}
{{ V}}_{d,d}=O_{\mathbb{P}}(M^{-1/2}\rho_{T}^{-1} )=O_{\mathbb{P}} \big ((M N\bf)^{-1/2} \big ).
\end{equation} 
Consequently, by \eqref{deq12} we have 
\begin{equation}\label{hx1}
  \hat{\mathbb{T}}_d \  { \stackrel{\mathcal{D}}{\longrightarrow} } \ \mathbb{T}_d \quad \mbox { if } \quad
d=d^*, \   s_{d^*} = \nu~,
\end{equation}
\begin{equation}\label{hx2}
  \hat{\mathbb{T}}_d \  { \stackrel{\mathbb{P}}{\rightarrow} } \
    - \infty ~~~ \mbox{if}  ~d < d^* ,~
    \hat{\mathbb{T}}_d \  { \stackrel{\mathbb{P}}{\rightarrow} } \
    + \infty ~~~ \mbox{if} ~ d > d^* ,~ \hat{\mathbb{T}}_d \  { \stackrel{\mathbb{P}}{\rightarrow} } \
    + \infty ~~~ \mbox{if} ~  d=d^* \mbox{ and } ~s_{d^*}> \nu . 
\end{equation}
We now define
\begin{equation}\label{hx3}
  \hat{d} = \min  \big \{ d \mid \hat{\mathbb{T}}_d > q_\alpha  \big \}
\end{equation}
as an estimator of $ d^{*}$,  where $q_\alpha$ is the $\alpha$-quantile of the distribution of $\mathbb{T}_d$ in \eqref{deq12}.
The following result provides consistency 
of the estimator $\hat d$.

\begin{thm} \label{thm3} 
 If the assumptions of  \autoref{thmmain} are satisfied, we have for the estimator in \eqref{hx3}
\begin{eqnarray*}
&&  \lim_{T \to \infty}  \mathbb{P} \big (\hat{d} < d^* \big ) = 0  ~~\text{and } ~~  \lim_{T \to \infty}  \mathbb{P} \big (\hat{d} > d^* \big ) \leq  \alpha .
\end{eqnarray*}
In particular, if $\alpha = \alpha_T$ in \eqref{hx3} depends on $T$ such that $\alpha_T \to 0 $ as $T \to \infty $, we have
$$
\lim_{T \to \infty}  \mathbb{P} \big (\hat{d} \not =  d^* \big )  = 0 ~.
$$
\end{thm}

 \subsubsection{The null hypothesis of choosing a too ``large''  model}  \label{sec422}

Assume that one has decided for a specific
class of models, say   ${\cal P}_{d_{0}}$ and one is interested in the question if
a model from this class  already achieves the desired accuracy. 
This problem  can be addressed by testing the hypotheses  in \eqref{hx0}, that  is, $H_0: d^* \leq d_0 $ vs. $H_1: d^* > d_0$.  A natural decision rule is obtained by rejecting $H_0$
 if 
 the estimator $\hat d$ 
 is larger than $d_0$,
 and the 
following result shows that this decision rule defines a reasonable test.

\begin{thm} \label{thm4} 
If the assumptions of \autoref{thmmain} are satisfied then  the test that 
 rejects the null hypothesis
 $H_0: d^* \leq d_0 $
 in \eqref{hx0} 
 whenever $ \hat d > d_0$, has asymptotic level $\alpha$ and is consistent for the hypotheses in \eqref{hx0}.
\end{thm}

 \subsubsection{The null hypothesis of choosing a too   ``small''  model} 
 \label{sec423}
 
 We conclude this section discussing tests for the hypotheses 
        \begin{equation}\label{1.5}
             H_0 : d^* > d_0 \quad \mbox { versus } \quad H_1 : d^* \leq d_0~,
           \end{equation}
           which are obtained from \eqref{hx0} by interchanging the null hypothesis and the alternative.
         Note that this formulation allows us to conclude at a controlled type I error that  $s_{d_0} \geq s_{d^*} > \nu $. 
       For example, in the context of (time-varying dynamic  functional)  PCA a decision in favor of $H_{1}$ means that if one works with 
 $d_{0}$ principal components then at least $100 \cdot \nu \% $ of the total variance is explained, and the probability of an error of such a decision is at most $\alpha$.
 
 Unfortunately, the solution for the testing problem  \eqref{1.5} cannot be directly obtained from  the corresponding test for the hypotheses
 \eqref{hx0}. To see this, note that 
 a natural decision rule is to reject the null hypothesis in \eqref{1.5}  whenever  the estimator in \eqref{hx3} satisfies 
\begin{equation}\label{deq16a}
  \hat{d} \leq d_0.
\end{equation}
However, in this case the probability of a type I error is  asymptotically independent of the quantile $q_\alpha$ used in the definition of 
the estimator \eqref{hx3} since
$ \mathbb{P}_{H_0} \big (\hat{d} \leq d_0  \big )  = \mathbb{P}_{d^* > d_0} ( \cup^{d_0}_{d=1}  \{ \hat{\mathbb{T}}_d > q_\alpha  \} )=  1 - \mathbb{P}_{d^* > d_0} \big ( \cap^{d_0}_{d=1} \{ \hat{\mathbb{T}}_d \leq q_\alpha  \} \big) \rightarrow 0~,$
where the convergence follows from the fact that, by \eqref{hx2}, we have $\hat{\mathbb{T}}_d {\stackrel{\mathbb{P}}{\longrightarrow}}  -  \infty$ for all $d=1,\ldots,d^*-1$ (note that under the null hypothesis $d^*-1 \geq d_0$). In particular,  this property does not depend on the  choice of $q_\alpha$ in the definition of the
estimator \eqref{hx3}.  A similar calculation also shows that
$ \mathbb{P}_{H_1} (\hat{d} \leq d_0) = \mathbb{P}_{d^* \leq d_0} (\hat{d} \leq d_0) \rightarrow  1 - \alpha$   if  $d_0 = d^*$, $ s_{d^*}=\nu$ and 
$\mathbb{P}_{H_1} (\hat{d} \leq d_0)  \rightarrow  1$ if $ d^* \leq d_0$, $ s_{d^*}>\nu$. 
 In other words, one cannot choose $q_\alpha$ in \eqref{hx3} to control the type I error, and   additionally, the test is not consistent against all alternatives.

However,  a test for the hypotheses  in \eqref{1.5} can easily be obtained. Indeed, observe that these are in fact equivalent to the hypotheses in \eqref{deq14}. 
Therefore, it follows from \eqref{deq12} and the discussion in Section \ref{sec41} that
the decision rule, which rejects the null hypothesis in 
 \eqref{1.5} (or equivalently in \eqref{deq14}) whenever
 \begin{equation}  \label{deq18}
  \hat{s}_{d_0} > \nu  + q_{1-\alpha,d_0}  V_{{d_0},{d_0}},
 \end{equation}
 defines a reasonable test, where we recall that $q_{1-\alpha,d_0} $ denotes the $(1-\alpha)$-quantile of the distribution of the random variable $\mathbb{T}_{d_0}$ defined in \eqref{deq12}.  
      \begin{thm}  \label{thm5}  If the assumptions of \autoref{thmmain} are satisfied then the decision rule \eqref{deq18} defines an asymptotic and consistent level $\alpha$-test for the hypotheses \eqref{1.5} and \eqref{deq14}.
      \end{thm}
\purp{
Our approach can be used to quantify the coverage of ad-hoc procedures to select a model. 
To be precise,  consider  the following  ad-hoc rule: a
model from the class ${\cal P}_{d_0}$ already provides a good approximation of ${\cal P}_{d}$, that is $s_{d_0} > \nu $,   whenever $\hat s_{d_{0}}  > \nu $. Under suitable assumptions it 
follows from  \eqref{det301} 
that $ M^{1/2}\rho_{T} (\hat s_{d_{0}} -  s_{d_{0}})$ converges weakly to a centered normal distribution with variance $\sigma^2 g^2(1) >0$.  Therefore we obtain for the probability   that $d_0$ components are sufficient
\begin{eqnarray}
\label{detrep}~~~~~~~~~~~~~~~~~
  \mathbb{P}_{s_{d_0}  }
  \big ( \hat s_{d_{0}}  > \nu  \big )  
  = \mathbb{P}_{s_{d_0}  }
  \big ( M^{1/2}\rho_{T} (\hat s_{d_{0}} -  s_{d_{0}}) > 
  M^{1/2}\rho_{T} ( \nu -  s_{d_{0}}) \big )   \approx & 1-  \Phi  \Big (  {M^{1/2}\rho_{T} ( \nu -  s_{d_{0}})  \over \sigma g (1) }  \Big ) ~, 
\end{eqnarray}
where $\Phi$ is the cdf of the standard normal distribution. Thus,
if $\sigma^2$ can be estimated well, we can quantify this probability in dependence of  
the ``true" ratio 
$s_{d_0} $. 
However this variance  depends in a very complicated way on the underlying process (in particular on
the time varying spectral density operator), which makes its estimation difficult. Nevertheless, from \eqref{detrep} we obtain qualitative statements of the form 
$
\mathbb{P}_{s_{d_0}  }
  \big ( \hat s_{d_{0}}  > \nu  \big ) =1$, ${1\over 2}  $ and $0$ if  $ s_{d_0} > \nu$, $ s_{d_0} = \nu$ and $ s_{d_0} < \nu$, respectively.
  Moreover,  similar arguments as given in the proof of Theorem \ref{thm5} show that the self-normalization approach also allows quantitative statements about this probability. 
  }

\section{Finite sample performance}
\label{sec5new}
\def\theequation{ 5.\arabic{equation}}
\setcounter{equation}{0}
\purp{
To illustrate the finite sample sample performance, we provide results from a simulation study in which we focus on testing the hypotheses in \eqref{deq14} in the case of time-varying FPCA (see also \autoref{thm5}). 
In order to construct the estimator \eqref{eq:Fint}, we require specification of bandwidths $\bf$ and $N_T$, which induce an inherent trade-off between resolution in time and frequency direction. Given we are applying this to normalized cumulative local eigenvalue estimation, extra care needs to be taken. Indeed, it is well-known issue that -- even in the most simple settings of eigenvalue estimation -- the dispersion of empirical eigenvalues tends to be larger than that of the population counterparts due to the largest eigenvalues being biased upwards and the smaller ones biased downwards \citep[][]{Jolliffe}. In our context, this can lead to more sensitivity to the bandwidth choice and a data-driven bandwidth selection and bias reduction is therefore desirable, which is presented here. 
\\
First, we propose a data-adaptive procedure in spirit of the MV selection method \citep[][]{polromwol} to select the pair of  bandwidths. Second, given the optimal choice of bandwidths based on this procedure, we propose a jacknife estimator of the normalized cumulative empirical eigenvalues and show why this reduces the finite-sample bias. We explain the selection method first and then explain the jacknife in this context theoretically.\\   \textbf{Data-driven local eigenvalue estimation.} 
In order to choose the bandwidth that is appropriate for normalized eigenvalues we take a grid of block sizes and frequency bandwidths within the theoretically justifiable range under which \autoref{thmmain} holds, say, 
\[
N_1 < N_2 < \ldots < N_K \quad b_{f,1} < \ldots < b_{f,L} \tageq \label{eq:bwpairs}
\]
For each pair, we determine the estimate   $\hat{\F}^{(N_m,b_{f_l})}_{u_{i_m},\omega_j }$ in  \eqref{eq:Fint}, where we now reflect the dependence on $N_m,b_{f_l}$ in our notation, where 
$\omega_j = \frac{j\pi}{100}$ $j=1,\ldots, J$, and where $u_{i,m}$ is an element of the set $U_{{T/N_m}}$  in \eqref{eq:midset} determined by the block size $N_m$.  Denoting  the $k$th largest eigenvalue by $\lambda^{u_{i}, \omega_j}_{k,(N_m,b_{f_l})}$, we compute  for each $d =1,\ldots d_{max}$ the quantity 
\[
g_\F(l,m,d_0) = \max_{i_m,j} \frac{\sum_{k=1}^{d_0} \lambda^{u_{i_m}, \omega_j}_{k,(N_m,b_{f_l})}}{ \Tr\big(\hat{\F}_{(N_m,b_{f_l}}^{(u_{i_m},\omega_j)} \big)}
\]
and choose the pair $(N_m, b_{f_l})$ for which this is minimized over a local neighborhood of bandwidth pairs, that is,
\[
(N^\star, b^\star)= \argmin_{l,m} \, \text{se}\big\{g_\F(l+h_1,m+h_2,d)_{h_1,h_2=-5}^5\big\}.
\]
Note that here we are thus determining the optimal bandwidth pair by minimizing the variation over possible pairs of the normalized cumulative  eigenvalues \textit{after} maximizing on the time frequency grid. 
\\
To reduce the aforementioned bias we propose the following jacknife-version of \eqref{eq:pcahat}
\[
\hat{s}^{\text{Jack}}_{d}(\eta) = 
\frac{ \frac{1}{ M^\star}\sum_{u \in U_{ M^\star}} \int_a^b \sum_{k=1}^{d} 2 \hat{\lambda}_{k,(\frac{N^\star}{\sqrt{2}}, \frac{b^\star}{\sqrt{2}})}^{(u,\omega)}(\eta)  d\omega- \frac{1}{M^\star}\sum_{u \in U_{M^\star}} \int_a^b \sum_{k=1}^{d} \hat{\lambda}_{k,(N^\star, b^\star)}^{(u,\omega)}(\eta)  d\omega}{
\frac{1}{M^\star}\sum_{u \in U_{ M^\star}} \int_a^b \Tr\big(2\hat{\F}^{u,\omega}_{(\frac{N^\star}{\sqrt{2}}, \frac{b^\star}{\sqrt{2}})}(\eta) \big)d\omega -\frac{1}{M^\star}\sum_{u \in U_{ M^\star}} \int_a^b \Tr\big(\hat{\F}^{u,\omega}_{(N^\star, b^\star)}(\eta) \big)d\omega }~.
\]
This estimator has smaller bias than $\hat{s}_d(\eta)$, which follows from the following proposition,  \purp{which is proved in Section \ref{proofsec4}}.
\begin{proposition}
\label{propbias}
Define the estimator $
\hat{\F}^{\text{Jack}}_{u,\omega}(\eta)= 2\hat{\F}_{(u,\omega)}^{(N/\sqrt{2},\bf/{\sqrt{2}})}(\eta)-\hat{\F}_{(u,\omega)}^{(N,\bf)}(\eta)$
\begin{align*} \bigsnorm{\,\E\hat{\F}^{\text{Jack}}_{u,\omega}(\eta)-\F_{(u,\omega)}(\eta)}_{S_r} \le \bigsnorm{\,\E\hat{\F}_{u,\omega}(\eta)-\F_{(u,\omega)}(\eta)}_{S_r}
\end{align*}
\end{proposition}
\textbf{Simulations}. 
In all scenarios, the empirical rejection probabilities (ERP) are calculated over 500 repetitions. Because of the computational complexity of the proposed procedure, simulations were implements using $\tt{Rcpp}$ and the c++ $\tt{Eigen}$ library.  The measure in the 
self-normalizing statistic was chosen as $\zeta =\frac{1}{n-1}\sum_{i=1}^{n-1}\delta_{i/n}$, where $\delta_\eta$ denotes the Dirac measure at $\eta \in [0,1]$. Simulations reported below are conducted with $n=20$. Previous work indicate that  the behavior of self-normalized tests in other settings are robust under various choices of $n$ for which the positive mass is sufficiently bounded away from the boundaries; see specifically Remark 3.1 in \cite{vdd21}. We found the same to be true in this setting. Next, we take $J=100$ equidistant frequency points in $(0,\pi)$. The grid of bandwidth pairs \eqref{eq:bwpairs} used to find $(N^\star,b^\star)$ is chosen as $2.5T^{0.53} <N_{m} < 5T^{0.53}$ and $0 \le \frac{1}{b_{f,l}} \le 2.5T^{0.53}$.
\\
We simulate processes from time-varying functional autoregressive and moving averages via their basis representation That is, given an ONB $\{\phi_l\}$ of $\Hi$ and a sequence of bounded linear operators$\{A_{t,i}\}_{i=1}^p$, $\{B_{t,j}\}_{j=0}^q$, we can simulate such processes via the coefficients which are given by
$$\widetilde{X}_{t,T} = \sum_{i=1}^p \widetilde{A}_{t,i}\widetilde{X}_{t-i} +  \sum_{j'=0}^q \widetilde{B}_{t,j}\widetilde{\epsilon}_{t-j},$$
where $\widetilde{X}_t := \left(\langle X_t, \phi_1 \rangle, \dots, \langle X_t, \phi_{d_{max}} \rangle \right)^\top$, $\widetilde{\epsilon}_t := \left(\langle \epsilon_t, \phi_1 \rangle, \dots, \langle \epsilon_t, \phi_{d_{max}} \rangle \right)^\top$ and the $(l,l')$-th entry of $\widetilde{A}_{t,j}$, $\widetilde{B}_{t,j}$ are respectively given by $\langle A_{t,j}(\phi_l),\phi_{l'}\rangle$ and $\langle B_{t,j}(\phi_l),\phi_{l'}\rangle$.The matrix entries  are generated from $N\big (0, z_{l,l'}\big )$ with $z_{l,l'}$ specified below. Note that second-order stationarity and existence of a causal solution requires conditions on the operator norms \citep[see][]{vde16}. Empirical rejection curves 
of the test \eqref{deq18} as a function of $\nu$ at the $5\%$ level are reported in Figure \ref{fig3} for sample sizes $T=2^i$, $i=10,\ldots, 15$ for the following models:
\begin{itemize}
\item[I)] tvFunctional noise generated via Gaussian coefficients with variances 
$\text{Var}(\langle \epsilon_t, \phi_l \rangle)=\cos(1/2+\sin(\pi(t/T))+0.3\cos(\pi(t/T))) \exp(-l/d^2)$, $l=1,\ldots 6$.
\item[II)] tvFAR(1) generated via Gaussian coefficients with variances 
$\text{Var}(\langle \epsilon_t, \phi_l \rangle) \exp(-l/d^2)$, $l=1,\ldots 6$, $z_{l,l^\prime}=\exp(-i-j)  $ and time-varying norm sequence $\{\kappa_{1,t}\}_{t=1}^T$ defined by $\kappa_{1,t} = 0.3\cos(1/2(t/T))$.   
\item[III)] tvFMA(1) with $B_{t,1}=\frac{1}{2}(1-b\sin(0.5\pi(\frac{t}{T})))I_d$, $b=-0.5$, and variance process as in (II). 
\end{itemize}
It may be noted  that the overall behavior as described in the previous section holds true in all models. In fact we do observe that around the true value (vertical dotted line), the size is close to the nominal level (vertical dotted line).  For the second cumulative normalized eigenvalue is somewhat undersized for model II, and somewhat oversized for model III.  For larger values of $d_0$, where the additional variation explained of including them gets very small, the curve (as a function of $\nu$) gets very steep leading to slightly oversized  in both cases. Overall we may conclude the results are satisfying.  
Note that smaller values of $\nu$ refer to alternative in \eqref{deq14} and in this case the ERPs quickly increase. Similarly,  larger  distances between  $\nu$  and  the threshold (specified by the vertical dotted line) correspond  to cases in the interior of the null hypothesis. In these cases  ERPs are strictly smaller than the nominal level confirming our theoretical results in Theorem \ref{thm5}.
\begin{figure}[h!]
\centering
\begin{subfigure}[b]{0.33\linewidth}
\includegraphics[width=\linewidth]{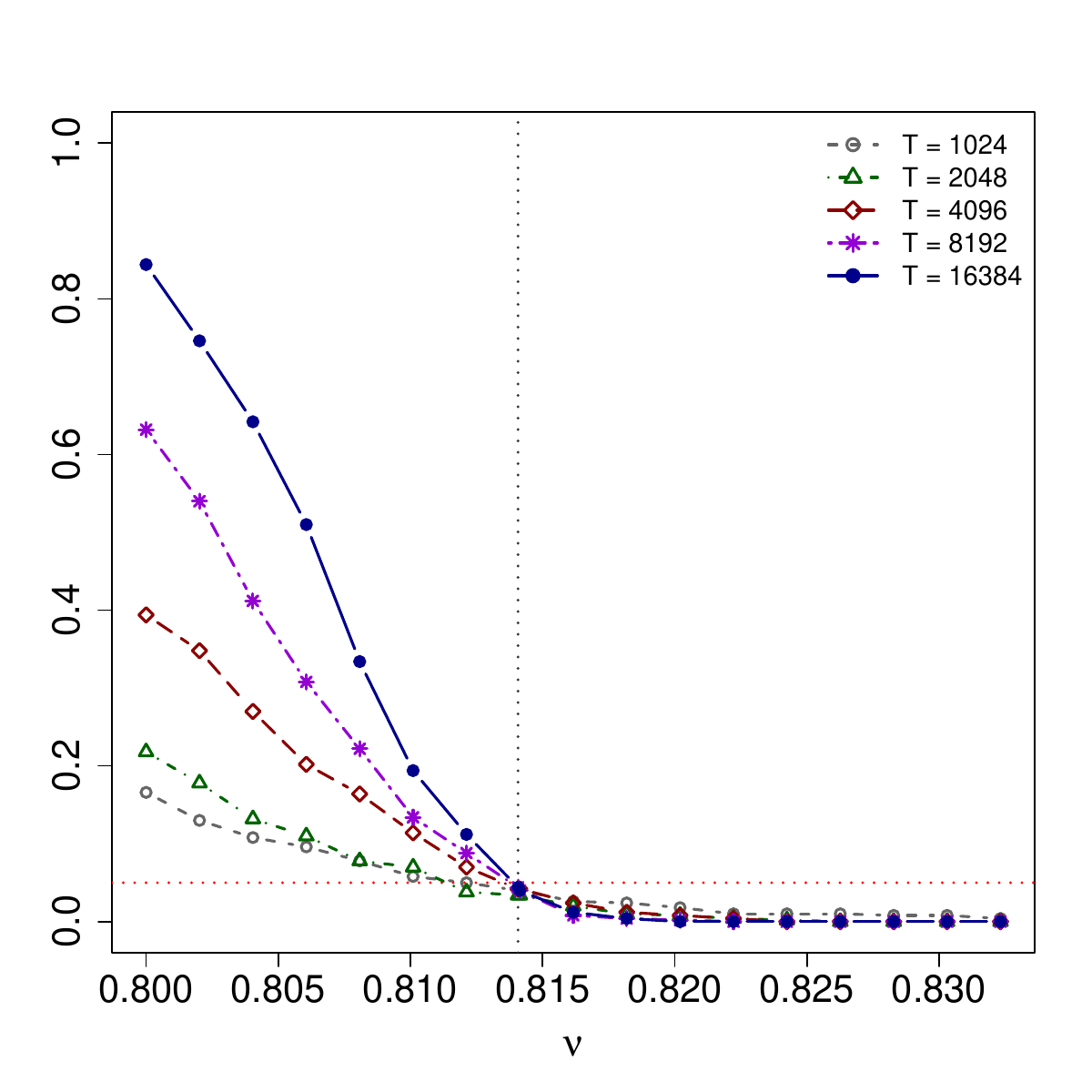}
\setlength{\abovecaptionskip}{-10pt}
\setlength{\belowcaptionskip}{-8pt} 
\end{subfigure}\hfil
\begin{subfigure}[b]{0.33\linewidth}
\includegraphics[width=\linewidth]{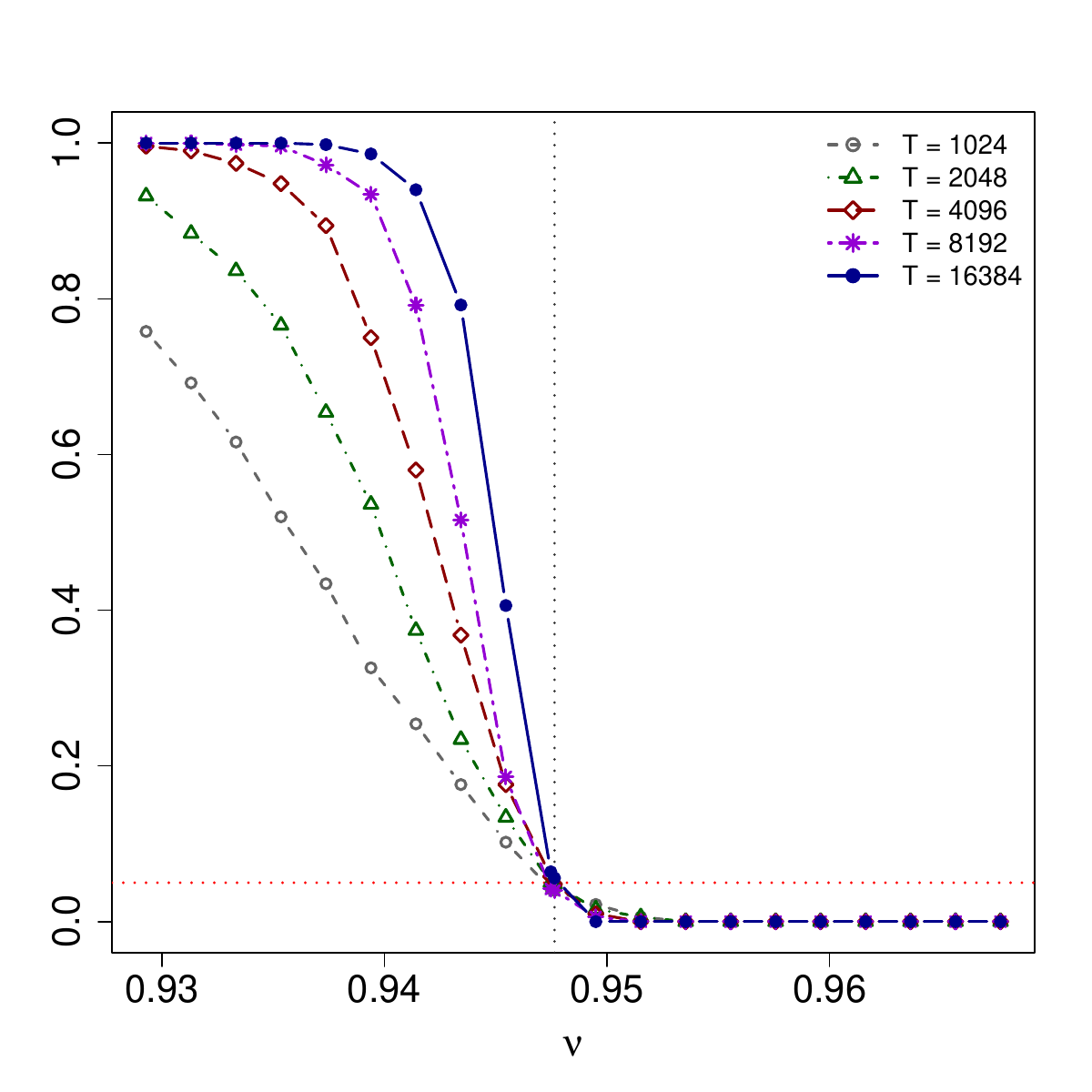}
\setlength{\abovecaptionskip}{-10pt}
\setlength{\belowcaptionskip}{-8pt} 
 \end{subfigure}\hfil
\begin{subfigure}[b]{0.33\linewidth}
\includegraphics[width=\linewidth]{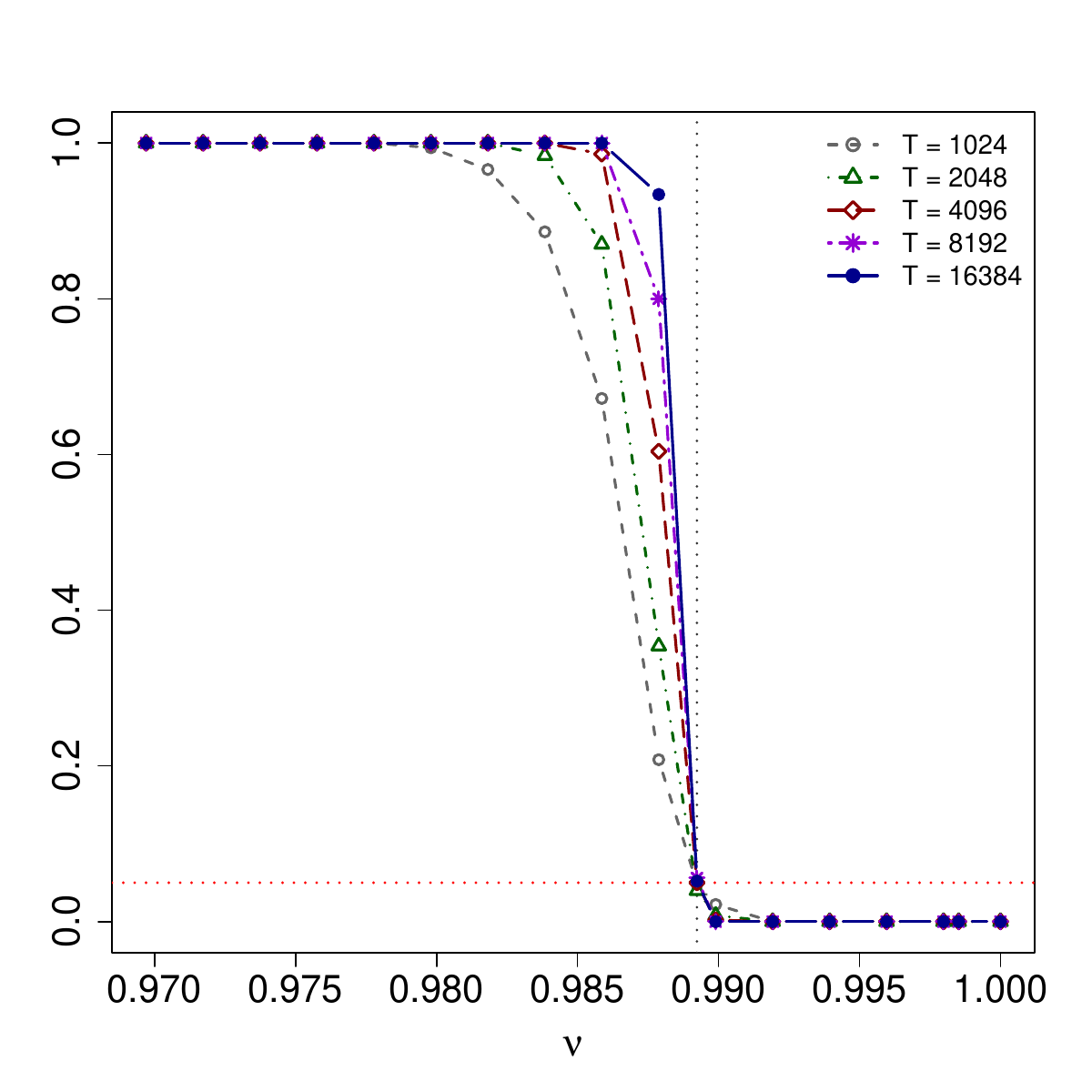}
\setlength{\abovecaptionskip}{-10pt}
\setlength{\belowcaptionskip}{-8pt} 
\end{subfigure}\hfil
\begin{subfigure}[b]{0.33\linewidth}
\includegraphics[width=\linewidth]{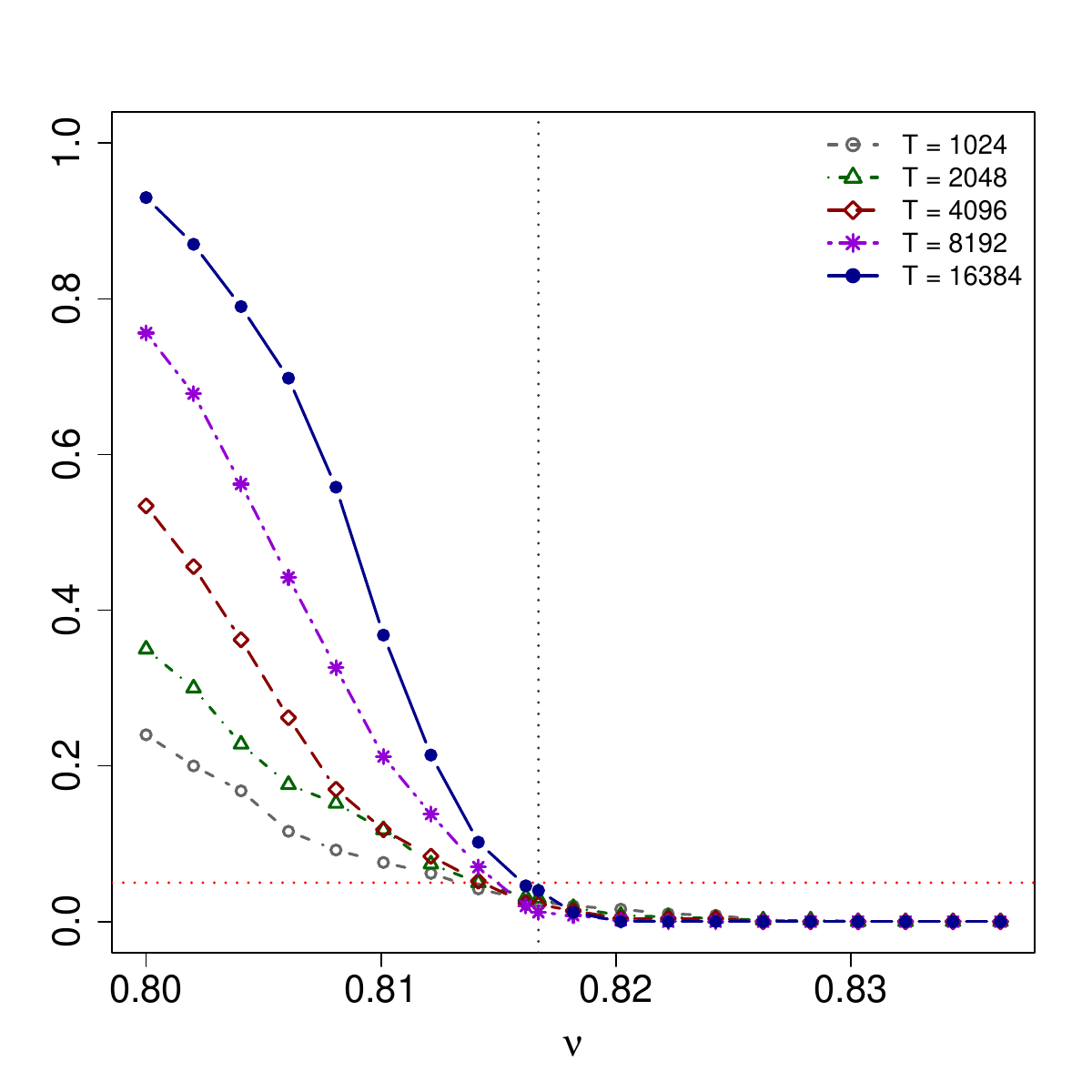}
\setlength{\abovecaptionskip}{-10pt}
\setlength{\belowcaptionskip}{-8pt} 
\end{subfigure}\hfil
\begin{subfigure}[b]{0.33\linewidth}
\includegraphics[width=\linewidth]{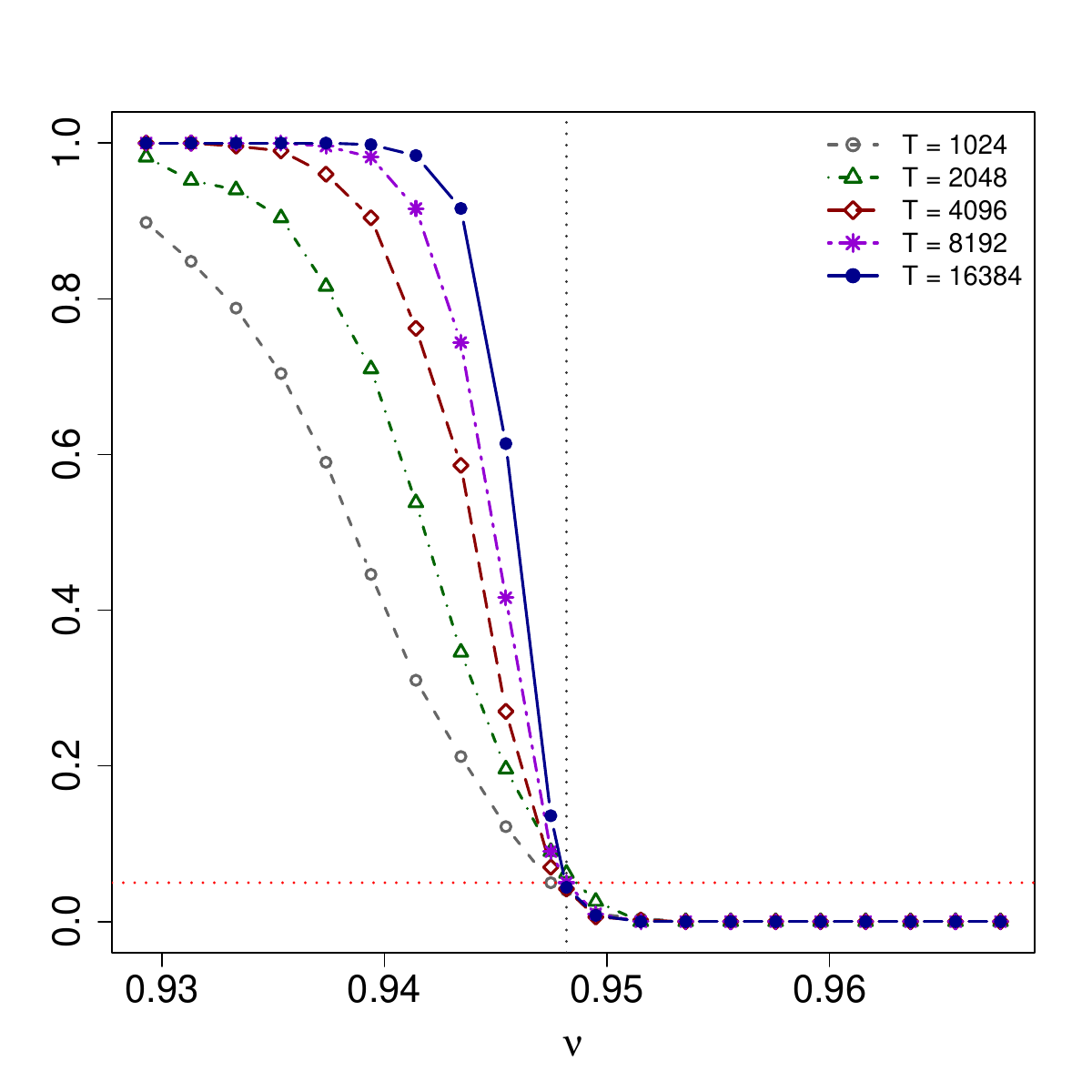}
\setlength{\abovecaptionskip}{-10pt}
\setlength{\belowcaptionskip}{-8pt} 
 \end{subfigure}\hfil
\begin{subfigure}[b]{0.33\linewidth}
\includegraphics[width=\linewidth]{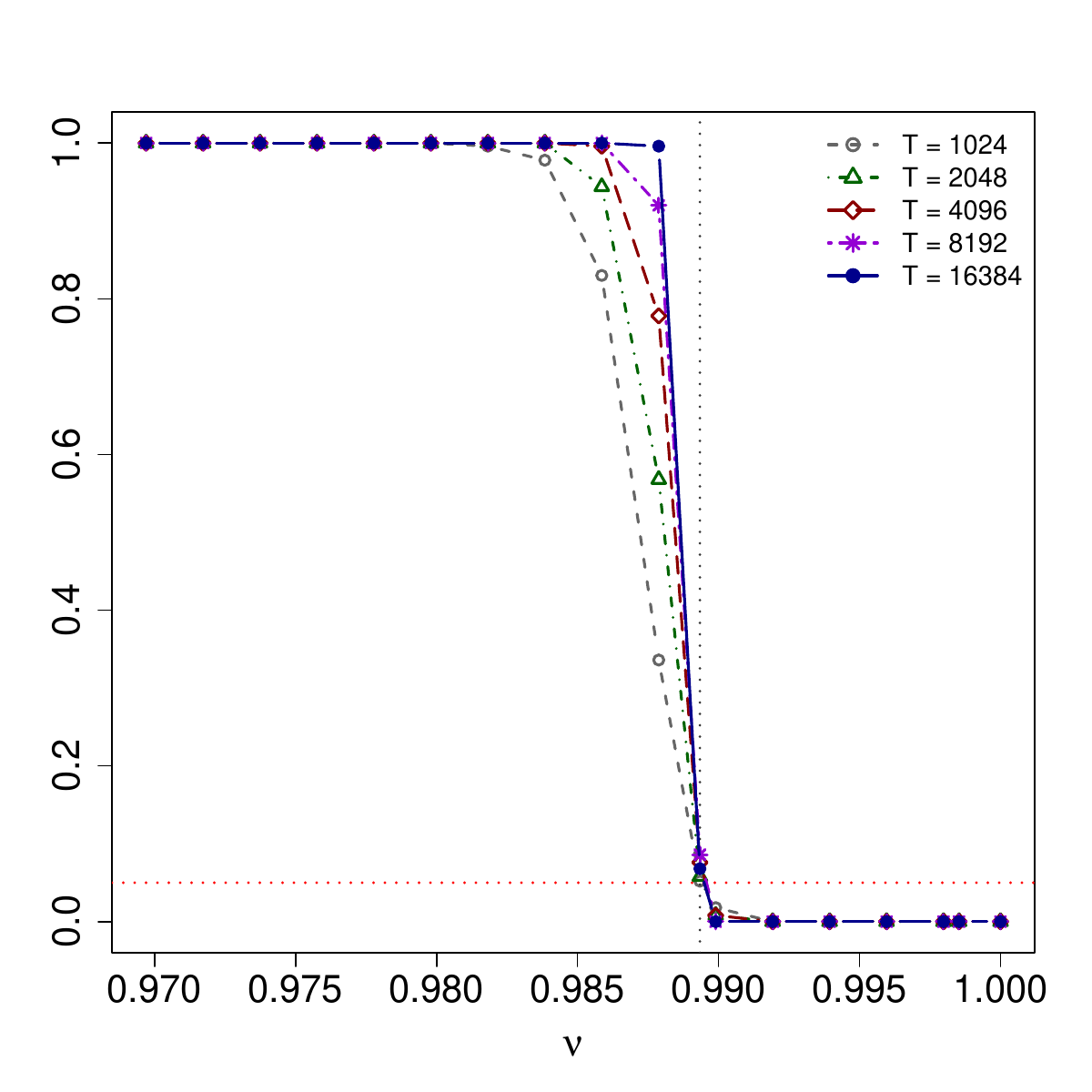}
\setlength{\abovecaptionskip}{-10pt}
\setlength{\belowcaptionskip}{-8pt} 
\end{subfigure}\hfil
\setlength{\belowcaptionskip}{-8pt} 
\begin{subfigure}[b]{0.33\linewidth}
\includegraphics[width=\linewidth]{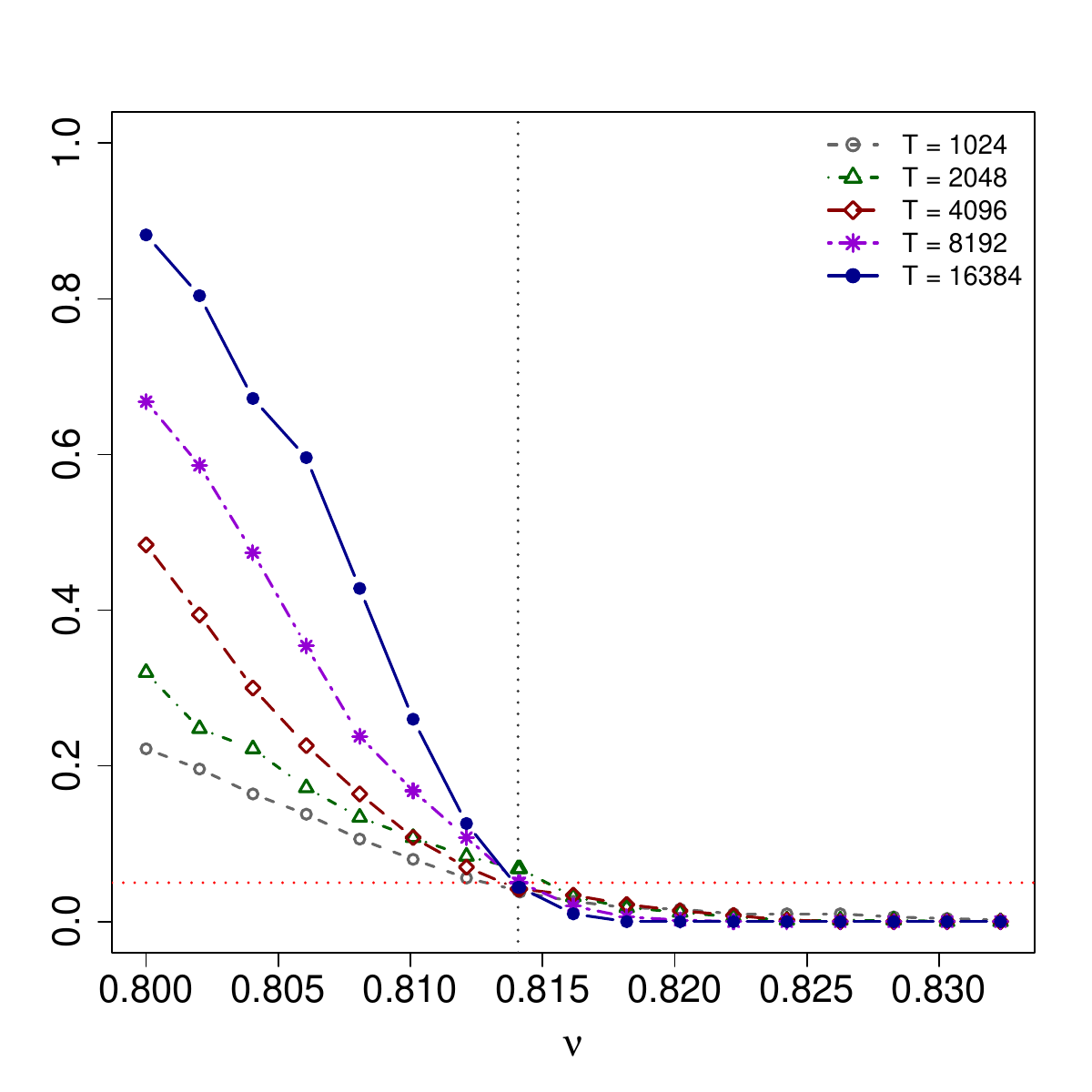}
\setlength{\abovecaptionskip}{-10pt}
\setlength{\belowcaptionskip}{-8pt} 
\caption{}
\end{subfigure}\hfil
\begin{subfigure}[b]{0.33\linewidth}
\includegraphics[width=\linewidth]{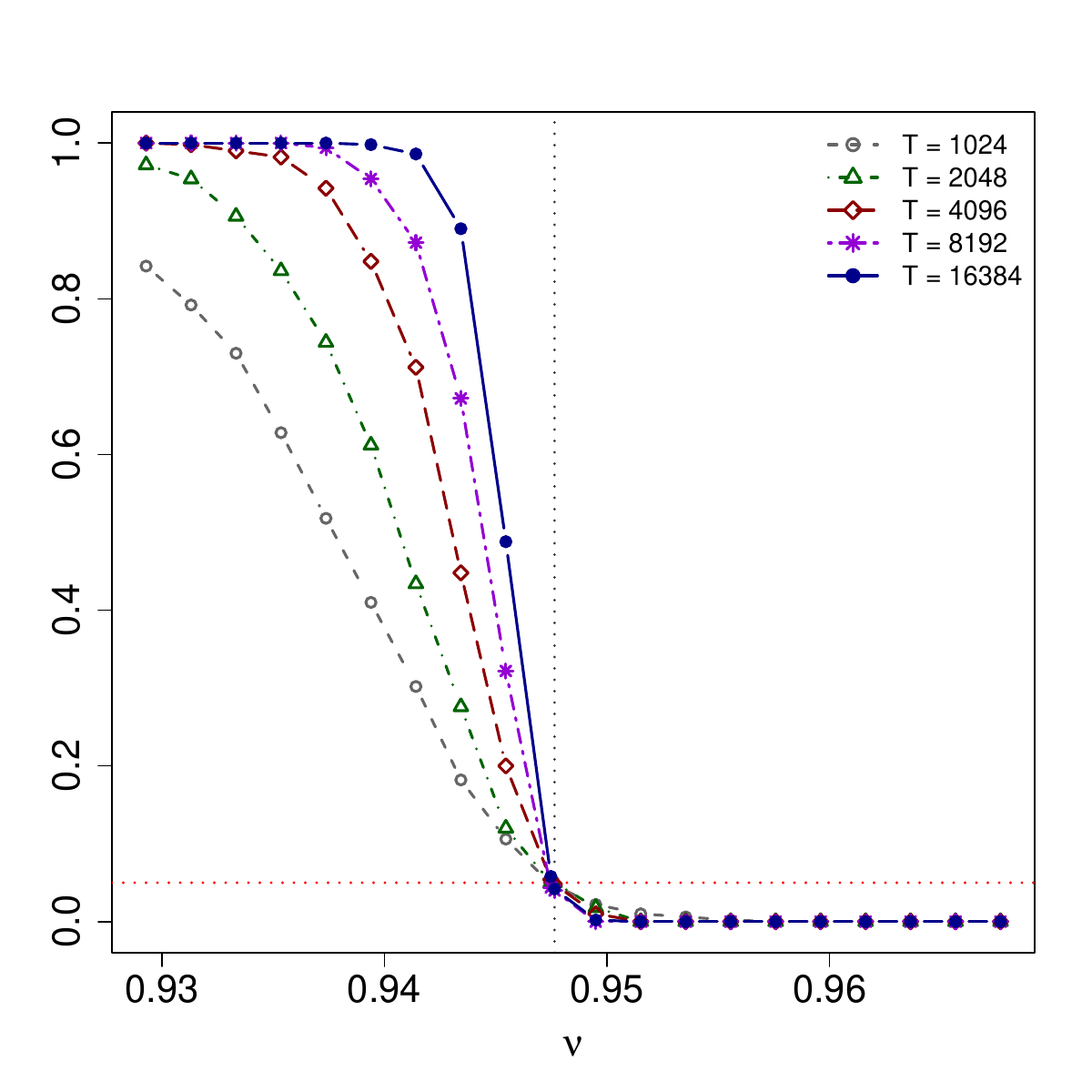}
\setlength{\abovecaptionskip}{-10pt}
\setlength{\belowcaptionskip}{-8pt} 
\caption{}
 \end{subfigure}\hfil
\begin{subfigure}[b]{0.33\linewidth}
\includegraphics[width=\linewidth]{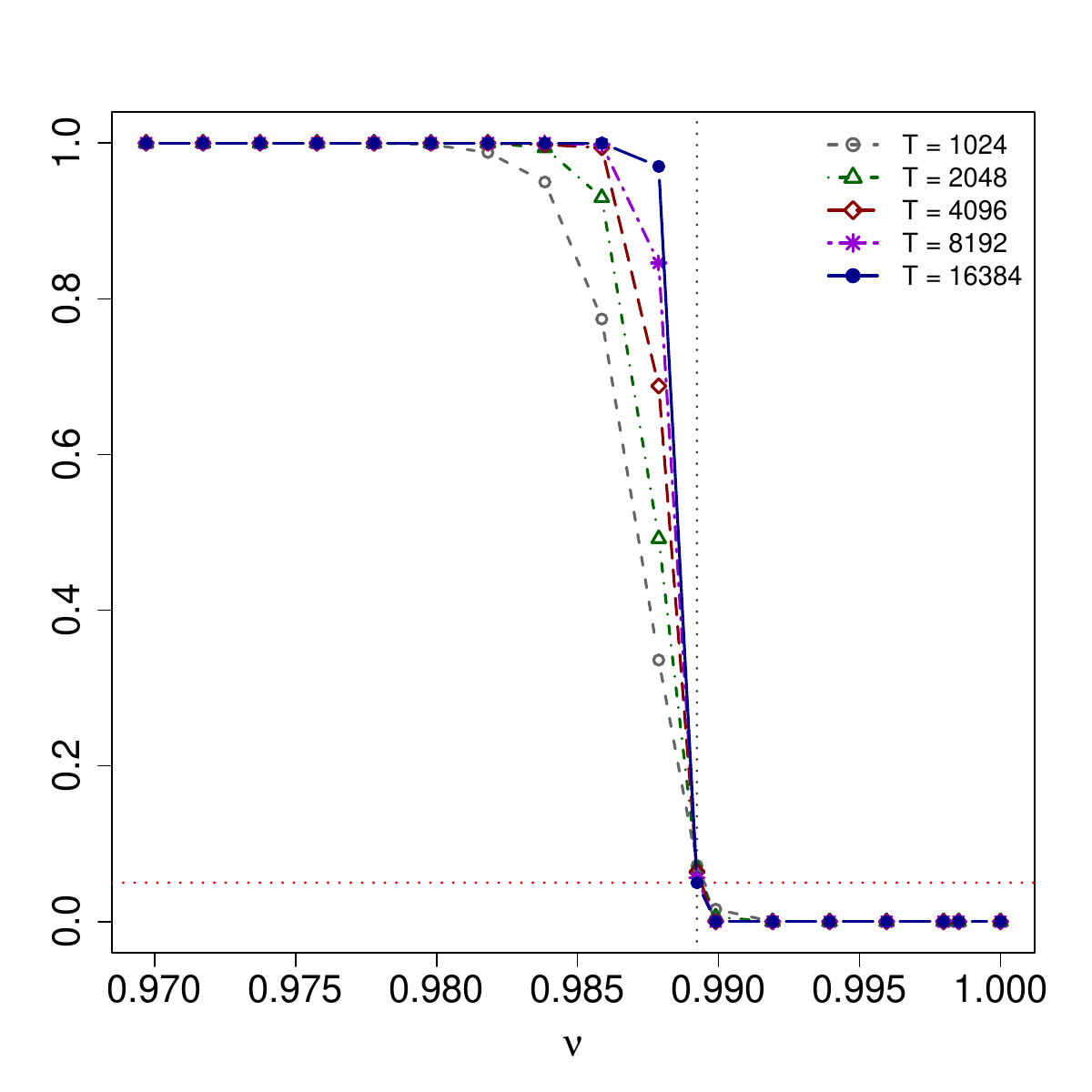}
\setlength{\abovecaptionskip}{-10pt}
\setlength{\belowcaptionskip}{-8pt} 
\caption{}
\end{subfigure}\hfil
\setlength{\belowcaptionskip}{-8pt} 
\caption{\label{fig3}\it ERP of the test \eqref{deq18} for the hypothesis  \eqref{deq14} for $d_0=2,3,4$ (column (a)--(c), resp.)  plotted as a function of $\nu$ at the nominal level 0.05 (horizontal dotted line). The true value $s_{d_0}$ is marked in the three panels by the vertical dotted line.
Model I: first row;  model II: second row; 
model III: third row.
}  
\end{figure}
}
\medskip 

\textbf{Acknowledgements.} 
This work  was partially supported by the  
 DFG Research unit 5381 {\it Mathematical Statistics in the Information Age}, project number 460867398. Anne van Delft was partially supported by NSF grant DMS-2311338. The authors would like to thank the associate editor and two unknown referees for their constructive comments on an earlier version of our paper.

\newpage

\appendix

\section{Preliminaries and inequalities for $S_1(\Hi)$-valued random variables} 
\label{secA} 
\def\theequation{A.\arabic{equation}}
\setcounter{equation}{0}
\subsection{Some properties of the space of trace-class operators} \label{prel}

We begin by providing some background results that we will make use of in order to prove the  main results. In the following, we denote the law of a  random element $X\in V$ by $\mu_X$, that is, $\mu_X=\pro \circ(X)^{-1}$. Suppose that $V$ has a Schauder basis $\{e_n\}_{n \ge 1}$. We denote $E_k$ to be the subspace of $V$ generated by $\{e_1,\ldots,e_k\}$ and we let $Q_k=\{q_1,\ldots,q_k\}$ be the subspace of $V^\pr$ generated by the coordinate functions. The natural projections associated with $\{e_n\}$ are the linear mappings $\Pi_k: V\to V$ 
defined by 
\begin{equation}
\Pi_k(v) =\sum_{j}^{k} \inprod{v}{q_j} e_{j}, \quad v \in V.
\label{eq51}  
\end{equation}
The map $\Pi_k$ has range $E_k$ and satisfies $\lim_{k\to \infty}\norm{\,\Pi_k(v) -v}_V = 0$ for all $v \in V$. Let $\Pi^\pr_{k}: V^\pr \to V^\pr$ be the operator adjoint to $\Pi_k$, that is, $\Pi^\pr_{k}(v^\pr) =\sum_{j=1}^k \inprod{e_j}{v^\pr}q_j$. Then for any $v \in V^\pr$, $\Pi^\pr_{k}(v^\pr) \to v^\pr$ in the $w^{\star}$-topology  \citep[see e.g.,][]{Day}. Furthermore, it follows from the Uniform Boundedness Principle that $C=\sup_n\|\Pi_n\|_V <\infty$ 
\citep[see e.g.,][]{Heil}, and since  $\|\Pi_n\|_V=\|\Pi^\pr_n\|_{V^\pr}$, we also have $\sup_n \|\Pi^\pr_n\|_{V^\pr}=C$. The constant $1\le C<\infty$ is known as the \textit{basis constant} and it can be shown that the coordinate functions satisfy $\sup_{n}\|q_n\|_{V^\prime} \le 2 C$.  This implies the countable set $Q=\cup_{k=1}^{\infty}Q_k$ is sequentially $w^\star$-dense in $V^\pr$ \citep[see e.g.,][]{Conway}. \\

Next, we focus on the context of the separable Banach space $(S_1(\Hi), \norm{\cdot}_{S_1})$. Using the duality pairing 
\[
\opl(\Hi) \times S_1(\Hi) \to \cnum, \quad (B, A) \mapsto \Tr(BA),
\]
and by defining $f_B: S_1(\Hi) \to \cnum$ with $f_B(A)=\Tr(BA)$, we can identify the Banach space $(\opl(\Hi), \norm{\cdot}_{S_\infty})$ with the topological dual space of $(S_1(\Hi), \norm{\cdot}_{S_1})$ through the isometric isomorphism 
$$
f: \opl(\Hi) \to \big(S_1(\Hi)\big)^{\pr}, \quad B \mapsto f_B,
$$
that is, $\opl(\Hi) \cong \big(S_1(\Hi)\big)^{\pr}$. Hence, we can represent each linear functional on $S_1(\Hi)$ as  a bounded linear operator on $\Hi$. We note a similar duality pairing exists between the subspaces of self-adjoint elements, $(S_1(\Hi)^\dagger,\norm{\cdot}_{S_1})$ and $(\opl(\Hi)^\dagger,\norm{\cdot}_{S_{\infty}})$, respectively. The space $(S_1(\Hi), \norm{\cdot}_{S_1})$ is a separable Banach space with a Schauder basis, as demonstrated in the following lemma. 
\begin{lemma} \label{lem:schauder}
Let $\{e_i\}_{i\ge 1}$ be an ONB of $\Hi$. Then $\{e_{ij}= e_i \otimes e_j\}_{i,j \ge 1}$ is a Schauder basis for  
$(S_1(\Hi), \norm{\,\cdot}_{S_1})$. Hence, for any $A \in S_1(\Hi)$ we can write $A=\sum_{i,j=1}^\infty \beta_{ij}(A) e_{ij}$ 
for the unique sequence of scalars (depending on $A$) given by $\beta_{ij}(A)=f_{e^\dagger_{ij}}(A)=\Tr(e_{ji}A)
$ where the convergence with respect to  $\norm{\,\cdot}_{S_1}$.
\end{lemma}

\begin{proof}
Note that uniqueness of the coefficients follows from the fact that $\{e_{ij}\}_{i,j \ge 1}$ is an orthonormal basis. 
We recall that we can write an element $A \in S_1(\Hi)$ as $A=B+\im C$, where $B, C \in S_1(\Hi)^\dagger$ given by $B=\frac{1}{2}(A+A^\dagger)$ and $C=\frac{1}{2i}(A-A^\dagger)$. Furthermore, any $T \in S_1(\Hi)^\dagger$, has a Jordan decomposition, i.e., there exists unique $T_{+}, T_{-} \in S_1(\Hi)^+$ such that $T=T_{+}-T_{-}$ and $T_{-}T_{+}=T_{+}T_{-}=O$ (see e.g., \cite{Conway}). 
Thus, we may write $A \in S_1(\Hi)$ as $A=B_{+}-B_{-}+\im (C_{+}-C_{-})$, where $B_{+}, B_{-}, C_{+}, C_{-} \in  S_1(\Hi)^+$. Without  loss of generality,  assume that $A \in S_1(\Hi)^+$. We will show that  $\lim_{K \to \infty}\|\Pi_K(A) - A\|_{S_1} \to 0$, where 
\[
 \Pi_K(A)= \sum_{j=1}^{K} \sum_{i=1}^K
\Tr(e^{\dagger}_{ij}A) e_{ij}
\]
First we remark that for a $A \in S_1(\Hi)^+$, $A^{1/2} \in S_2(\Hi)$, and thus the Cauchy Schwarz inequality yields
\[
\|A-\sum_{j=1}^K \Tr(e_{jj}A) e_{jj}\|_{S_1} \le  \sum_{j>K}|\Tr(e_{jj}A)|_{S_1} \le \sum_{j >K} \|A^{1/2}(e_j)\|^2_{\Hi}  \to 0 \quad K \to \infty.
\]
But note that,
\begin{align*}
 \Pi_K(A)=\sum_{j=1}^{K} \sum_{i=1}^K
\Tr(e^{\dagger}_{ij}A) e_{ij} =  \sum_{j=1}^K \Tr(e_{jj}A) e_{jj}, 
\end{align*}
and thus $\lim_{K \to \infty}\|\Pi_K(A) - A\|_{S_1} \to 0$ for $A \in S_1(\Hi)^+$. The result now follows from the triangle inequality.
\end{proof}
We will use \autoref{lem:schauder} to make precise the sequentially $w^\star$-dense set $Q$ of $(S_1(\Hi))^\pr$.
The projection mappings $\Pi_K: S_1(\Hi) \to S_1(\Hi)$ are given by 
$
\Pi_K(A) =\sum_{ij=1}^{K} f_{e^{\dagger}_{ij}}(A) e_{ij},
$
which have range $E_k$ and satisfy $\lim_{k\to \infty}\norm{\,\Pi_K(A) -A}_{S_1} = 0$, for all $A \in S_1(\Hi)$. Let $\Pi^\pr_{K}: \big(S_1(\Hi)\big)^{\pr} \to \big(S_1(\Hi)\big)^{\pr}$ be the operator adjoint to $\Pi_K$, that is, $\Pi^\prime_{k}(B) =\sum_{ij \le K}\Tr(B^\dagger e_{ji}) e_{ji} $. The results at the beginning of this section then give the following. 
\begin{proposition}\label{prop:Qstar}
If $B \in \big(S_1(\Hi)\big)^{\pr}$, then $
\lim_{k\to \infty} \Pi^\pr_{k}(B) = B$ in the $w^{\star}$-topology. 
\end{proposition}
Hence, the set $Q_k$ can be identified as the subspace of $(S_1(\Hi))^{\pr}$ generated by $\{e_{ij}: i,j=1,\ldots,k\}$. It follows from Proposition \autoref{prop:Qstar} that the countable set $Q=\cup_{k=1}^{\infty} Q_n$ 
is a sequentially $w^{\star}$-dense subspace in the topological dual of $S_1(\Hi)$.

\begin{proposition}\label{prop:S1inHprod}
Let $x, y \in \Hi$. Then $\norm{x \otimes y}_{S_1} = \|x\|_{\Hi}\|y\|_{\Hi}$.
\end{proposition}
\begin{proof}
By definition, we have 
\begin{align*}
\norm{x \otimes y}_{S_1} &
= \Tr\big( |x \otimes y|\big)
= \sum_{i=1}^{\infty} \biginprod{ \big( (x \otimes y)^{\dagger}(x\otimes y)\big)^{1/2} e_i}{e_i}
\end{align*}
for an arbitrary ONB $\{e_i\}_{i\ge 1}$ of $\Hi$. Observe that we can write $
(x \otimes y)^{\dagger}(x\otimes y) = \|x\|_{\Hi}^2 \|y\|_{\Hi}^2 (u \otimes u)
$
where $u = y/\|y\|_{\Hi}^2$. Then note that we can construct another ONB $\{u_i\}_{i \ge 1}$ of $\Hi$ with $u_1=u$. Hence, 
\begin{align*}
 \Tr\big( |x \otimes y|\big) = \Tr\big( \|x\|_{\Hi} \|y\|_{\Hi} (u \otimes u)\big)= \sum_{i=1}^{\infty} \biginprod{ \big(  \|x\|_{\Hi} \|y\|_{\Hi} (u \otimes u) u_i}{u_i} =  \|x\|_{\Hi} \|y\|_{\Hi}.
\end{align*}
\end{proof}

\subsection{Inequalities $S_1(\Hi)$-valued random variables} \label{sec:S1ineq}
\def\theequation{A.\arabic{equation}}
\setcounter{equation}{0}

\begin{lemma}\label{lem:Burkh}
Let $V$ be a separable Banach space and let $\{M_k\} \in \op^p_{V}, p>1$, be a martingale with respect to $\{\G_k\}$ with $\{D_k\}$ denoting its difference sequence, and let $\{A_k\}$ be a sequence of bounded linear operators. Then,
\begin{enumerate}
\item[i)] if $V=\Hi$, 
\[
\Bignorm{\sum_{k=1}^{n} A_k (D_k)}^q_{\Hi,p} \le K^q_{p} \sum_{k=1}^{n} \norm{A_k}_{\infty}^q \|D_k\|^q_{\Hi,p}, \quad q=\min(2,p),
\]
where $K^q_{p}= (p^\star-1)^q$ with $p^\star=\max(p,\frac{p}{p-1})$,
\item[ii)] if $V=S_1(\Hi)$ and $p>1$ a positive integer, then for an orthonormal basis $\{e_j\}$ of $\Hi$ 
\[
\Bignorm{\sum_{k=1}^{n} A_k (D_k)}^q_{S_1,p} \le  K^q_{p} \Big(\sum_j \sqrt[q]{\sum_k^n \|A_k\|^q_{\infty} \|(D_k)(e_{j})\|^q_{\Hi,p}}\Big)^q \quad q=\min(2,p).
\]
\end{enumerate}
\end{lemma}
\begin{proof}
We only prove the second part as the first can for example be found in \cite{vdd21}. We recall the following facts \citep[see e.g.,][]{Pedersen,Conway}. If $B \in \opl(\Hi)$ then there exists a unique partial isometry $U$ such that $B=U|B|$ where $|B|=\big(B^{\dagger}B\big)^{1/2}$.
The decomposition $B=U|B|$  is the polar decomposition of $B$. Since the $r$-th Schatten class is an ideal in $\opl(\Hi)$, we have that $|B| =U^\dagger B\in S_r(\Hi)$ if $B \in S_r(\Hi)$.  
We furthermore recall that the partial isometries $U$ and $U^\dagger$ have operator norm 1 or 0, where the latter is true only if it is the zero map. Using the definition of $\|\cdot\|_{S_1}$ and Cauchy Schwarz's inequality
\[
\|B\|_{S_1} = \Tr(U^\dagger B)=\sum_{j\in \nnum}\inprod{U^\dagger B(e_j)}{e_j} \le \sum_{j\in \nnum}\norm{B(e_j)}_{\Hi} \norm{U(e_j)}_{\Hi} \le \sum_{j\in \nnum}\norm{B(e_j)}_{\Hi} \tageq \label{eq:upbtr}
\]
Furthermore, H{\"o}lder's inequality implies in turn
\begin{align*}
\E \|B\|^p_{S_1} \le \E\Big(\sum_{j} \|B(e_j)\|_{\Hi})^p 
\le  \sum_{j_1} \cdots \sum_{j_p} 
\prod_{i=1}^p \sqrt[p]{\E \|B(e_{j_i})\|^p_{\Hi}} 
& =\Big(\sum_j \sqrt[p]{\E \|B(e_{j})\|^p_{\Hi}}\Big)^p
\\& =\Big(\sum_j \|B(e_{j})\|_{\Hi,p}\Big)^p
\\& =\Big(\sum_j \sqrt[q]{\|B(e_{j})\|^q_{\Hi,p}}\Big)^p. \tageq\label{eq:upbtr2}
\end{align*}

Taking $B=\sum_{k=1}^{n} A_k (D_k)$, and using the first part of \autoref{lem:Burkh}
\begin{align*}
\Bignorm{\sum_{k=1}^{n} A_k (D_k)}^q_{S_1,p} &\le \Big(\sum_j  \|A_k (D_k)(e_{j})\|_{\Hi,p}\Big)^q 
 \le \Big(\sum_j \sqrt[q]{K^q_p\sum_k^n \|A_k\|^q_{\infty} \|(D_k)(e_{j})\|^q_{\Hi,p}}\Big)^q,
\end{align*}
from which the statement follows.
\end{proof}

\begin{lemma} \label{lem:Burklin} 
Let $\{X_{t,T}| t=1,\ldots, T; T \in \mathbb{N}\}$ be a zero mean stochastic process in $\op^p_{\Hi}$, that satisfies  \autoref{def:locstat} and \autoref{as:depstrucnonstat} and let $\{A_n\}$ be a sequence of bounded linear operators on $\Hi$. Then, for  $p\ge 2$
\[
\sup_{T \in \mathbb{N}}\bignorm{\sum_{i=1}^{n} A_i(X_{i,T})}_{\Hi,p} \le K_p\Big(\sum_{i=1}^{n} \norm{A_i}^2_{S_\infty}\Big)^{1/2} \sum_{j=0}^{\infty}\sup_{T \in \nnum}\nu^{X_{\cdot,\cdot}}_{p}(j,T)
\]
\end{lemma}
\begin{proof}
Let $X_{t,T}$ satisfy  \autoref{as:depstrucnonstat} with $p\ge 1$. Then we have the representation $
X_t = \sum_{j=0}^{\infty} P_{t-j}(X_t)$.
Furthermore,  
\[
P_k (X_t) = \E[X_t | \G_{k}]-  \E[X_t | \G_{k-1}] = \E\Big[X_t- \E[X_t |\G_{t,\{k\}}] \Big\vert \G_k \Big]
\]
where $\G_{t,\{k\}} = \sigma(\epsilon_t, \epsilon_{t-1}, \ldots, \epsilon_{k+1}, \epsilon^{\prime}_{k}, \epsilon_{k-1}, \ldots)$, and thus 
\[
\|P_k (X_{t,T})\|_{\Hi,p} \le \sup_{t \in \znum}\bignorm{\mathfrak{G}(t,T,\G_t)- \mathfrak{G}(t,T,\G_{t,\{k\}})}_{\Hi,p}~.
\]
From the first part of \autoref{lem:Burkh}, we find for $p\ge 2$
\begin{align*}
\Bignorm{\sum_{i=1}^{n} A_i(X_{i,T})}_{\Hi,p} 
 \le \sum_{j=0}^{\infty}\Bignorm{\sum_{i=1}^{n} A_i \Big(P_{i-j}(X_{i,T}\Big)}_{\Hi,p} 
\le K_p \Big(\sum_{i=1}^{n} \norm{A_i}^2_{S_\infty}\Big)^{1/2} \sum_{j=0}^{\infty}\sup_{T \in \nnum}\nu^{X_{\cdot,\cdot}}_{p}(j,T).
\end{align*}
\end{proof}

\begin{lemma} \label{lem:expin}
Let $Z_1, \ldots, Z_T$ be independent zero-mean  $V=S_1(\Hi)$-valued random variables such that for some $q>2$, $\E\|Z_i\|_{S_1}^q<\infty$, $i=1,\ldots, T$. Then for $0 \le \beta \le 1$, $\delta >0$ and any $x>0$,
\begin{align*}
\mathbb{P}\Bigg( \bignorm{\sum_{i=1}^T Z_i}_{S_1} \ge x \Bigg) \le \exp\Big(-\frac{x^2}{(2+\beta)\Lambda_n} \Big)+ C \E \bignorm{\sum_{i=1}^T Z_i}^q_{S_1} x^{-q},\tageq \label{eq:expin1}
\end{align*}
where $\Lambda_n =\sup\big\{ \sum_{j=1}^n \E \big|v^\prime(Z_j)\big|^2:v^\prime \in V^\prime_1\big\}$, where $V^\prime_1$ denotes the unit ball in $V^\prime$, and $C$ is a positive constant depending on $\beta, \delta$ and $q$.
\end{lemma}
\begin{proof}
 It was proved in \cite{tj74} that $S_1(\Hi)$ is a Banach space of cotype 2. It can be shown using Kahane's inequality, and the randomization principle that it holds for any $2\le q <\infty$ that there exists a constant $C_{2,q}>0$ such that
\begin{align*}
\E\sum_{i=1}^T \norm{ Z_i}^q_{S_1}  \le C_{2,q}\E\bignorm{\sum_{i=1}^T Z_i}^q_{S_1} \tageq \label{eq:inraad}
\end{align*}
With this in place, let us turn to \eqref{eq:expin1}.
The case $x \ge (1+\beta)\big(\E \bignorm{\sum_{i=1}^T Z_i}^q_{S_1}\big)^{1/q}$, immediately follows from \eqref{eq:inraad} and the following Fuk-Nagaev Type inequality provided in Theorem 3.1  of  \cite{el08}.

\begin{thm} 
\label{Einm}
Let $V$ be a separable Banach space, and let $Z_1, \ldots, Z_n$ be independent $V$-valued random variables with mean zero such that for some $q>2$, $\E\|Z_i\|^q<\infty$, $1\le i\le n$. Then for $0 \le \beta \le 1$, $\delta >0$ and any $x>0$,
\begin{align*}
\mathbb{P}\Bigg(\max_{1\le i\le n} \bignorm{\sum_{i=1}^n Z_i}_{V} \ge (1+\beta)\E\bignorm{\sum_{i=1}^n Z_i}_{V}+x \Bigg) \le \exp\Big(-\frac{x^2}{(2+\beta)\Lambda_n}  \Big)+ \tilde{C} \E\sum_{i=1}^n\norm{ Z_i}^q_{V}x^{-q},
\end{align*}
where $\Lambda_n =\sup\{\sum_{j=1}^n \E |v^\prime(Z_j)|^2: v^\prime \in V^\prime_1\}$, where $\tilde{C}$ is a positive constant depending on $\beta, \delta$ and $q$.
\end{thm}
Indeed, it suffices to take $C \ge \tilde{C}C_{2,q}$.
 On the other hand,
if $x< (1+\beta)\big(\E \bignorm{\sum_{i=1}^T Z_i}^q_{S_1}\big)^{1/q}$ then \eqref{eq:expin1} is trivially true for $C\ge (1+\beta)^q$. Hence, the result holds by taking $C= \max\big((1+\beta)^q, \tilde{C}C_{2,q}\big)$.

\end{proof}

 \label{sec:thm31proof}
\def\theequation{B.\arabic{equation}}
\setcounter{equation}{0}
\section{Proof of Theorem \ref{thm:Conv}}

The proof is intricate and consists of several auxiliary statements. We first prove the  result  in the case, where $\phi (z) =z  $  and relegate the proof for general analytic functions to the end of this section. 

\subsection{ The case $\phi (z) =z  $} \label{sec821}

We start by introducing  an appropriate double indexed process $\big\{M^{-1/2}\sum_{u \in U_M}\mathcal{Y}^{(u)}_{m,N} \big\}_{m,N} $ in $ C_{S_1}$, which  will be used of an approximation of the process $M^{1/2}\rho_{T} \big(
\mathrm{L}^{a,b}_{U_M}\circ  \G_\Upsilon (\hat {\cal F} )- 
 \mathrm{L}^{a,b}\circ \G_\Upsilon ({\cal F}   ) \big ) $. This 
requires  some additional notation. Firstly, recall the $\sigma$-algebra 
$\G_\ell  = \sigma(\epsilon_\ell, \epsilon_{\ell -1} , \ldots  )$ from
Section \ref{sec21}, and furthermore denote
 for $k \leq \ell $,  
$\G^\ell_k = \sigma(\epsilon_\ell, \epsilon_{\ell -1} , \ldots,\epsilon_{k})$.

We consider $m$-dependent
 martingale difference sequences defined by 
\[
\dm{u}{\omega }{t} =  \sum_{l=0}^{\infty} P_{t} \Big({X}^{(u)}_{m,t+l}  \Big) e^{-\im \omega   l} \tageq \label{eq:dmuv}
\]
where $P_j(Y) = \E [ Y| \G_{j}] - \E[Y|\G_{j-1}], Y \in \mathcal{L}^2_{\Hi}$ and where 
$
{X}^{(u)}_{m,t}= \E \big[{X}^{(u)}_{t}  \Big\vert \G^{t}_{t-m} \big]$.
denotes an $m$-dependent version of $X^{(u)}_t$. We will have specific interest in
\begin{align*}
\dmpi{u}{\omega}{N}{t}=\dm{\tu{N}{u}+t}{\omega}{\tu{N}{u}+t} & = \sum_{l=0}^{\infty} P_{\tu{N}{u}+t} \Big({X}^{(\tu{N}{u}+t)/T}_{m,\tu{N}{u}+t+l}  \Big) e^{-\im \omega  l}\tageq \label{eq:dmpi}
\end{align*}
Under  \autoref{as:depstruc}, {a similar argument as in the proof of \autoref{lem:Burklin} yields}
$$
\|\dmpi{u}{\omega}{N}{t}\|_{\hi,2} \le \sup_{u}\|\tilde{D}^{(u,\omega)}_{0}\|_{\hi,2} \le \sum_{j=0}^{\infty}\nu_{\hi,2}(j)< \infty,
$$
where ${D}^{(u,\omega)}_{0}$ $= \sum_{l=0}^{\infty} P_0(X^{(u)}_{l})e^{-\im \omega l}$.  Next, we define 
the random elements 
\[
V^{(u,\omega)}_{m,N,t} =  \dmpi{u}{\omega}{N}{t} \otimes \sum_{s=1}^{t-1} \tilde{w}_{\bf,t,s}^{(\omega)} \dmpi{u}{\omega}{N}{s} \quad  t=2,3, \ldots, N.
\tageq \label{eq:VNTfix}
\]
For $t=1$, we set $V^{(u),\omega}_{m,N,1} = \mathrm{0}_{\Hi}$, where $\mathrm{0}_{\Hi}$ denotes the zero element of $\Hi$. 
{Since our assumptions imply that $V^{(\cdot,\cdot)}_{m,N,t} \in L^1_{S_1}([0,1]\times[0,\pi])$ the  integral 
\[
W^{(u)}_{m,N,t} 
=  \mathcal{\Phi}^{-1}_{\bf }\int^{b}_{a} \Upsilon^{(\omega)}_{u} \Big(V^{(u,\omega)}_{m,N,t} +V^{\dagger(u,\omega)}_{m,N,t} \Big) d\omega \quad t=1, 2, \ldots, k, 1\le k \le N, \tageq \label{eq:VNT}
\]
is well-defined,} 
and where we use the notation 
\[
\mathcal{\Phi}^2_{\bf}=\sum_{t,s=1}^{N} \tilde{w}^2_{\bf,t,s} \sim ( \kappa_f N)\bf^{-1}.\] 
Note that $W^{(u)}_{m,N,t}$ 
depends on the parameters $a,b$, but for the sake of simplicity this  dependence is not reflected in the notation. {Under the stated conditions, we can show (using \autoref{lem:Burkh}(ii)) that $\big\{W^{(u)}_{m,N,t}\big\}_{t}$ is a martingale difference sequence in $\op^2_{S_1}$ with respect to the filtration $\{\G_{\tu{N}{u}+t}\}_t$, for any $N \in \mathbb{N}$, uniformly in $u \in [0,1]$.} The process $\big\{M^{-1/2} \sum_{u \in U_M}\mathcal{Y}^{(u)}_{m,N} \big\}_{m,N}$ is then defined as an interpolated version of the partial sum process of  $\big\{W^{(u)}_{m,N,t}\big\}_{t}$, i.e., 
\[
\mathcal{Y}^{(u)}_{m,N}(\eta)=\sum_{t=1}^{\flo{\eta N}}W^{(u)}_{m,N,t}  + (\eta N- \flo{\eta N}) W^{(u)}_{m,N,\flo{\eta N}+1}
\quad \eta\in [0,1]. \tageq \label{eq:BMcont}
\]
We shall show that the process in \eqref{eq:BMcont} defines the distributional properties  in \autoref{thm:Conv}. More specifically, note that $(C_{S_1}, \norm{\cdot}_{C_{S_1}})$ is a separable Banach space and therefore a complete metric space. Let $A$ be a closed set of $(C_{S_1}, \norm{\cdot}_{C_{S_1}})$ and consider the set $A_{\delta} =\{x:\|x-y\|_{C_{S_1}}\le  \delta, y \in A\}$. Observe that
\begin{align*}
& \limsup_{T\to \infty} \,\mathbb{P}\Big(M^{1/2} \rho_{T} \big(
\mathrm{L}^{a,b}_{U_M}\circ  \G_\Upsilon (\hat {\cal F} )- 
 \mathrm{L}^{a,b}\circ \G_\Upsilon ({\cal F}   )  \big )  \in A\Big)
\\& \le  \lim_{m\to \infty} 
\limsup_{T\to \infty} 
\, \mathbb{P}\Big( \bignorm{M^{1/2} \rho_{T} \Big(
 \mathrm{L}^{a,b}_{U_M}\circ \G_\Upsilon (\hat {\cal F} ) -
 \mathrm{L}^{a,b}\circ \G_\Upsilon ( {\cal F} )   \Big)- \frac{1}{M^{1/2}}\sum_{u \in U_M}\mathcal{Y}^{(u)}_{m,N} }_{C_{S_1}} > \delta\Big)
 \tageq \label{eq:processtapprox}
  \\& + \lim_{m\to \infty} 
  \limsup_{T\to \infty} 
\, \mathbb{P}\Big(  \frac{1}{M^{1/2}}\sum_{u \in U_M}\mathcal{Y}^{(u)}_{m,N}  \in A_{\delta}\Big) \tageq \label{eq:avBM}
\end{align*}
For the first term \eqref{eq:processtapprox}, we have the following result, which is proved in
Section 
\ref{sec831}. The proof consists of appropriately decomposing the error and bounding each term (see \eqref{eq:approxloc}-\eqref{eq:approxmart}). 

\begin{thm} \label{thm:fullapproximation}
Under the conditions of  \autoref{thm:conv_an}(a) or (b) we have
\begin{align*}
\lim_{m\to \infty} \limsup_{T\to \infty}\,  \mathbb{P}\Big( \bignorm{M^{1/2}\rho_{T} \Big(
 \big (\mathrm{L}^{a,b}_{U_M}\circ  \G_\Upsilon (\hat {\cal F} )  \big )  - \big (
 \mathrm{L}^{a,b}\circ  \G_\Upsilon ({\cal F} )  \big )  \Big) - \frac{1}{M^{1/2}}\sum_{u \in U_M}\mathcal{Y}^{(u)}_{m,N} }_{C_{S_1}} > \delta\Big) =0.
\end{align*}
\end{thm}
 For  the second  term \eqref{eq:avBM}, we proceed in several steps.  First, we 
derive the asymptotic properties  of the process 
 $\big\{\mathcal{Y}^{(u)}_{m,N} \big\}_{N \ge 1}$ 
 as $N \to \infty $ for  fixed $m$ and $u$
 (Theorem \ref{thm:BMfixu}). Then, we study  the average 
$\big\{ \frac{1}{M^{1/2}} \sum_{u\in U} \mathcal{Y}^{(u)}_{m,N}  \big\}_{N \ge 1} $ 
 for fixed $m$   (Theorem \ref{thm:bigapprox}).  
 Finally, we investigate the term \eqref{eq:avBM}.
Our  first result in this sequence of arguments --Theorem \ref{thm:BMfixu} below--
shows that, for any fixed $m \ge 1, u \in [0,1]$, the law of this process converges as $N \to \infty$ to a Wiener measure {$\mathbb{W}_{\mu^{(u)}_{\Upsilon,m}}$ on $S_1(\Hi)$ generated by a complex-valued Gaussian measure $\mu^{(u)}_{\Upsilon,m}$}. 
A proof can be found in  Section  \ref{sec832}. The statement makes use of the notation $(A \widetilde{\otimes} B)C = ACB^{\dagger}$, $A,B,C \in S_\infty(\Hi)$, for the Kronecker tensor product and  $(A \widetilde{\otimes}_{\top} B)C = (A \widetilde{\otimes} \overline{B})\overline{C}^{\dagger}$ for the transpose Kronecker tensor product, respectively. 
\begin{thm}\label{thm:BMfixu}
Suppose the conditions of \autoref{thm:conv_an}(a) or (b) hold and let $u \in [0,1]$ and $[a,b]  \subseteq [0,\pi], a\le b$ be fixed. Then there exists a zero-mean Gaussian measure $\mu^{(u)}_{\Upsilon,m}$ on $S_1(\Hi)$ with 
\[\int_{S_1} |f(x)|^2 d\mu^{(u)}_{\Upsilon,m}(x) = \Gamma^{(u)}_{\Upsilon,m}(f) \text{ and } \int_{S_1} f^2(x) d\mu^{(u)}_{\Upsilon,m}(x)=\Sigma^{(u)}_{\Upsilon,m}(f)\] for each 
 $f \in \big(S_1(\Hi)\big)^{\prime}$ and 
where $\Gamma^{(u)}_{\Upsilon,m}$ and  $\Sigma^{(u)}_{\Upsilon,m}$ are respectively given by
\begin{align*} 
\Gamma^{(u)}_{\Upsilon,m}& =8\pi^2 \int_a^b \Big\{(\Upsilon_{u,\omega} \widetilde{\otimes} \Upsilon_{u,\omega}) (\F_m^{(u,\omega)} \widetilde{\otimes} \F_m^{(u,\omega)} )
+\mathrm{1}_{\{0,\pi\}}(\omega) \Big[(\Upsilon_{u,\omega}\widetilde{\otimes} \Upsilon_{u,\omega})
({\F}^{(u, \omega)}_{m} \widetilde{\otimes}_\top {\F}^{(u, \omega)}_{m}) \Big]\Big\}d\omega, \tageq \label{eq:varfixedum}
\end{align*}
and 
\begin{align*}\Sigma^{(u)}_{\Upsilon,m}& =8\pi^2 \int_a^b \Big\{
 \mathrm{1}_{\{0,\pi\}}(\omega) \big[(\Upsilon_{u,\omega} \widetilde{\otimes} \overline{\Upsilon}_{u,\omega})({\F}^{(u, \omega)}_{m} \widetilde{\otimes} {\F}^{(u, \omega)}_{m})\big]+ 
(\Upsilon_{u,\omega} \widetilde{\otimes} \overline{\Upsilon}_{u,\omega}) (\F_m^{(u,\omega)} \widetilde{\otimes}_\top \overline{\F}_m^{(u,\omega)})\Big\}d\omega,
\tageq \label{eq:psvarfixedum}
\end{align*}
such that the sequence of random elements $\big\{\mathcal{Y}^{(u)}_{m,N} \big\}_{N \ge 1}$ in $C_{S_1}$ defined in \eqref{eq:BMcont} satisfies
\[
\mathbb{P} \circ \Big(\mathcal{Y}^{(u)}_{m,N}\big)^{-1} \Rightarrow \mathbb{W}_{\mu^{(u)}_{\Upsilon,m}}  \quad N \to \infty,
\]
where $\mathbb{W}_{\mu^{(u)}_{\Upsilon,m}}$ denotes the Wiener measure on $S_1(\Hi)$ induced by the Gaussian measure $\mu^{(u)}_{\Upsilon,m}$.
\end{thm}

\noindent
The next theorem provides the distribution of the average of re-scaled midpoints, for fixed $m$ as $N \to \infty$, and is proved in 
Section \ref{sec833}. 
 
\begin{thm} \label{thm:bigapprox}
Suppose the assumptions underlying \autoref{thm:conv_an}(a) or (b) hold true.  Let $[a,b] \subseteq [0,\pi]$ and let $U_M$ be as defined in \eqref{eq:midset}. Then, for any fixed $m$,
as $N \to \infty$ (and $M \to \infty$ such that \autoref{as:bandwidth} holds)
\[
\Bigg\{ \frac{1}{M^{1/2}} \sum_{u\in U_M} \mathcal{Y}^{(u)}_{m,N}(\eta)\Bigg\}_{\eta  \in I}
\stackrel{\mathcal{D}}{\Longrightarrow} 
\Big\{\mathbb{W}_{\mu_{\Upsilon,m}}(\eta)\Big\}_{\eta \in I }   ~,
\]
where $\mathbb{W}_{\mu_{\Upsilon,m}}$ denotes Brownian motion on $S_1(\Hi)$ corresponding to a zero-mean Gaussian measure  with covariance operator $\tau_{\Upsilon_m}^2 =\int_0^1\Gamma^{(u)}_{\Upsilon,m}du$ with $\Gamma_{\Upsilon_m}^{u}${ and pseudo-covariance operator $\tilde{\tau}_{\Upsilon_m}^2 =\int_0^1\Sigma^{(u)}_{\Upsilon,m} du$ with $\Gamma^{(u)}_{\Upsilon_m}$ and $\Sigma^{(u)}_{\Upsilon_m}$  as given in \autoref{thm:BMfixu}. }
\end{thm}
We are now in a position to investigate  the 
asymptotic properties of \eqref{eq:avBM}. 
In particular, this follows from \autoref{thm:bigapprox} if we can show 
\[
\lim_{m\to \infty}
 \mathbb{P}\Big(  \mathbb{W}_{\mu_{\Upsilon,m}}  \in B \Big)=\mathbb{P}\Big(  \mathbb{W}_{\mu_{\Upsilon}}  \in B\Big),
\]
This follows if, for all $u \in [0,1]$
\begin{equation}
    \label{eq1b}
\lim_{m\to\infty} \Gamma^{(u)}_{\Upsilon,m}= \Gamma^{(u)}_{\Upsilon} ~~~\text{ and  } ~~\lim_{m\to\infty} \Sigma^{(u)}_{\Upsilon,m}=\Sigma^{(u)}_\Upsilon
\end{equation}
in $S_1(\Hi \otimes \Hi)$, where 
\begin{align*} 
\Gamma_\Upsilon^{(u)}& =8\pi^2 \int_a^b \Big\{(\Upsilon_{u,\omega} \widetilde{\otimes} \Upsilon_{u,\omega}) (\F^{(u,\omega)} \widetilde{\otimes} \F^{(u,\omega)} )
+\mathrm{1}_{\{0,\pi\}}(\omega) \Big[(\Upsilon_{u,\omega}\widetilde{\otimes} \Upsilon_{u,\omega})
({\F}^{(u, \omega)} \widetilde{\otimes}_\top {\F}^{(u, \omega)}) \Big]\Big\}d\omega \tageq \label{eq:var}
\\\Sigma_\Upsilon^{(u)}& =8\pi^2 \int_a^b \Big\{
 \mathrm{1}_{\{0,\pi\}}(\omega) \big[(\Upsilon_{u,\omega} \widetilde{\otimes} \overline{\Upsilon}_{u,\omega})({\F}^{(u, \omega)} \widetilde{\otimes} {\F}^{(u, \omega)})\big]+ 
(\Upsilon_{u,\omega} \widetilde{\otimes} \overline{\Upsilon}_{u,\omega}) (\F^{(u,\omega)} \widetilde{\otimes}_\top \overline{\F}^{(u,\omega)})\Big\}d\omega.
\tageq \label{eq:psvar}
\end{align*}
 Now, $\sup_{u,\omega} \norm{\F^{(u,\omega)}_m}_{S_1} \le \sup_{u,\omega} \sum_{r \in \znum}\norm{\mathcal{C}_{u,r}}_{S_1} < \infty $ for any $m\ge 1$ and \eqref{eq1b}  follows from the dominated convergence theorem. 
 
\subsection{ A general function  $\phi$} \label{sec822}

In the previous subsection we proved that 
\[\Big\{
 M^{1/2}\rho_{T} \Big(
 \big (\mathrm{L}^{a,b}_{U_M}\circ  {\cal G}_\Upsilon (\hat \F)  \big ) (\eta) - \big (
 \mathrm{L}^{a,b}\circ {\cal G}_\Upsilon  ( \F)  \big ) (\eta)  \Big\}_{\eta \in I } \stackrel{\mathcal{D}}{\Longrightarrow} \Big\{\mathbb{W}_{\mu_{\Upsilon}}(\eta)\Big\}_{\eta \in I}~, 
 \tageq \label{eq:BrlimFu2}
 \]
 for the case where ${\cal G}_{\Upsilon ,u,\omega} ( A ) (\eta)    = \Upsilon_{u,\omega} \big( \phi(\eta{A (\eta )}_{u,\omega})\big) $ with $\phi(z)=z$. We conclude the proof considering  the situation 
where the function $\phi$ is analytic on $D$.

In this respect, note that we only have to reconsider 
the proof of \autoref{thm:fullapproximation}
for a general function $\phi$ (the proof of \autoref{thm:BMfixu} and \autoref{thm:bigapprox} remain unchanged), and that the important step 
is the investigation of  the first term of \eqref{eq53},
as the  second term is already proved for a general function $\phi$.
Let 
\[\hat{\Delta}_{u,\omega}(\eta)=\eta(\hat{\F}_{u,\omega}(\eta)- {\F}_{u,\omega})\]
 and consider the set 
$$\Omega_T= \big \{x \in \Omega: \sup_{u \in U_M} \sup_\omega \sup_{\eta} \norm{ \hat{\Delta}_{u,\omega}(\eta) }_{S_1} \le \varepsilon \big \}.
$$
Under the conditions of \autoref{thm:Conv}(ii),  \autoref{thm:maxdevcon} implies $\mathbb{P}(\Omega_T) =1$ for $T$ sufficiently large. Therefore,  by \autoref{thm:PerEx}, we obtain for the  first term in  \eqref{eq53}  
 \begin{align*}
&\frac{\rho_T}{M^{1/2}}\sum_{u \in U_M}\int_a^b  \Upsilon_{u,\omega}\Big( \phi( \eta\hat{\F}_{u,\omega}(\eta))-\phi(\eta {\F}_{u,\omega})\Big)d\omega  \\&= \frac{\rho_T}{M^{1/2}}\Big(
\Big\{\sum_{u \in U_M}\int_a^b \Upsilon_{u,\omega} \Big(\phi^\prime_{ \eta \F_{u,\omega}}\big(\hat{\Delta}_{u,\omega}(\eta)\big) +R_{u,\omega}(\eta)\Big) d\omega\Big\}{\mathrm{1}_{\Omega_T}} \\
& ~~~~
+ \Big\{ \sum_{u \in U_M}\int_a^b\Upsilon_{u,\omega}\Big(\phi( \eta\hat{\F}_{u,\omega}(\eta))-\phi(\eta {\F}_{u,\omega})\Big)d\omega\Big\}{\mathrm{1}_{\Omega^\complement_T}}\Big)
\\
& = \frac{\rho_T}{M^{1/2}}\Big\{\sum_{u \in U_M}\int_a^b\Upsilon_{u,\omega}\Big(\phi^\prime_{ \eta \F_{u,\omega}}\big(\hat{\Delta}_{u,\omega}(\eta)\big) +R_{u,\omega}(\eta) \Big)d\omega\Big\}{\mathrm{1}_{\Omega_T}} + o_p(1).
\end{align*}
Using that we can write \begin{align*}
&\Upsilon_{u,\omega}\Big(\phi^\prime_{\eta\F_{u,\omega}}\big(\hat{\Delta}_{u,\omega}(\eta)\big) \Big)+\Upsilon_{u,\omega}\Big(R_{u,\omega}(\eta) \Big)
 = \Upsilon_{u,\omega}\circ \phi^\prime_{\eta \F_{u,\omega}} \big( \hat{\Delta}_{u,\omega}(\eta)\big)+\Upsilon_{u,\omega}\Big(R_{u,\omega}(\eta) \Big)
\end{align*}
where, by \autoref{thm:PerEx} and \autoref{thm:maxPS1}, 
 $$ \sup_{\eta \in [0,1]}\bignorm{\Upsilon_{u,\omega}\big(R_{u,\omega}(\eta) \big)}_{S_1} =O_p(\sup_{\eta}\bignorm{\Upsilon_{u,\omega}\big(\eta \hat{\Delta}_{u,\omega}(\eta)\big)}^2_{S_1}) =O_p(\bf^{-1} N^{-1})=o_p(\rho^{-1}_T M^{-1/2}),
$$
Furthermore, by assumption $\phi^\prime_{\eta \F_{u,\omega}} = h(\eta) \phi^\prime_{\F_{u,\omega}}$.
We now show that the results for $\phi (z) = z$ are applicable for the first term in the decomposition, while the contribution from the second term is negligible.
Observe that $\phi^\prime_{\F_{u,\omega}} \in \mathfrak{L}(C_{S_1})$ for each  $u,\omega $, and thus $\Upsilon_{u,\omega} \circ\, \phi^\prime_{\F_{u,\omega}}\in \mathfrak{L}(C_{S_1})$ for each  $u,\omega$, since $\Upsilon_{u,\omega}  \in \mathfrak{L}(C_{S_1})$. To verify that \autoref{as:mappings}(i)--(ii) are satisfied
for $\Upsilon_{u,\omega} \circ\, \phi^\prime_{\F_{u,\omega}}$, note that 
\[
\bignorm{\Upsilon_{u,\omega}\circ\, \phi^\prime_{\F_{u,\omega}} \big( \hat{\Delta}_{u,\omega}(\cdot)\big) \Big)}_{C_{S_1}} \le \bignorm{\Upsilon_{u,\omega}\circ\, \phi^\prime_{\F_{u,\omega}}}_{\mathfrak{L}(C_{S_1})} \bignorm{ \hat{\Delta}_{u,\omega}(\cdot)}_{C_{S_1}}~,
\]
where the time-frequency integral of the $p$th power of the first term on the right-hand side 
is bounded by
\[
 \int_0^1 \int_0^\pi \norm{\Upsilon_{u,\omega}}^p_{\mathfrak{L}(C_{S_1})} \norm{ \phi^\prime_{\F_{u,\omega}}}^p_{\mathfrak{L}(C_{S_1})}  d\omega du \le \sup_{A \in \mathbb{D}_1} \norm{ \phi^\prime_{A}}^p_{\mathfrak{L}(C_{S_1})}  \int_0^1 \int_0^\pi \norm{\Upsilon_{u,\omega}}^p_{\mathfrak{L}(C_{S_1})}  d\omega du <\infty
\]
and the set $\mathbb{D}_1$ is defined in \eqref{eq6}. 
To verify that the map satisfies \autoref{as:mappings}(iii), consider the map $F=\phi^\prime: \mathbb{D}_1 \to \mathfrak{L}(S_1(\Hi)), A \mapsto \phi^\prime_A$, which assigns to each element $A \in \mathbb{D}_1$ its Fr{\'e}chet derivative $ \phi^\prime_A$ and let $f:[0,1] \to \mathbb{D}_1, f(u)=\F_{u,\omega}$. Then setting $g:u \mapsto F(f(u)) $ we have
\[
g(u)-g(v)=Dg(v)(u-v)+\frac{1}{2}D^2g(v)(u-v,u-v)+R~,
\]
where
\[Dg(v)(\cdot)
=\phi^{\prime \prime}_{\F_{v,\omega}}(\F^\prime_{v,\omega}(\cdot))\]
and 
\begin{align*}D^2g(v)(\cdot,\cdot)
&=\phi^{\prime \prime \prime}_{\F_{v,\omega}}(\F^\prime_{v,\omega}(\cdot),\F^\prime_{v,\omega}(\cdot))+\phi^{ \prime \prime}_{\F_{v,\omega}}(\F^{\prime\prime}_{v,\omega}(\cdot,\cdot)) ~. 
\end{align*}
Therefore, 
\autoref{as:mappings}(iii) is satisfied since  the involved mappings, such as $\phi^{\prime \prime \prime}_{x}$ at $x=\F_{u_{i,T},\omega}$, are well-defined on the set  $\Omega_T$.
The claim  now follows because the second term in \eqref{eq53} is of order $o_p(1)$; see 
 \eqref{intapprox}.

\subsection{Proof of   \autoref{thm:fullapproximation},  \autoref{thm:BMfixu} and \autoref{thm:bigapprox}} 
\label{sec83}

\subsubsection{Proof of \autoref{thm:fullapproximation}} \label{sec831}

Note that
\begin{equation}
\label{eq53}
\begin{split}
&
M^{1/2}\bignorm{\rho_{T} \Big(
 \mathrm{L}^{a,b}_{U_M}\circ \G_\Upsilon (\hat {\cal F} )  - 
 \mathrm{L}^{a,b}\circ \G_\Upsilon ({\cal F} ) \Big) - \frac{1}{M}\sum_{u \in U_M}\mathcal{Y}^{(u)}_{m,N}(\cdot) }_{C_{S_1}}
\\
& ~~~~~~~~~~~ \le  
 M^{1/2} \bignorm{\rho_{T} \Big(
 \mathrm{L}^{a,b}_{U_M}\circ\G_\Upsilon (\hat {\cal F} ) 
 -
\mathrm{L}^{a,b}_{U_M}\circ\G_\Upsilon ( {\cal F} )  \Big)
 - \frac{1}{M}\sum_{u \in U_M}\mathcal{Y}^{(u)}_{m,N}(\cdot) }_{C_{S_1}}
 \\
 & ~~~~~~~~~~~  ~~~~~~~~~~~ 
 +   M^{1/2} \rho_T\bignorm{\Big( \mathrm{L}^{a,b}_{U_M}\circ \G_\Upsilon ({\cal F} )
- 
 \mathrm{L}^{a,b}\circ \G_\Upsilon ({\cal F} )  \Big)}_{C_{S_1}}
 \end{split}
\end{equation}

We first treat the second term. To this end, define 
$ x_{T,i} =  \frac{i}{M}$ such that 
the set $U_M$ in \eqref{eq:midset} corresponds to the midpoints of the intervals $[  x_{i-1,T} ,x_{i,T} ] $,
that is  $u_{i,T} = ( x_{i-1,T} + x_{i,T})/2 $ 
for $i=1, \ldots, M$. We focus on the analytic case, as the non-analytic case follows from a much simplified argument. 
For some constant $\varepsilon>0$, define the set
$$
E_T =  \bigcap _{i=1}^M  \Big\{\sup_{u \in [{x_{i-1,T}},{x_{i,T}] } }\sup_{\omega}\norm{\cdot({\F}_{u,\omega}-{\F}_{u_{i,T},\omega})}_{C_{S_1}}\le \varepsilon\Big\},
$$
We recall that  the mappings $u\mapsto \Upsilon_{u,\omega} $ and $u \mapsto \eta\F_{u,\omega}$ are twice continuously Fr{\'e}chet differentiable maps. Let $f_1:[0,1] \to \mathfrak{L}(S_1)$, $f_1(u)=\Upsilon_{u,\omega}$ and $f_2:[0,1] \to S_1$, $f_2(u)=\eta \F_{u,\omega}$, and consider the bilinear map  $F: \mathfrak{L}(S_1)\times {S_1} \to {S_1}$, defined by $F\big(f_1(u), f_2(u)\big)= f_1(u)\circ \phi(f_2(u))$. Then for $g_{\eta,\omega}=g:u \mapsto F\big(f_1(u), \phi(f_2(u))\big) $ 
we may write (dropping the dependency on $\eta,\omega$)
\begin{align*}
    g(u)-g(v)= 
    Dg(v)(u-v)+\frac{1}{2}D^2g(v)(u-v,u-v)+R~,
\end{align*}
where 
\[
Dg(v)(u-v)= D_1 F\big(\Upsilon^\prime_v[u-v], \phi(\F_{v})\big)+
D_2 F\big(\Upsilon_{v}, \phi^\prime_{\F_v}(\F^\prime_v[u-v])\big)
\]
 and
\begin{align*}
D^2g(v)(u-v,u-v)&= \frac{1}{2} D_{12} F\big(\Upsilon^\prime_v[u-v], \phi^\prime_{\F_{v}}(\F^\prime_v[u-v])\big)+
\frac{1}{2} D_{21} F\big(\Upsilon^\prime_v[u-v], \phi^\prime_{\F_v}(\F^\prime_v[u-v])\big)
\\&+
\frac{1}{2}D_{22} F\big(\Upsilon_v, \phi^{\prime\prime}_{\F_v}(\F^\prime_v[u-v],\F^\prime_v[u-v])+ \phi^{\prime}_{\F_v}(\F^{\prime\prime}_v[u-v,u-v])\big).
\end{align*}
Using bilinearity of $F$ and of the second derivative maps, and observing the definitions 
\eqref{eq:Lin} and \eqref{eq:Lin2}, Assumptions \eqref{as:smooth} and \eqref{as:mappings} and  \autoref{thm:PerEx} yield
\begin{align*}
&\bignorm{\Big( \mathrm{L}^{a,b}_{U_M}\circ \G_\Upsilon ({\cal F} )
- 
 \mathrm{L}^{a,b}\circ \G_\Upsilon ({\cal F} )  \Big)}_{C_{S_1}} 1_{E_T}
\\& \le  \int_a^b \sup_{\eta}\bignorm{\sum_{i=1}^{M} \int_{x_{i-1,T}}^{x_{i,T}}  g(u)-g(u_{i,T}) du}_{{S_1}} d\omega
\\& \le \int_a^b \sup_{\eta}\Bignorm{\sum_{i=1}^{M} \int_{x_{i-1,T}}^{x_{i,T}} (u-u_{i,T})Dg(u_{i,T})+\frac{1}{2}(u-u_{i,T})^2D^2g(u_{i,T}) du\, }_{{S_1}} d\omega
+ \norm{R}_{C_{S_1}}=O(M^{-2})~,
\end{align*}
where $O(\norm{R}_{C_{S_1}})=O(M^{-2})$, and where it was used that  $\int_{x_{i-1,T}}^{x_{i,T}} (u-u_{i,T}) du=0$ and $\int_{x_{i-1,T}}^{x_{i,T}} (u-u_{i,T})^2 du =O(M^{-3})$. 
 Finally, the map $u \mapsto \eta\F_{u,\omega}$ is in particular Lipschitz uniformly in $\omega, \eta$ w.r.t. $\norm{\cdot}_{S_1}$ which implies $\mathrm{1}_{E^\complement_T}=0$ for sufficiently large $T$. Thus, 
\[
\lim_{T\to \infty}M^{1/2}\rho_T \bignorm{\Big( \mathrm{L}^{a,b}_{U_M}\circ \G_\Upsilon ({\cal F} )
- 
 \mathrm{L}^{a,b}\circ \G_\Upsilon ({\cal F} )  \Big)}_{C_{S_1}} =O\Big(\frac{N^{1/2}\bf^{1/2}}{M^{3/2}}\Big)=o(1). \tageq \label{intapprox}
\]

Next, we will show for the first term in \eqref{eq53} that 
\begin{align*}
& \lim_{m\to \infty} \limsup_{T\to \infty}\, \mathbb{P}\Big( \bignorm{\rho_{T}M^{1/2} \Big(
\mathrm{L}^{a,b}_{U_M}\circ  \G_\Upsilon (\hat {\cal F} )
 -
\mathrm{L}^{a,b}_{U_M}\circ \G_\Upsilon ( {\cal F} )  \Big)
 - \frac{1}{M^{1/2}}\sum_{u \in U_M}\mathcal{Y}^{(u)}_{m,N}(\cdot) }_{C_{S_1}} 
> \epsilon\Big)=0. \tageq \label{eq:fullbigapprox}
\end{align*}

Recall the notation \eqref{eq:VNTfix} and define
\[
\ldmi{}{u}{}{N,\flo{\eta N}} =\sum_{t=2}^{\flo{\eta N}} V^{(u,\omega)}_{m,N,t}
+\big(\eta N - \flo{\eta N}\big) V^{(u,\omega)}_{m,N,\flo{\eta N}+1}. \tageq \label{eq:Mar}
\]
Observe that we can write
\begin{equation}
 \mathcal{Y}^{(u)}_{m,N}(\eta) = \frac{1}{\mathcal{\Phi}_{\bf}} \int^b_a \Upsilon_{u,\omega}(\ldmi{}{u}{}{N,\flo{\eta N}}+ \ldmi{\dagger}{u}{}{N,\flo{\eta N}})d\omega ~.
\end{equation}
Then using the triangle inequality and Jensen's inequality, 
\begin{align*}
&\bignorm{\rho_{T}M^{1/2} \Big(
\mathrm{L}^{a,b}_{U_M}\circ  \G_\Upsilon (\hat {\cal F} )
 -
\mathrm{L}^{a,b}_{U_M}\circ \G_\Upsilon ( {\cal F} )  \Big)
 - \frac{1}{M^{1/2}}\sum_{u \in U_M}\mathcal{Y}^{(u)}_{m,N}(\cdot) }_{C_{S_1}} 
\\& \le
\frac{1}{M^{1/2}}\int_a^b \bignorm{  \rho_{T}   \sum_{u\in U_M}\Upsilon_{u,\omega}\eta \big(\hat{\F}_{u,\omega}(\eta)-\F_{u,\omega}\big) -  \sum_{u\in U_M} \Phi^{-1}_{\bf}\Upsilon_{u,\omega}\big(   \ldmi{}{u}{}{N,\flo{\eta N}}+\ldmi{\dagger}{u}{}{N,\flo{\eta N}} \big) }_{C_{S_1}}  d\omega
\\& \le \frac{1}{M^{1/2}}\int_a^b \bignorm{  \sum_{u\in U_M} \Big[ \rho_{T} \Upsilon_{u,\omega}\eta \big(\hat{\F}_{u,\omega}(\eta)-\F_{u,\omega}\big) -  
   \Phi^{-1}_{\bf} \Upsilon_{u,\omega} \flo{\eta N}\big( \hat{\F}_{u,\omega}(\eta)- \F_{u,\omega}\big) \Big]}_{C_{S_1}}  d\omega
\\& + \frac{1}{M^{1/2}}\int_a^b \bignorm{     \sum_{u\in U_M}\Big[  \Phi^{-1}_{\bf} \Upsilon_{u,\omega} \flo{\eta N}\big( \hat{\F}_{u,\omega}(\eta)- \F_{u,\omega}\big)-  \Phi^{-1}_{\bf}\Upsilon_{u,\omega}\big(   \ldmi{}{u}{}{N,\flo{\eta N}}+\ldmi{\dagger}{u}{}{N,\flo{\eta N}}\big)  \Big]}_{C_{S_1}}  d\omega
\end{align*}
For the first term, note that under \autoref{as:Weights}
\begin{align*}
\sup_{\eta \in (1/N,1]}\Big \vert \frac{1}{\flo{\eta N}}\eta \rho_T - \Phi^{-1}_{\bf}\Big \vert  
&=\sup_{\eta \in (1/N,1]}\Big \vert \frac{\eta N}{\flo{\eta N}}\frac{1}{N}\rho_T - \Phi^{-1}_{\bf}\Big \vert 
 \\ &
=O(\frac{1}{N}\frac{\rho_T}{N})+\Big \vert \frac{\rho_T}{N} - \Phi^{-1}_{\bf}\Big \vert 
 = O(\bf^{1/2}  N^{-3/2}) +o(\bf^{1/2} N^{-1/2}) ,
\end{align*}
where we used that $\Phi^2_{\bf} = \frac{N \kappa_f}{\bf}(1+o(1))$ and that
\[
\sup_{\eta \in (1/N,1]}|\frac{\eta N}{\flo{\eta N}}-1 | = O(1/N). \tageq \label{eq:etaNapp}
\]
Then, 
\begin{align*}
& \frac{1}{M^{1/2}}\int_a^b \bignorm{  \sum_{u \in U_M}  \rho_{T} \Upsilon_{u,\omega} \big(\eta \hat{\F}_{u,\omega}(\eta)-\eta \F_{u,\omega}\big) -  
   \Phi^{-1}_{\bf} \flo{\eta N}\Upsilon_{u,\omega} \big( \hat{\F}_{u,\omega}(\eta)- \F_{u,\omega}\big) }_{C_{S_1}}  d\omega
\\& \le  \frac{(b-a)}{M^{1/2}}
\sup_{\eta \in (1/N,1]}\Big\vert \frac{1}{\flo{\eta N}} \rho_T - \Phi^{-1}_{\bf}\Big\vert \int_a^b
\bignorm{  \sum_{u \in U_M}   \flo{\eta N}\Upsilon_{u,\omega} \big(\hat{\F}_{u,\omega}(\eta)- \F_{u,\omega}\big) }_{C_{S_1}} d\omega\overset{p}{\to} 0  ~, 
\end{align*}
as $N \to \infty$. where the last line follows from \autoref{thm:maxPS1}. The second term converges to probability by the following lemma, the proof of which is deferred to \autoref{sec:Cbounds}.
\begin{lemma} \label{lem:martapproxprob}
Under the conditions of \autoref{thm:Conv}
\begin{align*}
\lim_{m\to \infty}
\limsup_{T \to \infty}  \frac{1}{M^{1/2}}\int_a^b \bignorm{     \sum_{u\in U_M}\Phi^{-1}_{\bf} \Upsilon_{u,\omega}\Big[   \flo{\eta N}\big( \hat{\F}_{u,\omega}(\eta)- \F_{u,\omega}\big)- \big(   \ldmi{}{u}{}{N,\flo{\eta N}}-\ldmi{\dagger}{u}{}{N,\flo{\eta N}}\big)  \Big]}_{C_{S_1}}  d\omega=0.
\end{align*}
\end{lemma}

\subsubsection{Proof of \autoref{thm:BMfixu}} \label{sec832}
To ease notation, we write $V=S_1(\Hi)$ in the following. First, we show convergence of the  finite-dimensional distributions. For this purpose, let $\mathcal{Y}^{(u)}_{m,N}$ be as defined in \eqref{eq:BMcont}, let $0=\eta_0< \eta_1 < \ldots < \eta_k =1$ with fixed $k \in \mathbb{N}$, and define the evaluation functionals $\pi^V_{\eta_1, \ldots, \eta_k}: C_V[0,1] \to V^k$ by
\[
\xi \to (\xi_{\eta_1}, \ldots, \xi_{\eta_k})~. 
\]
We will show that that there exists a zero-mean Gaussian measure $\mu^{(u)}_{\Upsilon,m}$ on $V$ with $\int_V |f(x)|^2 d\mu^{(u)}_{\Upsilon,m}(x) = \Gamma^{(u)}_{\Upsilon,m}(f)$ and $\int_V f^2(x) d\mu^{(u)}_{\Upsilon,m}(x)=\Sigma^{(u)}_{\Upsilon,m}(f)$ for each $f \in V^{\prime}$ such that for any $0=\eta_0< \eta_1 < \ldots < \eta_k =1$, $k \in \nnum$
\begin{align}
\label{eq50}
\mathbb{P} \circ (\pi^V_{\eta_1, \ldots, \eta_k}\circ \mathcal{Y}^{(u)}_{m,N})^{-1} \Rightarrow  \mathbb{W}_{\mu^{(u)}_{\Upsilon,m}} \circ (\pi^V_{\eta_1, \ldots \eta_k})^{-1} \quad \text{ as $N \to \infty$}
\end{align}
where $ \mathbb{W}_{\mu^{(u)}_{\Upsilon,m}} $ is the Wiener measure on $C_V[0,1]$ induced by the mean zero Gaussian measure $\mu^{(u)}_{\Upsilon,m}$. 
To do so, we make use of the following Lemma, which is proved in Section \ref{sec:supstat}.
\begin{thm} \label{lem:fidis_dis}
Let $\{\mathcal{Y}^{(u)}_{m,N}(\eta)\}_{\eta \in [0,1]}$ be as defined in \eqref{eq:BMcont} and suppose the conditions of \autoref{thm:BMfixu} hold. Then, for any $0=\eta_0< \eta_1 < \ldots < \eta_k =1$, $k \in \mathbb{N}$ and $f \in V^{\pr}$, 
\[
\mathbb{P} \circ (\pi^{\cnum}_{\eta_1, \ldots, \eta_k}\circ f(\mathcal{Y}^{(u)}_{m,N}))^{-1} \Rightarrow  \mathbb{W}^{\cnum}_{\mu^{(u)}_{\Upsilon,m}} \circ (\pi^\cnum_{\eta_1, \ldots \eta_k})^{-1} \quad \text{ as $N \to \infty$}
\]
where $ \mathbb{W}^{\cnum}_{\mu^{(u)}_{\Upsilon,m}}$ is the Wiener measure on $C_{\cnum}[0,1]$ induced by the zero-mean Gaussian measure on $\cnum$ with covariance $\Gamma^{(u)}_{\Upsilon,m}(f)$ and pseudocovariance $\Sigma^{(u)}_{\Upsilon,m}(f)$ for all 
$f \in V^{\pr}$. 
\end{thm}
We remark first that existence of Gaussian measures on $V$ was established in \cite{Kuelbs73}.
Note then that it  follows immediately from \autoref{lem:fidis_dis}, that
for a proof of \eqref{eq50}  it is sufficient to show that the sequence 
\[
\Big\{ \mathbb{P} \circ (\pi^V_{\eta_1, \ldots, \eta_k}\circ \mathcal{Y}^{(u)}_{m,N})^{-1} \Big\}_{N \ge 1} 
\]
is tight on $V^k$. Now, a sequence of probability measures is tight on $V^k$ if and only if each of the $k$ sequences of marginals is tight on $V$ for any fixed $\eta\in[0,1]$.  It is thus sufficient to show that 
\[
\Big\{ \mathbb{P} \circ \big(\pi^V_{\eta}\circ \mathcal{Y}^{(u)}_{m,N}\big)\Big\}_{N \ge 1} 
\]
is tight. It follows from \autoref{lem:fidis_dis} and \autoref{lem:schauder} that for any finite $K$ 
\[
\Big\{ \mathbb{P} \circ (\Pi_{K} \circ \pi^V_{\eta}\circ \mathcal{Y}^{(u)}_{m,N})^{-1} \Big\}_{N \ge 1} 
\]
is tight, where $\Pi_K$ is defined in \eqref{eq51}. 
It therefore remains to show that 
\[
\lim_{K\to \infty}\sup_N\mathbb{P}\big(\sup_{\eta \in [0,1]}\bignorm{\mathcal{Y}^{(u)}_{m,N}(\eta)-\Pi_K \big(\mathcal{Y}^{(u)}_{m,N}(\eta)\big)}_{S_1}> \epsilon \big)=0. \tageq \label{eq:tightness}\]
Recall that 
$\big\{\mathcal{Y}^{(u)}_{m,N}(\eta)\big\}_{N \ge 1}$ is a martingale in $\op^2_{S_1(\Hi)}$. 
Convexity of the norm implies therefore that $\bignorm{\mathcal{Y}^{(u)}_{m,N}(\eta)-\Pi_K \big(\mathcal{Y}^{(u)}_{m,N}(\eta)\big)}_{S_1}$ is a nonnegative submartingale in $\op^2_{\cnum}$.
Hence, by Doob's inequality and the second part of \autoref{lem:Burkh},
\begin{align*}
\mathbb{P}\big(\sup_{\eta \in [0,1]}\bignorm{\mathcal{Y}^{(u)}_{m,N}(\eta)-\Pi_K \big(\mathcal{Y}^{(u)}_{m,N}(\eta)\big)}_{S_1}> \epsilon \big) 
& \le \frac{c}{\epsilon^2}\E\bignorm{\mathcal{Y}^{(u)}_{m,N}(1)-\Pi_K \big(\mathcal{Y}^{(u)}_{m,N}(1)\big)}^2_{S_1} 
\\&  \le K_2^2 \frac{c}{\epsilon^2} \Big( \sum_{r \in \nnum}\sqrt[2]{\sum_{t=1}^N \bignorm{\big(W^{(u)}_{m,N,t}-\Pi_K( W^{(u)}_{m,N,t})\big)(e_r)}^2_{\Hi,2} }\Big)^{2}
\\& \le K_2^2\frac{c}{\epsilon^\delta}{K}(\delta)~,
\end{align*}
where we used in the last inequality that for any $\delta>0$ there exists a positive integer $K^{\star}$ such that
\[
 \Big(\sum_{r \in \nnum}\sqrt[2]{\sum_{t=1}^N \bignorm{\big(W^{(u)}_{m,N,t}-\Pi_K( W^{(u)}_{m,N,t})\big)(e_r)}^2_{\Hi,2} }\Big)^2 <\delta \quad K \ge K^\star  ~. \tageq \label{eq:unif}
\]
To see this, we write
\begin{align*}
 &\sum_{r \in \nnum}\sqrt[2]{\sum_{t=1}^N \bignorm{\big(W^{(u)}_{m,N,t}-\Pi_K( W^{(u)}_{m,N,t})\big)(e_r)}^2_{\Hi,2} } 
 \\& =\sum_{r \le K}\sqrt[2]{\sum_{t=1}^N \bignorm{\big(W^{(u)}_{m,N,t}-\Pi_K( W^{(u)}_{m,N,t})\big)(e_r)}^2_{\Hi,2} }
 + \sum_{r >K}\sqrt[2]{\sum_{t=1}^N \bignorm{\big(W^{(u)}_{m,N,t}-\Pi_K( W^{(u)}_{m,N,t})\big)(e_r)}^2_{\Hi,2} }, \tageq \label{eq:unif2}
\end{align*}
and recall that $\Pi_K \to I$ as $K \to \infty$ in the strong operator topology.  This implies for fixed $r$ that  $\norm{(W^{(u)}_{m,N,t}-\Pi_K ( W^{(u)}_{m,N,t})(e_r)}^2_{\Hi,2} \to 0$ as $K \to \infty$. Hence, for \eqref{eq:tightness} to hold it suffices to show  that \[\sup_N \big(\sum_{r \in \nnum} \sqrt[2]{\sum_{t=1}^N \E\bignorm{W^{(u)}_{m,N,t}(e_r)}^2_{\Hi} }\big)^2<\infty \tageq \label{eq:supN}\] for any fixed $m$, since then the first term in \eqref{eq:unif2} converges to zero by the dominated convergence theorem, and the second can be immediately seen to converge to zero. Using  Jensen's inequality,  Cauchy Schwarz's inequality and \autoref{lem:Burkh}
\begin{align*}
  \E\norm{W^{(u)}_{m,N,t}(e_r)}^2_{\Hi}& 
  \le C\Phi^{-2}_{\bf} \int_a^b \E \bignorm{\Upsilon_{u,\omega} (V^{(u,\omega)}_{m,N,t}+V^{\dagger(u,\omega)}_{m,N,t})(e_r)}^2_{\Hi} d\omega
\\&  \lesssim \Phi^{-2}_{\bf} \int_a^b \norm{\Upsilon_{u,\omega}}^2_\infty\|\dmpi{u}{\omega}{N}{t}\|^2_{\hi,4}
\|\sum_{s \ne t}  \tilde{w}^{(\omega)}_{\bf,t,s}\dmpi{u}{\omega}{N}{s}(e_r)\|^2_{\cnum,4} d\omega
\\& \lesssim  \Phi^{-2}_{\bf} \sum_{t,s=1}^{N} w^2_{\bf,t,s} \sup_{u\in [0,1]}\sup_{\omega\in [a,b]}\bignorm{\dmpi{u}{\omega}{}{0} }^2_{\hi,4}  \bignorm{\dmpi{u}{\omega}{}{0}(e_r) }^2_{\cnum,4}.  
\end{align*}
Thus,
\begin{align*}
\Big(\sum_{r \in \nnum}\sqrt{\sum_{t=1}^{N} \E\norm{W^{(u)}_{m,N,t}(e_r)}^2_{\Hi}}\Big)^2 & =C \sup_{u,\omega}\bignorm{\dmpi{u}{\omega}{}{0} }^2_{\hi,4} \Big(\sum_{r \in \nnum}\sqrt{  \sup_{u,\omega} \bignorm{\dmpi{u}{\omega}{}{0}(e_r) }^2_{\cnum,4}}\Big)^2  <\infty,
\end{align*}
which shows \eqref{eq:supN} for any fixed $m$. 
Therefore \eqref{eq:unif} now follows,
which implies \eqref{eq50}.  

Finally, the following lemma shows the sequence $\{ \mathcal{Y}^{(u)}_{m,N}\}_{N \ge 1}$ is tight in $C_{S_1}$. This completes the proof of Theorem 
\ref{thm:BMfixu}.
\begin{lemma}
Under the conditions of \autoref{thm:Conv}, we have for each $\epsilon >0$,
\[
\lim_{h \to 0} \limsup_{N \to \infty} \pro\Big(\sup_{|\eta_1-\eta_2|<h} \bignorm{\mathcal{Y}^{(u)}_{m,N}(\eta_1)-\mathcal{Y}^{(u)}_{m,N}(\eta_2)}_{S_1} > \epsilon\Big) = 0.
\]
\end{lemma}
\begin{proof}
The argument is standard and mimics the one provided in \cite{Walk77}, where the second part of his proof should be replaced with the the first part of this proof, and which further relies on a submartingale inequality that is also valid in our case \citep[][Lemma 5]{Walk77}, and a slightly adapted argument of \citep[][p. 222]{Part67}. Details are omitted for the sake of brevity.
\end{proof}

\subsubsection{Proof of \autoref{thm:bigapprox}} \label{sec833}
Set $L=N-m$, and observe that $\mathcal{Y}^{(u)}_{m,L}(1)$ is measurable with respect to the $\sigma$-algebra
 \[
 \sigma(\epsilon_{\flo{uT}-\flo{L/2}+1-m}, \epsilon_{\flo{uT}-\flo{L/2}+2-m}, \ldots, \epsilon_{\flo{u T}-\flo{ L/2}+L}),
\]
 and forms an  $L+m$-dependent sequence. Consequently $\mathcal{Y}^{(u_i)}_{m,L}(k)$ and $\mathcal{Y}^{(u_j)}_{m,L}(k)$, for $u_i\ne u_j \in U_M$, and any $1 \le k, k^\prime\le L$ are independent. We make use of the following auxiliary lemma, of which the proof is postponed to  Section \ref{sec:supstat}. 
 
\begin{lemma} \label{lem:approxindep}
Under the conditions of \autoref{thm:conv_an}, for any fixed $m$,
\begin{equation}
    \label{errorLM}
\lim_{T \to \infty} \mathbb{P}\Big(\bignorm{\frac{1}{M^{1/2}} \sum_{u\in U_M} \mathcal{Y}^{(u)}_{m,N}-\frac{1}{M^{1/2}}  \sum_{u\in U_M} \mathcal{Y}^{(u)}_{m,L}}_{C_{S_1}} > \epsilon\Big)=o(1).
\end{equation}
\end{lemma}
Note that the second term in the norm 
of \eqref{errorLM} is  an average of independent processes that satisfy the conditions of \autoref{thm:BMfixu}. Thus, for any closed set $B \subset C_{S_1}$,
\begin{align*}
\limsup_{T\to \infty}
 \,\mathbb{P}\Big( \frac{1}{M^{1/2}}\sum_{u \in U_M}\mathcal{Y}^{(u)}_{m,N}   \in B\Big)
& \le  
\lim_{T \to \infty} \mathbb{P}\Big(\bignorm{\frac{1}{M^{1/2}} \sum_{u\in U_M} \mathcal{Y}^{(u)}_{m,N}-\frac{1}{M^{1/2}}  \sum_{u\in U_M} \mathcal{Y}^{(u)}_{m,L}}_{C_{S_1}} > \epsilon\Big)
\\&+ 
\limsup_{T\to \infty} \,
 \mathbb{P}\Big( \frac{1}{M^{1/2}}\sum_{u \in U_M}\mathcal{Y}^{(u)}_{m,L}   \in B_{\epsilon}\Big)
 \\& 
 =
  \mathbb{P}\Big(\Big\{\mathbb{W}_{\mu_{\Upsilon_m}}(\eta)\Big\}_{\eta \in I}\in B_{\epsilon} \Big),  
\end{align*}
where $B_{\epsilon} =\{x: \|x-y\|_{C_{S_1}} \le \epsilon, y \in B_{\epsilon} \}$. The result now follows.

\subsection{Proofs of \autoref{lem:martapproxprob}, \autoref{lem:fidis_dis},
and \autoref{lem:approxindep}} \label{sec:supstat}

\begin{proof}[Proof of \autoref{lem:martapproxprob}]
To prove that
\begin{align*}
\lim_{m\to \infty} \lim_{N \to \infty} \mathbb{P}\Big(\frac{1}{M^{1/2}}\int_a^b \bignorm{     \sum_{u\in U_M}\Phi^{-1}_{\bf} \Upsilon_{u,\omega}\Big[   \flo{\eta N}\big( \hat{\F}_{u,\omega}(\eta)- \F_{u,\omega}\big)- \big(   \ldmi{}{u}{}{N,\flo{\eta N}}+\ldmi{\dagger}{u}{}{N,\flo{\eta N}}\big)  \Big]}_{C_{S_1}}  d\omega>\epsilon \Big) =0.
\end{align*}
We  decompose the term inside the norm as 
\begin{align*}
&\E\Big( \sup_{\eta \in [0,1]}\bignorm{  \sum_{u\in U_M}\Phi^{-1}_{\bf} \Upsilon_{u,\omega}\Big[   \flo{\eta N}\big( \hat{\F}_{u,\omega}(\eta)- \F_{u,\omega}\big)- \big(   \ldmi{}{u}{}{N,\flo{\eta N}}+\ldmi{\dagger}{u}{}{N,\flo{\eta N}}\big)  \Big]  }_{S_1}\Big)^q
 \tageq\label{eq:approxall} \\& \le
 2^{q-1} \E\Big( \sup_{\eta \in [0,1]} \bignorm{ \sum_{u\in U_M}\Phi^{-1}_{\bf} \Upsilon_{u,\omega}\flo{\eta N} \big(\hat{\F}_{u,\lambda}(\eta)
 -\tilde{\F}_{u,\omega}(\eta)\big)}_{S_1} \Big)^q  \tageq \label{eq:approxloc}
 \\&+
 2^{q-1} \Big( \sup_{\eta \in [0,1]}\bignorm{ \sum_{u\in U_M}\Phi^{-1}_{\bf} \Upsilon_{u,\omega}\flo{\eta N} \big(\E\tilde{\F}_{u,\omega}(\eta)- \F^{(u,\omega)}\big)}_{S_1} \Big)^q  \tageq \label{eq:approxtrue}
  \\&+
 2^{q-1}\E\Big( \sup_{\eta \in [0,1]}\bignorm{  \sum_{u\in U_M}\Phi^{-1}_{\bf} \Upsilon_{u,\omega}\Big[   \flo{\eta N}\big( \tilde{\F}_{u,\omega}(\eta)- \E\tilde{\F}_{u,\omega}\big)- \big(   \ldmi{}{u}{}{N,\flo{\eta N}}+\ldmi{\dagger}{u}{}{N,\flo{\eta N}}\big)  \Big]  }_{S_1}\Big)^q \tageq \label{eq:approxmart}
\end{align*}
where $\hat{\F}_{u,\omega}(\eta)$ is defined in \eqref{eq:Fint}, 
and $\tilde{\F}_{u,\omega}(\eta)$ is defined  as in  \eqref{eq:Fint} but where the nonstationary process is replaced with its auxiliary counterpart. 
The result then follows from Markov's inequality, Tonelli's theorem, \autoref{lem:approxloc}, \autoref{lem:bias}, and
\autoref{thm:approx}, respectively.

\end{proof}

\begin{proof}[Proof of \autoref{lem:fidis_dis}]

We start with the following result.
\begin{lemma}\label{lem:fidis_dense}
Let $X, Y$ be $V$-valued random variables over a probability space $(\Omega, \mathcal{A}, \pro)$ and let $Q \subseteq V^\pr$ be sequentially weak$^\star$-dense. Assume $\mu_{q(X)} = \mu_{q(Y)}$ for all $q \in Q$, then $\mu_{v^\pr(X)} = \mu_{v^\pr(Y))}$ for all $v^\pr \in V^\pr$.
\end{lemma}
\begin{proof}
Let $v^\pr \in V^\pr$. Since the set $Q$ is sequentially weak$^\star$-dense in $V^\pr$, there exists a sequence $\{v^\pr_n\} \in Q$ such that $\lim_{n\to \infty} v^\pr_n(v) = v^\pr(v)$ for all $v \in V$, which implies for $V$-valued random variables $X$ and $Y$ that 
\[
v^\pr_n(X) \overset{a.s.}{\to} v^\pr(X) \text{ and } v^\pr_n(Y) \overset{a.s.}{\to} v^\pr(Y)  \quad \text{as $ n\to \infty$}. \tageq \label{eq:asconv}
\]
By assumption $\mu_{v^\pr_n(X)} = \mu_{v^\pr_n(Y)}$ since $v^\pr_n \in Q$  $\forall n$.  Hence, \eqref{eq:asconv}  implies that $\mu_{v^\pr(X)} = \mu_{v^\pr(Y)} \forall v^\pr \in V^\pr.$
\end{proof}
We will thus prove the following conditions for all $f\in Q$, where $Q \subseteq V^\pr$ is sequentially weak$^\star$-dense  
\begingroup
\addtolength{\jot}{-0.3em}
\begin{align*}
&\label{itm:lindeberg} \bullet~ \forall \epsilon>0,  \, \sum_{t=1}^{N} \E\Big[ \Big\vert\big(W^{(u)}_{m,N,t}\big)\Big\vert^2 \mathrm{1}_{|f(W^{(u)}_{m,N,t})|\ge \epsilon} \Big\vert |\G_{\tu{N}{u},t-1}\Big] \overset{p}{\to} 0 \quad \text{ as } N \to \infty; \hspace{80pt} \tag{Cond.1} \\
&
 \bullet~~  \label{itm:concov}  \sum_{t=1}^{\flo{\eta N}} \E\Big[ \Big\vert f\big(W^{(u)}_{m,N,t}\big)\Big \vert^2 \Big\vert |\G_{\tu{N}{u},t-1}\Big] \overset{p}{\to} \eta \Gamma^{(u)}_{\Upsilon,m}(f)\quad \text{ as } N \to \infty  \tag{Cond.2};\\&
 \bullet~~  \label{itm:concovp}  \sum_{t=1}^{\flo{\eta N}} \E\Big[ f\big(W^{(u)}_{m,N,t}\big) \overline{f\big(\overline{W}^{(u)}_{m,N,t}\big)} \Big\vert |\G_{\tu{N}{u},t-1}\Big] \overset{p}{\to} \eta \Sigma^{(u)}_{\Upsilon,m}(f)\quad \text{ as }N \to \infty \tag{Cond.3};\\&
 \bullet~~  \label{itm:mds} \sum_{t=1}^{\flo{\eta N}} \Big|\E\Big[f\big(W^{(u)}_{m,N,t}\big) \Big\vert |\G_{\tu{N}{u},t-1}\Big]\Big| \overset{p}{\to} 0\quad \text{ as } N \to \infty, \tag{Cond.4}
\end{align*}
\endgroup
where we used the notation 
$\G_{\tu{N}{u},t}= \G_{\tu{N}{u}+t}$. 
The result then follows from \autoref{lem:fidis_dense} and \cite{McL74}. Firstly, it may be verified using \autoref{lem:Burkh}(ii) that  $\Big\{W^{(u)}_{m,N,t}\Big\}_{t}$ is a martingale difference sequence in $\op^2_{S_1}$ with respect to the filtration $\{\G_{\tu{N}{u},t}\}_t$, for fixed $u \in [0,1]$ and any $N \in \mathbb{N}$. \eqref{itm:mds} is therefore immediate. Next, we verify \eqref{itm:concov} and \eqref{itm:concovp}. Firstly, for fixed $m \ge 1$, and $\gamma>2$ 
\begin{align*}
&\E\Big(\max_{1\le k\le N}\bignorm{\sum_{t=2}^{k} \dmpi{u}{\omega}{N}{t}  \otimes \sum_{s=t-4m+1 \vee 1}^{t-1} \tilde{w}_{\bf,t,s}^{(\omega)} \dmpi{u}{\omega}{N}{s}}_{S_1}\Big)^\gamma = o(\mathcal{\Phi}^\gamma_{\bf}). \tageq \label{eq:boundfirst}
\end{align*}
By \autoref{lem:Burkh}(ii)
\begin{align*}
\bignorm{\sum_{t=2}^{k} \dmpi{u}{\omega}{N}{t}  \otimes \sum_{s=t-4m+1 \vee 1}^{t-1} \tilde{w}_{\bf,t,s}^{(\omega)} \dmpi{u}{\omega}{N}{s}}_{S_1,2}^2 
& \le 
\Big(\sum_{r \in \nnum} \sqrt[2]{\sum_{t=2}^{k}\Bignorm{\Big( \dmpi{u}{\omega}{N}{t}  \otimes \sum_{s=t-4m+1 \vee 1}^{t-1} \tilde{w}_{\bf,t,s}^{(\omega)} \dmpi{u}{\omega}{N}{s}\Big)(e_r)}_{\hi,2}^2}\Big)^2.
\end{align*}
Then Hölder's inequality, \autoref{lem:Burkh}(i) and elementary calculations yield
\begin{align*}
&\sum_{t=2}^{k} \Bignorm{\Big(\dmpi{u}{\omega}{N}{t}  \otimes \sum_{s=t-4m+1 \vee 1}^{t-1} \tilde{w}_{\bf,t,s}^{(\omega)} \dmpi{u}{\omega}{N}{s}\Big)(e_r)}_{\hi,2}^2
\le 2 K^4_{4} \norm{D^{(u,\omega)}_{0}}^2_{\hi,4}  \norm{D^{(u,\omega)}_{0}(e_r)}^2_{\cnum,4}O\Big( (4m)^2  +k 4m\Big).
\end{align*}
Hence, 
\[
\bignorm{\sum_{t=2}^{k} \dmpi{u}{\omega}{N}{t}  \otimes \sum_{s=t-4m+1 \vee 1}^{t-1} \tilde{w}_{\bf,t,s}^{(\omega)} \dmpi{u}{\omega}{N}{s}}_{S_1,2}^2 \le  2 K^4_{4} \norm{D^{(u,\omega)}_{0}}^2_{\hi,4}  (\sum_{r \in \nnum}\norm{D^{(u,\omega)}_{0}(e_r)}_{\cnum,4})^2 g(1,k) ~,
\]
where $g(b,k) =(4m)^2  +\sum_{t=b+1}^{b+k} 4m$,
which can be easily seen to satisfy
$ g(b,k) + g(b+k,l) \leq g(b,k+l)
$
 for all $b \ge 0$ and $1\le k < k+l$. Therefore 
\eqref{eq:boundfirst}  follows  from Theorem 1 in \cite{Mor76}, and thus this term is negligible in the verification of \eqref{itm:lindeberg}-\eqref{itm:concovp}.  
It follows from Jensen's, Hölder's inequality and \autoref{as:mappings} that it suffices to focus on verifying the latter conditions with 
\[
\check{W}^{(u)}_{m,N,t}=
\int_a^b \Upsilon_{u,\omega}(\check{V}^{(u,\omega)}_{m,N,t} +\check{V}^{\dagger(u,\omega)}_{m,N,t} )
d \omega 
\tageq \label{eq:4mapprox}
\]
instead of  ${W}^{(u)}_{m,N,t}$, 
where $\check{V}^{(u,\omega)}_{N,t}= \dmpi{u}{\omega}{N}{t} \otimes \sum_{s=1}^{t-4m} \tilde{w}_{\bf,t,s}^{(\omega)} \dmpi{u}{\omega}{N}{s}$. 
Now,
\begin{align*}
\E \Big[\bignorm{\check{W}^{(u)}_{m,N,t}}^2_{S_1}| \G_{\tu{N}{u},t-1}\Big] =
\sum_{l=1}^{m}P_{\tu{N}{u},t-l}\Big(\bignorm{\check{W}^{(u)}_{m,N,t}}^2_{S_1}\Big)+\E \Big[\bignorm{\check{W}^{(u)}_{m,N,t}}^2_{S_1}| \G_{\tu{N}{u},t-m-1}\Big]. 
\end{align*}
Orthogonality of the projections, \autoref{lem:Burkh}(i), the contraction property of the conditional expectation
\begin{align*}
\Bignorm{\sum_{t=1+4m}^{N} \sum_{l=1}^{m}P_{\tu{N}{u},t-l}\Big(\bignorm{\check{W}^{(u)}_{m,N,t}}^2_{S_1}\Big)}^2_{\cnum,2}
&\le \Big(\sum_{l=1}^{m}   \Bignorm{ \sum_{t=1+4m}^{N} P_{\tu{N}{u},t-l}\Big(\bignorm{\check{W}^{(u)}_{m,N,t}}^2_{S_1}\Big)}_{\cnum,2} \Big)^2
\\&\le \Bigg(\sum_{l=1}^{m} \Big( \sum_{t=1+4m}^{N}   \Bignorm{ P_{\tu{N}{u},t-l}\Big(\bignorm{\check{W}^{(u)}_{m,N,t}}^2_{S_1}\Big)}^2_{\cnum,2} \Big)^{1/2}\Bigg)^2\\
 &\le \Big(\sum_{l=1}^{m}  \Big(\sum_{t=1+4m}^{N}\bignorm{\check{W}^{(u)}_{m,N,t}}^4_{S_1,4} \Big)^{1/2} \Big)^2.
 \end{align*}
By Jensen's inequality, Hölder's inequality, Tonelli's theorem, \autoref{prop:S1inHprod}, and \autoref{lem:Burkh}(i)
 \begin{align*}
 \bignorm{\check{W}^{(u)}_{m,N,t}}^4_{S_1,4}& 
 \lesssim \Phi^{-4}_{\bf}\int_a^b \bignorm{\Upsilon_{u,\omega} \big( \dmpi{u}{\omega}{N}{t}  \otimes \sum_{s=1}^{t-4m} \tilde{w}_{\bf,t,s}^{(\omega)}   \dmpi{u}{\omega}{N}{s}\big)}^4_{S_1,4}d\omega
  \\&
  \lesssim \Phi^{-4}_{\bf} \sup_{u \in [0,1]}  \sup_{\omega \in [a,b]} \| \dmpi{u}{\omega}{0}{}\|^8_{\hi,4} \Big( \max_{1+4m \le t \le N} \sum_{s}^{t-4m} |\tilde{w}_{\bf,t,s}^{(\omega)} |^2\Big)^2 \int_a^b  \norm{\Upsilon_{u,\omega} }_{\infty}^4 d\omega
 \end{align*}
where we used that for $|t-s|\ge 4m$, $\dmpi{u}{\omega}{N}{t} $ and $\dmpi{u}{\omega}{N}{s}$ are independent. Consequently, 
\begin{align*} \Bignorm{\sum_{t=1+4m}^{N} \sum_{l=1}^{m}P_{t-l}\Big(\bignorm{\check{W}^{(u)}_{m,N,t}}^2_{S_1}\Big)}^2_{\cnum,2}
& =O(\Phi^{-4}_{\bf})O(N b^{-2}_f) =O( b^2_f N^{-2} N b^{-2}_f) = o(1), \end{align*}
and thus
\begin{align*}
\sum_{t=1+4m}^{N}\E \Big[\bignorm{\check{W}^{(u)}_{m,N,t}}^2_{S_1}| \G_{\tu{N}{u},t-1}\Big] =\sum_{t=1+4m}^{N}\E \Big[\bignorm{\check{W}^{(u)}_{m,N,t}}^2_{S_1}| \G_{\tu{N}{u},t-m-1}\Big] +o_{p}(1). \tageq \label{eq:expmapprox}
\end{align*}
Next, we verify the conditional Lindeberg condition \ref{itm:lindeberg}, i.e.,  that for all $f \in Q$, 
$\eta \in [0,1]$ 
\[\forall \epsilon>0, \sum_{t=1}^{N} \E\Big[ |f\big(\check{W}^{(u)}_{m,N,t}\big)|^2 \mathrm{1}_{|f(\check{W}^{(u)}_{m,N,t})|\ge \epsilon} \Big\vert |\G_{\tu{N}{u},t-1}\Big] \overset{p}{\to} 0 \quad \text{ as } N \to \infty \tageq \label{eq:Bergercon3}\]
Using $\E\big[\|Y\|_V^2 \mathrm{1}_{\{\|Y\|_V>\epsilon\}}|\mathcal{A}^{\star} \big] \le \epsilon^{-2} \E[\|Y\|_V^4| \mathcal{A}^{\star}]$ for some $Y \in \op^4_V$ and $\mathcal{A}^{\star}$ a sub $\sigma$-algebra of $\mathcal{A}$, \eqref{eq:expmapprox}, and similar arguments as above, we find 
\begin{align*}
&\sum_{t=1+4m}^{N} \E\Big[ |f_B\big(\check{W}^{(u)}_{m,N,t}\big)|^2 \mathrm{1}_{|f_B(\check{W}^{(u)}_{m,N,t})|\ge \epsilon} \Big\vert |\G_{\tu{N}{u},t-1}\Big]  \\
&\le \epsilon^{-2}   \sum_{t=1+4m}^{N}\E\Big[ \big\vert f_B \big(\check{W}^{(u)}_{m,N,t} \big\vert^4 \Big\vert \G_{\tu{N}{u},t-m-1}\Big]  +o_p(1)
    \\& \lesssim  \epsilon^{-2}   \Phi^{-4}_{\bf} \sum_{t=1+4m}^{N}\int_a^b  \norm{B}^4_{S_{\infty}} \norm{\Upsilon_{u,\omega}}^4_{\infty}\E\Big[ \bignorm{\dmpi{u}{\omega}{N}{t} \otimes \sum_{s=1}^{t-4m}  \tilde{w}_{\bf,t,s}^{(\omega)} \dmpi{u}{\omega}{N}{s}}^4_{S_1} \Big\vert \G_{\tu{N}{u},t-m-1}\Big] d\omega +o_p(1) 
\end{align*}
for any $B \in \opl(\Hi)$. 
Consequently, taking expectations and using that $|t-s|\ge 4m$, an application of Tonelli's theorem, \autoref{prop:S1inHprod}, and of  \autoref{lem:Burkh}(i) yield
\begin{align*}
&   \le  \epsilon^{-2}   C \Phi^{-4}_{\bf} \sum_{t=1+4m}^{N}\int_a^b  \norm{\Upsilon_{u,\omega}}^4_{\infty}\E\Big[ \bignorm{\dmpi{u}{\omega}{N}{t} }^4_{\Hi} \bignorm{\sum_{s=1}^{t-4m}  \tilde{w}_{\bf,t,s}^{(\omega)} \dmpi{u}{\omega}{N}{s}}^4_{\Hi}\Big] d\omega
\\&\le   \epsilon^{-2}   C \Phi^{-4}_{\bf} \sum_{t=1+4m}^{N} \int_a^b \norm{\Upsilon_{u,\omega}}^4_{\infty}\E\bignorm{\sum_{s=1}^{t-4m}  \tilde{w}_{\bf,t,s}^{(\omega)} \dmpi{u}{\omega}{N}{s}}^4_{\Hi} \E\Big[ \bignorm{\dmpi{u}{\omega}{N}{t} }^4_{\Hi} d\omega
\\& \le
  \epsilon^{-2}  C \Phi^{-4}_{\bf}\sup_{u\in [0,1]}\sup_{\omega\in [a,b]} \bignorm{\dmpi{u}{\omega}{}{0} }^8_{\hi,4}  \sum_{t=1+4m}^{N} \Big(\sum_{s=1}^{t-4m} \big\vert\tilde{w}_{\bf,t,s}^{(\omega)} \big\vert^2 \Big)^2  \int_a^b \norm{\Upsilon_{u,\omega}}^4_{\infty} d\omega  =o(1),
\end{align*}
from which \eqref{itm:lindeberg} follows.

For \eqref{itm:concov}, it follows from \autoref{lem:fidis_dense} that it is sufficient to show that
\[ \sum_{t=1}^{\flo{\eta N}} \E\Big[ |f\big(\check{W}^{(u)}_{m,N,t}\big)|^2 \Big\vert |\G_{\tu{N}{u},t-m-1}\Big] \overset{p}{\to} \eta \Gamma^{(u)}_{\Upsilon,m}(f)\quad \text{ as } N \to \infty, \]
for all $f \in Q_K$, for some finite $K$. 
Recall that elements of the topological dual space of $S_1(\Hi)$ are given by functionals $f_B: S_1(\Hi) \to \cnum$ with $f_B(A)=\Tr(BA)$, $B \in \opl(\Hi)$ (see Section \ref{prel}). 
By continuity of the coordinate functionals
\begin{align*}
f_B(\check{W}^{(u)}_{m,N,t}) \overline{f_B(\check{W}^{(u)}_{m,N,t})} & = 
 \Phi^{-2}_{\bf}\Big[ \Tr\Big(B \Upsilon_{u}(\check{V}^{(u)}_{N,t})\Big)+\Tr\Big(B \Upsilon_{u,\omega}(\check{V}^{\dagger(u)}_{N,t})\Big)\Big] \Big[ \overline{\Tr\Big(B \Upsilon_{u}(\check{V}^{(u)}_{N,t})\Big)}+\overline{\Tr\Big(B \Upsilon_{u}(\check{V}^{\dagger(u)}_{N,t})\Big)}\Big]
\end{align*}
where we use the notation 
$$
\Upsilon_{u}(\check{V}^{(u)}_{N,t}) =
\int_a^b  \Upsilon_{u,\omega}(\check{V}^{(u,\omega)}_{m,N,t} +\check{V}^{\dagger(u,\omega)}_{m,N,t} )
d \omega .
$$
Observe that $\sum_{s=1}^{t-4m} {\tilde{w}}_{\bf,t,s}^{(\lambda)} \dmpi{u}{\omega}{N}{s}$ is $\G_{\tu{N}{u},t-4m}$-measurable and $\dmpi{u}{\omega}{N}{t}$ is independent of the sigma algebra generated by $\G_{\tu{N}{u},t-m-1}$. Furthermore, $\{\tilde{w}_{\bf,t,s}^{(\lambda)} \}$ satisfies Assumption 3.2 of \cite{vD19}. Fubini's theorem and a slight adjustment of Lemma B.2-B.3 of aforementioned work then yield
\begin{align*}
& \sum_{t=1+4m}^{\flo{\eta N}}\E\Big[ \Tr\Big(B\Upsilon_{u}(\check{V}^{(u)}_{N,t})\Big)
\Tr\Big((\Upsilon_{u}(\check{V}^{(u)}_{N,t}))^\dagger B^\dagger\Big)\Big\vert \G_{\tu{N}{u},t-m-1}\Big]
\\&=\sum_{t=1+4m}^{\flo{\eta N}}\E\Big[ \Tr\Big((B  \widetilde{\otimes} B)\big(\Upsilon_{u}(\check{V}^{(u)}_{N,t}) \widetilde{\otimes}\Upsilon_{u}(\check{V}^{(u)}_{N,t})\big)\Big)\Big\vert \G_{\tu{N}{u},t-m-1}\Big]
\\& = \sum_{t=1+4m}^{\flo{\eta N}}\int_a^b \int_a^b  \Tr\Big((B  \widetilde{\otimes} B)(\Upsilon_{u,\lambda}\widetilde{\otimes}  \Upsilon_{u,\omega})
\\& ~~~~~~~~~~~~~~~~~~~~~~~~~~~~~~~\times
\E \Big[ (\dmpi{u}{\lambda}{N}{t} \otimes \sum_{s=1}^{t-4m} \tilde{w}_{\bf,t,s}^{(\lambda)} \dmpi{u}{\lambda}{N}{s})
{\bigotimes} (\dmpi{u}{\omega}{N}{t} \otimes \sum_{s=1}^{t-4m} \tilde{w}_{\bf,t,s}^{(\omega)}\dmpi{u}{\omega}{N}{s} \Big \vert \G_{\tu{N}{u},t-m-1}\Big] \Big)d\lambda d\omega
\\& =\lim_{n\to \infty}\frac{(b-a)}{n}  \sum_{i=1}^{n} \sum_{t=1+4m}^{\flo{\eta N}} \sum_{s=1}^{t-4m}
w^2_{\bf,t,s}  \Tr\Big((B  \widetilde{\otimes} B)(\Upsilon_{u,\lambda_i}\widetilde{\otimes}  \Upsilon_{u,\lambda_i})
\E \big[ (\dmpi{u}{\lambda_i}{N}{t} \otimes \dmpi{u}{\lambda_i}{N}{t}\big] {\widetilde{\otimes}}\E[\dmpi{u}{\lambda_i}{N}{s}
{\otimes} \dmpi{u}{\lambda_i}{N}{s}\Big]\Big)
\\&+o_p(\Phi^2_{\bf}).
\end{align*}
where $\lambda_1, \ldots, \lambda_n$ denote the $n$ midpoints corresponding to $n$ equally spaced subintervals of  $[a,b]$. 
Note that the process $\{{X}^{(u)}_{m,t}\}_{t \in \znum}$ satisfies the conditions of Proposition  3.2 of \cite{vD19}, uniformly in $u \in [0,1]$. Therefore, 
$\Tr(\E\big[ \dmpi{u}{\lambda_i}{N}{t} \otimes \dmpi{u}{\lambda_i}{N}{t} \big]) =\Tr(2\pi {\F}_{m}^{(\tu{N}{u}+t,\lambda_i)})$, 
 where 
${\F}_{m}^{(\tu{N}{u}+t,\lambda_i)}$ 
 is the time-varying spectral density operator of $\{{X}^{(u)}_{m,t}\}_{t}$ evaluated at time $\tu{N}{u}+t$, i.e., ${\F}_{m}^{(\tu{N}{u}+t,\lambda_i)} = \sum_{|r|\le m} e^{-\im r \lambda_i} \mathcal{C}_{\tu{N}{u}+t,r}$. The given  assumptions imply, $u \mapsto {\F}_{m}^{(u,\omega)}$ is a Lipschitz-continuous map from $[0,1]$ to $S_1(\Hi)$, uniformly in $\omega$, which  also yields that
$\sup_{\omega}\vert \Tr({\F}^{(\tu{N}{u}+t, \omega)}_{m}  -{\F}^{(u, \omega)}_{m})\vert=O(N/T).$
Hence, 
\begin{align*}
& \Phi^{-2}_{\bf}
\sum_{t=1+4m}^{\flo{\eta N}}\E \Big[\Tr\Big(B\Upsilon_{u}(\check{V}^{(u)}_{N,t})\Big)
\Tr\Big((\Upsilon_{u}(\check{V}^{(u)}_{N,t}))^\dagger B^\dagger\Big)
 | \G_{\tu{N}{u},t-m-1}\Big] 
\\
&=\frac{4\pi^2\sum_{t=1+4m}^{\flo{\eta N}} \sum_{s=1}^{t-4m} w_{\bf,t,s}^2}{\sum_{t=1}^{N} \sum_{s=1}^{N}  w_{\bf,t,s}^2 } \lim_{n\to \infty} \frac{(b-a)}{n}\sum_{i=1}^n \Tr\Big((B  \widetilde{\otimes} B)(\Upsilon_{u,\lambda_i}\widetilde{\otimes}  \Upsilon_{u,\lambda_i})({\F}^{(u, \lambda_i)}_{m}{\widetilde{\otimes}} {\F}^{(u, \lambda_i)}_{m}\Big)  +o_p(1)  +o(1)
\end{align*}
where we used self-adjointness of ${\F}^{(u, \lambda_i)}_{m}$. Furthermore, almost surely,\small{
\begin{align*}
& \E\Big[ \Tr\Big(B\Upsilon_{u}(\check{V}^{(u)}_{N,t})\Big)\Tr\Big(\big(\Upsilon_{u}(\check{V}^{\dagger(u)}_{N,t})\big)^\dagger B^\dagger\Big)\Big\vert \G_{\tu{N}{u},t-m-1}\Big]
=\E\Big[ \Tr\Big((B  \widetilde{\otimes} B)(\Upsilon_{u}(\check{V}^{(u)}_{N,t}) \widetilde{\otimes} \Upsilon_{u}(\check{V}^{\dagger(u)}_{N,t})\Big)\Big\vert \G_{\tu{N}{u},t-m-1}\Big]
\\& = \int_a^b \int_a^b  \Tr\Big((B  \widetilde{\otimes} B)(\Upsilon_{u,\lambda}  \widetilde{\otimes}\Upsilon_{u,\omega})\big(
\E \Big[ (\dmpi{u}{\lambda}{N}{t} \otimes \sum_{s=1}^{t-4m} \tilde{w}_{\bf,t,s}^{(\lambda)} \dmpi{u}{\lambda}{N}{s})
{\bigotimes} ( \sum_{s=1}^{t-4m} \tilde{w}^{(\omega)}_{\bf,t,s}\dmpi{u}{\omega}{N}{s} \otimes\dmpi{u}{\omega}{N}{t} )\Big \vert \G_{\tu{N}{u},t-m-1}\Big] \Big)d\lambda d\omega
\\& = \int_a^b \int_a^b  \Tr\Big((B  \widetilde{\otimes} B)(\Upsilon_{u,\lambda}  \widetilde{\otimes} \Upsilon_{u,\omega}))
\E \Big[ (\dmpi{u}{\lambda}{N}{t} \otimes \dmpi{u}{-\omega}{N}{t})  \widetilde{\bigotimes}_\top (\sum_{s=1}^{t-4m} \tilde{w}^{(-\lambda)}_{\bf,t,s}\dmpi{u}{-\lambda}{N}{s} {\otimes} ( \sum_{s=1}^{t-4m} \tilde{w}^{(\omega)}_{\bf,t,s}\dmpi{u}{\omega}{N}{s} )\Big \vert \G_{\tu{N}{u},t-m-1}\Big] \Big)d\lambda d\omega
\end{align*}}
\normalsize{where we used that for $a,b,c,d \in \Hi$,} 
 $  (a \otimes b){\bigotimes}  (c \otimes d) 
 =(a \otimes c)\widetilde{\bigotimes}  (b \otimes d) = (a  \otimes \overline{d}) \widetilde{\bigotimes}_\top (\overline{b} \otimes c)$. By a similar argument as for the first term, we thus find
\begin{align*}
& \Phi^{-2}_{\bf}\sum_{t=1+4m}^{\flo{\eta N}}
\E \Big[\Tr\Big(B\Upsilon_{u}(\check{V}^{(u)}_{N,t})\Big)\Tr\Big(\big(\Upsilon_{u}(\check{V}^{\dagger(u)}_{N,t})\big)^\dagger B^\dagger\Big)
 | \G_{\tu{N}{u},t-m-1}\Big] 
\\&=\frac{4\pi^2\sum_{t=1+4m}^{\flo{\eta N}} \sum_{s=1}^{t-4m} w_{\bf,t,s}^2}{\sum_{t=1}^{N} \sum_{s=1}^{N}  w_{\bf,t,s}^2} \lim_{n\to \infty} \frac{(b-a)}{n}\sum_{i=1}^n \Tr\Big((B  \widetilde{\otimes} B)\mathrm{1}_{\{0,\pi\}}(\lambda_i) (\Upsilon_{u,\lambda_i}\widetilde{\otimes} \Upsilon_{u,\lambda_i})
({\F}^{(u, \lambda_i)}_{m} \widetilde{\otimes}_\top {\F}^{(u, -\lambda_i)}_{m}) \Big) +o_p(1) 
\end{align*}
For the third term, we have
\small{
\begin{align*}
& \E\Big[ \Tr\Big(B \Upsilon_{u}(\check{V}^{\dagger(u)}_{N,t})\Big)
\Tr\Big(\big(\Upsilon_{u}(\check{V}^{(u)}_{N,t})\big)^\dagger B^\dagger\Big)\Big\vert \G_{\tu{N}{u},t-m-1}\Big]
= \E\Big[ \Tr\Big((B  \widetilde{\otimes} B)\Big(\Upsilon_{u}(\check{V}^{\dagger(u)}_{N,t}) \widetilde{\otimes} \Upsilon_{u}(\check{V}^{(u)}_{N,t})\Big)\Big\vert \G_{\tu{N}{u},t-m-1}\Big]
\\
& = \int_a^b \int_a^b  \Tr\Big((B  \widetilde{\otimes} B)(\Upsilon_{u,\lambda} \widetilde{\otimes} \Upsilon_{u,\omega})
\E \Big[ ( \sum_{s=1}^{t-4m} \tilde{w}_{\bf,t,s}^{(\lambda)} \dmpi{u}{\lambda}{N}{s} \otimes \dmpi{u}{\lambda}{N}{t} )
{\bigotimes} (\dmpi{u}{\omega}{N}{t} \otimes\sum_{s=1}^{t-4m} \tilde{w}^{(\omega)}_{\bf,t,s}\dmpi{u}{\omega}{N}{s} )\Big \vert \G_{\tu{N}{u},t-m-1}\Big]\Big) d\lambda d\omega
\end{align*}}\normalsize{}
which yields, similar to the previous term
\begin{align*}
& \Phi^{-2}_{\bf}\sum_{t=1+4m}^{\flo{\eta N}}
\E \Big[\Tr\Big(B \Upsilon_{u}(\check{V}^{\dagger(u)}_{N,t})\Big)
\Tr\Big(\big(\Upsilon_{u}(\check{V}^{(u)}_{N,t})\big)^\dagger B^\dagger\Big)
 | \G_{\tu{N}{u},t-m-1}\Big] 
\\&=\frac{4\pi^2\sum_{t=1+4m}^{\flo{\eta N}} \sum_{s=1}^{t-4m} w_{\bf,t,s}^2}{\sum_{t=1}^{N} \sum_{s=1}^{N}  w_{\bf,t,s}^2} \lim_{n\to \infty} \frac{(b-a)}{n}\sum_{i=1}^n \Tr\Big((B  \widetilde{\otimes} B)\mathrm{1}_{\{0,\pi\}}(\lambda_i) (\Upsilon_{u,\lambda_i} \widetilde{\otimes} \Upsilon_{u,\lambda_i})({\F}^{(u, -\lambda_i)}_{m} \widetilde{\otimes}_\top {\F}^{(u, \lambda_i)}_{m}) \Big) +o_p(1)  
\end{align*}
Finally,\small{
\begin{align*}
& \E\Big[\Tr\Big(B \Upsilon_{u}(\check{V}^{\dagger(u)}_{N,t})\Big)\Tr\Big(\big(\Upsilon_{u}(\check{V}^{\dagger(u)}_{N,t})\big)^\dagger B^\dagger\Big)\Big\vert \G_{\tu{N}{u},t-m-1}\Big]
=\E\Big[ \Tr\Big((B  \widetilde{\otimes} B)(\Upsilon_{u,\omega}(\check{V}^{\dagger(u)}_{N,t}) \widetilde{\otimes} \Upsilon_{u,\omega}(\check{V}^{\dagger(u)}_{N,t})\Big)\Big\vert \G_{\tu{N}{u},t-m-1}\Big]
\\& = \int_a^b \int_a^b  \Tr\Big((B  \widetilde{\otimes} B)(\Upsilon_{u,\lambda} \widetilde{\otimes} \Upsilon_{u,\omega})
\E \Big[ ( \sum_{s=1}^{t-4m} \tilde{w}_{\bf,t,s}^{(\lambda)} \dmpi{u}{\lambda}{N}{s}\otimes \dmpi{u}{\lambda}{N}{t})
{\bigotimes} ( \sum_{s=1}^{t-4m} \tilde{w}^{(\omega)}_{\bf,t,s} \dmpi{u}{\omega}{N}{s} \otimes \dmpi{u}{\omega}{N}{t}) \Big \vert \G_{\tu{N}{u},t-m-1}\Big] \Big)d\lambda d\omega,
\end{align*}}\normalsize{}
almost surely, from which we find
\begin{align*}
& \Phi^{-2}_{\bf} \sum_{t=1+4m}^{\flo{\eta N}}
\E \Big[ \Tr\Big(B \Upsilon_{u}(\check{V}^{\dagger(u)}_{N,t})\Big)\Tr\Big(\big(\Upsilon_{u}(\check{V}^{\dagger(u)}_{N,t})\big)^\dagger B^\dagger\Big)
 | \G_{\tu{N}{u},t-m-1}\Big] 
\\&=\frac{4\pi^2\sum_{t=1+4m}^{\flo{\eta N}} \sum_{s=1}^{t-4m} w_{\bf,t,s}^2}{\sum_{t=1}^{N} \sum_{s=1}^{N}  w_{\bf,t,s}^2} \lim_{n\to \infty} \frac{(b-a)}{n}\sum_{i=1}^n \Tr\Big((B  \widetilde{\otimes} B)(\Upsilon_{u,\lambda_i} \widetilde{\otimes} \Upsilon_{u,\lambda_i})({\F}^{(u, \lambda_i)}_{m} \widetilde{\otimes} {\F}^{(u, \lambda_i)}_{m})  +o_p(1)  
\end{align*}
A change of variables yields $
\sum_{t=1}^{k} \sum_{s=1}^{k}   w_{\bf,t,s}^2 =\sum_{|h|< k} (k-|h|) w^2(\bf(h))$
and
\begin{align*}
\frac{\sum_{t=1+4m}^{\flo{\eta N}} \sum_{s=1}^{t-4m}  w{\bf,t,s}^2}{\sum_{t=1}^{N} \sum_{s=1}^{N} w_{\bf,t,s}^2}
& = \frac{1}{2}\frac{\frac{\sum_{|h|< \flo{\eta N}} (\flo{\eta N}-|h|)  w(\bf h )^2}{\flo{\eta N}\bf^{-1} \int |w(x)|^2 d x} \flo{\eta N}\bf^{-1} \int |w(x)|^2 d x}{\frac{\sum_{|h|< N} (N-|h|)  w(\bf h )^2}{ N \bf^{-1} \int |w(x)|^2 d x} N \bf^{-1} \int |w(x)|^2 d x} 
\to \eta/2 
\end{align*}
where we used that the sum is zero for $\eta < 1/N$ and that $\sup_{\eta \in (1/N,1]}|\frac{\flo{\eta N}}{\eta N}-1 | = O(1/N)$. From Fubini's theorem and continuity of the inner product we thus obtain
\begin{align*}
 \sum_{t=1}^{\flo{\eta N}}\E\Big[ \Big\vert f_B\big(\check{W}^{(u)}_{N,t}\big)\Big\vert^2  \Big\vert \G_{\tu{N}{u},t-1}\Big] \overset{p}{\to}  \eta&  \Tr\Big( (B \widetilde{\otimes} B) \Gamma^{(u)}_{\Upsilon,m} \Big) 
\tageq \label{eq:concov}
\end{align*}
where 
\begin{align*}
\Gamma^{(u)}_{\Upsilon,m}& =8\pi^2 \int_a^b \Big\{(\Upsilon_{u,\omega} \widetilde{\otimes} \Upsilon_{u,\omega}) (\F_m^{(u,\omega)} \widetilde{\otimes} \F_m^{(u,\omega)} )
+\mathrm{1}_{\{0,\pi\}}(\omega) \Big[(\Upsilon_{u,\omega}\widetilde{\otimes} \Upsilon_{u,\omega})
({\F}^{(u, \omega)}_{m} \widetilde{\otimes}_\top {\F}^{(u, \omega)}_{m}) \Big]\Big\}d\omega
\end{align*}
This verifies \eqref{itm:concov}. \eqref{itm:concovp} can be verified similarly. Indeed, we obtain

\begin{align*}
f_{B}(\check{W}^{(u)}_{m,N,t}) {f_{\overline{B}}(\check{W}^{(u)}_{m,N,t})} &=f_{B}
\Big(\check{W}^{(u)}_{m,N,t}\Big) \overline{f_{B}\Big(~\overline{\check{W}^{(u)}_{m,N,t}}
~\Big) }
\\& 
=\Phi^{-2}_{\bf}\Big[ \Tr\Big(B\Upsilon_{u}(\check{V}^{(u)}_{N,t}) \widetilde{\otimes} B \overline{\Upsilon}_{u}(\overline{\check{V}}^{(u)}_{N,t}) \Big)+ \Tr\Big(B\Upsilon_{u}(\check{V}^{(u)}_{N,t}) \widetilde{\otimes} B \overline{\Upsilon}_{u}(\overline{\check{V}}^{\dagger(u)}_{N,t}) \Big)
\\& \phantom{\Phi^{-2}_{\bf}\Big[}
+\Tr\Big(B \Upsilon_{u}(\check{V}^{\dagger(u)}_{N,t}) \widetilde{\otimes}  B \overline{\Upsilon}_{u} (\overline{\check{V}}^{(u)}_{N,t}) \Big)+\Tr\Big(B \Upsilon_{u}(\check{V}^{\dagger(u)}_{N,t}) \widetilde{\otimes}  B \overline{\Upsilon}_{u}(\overline{\check{V}}^{\dagger(u)}_{N,t}) \Big)\Big].
\end{align*} 

Similar calculations as for the conditional covariance then yield
\begin{align*}
 \sum_{t=1}^{\flo{\eta N}}\E\Big[  f_B\big(\check{W}^{(u)}_{N,t}\big)
f_{\overline{B}}\big(\check{W}^{(u)}_{N,t}\big)  \Big\vert \G_{\tu{N}{u},t-1}\Big] \overset{p}{\to}  \eta&  \Tr\Big( (B  \widetilde{\otimes }B) \Sigma^{(u)}_{\Upsilon,m}  \Big) 
\tageq \label{eq:concov}
\end{align*}
where 
\begin{align*}
\Sigma^{(u)}_{\Upsilon,m}& =8\pi^2 \int_a^b \Big\{
 \mathrm{1}_{\{0,\pi\}}(\omega) \big[(\Upsilon_{u,\omega} \widetilde{\otimes} \overline{\Upsilon}_{u,\omega})({\F}^{(u, \omega)}_{m} \widetilde{\otimes} {\F}^{(u, \omega)}_{m})\big]+ 
(\Upsilon_{u,\omega} \widetilde{\otimes} \overline{\Upsilon}_{u,\omega}) (\F_m^{(u,\omega)} \widetilde{\otimes}_\top \F_m^{(u,-\omega)})\Big\}d\omega.
\end{align*}

\end{proof}

\begin{proof}[Proof of \autoref{lem:approxindep}]
Let $q \ge 2$ and recall the representation of 
$\mathcal{Y}^{(u)}_{m,N}$ in \eqref{eq:BMcont}. Applying Jensen's and Minkowski's inequality yields
\begin{align*}
&\E\Big(\bignorm{\frac{1}{M^{1/2}} \sum_{u\in U_M} \mathcal{Y}^{(u)}_{m,N}-\frac{1}{M^{1/2}} \sum_{u\in U_M} \mathcal{Y}^{(u)}_{m,L}}_{C_{S_1}}\Big)^q
\\& \lesssim M^{-q/2}\int_a^b 
 \E\Big(\sup_{\eta \in [0,1]}  \bignorm{\sum_{i=1}^{M}\sum_{t=1}^{\flo{\eta N}}\mathcal{\Phi}^{-q}_{\bf } \Upsilon_{u_i,\omega}(V^{(u_i,\omega)}_{m,N,t}) -\sum_{i=1}^{M}  \sum_{t=1}^{\flo{\eta L}}\mathcal{\Phi}^{-1}_{b_{L} } \Upsilon_{u_i,\omega}(V^{(u_i,\omega)}_{m,L,t} ) }_{S_1}\Big)^q d\omega
\\& + 
 M^{-q/2} \int_a^b 
 \E\Big(\sup_{\eta \in [0,1]}  \bignorm{\sum_{i=1}^{M}\Big[\mathcal{\Phi}^{-1}_{\bf } f(\eta,N)\Upsilon_{u_i,\omega}(V^{(u_i,\omega)}_{m,N,\flo{\eta N}+1}) -\mathcal{\Phi}^{-1}_{b_L}  f(\eta,L)\Upsilon_{u_i,\omega}(V^{(u_i,\omega)}_{m,L,\flo{\eta L}+1} )\Big] }_{S_1}\Big)^q d\omega \tageq \label{eq:micm1} 
\end{align*}
 For the first term we further decompose the error as
\begin{align*}
&M^{-q/2}\mathcal{\Phi}^{-q}_{\bf }\E\Big(\sup_{\eta \in [0,1]}  \bignorm{\sum_{i=1}^{M}\sum_{t=1}^{\flo{\eta N}}\Upsilon_{u_i,\omega}(V^{(u_i,\omega)}_{m,N,t}) -\Big[1-1+\frac{\mathcal{\Phi}_{\bf }}{\mathcal{\Phi}_{b_L }}\Big]\sum_{i=1}^{M}\sum_{t=1}^{\flo{\eta L}}\Upsilon_{u_i,\omega}(V^{(u_i,\omega)}_{m,L,t})  }_{S_1}\Big)^q 
\\& ~~~~~~~~~~~~~~~~~ ~~~~~~~~ \le    2^{q-1}M^{-q/2}\mathcal{\Phi}^{-q}_{\bf }\E\Big(\sup_{\eta \in [0,1]}  \bignorm{\sum_{i=1}^{M}\sum_{t=1}^{\flo{\eta N}}\Upsilon_{u_i,\omega}(V^{(u_i,\omega)}_{m,N,t}) -\sum_{i=1}^{M}\sum_{t=1}^{\flo{\eta L}}\Upsilon_{u_i,\omega}(V^{(u_i,\omega)}_{m,L,t})  }_{S_1}\Big)^q 
\tageq \label{eq:micm2a} \\
&  ~~~~~~~~~~~~~~~~~  ~~~~~~~~~~~~~~~~~    +
 2^{q-1} M^{-q/2}\Big[1-\frac{\mathcal{\Phi}_{\bf }}{\mathcal{\Phi}_{b_L }}\Big]^q\mathcal{\Phi}^{-q}_{\bf }\E\Big(\sup_{\eta \in [0,1]}  \bignorm{\sum_{i=1}^{M} \sum_{t=1}^{\flo{\eta L}}\Upsilon_{u_i,\omega}(V^{(u_i,\omega)}_{m,L,t}  )}_{S_1}\Big)^q  \tageq \label{eq:micm2}
\end{align*}
The  term in  \eqref{eq:micm2} converges to zero,
as $N \to \infty$,  since, by \autoref{lem:Vmax} below, we have \[
\E\Big(\max_{1\le k \le L}  \bignorm{\sum_{i=1}^{M} \sum_{t=1}^{k}\Upsilon_{u_i,\omega}(V^{(u_i,\omega)}_{m,L,t}  )}_{S_1}\Big)^q =O(M^{q/2}N^{q/2} \bf^{-q/2})
\]
  and thus we get
$$ O \Big ( 
M^{-q/2}\Big[1-\frac{\mathcal{\Phi}_{\bf }}{\mathcal{\Phi}_{b_L }}\Big]^q\mathcal{\Phi}^{-q}_{\bf }
 \Big ) \cdot  O \big (M^{q/2}N^{q/2} \bf^{-q/2}  \big )
 =  O \Big ( 
 \Big[1-\frac{\mathcal{\Phi}_{\bf}}{\mathcal{\Phi}_{b_L }}\Big]^q
 \Big ) $$
\autoref{lem:tedious} below then yields that the  term in \eqref{eq:micm2a} is of order $o(1)$ as $N \to \infty$ for any fixed $m$. Finally, it follows from \autoref{lem:tedious2} that the second term of \eqref{eq:micm1} converges  to zero as $N \to \infty$ for any fixed $m$ under \autoref{as:bandwidth}. This proofs the statement. \end{proof}

\section{Concentration bounds and inequalities} 
\label{sec:Cbounds}
\def\theequation{C.\arabic{equation}}
\setcounter{equation}{0}

In this section, we make use of the following shorthand notation
\begin{equation}
\label{eqnr106}
\tilde{X}^{(u,t)}_{v;t}=X^{(\tu{N}{u}+t)/T}_{\tu{N}{v}+t}, \quad
 {X}_{u;t}= X_{\tu{N}{u}+t,T}  \quad \barbelow{\ddot{X}}^{(u)}_{l}= {X}^{(u)}_{\tu{N}{u_{i_l}}+t_l},
\end{equation}  
and we recall that  $\tilde{\F}_{u,\omega}(\eta)$
is defined  as in  \eqref{eq:Fint} but where the nonstationary process is replaced with its auxiliary counterpart. 
Furthermore,
we also make use of the non-interpolated partial sum
\[\hat{\F}_{k}^{(u,\omega)}=\frac{1}{k}\sum_{s=1}^{k}\Big(\sum_{t=1}^{k} 
\tilde{w}^{(\omega)}_{\bf,s,t}
({X}_{u;s} \otimes {X}_{u;t})\Big) \tageq \label{eq:Fhatnonstat}~,
\]
and its auxiliary stationary counterpart
\[
\tilde{\F}_{k}^{(u,\omega)}=\frac{1}{k} \sum_{s=1}^{k}\sum_{t=1}^{k}\tilde{w}^{(\omega)}_{\bf,s,t} (\tilde{X}^{(u,s)}_{u;s} \otimes \tilde{X}^{(u,t)}_{u;t}). \tageq \label{eq:Fstat}
\]

\begin{thm} \label{thm:maxPS1}
Suppose \autoref{as:depstrucnonstat}-\ref{as:depstruc} hold with $p \ge 6$ and
that \autoref{as:smooth}-\ref{as:mappings} are satisfied. Then, for all $\omega \in [a,b]$,
\begin{align*}
&\sup_{\eta\in I} \bignorm{\sum_{u \in U_M} \flo{\eta N} \Upsilon_{u,\omega} 
(\hat{\F}_{u,\omega}(\eta)-\E\tilde{\F}_{u,\omega}(\eta))}_{S_1}  =O_p((MN)^{1/2}\bf^{-1/2});
\\&\sup_{\eta\in I} \bignorm{\sum_{u \in U_M} \flo{\eta N} \Upsilon_{u,\omega} 
(\E\tilde{\F}_{u,\omega}(\eta)-\F_{u,\omega})}_{S_1}= o((MN)^{1/2}\bf^{-1/2}).
\end{align*}
Furthermore,  for any fixed $u \in (0,1)$
\begin{align*}
&\sup_{\eta\in I} \bignorm{\flo{\eta N} \Upsilon_{u,\omega} 
(\hat{\F}_{u,\omega}(\eta)-\E\tilde{\F}_{u,\omega}(\eta))}_{S_1}  =O_p(N^{1/2}\bf^{-1/2});
\\&\sup_{\eta\in I} \bignorm{ \flo{\eta N} \Upsilon_{u,\omega} 
(\E\tilde{\F}_{u,\omega}(\eta)-\F_{u,\omega})}_{S_1}= o(N^{1/2}\bf^{-1/2}).
\end{align*}
\end{thm}

\begin{proof}[Proof of \autoref{thm:maxPS1}]
We have
 \begin{align*}
\hat{\F}_{u,\omega}(\eta)-\tilde{\F}_{u,\omega}(\eta)
+\tilde{\F}_{u,\omega}(\eta)
-\E\tilde{\F}_{u,\omega}(\eta)
+\E\tilde{\F}_{u,\omega}(\eta)
 -\F_{u,\omega}.
\end{align*}
\autoref{lem:approxloc},  \autoref{lem:approxvarint} and \autoref{lem:bias}, yield respectively
\begin{align*}
\E \big(\sup_{\eta}\bignorm{\sum_{u \in U_M} \flo{\eta N} \Upsilon_{u,\omega}\big(\hat{\F}^{(u,\omega) }(\eta)
 -\tilde{\F}^{(u,\omega)}(\eta)\big)}_{S_1}\big)^\gamma &=O( {(MN)}^{\gamma}  T^{-\zeta \gamma}  \bf^{-\gamma/2});
 \tageq  \label{eq:statap}
\\ 
\E\Big( \sup_{\eta}\bignorm{\sum_{u \in U_M} \flo{\eta N} \Upsilon_{u,\omega}\Big(\tilde{\F}^{(u,\omega)}(\eta)
 -\E\tilde{\F}^{(u,\omega)}(\eta)\Big)}_{S_1}\Big)^{\gamma} &=O( (MN)^{\gamma/2} \bf^{-\gamma/2});  \tageq \label{eq:varer}
\\
\sup_{\eta}\bignorm{\sum_{u \in U_M}  \flo{\eta N} \Upsilon_{u,\omega}\big(\E\tilde{\F}^{(u,\omega)}(\eta)
 -{\F}_{u,\omega}\big)}_{S_1} &=O\Big(\frac{M N^{2}}{T^\zeta\bf^{1/2}}\Big)+O(M)+O(MN\bf^{2}\big).
 \tageq \label{eq:bi}
\end{align*}
Hence, 
\begin{align*}
&\sup_{\eta \in I}\bignorm{\sum_{u \in U_T} \flo{\eta N} \Upsilon_{u,\omega} 
(\hat{\F}_{u,\omega}(\eta)-\F_{u,\omega})}_{S_1} =
O_p((MN)^{1/2} \bf^{-1/2}) +O\Big(\frac{M N^{2}}{T^\zeta\bf^{1/2}}\Big)+O(M)+O(MN\bf^{2}\big),
\end{align*}
where we used that \eqref{eq:statap} is absorbed in \eqref{eq:varer}. Note that we require the last three terms to be of lower order than the first term, that  is 
\begin{align*}
&MN^2 T^{-\zeta}= o(M^{1/2}N^{1/2}) \Rightarrow \frac{M^{1/2}N^{3/2}}{T^\zeta} =o(1) \Rightarrow N=o(T^{2/3\zeta} M^{-1/3})~, \tageq \label{eq:bias12}
\\
&M =o(M^{1/2}N^{1/2}b^{-1/2}) \Rightarrow \frac{M \bf}{N} = M N^{-\kappa-1} =o(1)  \Rightarrow M=o(N^{1+\kappa})~, \tageq \label{eq:bias22}\\
&MNb^\iota= o(M^{1/2}N^{1/2}b^{-1/2}) \Rightarrow MN \bf^{2\iota}= M N^{1-(2\iota+1)\kappa} = o(1) \Rightarrow M = o(N^{(2\iota+1)\kappa-1}) ~.
\tageq \label{eq:bias32}
\end{align*}
Under \autoref{as:bandwidth}, \eqref{eq:bias22} is immediately satisfied, whereas $M \to \infty$ together with
\eqref{eq:bias32} 
implies the necessary condition $\kappa >1/(2\iota+1)$, 
and hence appears in \autoref{as:bandwidth}. Furthermore, to determine the rate of $M$ with respect to $N$, we have to distinguish cases since both $M=o(N^{1-\kappa})$ and \eqref{eq:bias32} must be satisfied.  
Observe that
$(2\iota+1)\kappa-1 \ge 1-\kappa$ if $\kappa\ge 1/(1+\iota)$ and so in this case we must ensure $M=o(N^{1-\kappa})$.
Using then that $N=T^{\alpha}$, we obtain
that 
 \eqref{eq:bias12} holds if
$\alpha< \frac{2\zeta}{4-\kappa}$
, that is
$$
N^3 M T^{-2\zeta}  = 
o( N^{4-k}T^{-2\zeta} ) = o( T^{\alpha(4-\kappa )-2\zeta}) =o(1).
$$
On the other hand, if 
$\kappa <1/(1+\iota)$, then $M = o(N^{(2\iota+1)\kappa-1})$ and \eqref{eq:bias12} holds if $\alpha <\frac{2\zeta}{(2\iota+1)\kappa-1+3 }$, that is
$$
N^3 M T^{-2\zeta}  = 
o( N^{(2\iota+1)\kappa-1+3}T^{-2\zeta}) = o(T^{\alpha((2\iota+1)\kappa-1+3)-2\zeta }) =o(1).
$$
The second part of the statement follows similarly: see also Remark  \ref{Rem:fixu}.
\end{proof}

\begin{thm}\label{thm:maxdev}
Suppose \autoref{def:locstat} with $\zeta=1$, \autoref{as:depstrucnonstat}-\autoref{as:depstruc}  with $p\ge 6$ and \autoref{as:Weights} holds true. In addition suppose that $\sum_{r\in \nnum}\sum_{j=l}^{\infty} \nu^{X_{\cdot}^{\cdot}(e_r)}_{\cnum,p}(j) =O({(l+1)}^{-\rho})$, $0<\rho<1$. Then
\begin{align*}\sup_{u \in U_M} \sup_{\omega }\max_{1\le k\le N}\bignorm{k(\hat{\F}^{(u,\omega)}_k-\E\tilde{\F}^{(u,\omega)}_k)}_{S_1}
&= O_p\Bigg(\big(\bw^\gamma T \log(N)^{2\gamma+1}+\mathcal{R}_{N,\bf,\gamma} \mathrm{1}_{\{\rho<1/2-1/\gamma\}}\big)^{1/\gamma}\Bigg) 
\\&+O_p\Big( \big(T \bw^{3\gamma/4+1/2}\log(\bw)^{2\gamma+1}\big)^{1/\gamma}\Big)+O_p\Big(\sqrt{N\bw \log{T}}\Big)
\\&+  O_p(\frac{M^{1/\gamma} N \bw^{1/2+1/\gamma} }{T}),
\end{align*}
where $\gamma=p/2$, $\tilde{b}=\flo{1/\bf}$ and 
$ \mathcal{R}_{N,\bf,\gamma}=T\bw N^{\gamma(1/4-\rho/2)-1/2} \bw^{\gamma(3/4-\rho /2)-1/2} \log_2(N/\bw)^{2\gamma+1}$.
\end{thm}

\begin{proof}[Proof of \autoref{thm:maxdev}]
Without loss of generality, we focus on the case $M=T/N$. We recall the notation $\tu{N}{u}=\flo{uT}-\flo{N/2}$ and write
$
\hat{\F}^{(u,\omega)}_k=\hat{\F}^{u,\omega}_{k,+}+\hat{\F}^{u,\omega}_{k,-}$
where 
\begin{align*}
 k\hat{\F}^{u,\omega}_{k,+}&=\frac{1}{2\pi} \sum_{r=0}^{\bw \wedge k}  w(\bf(r)) \sum_{t=\tu{N}{u}+1 }^{\tu{N}{u}+k-r} e^{-\im r \omega } (X_t\otimes X_{t+r}\big) ~,
\\ k\hat{\F}^{u,\omega}_{k,-}&=\frac{1}{2\pi} \sum_{r=-(\bw \wedge k)}^{-1}  w(\bf(r)) \sum_{t=\tu{N}{u}+1-r}^{\tu{N}{u}+k} e^{-\im r \omega } (X_t\otimes X_{t+r}\big)~.
\end{align*}
Note that  
$\omega \mapsto k\hat{\F}^{(u,\omega)}_k$ and $\omega \mapsto k\tilde{\F}^{(u,\omega)}_k$ are trigonometric polynomials of at most degree $\bw$ on $S_1(\Hi)$ for fixed $u$ and $k$. A slight adaptation of Theorem 7.28, Chapter X in \cite{Zygmund} yields
\[
\sup_{\omega }\bignorm{k(\hat{\F}^{(u,\omega)}_k-\E\tilde{\F}^{(u,\omega)}_k)}_{S_1} \le 2 \max_{0 \le j \le 4 \bw } \bignorm{k(\hat{\F}^{(u,\omega_j)}_k-\E\tilde{\F}^{(u,\omega_j)}_k)}_{S_1}
\]
where $\omega_j=\frac{\pi j}{2\bw}, 0\le j\le 4 \bw$. Using De Morgan's laws and Boole's inequality 
\begin{align*}
&\mathbb{P}\Big(\sup_{u \in U_M} \max_{1\le k \le N}\sup_{\omega }\bignorm{k(\hat{\F}^{(u,\omega)}_k-\E\tilde{\F}^{(u,\omega)}_k)}_{S_1} \ge x\Big)
\\& 
\le (4\bw+1) M\max_{j} \max_i  \mathbb{P}\Big(  \max_{1\le k \le N} \bignorm{k(\hat{\F}^{(u_i,\omega_j)}_k-\tilde{\F}^{(u_i,\omega_j)}_k)}_{S_1} \ge x/4\Big)\tageq\label{eq:maxFFtilde}
\\
& + (4\bw+1) M\max_{j} \max_i \mathbb{P}\Big(  \max_{1\le k \le N} \bignorm{k(\tilde{\F}^{(u_i,\omega_j)}_k-\E\tilde{\F}^{(u_i,\omega_j)}_k)}_{S_1} \ge x/4\Big) \tageq\label{eq:maxvarFtilde}.
\end{align*}
For \eqref{eq:maxFFtilde}, we note that by Markov's inequality and \autoref{lem:approxloc} 
\begin{align*}
& (4\bw+1) M\max_{j} \max_i \mathbb{P}\Big( \max_{1\le k \le N} \bignorm{k(\hat{\F}^{u_i,\omega_j}_{k}-\tilde{\F}^{u_i,\omega_j}_{k})}_{S_1} \ge x/4 \Big)
\\& \le  (4\bw+1) M\max_{j} \max_i \frac{1}{(x/4)^\gamma} \E \big(\max_{1\le k \le N}\bignorm{ k\big(\hat{\F}_{k}^{u_i,\omega }
 -\tilde{\F}_{k}^{(u_i,\omega )})}_{S_1}\big)^\gamma
\lesssim \frac{1}{(x/4)^\gamma} 
M N^{\gamma} \bw^{\gamma/2+1} T^{-\gamma},
\end{align*}
which implies
\[
\sup_{u \in U_M} \max_{1\le k \le N}\sup_{\omega }\bignorm{k\big (\hat{\F}^{(u,\omega)}_k-\tilde{\F}^{(u,\omega)}_k \big )}_{S_1}  = O_p \Big (\frac{M^{1/\gamma} N \bw^{1/2+1/\gamma} }{T} \Big ).
\]
To determine a bound on \eqref{eq:maxvarFtilde}, we write
\[
 \mathbb{P}\Big(  \max_{1\le k \le N} \bignorm{k(\tilde{\F}^{(u_i,\omega_j)}_k-\E\tilde{\F}^{(u_i,\omega_j)}_k)}_{S_1} \ge x/4\Big)
\le \sum_{\mathcal{r}\in \{+,-\}} \mathbb{P}\Big( \max_{1\le k \le N} \bignorm{k(\tilde{\F}^{(u_i,\omega_j)}_{k,\mathcal{r}}-\E\tilde{\F}^{(u_i,\omega_j)}_{k,\mathcal{r}})}_{S_1} \ge x/8\Big)
\]
we focus on the case $\mathcal{r}=+$ as the other case is similar. We then further decompose the partial sums $G_{i,k}=
k(\tilde{\F}^{u_i,\omega_j}_{k,+}-\E\tilde{\F}^{u_i,\omega_j}_{k,+})$ into  $W=\flo{N/\bw}$ blocks, each of length $\bw$. In the following, we use the notation
$\phi(\bf x ) =(2\pi)^{-1}w(\bf x) e^{\im \omega  x }$. Then, elementary calculations give 
\begin{align*}
k\tilde{\F}^{u_i,\omega_j}_{k,+}&=\frac{1}{2\pi} \sum_{r=0}^{(\bw \wedge k)}  w(\bf r) \sum_{t=\tu{N}{u}+1 }^{\tu{N}{u}+k-r} e^{-\im r \omega } (\Xu{t}\otimes \Xu{t+r}\big)
\\ & =\sum_{s=(i-1)N+1}^{(i-1)N+k} \,\sum_{t=s-(\bw \wedge k) \vee (i-1)N+1}^{s} \phi(\bf (s-t)) (\Xu{t}\otimes \Xu{s}\big).
\end{align*}
Without loss of generality, we focus on the first block $i=1$, which is thus given by
\[
G_k=G_{1,k}=\sum_{s=1}^{k} \sum_{t=s-(\bw \wedge k) \vee 1}^{s} \phi(\bf(s-t)) \Big(\Xu{t}\otimes \Xu{s} -\E[\Xu{t}\otimes \Xu{s}]\Big)~.
\]
Partitioning the time stretch of length $N$ into blocks of length $\bw$ we get 
\begin{align*}
\mathbb{P}\Big( \max_{1\le k \le N} \bignorm{\tilde{\F}^{u_1,\omega_j}_{k,+}-\E\tilde{\F}^{u_1,\omega_j}_{k,+}}_{S_1} \ge x\Big)
& =\mathbb{P}\Big( \max_{1\le k \le N} \bignorm{G_{k}}_{S_1} \ge x\Big)
\\& 
\le  \mathbb{P}\Big( \max_{0\le w \le W} \bignorm{G_{w\bw}}_{S_1} \ge x/2\Big)\tageq \label{eq:l195}
\\& +\sum_{w=0}^W \mathbb{P}\Big( \max_{0 \le r \le \bw-1} \bignorm{G_{w\bw+r}-G_{w\bw}}_{S_1} \ge x/2\Big)~. 
\tageq \label{eq:l196}
\end{align*}
To deal with these two terms, we proceed in a similar manner as \cite{ZW21}. We write the dyadic expansion of a positive integer $w \le \lceil N/\bw \rceil$ in the following form
\[
w =\sum_{j=1}^J 2^{d_j} \quad 0 \le d_{J} < \ldots < d_1 \le d~,  
\]
where $d= \lceil \log_2(N/\bw )\rceil$. Then we can write $
G_{w\bw} = \sum_{j=1}^J G_{w(j)\bw}-G_{w(j-1)\bw}$
where $w(j) = \sum_{r=1}^j 2^{d_r}$ and $w(0)=0$, and also
\begin{align*}
\sum_{j=1}^J \norm{G_{w(j)\bw}-G_{w(j-1)\bw}}_{S_1}
& 
 \le \sum_{j=1}^J \max_{1\le v \le 2^{d-d_j}} \norm{G_{2^{d_j}v \bw}-G_{2^{d_j}(v-1)\bw}}_{S_1}
\\& 
\le \sum_{i=0}^d \max_{1\le v \le 2^{d-i}}\norm{G_{2^{i}v \bw}-G_{2^{i}(v-1)\bw}}_{S_1}.
\end{align*}
Therefore, for \eqref{eq:l195} 
\begin{align*}
 \mathbb{P}\Big( \max_{0\le w \le W} \norm{G_{w\bw}}_{S_1} \ge x/2\Big)
\le \sum_{i=0}^d 2^{d-i} \max_{1\le v \le 2^{d-i}} \mathbb{P}\Big( \norm{G_{2^{i}v \bw}-G_{2^{i}(v-1)\bw}}_{S_1} \ge \epsilon_i x /2\Big)~, 
\end{align*}
where 
\[
\epsilon_i(d)=
\begin{cases} 
\frac{3}{\pi^2} \frac{1}{(i+1)^2} &\text{ if } 0\le i\le d/2\\
\frac{3}{\pi^2} \frac{1}{(d+1-i)^2} &\text{ if } d/2< i\le d, \tageq\label{epsi}
\end{cases}
\]
which satisfies $\sum_{i=0}^d \epsilon_i(d)\le 1$, and furthermore,
\begin{align*}
\sum_{k=0}^d \frac{a^{k}}{\epsilon^{c}_k(d)} 
 & = \big ( 3/\pi^2\big )^c \Big( \sum_{0\le k\le d/2} (k+1)^{2 c} a^{k} + \sum_{d/2 <k\le d} (d-k+1)^{2c} a^{k}\Big)
 \\&  =\big ( 3/\pi^2 \big )^c \Big( \sum_{0\le k\le d/2} (k+1)^{2 c} a^{k} + \sum_{0\le k\le d/2} (k+1)^{2c} a^{d-k}\Big)
 \\&=O\big( (d/2+1)^{2c+1} a^{d/2}\big),  \tageq\label{epsisum}
\end{align*}
for some constants $a, c \ge 1$. It follows from \autoref{thm:quadtail} for fixed $v$, which yields $2^i$ blocks of length $\bw$, $2^i \bw\le T$,
that 
\[
 \mathbb{P}\Big( \norm{G_{2^{i}v \bw}-G_{2^{i}(v-1)\bw}}_{S_1} \ge \epsilon_i x\Big) \le 
  \begin{cases}
C\exp\Big(-\frac{ (\epsilon_i x)^2  }{(2+\beta) 2^i \bw^2 C_X}\Big)+C_{\gamma,\rho} \frac{2^i \bw^{\gamma}}{ (\epsilon_i x)^\gamma} & \text{ if } \rho >1/2-1/\gamma\\
C\exp\Big(-\frac{(\epsilon_i x)^2   }{(2+\beta) 2^i\bw^2 C_X}\Big)+C_{\gamma,\rho}\Big( \frac{2^i  \bw^{\gamma}}{ (\epsilon_i x)^\gamma}+\frac{(2^i \bw)^{\gamma/2-\rho \gamma} \bw^{\gamma/2}}{ (\epsilon_i x)^\gamma}\Big)
& \text{ if } \rho <1/2-1/\gamma.
\end{cases}\]
Now, because $2^i \le N/\bw$ and the term in the exponent is negative for any $i$;
$$
\sum_{i=0}^d 2^{d-i} C\exp\Big(-\frac{ (\epsilon_i x)^2  }{(2+\beta) 2^i \bw^2C_X}\Big) =O\Big(2^d C\exp\big(-\frac{x^2  }{(2+\beta) (N/\bw) \bw^2 C_X}\big)\Big) = O\Big(N/\bw C\exp\big(-\frac{x^2  }{(2+\beta) N\bw C_X}\big)\Big) ,
$$
and
$$
\sum_{i=0}^d 2^{d-i}\frac{2^i \bw^{\gamma}}{ (\epsilon_i(d) x)^\gamma}=6/\pi^2 \frac{2^d\bw^{\gamma}}{x^\gamma} \sum_{i=0}^d (1+i)^{2\gamma} = O(N/\bw \bw^\gamma \log_2(N/\bw)^{2\gamma+1})=O(N \bw^{\gamma-1} \log_2(N/\bw)^{2\gamma+1}).
$$
Furthermore, using \eqref{epsisum} with $a=2^{\gamma/2-\rho \gamma-1}$ and $c=\gamma$ 
\begin{align*}
\sum_{i=0}^d 2^{d-i}\frac{(2^i \bw)^{\gamma/2-\rho \gamma} \bw^{\gamma/2}}{\epsilon_i^\gamma}
&=O\big(2^d\bw^{\gamma-\rho \gamma} (d/2+1)^{2\gamma+1} (2^{\gamma(1/2-\rho)-1})^{d/2}\big)
\\&=
O\big(N/\bw \bw^{\gamma-\rho \gamma} (\log_2(N/\bw))^{2\gamma+1} (N/\bw)^{\gamma/2(1/2-\rho)-1/2}\big).
\\
&=
O(N  N^{\gamma(1/4-\rho/2)-1/2} \bw^{\gamma(3/4-\rho/2)-1/2} \log_2(N/\bw)^{2\gamma+1}).
\end{align*}
This means that
\begin{align*}
\mathbb{P}\Big( \max_{0\le w \le W} \norm{G_{w\bw}}_{S_1} \ge \frac{x}{2} \Big)
& = O\Big(N  \Big[\log(N)^{2\gamma+1}\bw^{\gamma-1} +  N^{\gamma(1/4-\rho/2)-1/2} \bw^{\gamma(3/4-\rho/2)-1/2} \log_2(N/\bw)^{2\gamma+1}\mathrm{1}_{\{\rho<1/2-1/\gamma\}}\Big]\Big)
\\
& +O\Big(N/\bw C\exp\big(-\frac{x^2  }{(2+\beta) N\bw C_X}\big)\Big) ~. 
 \end{align*}
Similarly, let $\tilde{d}=\lceil \log_2(\tilde{b}-1)\rceil$, then 
we obtain  for \eqref{eq:l196}
\begin{align*}
&\sum_{w=0}^W \mathbb{P}\Big( \max_{0 \le r \le \bw-1} \bignorm{G_{w\bw+r}-G_{w\bw}}_{S_1} \ge x/2\Big) 
\\& \le \sum_{w=0}^W \sum_{i=0}^{\tilde{d}} 2^{\tilde{d}-i} \max_{1\le v \le 2^{\tilde{d}-i}} \mathbb{P}\Big( \norm{G_{w \bw+2^{i}v}-G_{w\bw+2^i(v-1)}}_{S_1} \ge \epsilon_i(\tilde{d}) x /2\Big),
 \end{align*} 
 where $\epsilon_i(\tilde{d})$ 
 is defined by  \eqref{epsi}.
Now, Markov's inequality and \autoref{lem:approxvar} yield 
 \begin{align*}
 &\sum_{w=0}^W \sum_{i=0}^{\tilde{d}} 2^{\tilde{d}-i} \max_{1\le v \le 2^{\tilde{d}-i}}  \mathbb{P}\Big( \norm{G_{w \bw+2^{i}v}-G_{w\bw+2^i(v-1)}}_{S_1} \ge \epsilon_i(\tilde{d}) x\Big) \\&\le \sum_{w=0}^W \sum_{i=0}^{\tilde{d}} 2^{\tilde{d}-i} \max_{1\le v \le 2^{\tilde{d}-i}} \frac{\E\norm{G_{w \bw+2^{i}v}-G_{w\bw+2^i(v-1)}}^\gamma_{S_1}}{(\epsilon_i(\tilde{d}) x)^{\gamma}}
 \\&=(W+1) \sum_{i=0}^{\tilde{d}} 2^{\tilde{d}-i}  \frac{O((2^i)^{\gamma/2}\bf^{-\gamma/2}))}{(\epsilon_i(\tilde{d}) x)^{\gamma}} = O(N\bw^{-1} \bw \bw^{\gamma/2}) O\big(\log(\bw)^{2\gamma+1} {\bw}^{(\gamma/4-1/2)}\big) =O(N \bw^{3\gamma/4-1/2}\log(\bw)^{2\gamma+1})~, 
 \end{align*}
where we used \eqref{epsisum} with $a=2^{\gamma/2-1}$ and $c=\gamma$. Altogether, we obtain 
\begin{align*}
& (4\bw+1) M \max_i \max_j \mathbb{P}\Big(\max_{1\le k \le N}\bignorm{k(\tilde{\F}^{(u_i,\omega_j)}_k-\E\tilde{\F}_{u_i,\omega_j})}_{S_1} \ge x\Big)
\\&\le
 (4\bw+1) M \max_i \max_j\sum_{\mathcal{r}\in \{+,-\}}  \mathbb{P}\Big( \max_{1\le k \le N} \bignorm{k(\tilde{\F}^{(u_i,\omega_j)}_{k,\mathcal{r}}-\E\tilde{\F}^{(u_i,\omega_j)}_{k,\mathcal{r}})}_{S_1} \ge x/4\Big)
\\& \le  \frac{1}{(x/4)^{\gamma}} C_{\gamma,\rho} \Big(\bw MN  \Big[\log(N)^{2\gamma+1}\bw^{\gamma-1} +  N^{\gamma(1/4-\rho/2)-1/2} \bw^{\gamma(3/4-\rho/2)-1/2} \log_2(N/\bw)^{2\gamma+1}\mathrm{1}_{\{\rho<1/2-1/\gamma\}}\Big]
\\&\phantom{\frac{1}{x^{q/4}} C_{\gamma,\rho} \Big(}
+\bw M N \bw^{3\gamma/4-1/2}\log(\bw)^{2\gamma+1}\Big)
+ \bw M N/\bw C\exp\big(-\frac{(x/4)^2  }{(2+\beta) N\bw C_X}\big)\Big) 
\\& \le  \frac{1}{(x/4)^{\gamma}} C_{\gamma,\rho} \Big(\bw^\gamma T \log(N)^{2\gamma+1}+T \bw^{3\gamma/4+1/2}\log(\bw)^{2\gamma+1}+ \mathcal{R}_{N,\bf,\gamma}\mathrm{1}_{\{\rho<1/2-1/\gamma\}}\Big) + T C\exp\big(-\frac{(x/4)^2  }{(2+\beta) N\bw C_X}\big)
\end{align*}
where 
\[
 \mathcal{R}_{N,\bf,\gamma}=T\bw N^{\gamma(1/4-\rho/2)-1/2} \bw^{\gamma(3/4-\rho /2)-1/2} \log_2(N/\bw)^{2\gamma+1}.
\]

\end{proof}

\begin{thm}\label{thm:maxdevcon}
Suppose that  the conditions of \autoref{thm:Conv}(b)  are satisfied and that 
$$\gamma=p/2>\begin{cases}
\frac{1}{\alpha(1-\kappa)} & \text{ if } \rho\ge 1/2-1/\gamma\\
\max\Big(\frac{1}{\alpha(1-\kappa)}, \frac{\frac{1}{\alpha}-\frac{1}{2}(1-\kappa)}{\frac{3}{4}(1-\kappa)+\frac{\rho}{2}(\kappa+1)}\Big) & \text{ if } \rho<1/2-1/\gamma 
\end{cases}~.
$$
 Then
\begin{align}
\label{de101}
&\sup_{u \in U_T} \sup_{\omega }\sup_{\eta \in I}\bignorm{\eta(\hat{\F}_{u,\omega}(\eta)-{\F}_{u,\omega})}_{S_1}
= o_p(1) 
\end{align}
as $T \to \infty$. 
\end{thm}

\begin{proof}[Proof of \autoref{thm:maxdevcon}]
From \autoref{thm:maxdev} (the interpolation term does not change the order) and \autoref{lem:bias} 
we have  
\begin{align*}
\sup_{u \in U_T} \sup_{\omega }\sup_{\eta \in I}\bignorm{\eta(\hat{\F}_{u,\omega}(\eta)-\E\tilde{\F}_{u,\omega}(\eta))}_{S_1}
&= O_p\Bigg(\frac{\big(\bw T^{1/\gamma} \log(N)^{2+1/\gamma}+\mathcal{R}^{1/\gamma}_{N,\bf,\gamma} \mathrm{1}_{\{\rho<1/2-1/\gamma\}}}{N}\Bigg) 
\\ &+O_p\Big(\sqrt{\frac{\bw \log{T}}{N}}\Big)+ O_p\Big(\frac{M^{1/\gamma}  \bw^{1/2+1/\gamma} }{T}\Big), \tageq \label{eq:maxdevvar}\\
\sup_{u \in U_T} \sup_{\omega }\sup_{\eta \in I}\bignorm{ \eta(\E\tilde{\F}_{u,\omega}(\eta)-{\F}_{u,\omega})}_{S_1}& =O\Big(\frac{\tilde{b}^{1/2} N}{T}\Big)+O\Big(\frac{1}{N}\Big)+O(\bf^{2}\big), \tageq \label{eq:maxdevbi}
\end{align*}
\[
 \mathcal{R}_{N,\bf,\gamma}=T\bw N^{\gamma(1/4-\rho/2)-1/2} \bw^{\gamma(3/4-\rho/2)-1/2} \log_2(N/\bw)^{2\gamma+1}.
\]
Under \autoref{as:bandwidth} it 
is easily seen from \eqref{eq:bias12} that the last term in \eqref{eq:maxdevvar} is of order $o_p(1)$ and that \eqref{eq:maxdevbi} is of order $o(1)$,  for any $\gamma > 2$. For the first term, we note that 
$\tilde b \sim N^\kappa$, $T =N^{1/\alpha}$. Thus
we require (note that $N=T^\alpha$, \autoref{as:bandwidth})
\[
N^{\gamma(\kappa-1) }T \log(N)^{2\gamma+1} =
N^{\gamma(\kappa-1) } N^{1/\alpha} \log(N)^{2\gamma+1} = o(1) 
\]
which holds if  ~
\[
\gamma> 
\frac{1}{\alpha(1-\kappa)}~.
\]
For the third term in \eqref{eq:maxdevvar}, 
we note that 
$
N^{(\kappa-1) } \log(T)=T^{\alpha(\kappa-1)} \log(T)
=  o(1)  .
$
Finally, in the  case $\rho<1/2-1/\gamma$, we have for the second term in \eqref{eq:maxdevvar}
\begin{align*} 
 \mathcal{R}_{N,\bf,\gamma} N^{-\gamma} &= N^{1/\alpha+\gamma(-3/4-\rho/2)-1/2+\kappa(\gamma(3/4-\rho /2)+1/2)} \log_2(N^{1-\kappa})^{2\gamma+1} =o(1) 
 \end{align*} 
which requires 
 $$
 ~\gamma> \frac{\frac{1}{\alpha}-\frac{1}{2}(1-\kappa)}{\frac{3}{4}(1-\kappa)+\frac{\rho}{2}(\kappa+1)}. 
 $$
Hence, altogether under \autoref{as:bandwidth}, we require
 $\gamma>1/(\alpha(1-\kappa))$ if $\rho\ge 1/2-1/\gamma$, and 
 $$
 \gamma >\max\Big(\frac{1}{\alpha(1-\kappa)}, \frac{\frac{1}{\alpha}-\frac{1}{2}(1-\kappa)}{\frac{3}{4}(1-\kappa)+\frac{\rho}{2}(\kappa+1)}\Big)
 $$
 if  $\rho> 1/2-1/\gamma$.
\end{proof}

\subsection{Moment inequalities} 

\begin{thm}\label{thm:approx}
Suppose that Assumptions \ref{as:depstruc} (with $p\ge 6$), \ref{as:Weights}, \ref{as:mappings} hold, and recall the definition of 
$\ldmi{}{u_i}{}{N,\flo{\eta N}} $ in \eqref{eq:Mar}. Then, for $\gamma=p/2$,
\begin{align*}
&\E \Big( \sup_{\eta}\bignorm{\sum_{i=1}^{M}\flo{\eta N} \Upsilon_{u_i,\omega}\big(\tilde{\F}_{u_i,\omega}(\eta)-\E \tilde{\F}_{u_i,\omega}(\eta)\big)- \sum_{i=1}^{M}\big(\ldmi{}{u_i}{}{N,\flo{\eta N}}+\ldmi{\dagger}{u_i}{}{N,\flo{\eta N}}\big) }_{S_1}\Big)^{\gamma}\\
& = O\Big( M^{\gamma/2}N^{\gamma/2} \bf^{-\gamma/2}  \big(\sum_{r \in \nnum} \Lambda_{2\gamma,m}(e_r) \big)^{\gamma}+ m^{3/2 \gamma}  M^{\gamma/2}N^{\gamma/2}\big( o(\bf^{-1}) (1+m) \big)^{\gamma/2}\Big),
\end{align*}
where  $\Lambda_{2\gamma,m}(e_r) = 2\sum_{t=0}^{\infty}\min\Big(\nu^{{X^\cdot_\cdot}(e_r)}_{\hi, 2\gamma}(t),\big(\sum_{i=m}^{\infty}(\nu^{{X^\cdot_\cdot}(e_r)}_{\hi, 2\gamma}(i))^2\big)^{1/2}\Big)$. 
\end{thm}

\begin{proof}[Proof of \autoref{thm:approx}] 
To ease the exposition, we focus on showing this statement for the non-interpolated component, i.e., 
\begin{align*}
&\E \Big( \max_{1 \le k \le N}\bignorm{\sum_{i=1}^{M}k \Upsilon_{u_i,\omega}\big(\tilde{\F}^{(u_i,\omega)}_{k}-\E \tilde{\F}^{(u_i,\omega)}_{k}\big)- \sum_{i=1}^{M}\sum_{t=1}^{k}\Upsilon_{u_i,\omega}\big( V^{(u_i,\omega)}_{m,N,t}+ V^{\dagger (u_i,\omega)}_{m,N,t}\big) }_{S_1}\Big)^{\gamma}\\& = O( M^{\gamma/2}N^{\gamma/2} \bf^{-\gamma/2}  \Big(\sum_{r \in \nnum} \sqrt{\Lambda^2_{2\gamma,m}(e_r)} \Big)^{\gamma}+ m^{3/2 \gamma}  M^{\gamma/2}N^{\gamma/2}\big( o(\bf^{-1}) (1+m) \big)^{\gamma/2}). \tageq\label{eq:l56}
\end{align*}
The reader can easily verify that  writing out the interpolation error term (for given $u_i, \omega$) 
\begin{align*}
\tilde{\F}_{u_i,\omega}(\eta)-\tilde{\F}^{(u_i,\omega)}_{k}-\E\Big(\tilde{\F}_{u_i,\omega}(\eta)-\tilde{\F}^{(u_i,\omega)}_{k}\Big)-V^{(u_i,\omega)}_{m,N,\flo{\eta N}+1}-V^{\dagger(u_i,\omega)}_{m,N,\flo{\eta N}+1}
\end{align*}
in detail unveils that the subsequent argument can be applied in order to show this term is of lower order than \eqref{eq:l56}. To prove  \eqref{eq:l56}, we further use the decomposition
\[
\tilde{\F}_{k}^{(u,\omega)} = \tilde{\F}_{k,t<s}^{(u,\omega)} +\tilde{\F}_{k,t=s}^{(u,\omega)}+\tilde{\F}_{k,t>s}^{(u,\omega)}
\]
and similar for $\tilde{\F}_{m,k}^{(u,\omega)}$. The statement then follows from Minkowsk's inequality,  \autoref{lem:mdepap}, \autoref{lem:marappr}, and \autoref{lem:diag}, and from noting that noting that $\Big(\tilde{\F}_{k,t<s}^{(u,\omega)}\Big)^\dagger=\tilde{\F}_{k,t>s}^{(u,\omega)}$.
\end{proof}

\subsubsection{Auxiliary lemmas for the proof of   \autoref{thm:approx}}

\begin{lemma}[$m$-dependent approximation]\label{lem:mdepap}
 Under the conditions of \autoref{thm:approx}
 \begin{align*}
& \E\Big(\max_{1\le k \le N}\bignorm{\sum_{i=1}^{M}  k \Upsilon_{u_i,\omega} \big(\tilde{\F}^{(u_i,\omega)}_{k,t>s}-\E \tilde{\F}^{(u_i,\omega)}_{k,t>s}\big)-k\Upsilon_{u_i,\omega}\big(\tilde{\F}^{(u_i,\omega)}_{m,k,t>s}-\E \tilde{\F}^{(u_i,\omega)}_{m, k,t>s}\big)}_{S_1}\Big)^\gamma
\\&=O\Big ( (N\bf^{-1})^{\gamma/2} M^{\gamma/2} \big(\sum_{r \in \nnum} \Lambda_{2\gamma,m}(e_r) \big)^\gamma\Big ).
\end{align*}
\end{lemma}

\begin{proof}[Proof of \autoref{lem:mdepap}]
Without 
loss of generality, we prove the statement for the case $M=T/N$. To ease notation let for $2 \le t \le N$
\begin{align*}
W_{t,i} & =\Upsilon_{u_i,\omega} \Big( {X}^{(u_i;t)}_{u_i;t} \otimes \sum_{s=1}^{t-1} \tilde{w}^{(\omega)}_{b_f,s,t}   {X}^{(u_i;s)}_{u_i;s}\Big)~, 
\\
W^{(m)}_{t,i} & =\Upsilon_{u_i,\omega} \Big( {X}^{(u_i;t)}_{m,u_i;t}\otimes \sum_{s=1}^{t-1} \tilde{w}^{(\omega)}_{b_f,s,t}   {X}^{(u_i;s)}_{m,u_i;s}\Big)~, 
\\
\tilde{W}^{(m)}_{t,i}& =\Upsilon_{u_i,\omega} \Big( {X}^{(u_i;t)}_{u_i;t}\otimes \sum_{s=1}^{t-1} \tilde{w}^{(\omega)}_{b_f,s,t}  {X}^{(u_i;s)}_{m,u_i;s}\Big)~, 
\end{align*}
and, for $t=1$, set $W_{t,i}=W^{(m)}_{t,i}=\tilde{W}^{(m)}_{t,i}=0$. Denote $\barbelow{W} = \text{vec}\big(\{W_{t,i}\}_{t,i}\big)$ the vectorization of the matrix $\{W_{t,i}\}_{t,i}$ by stacking its columns on top of each other. Denote $i_l = \flo{(l-1)/N}+1$ and $t_l=\mathrm{Remainder}(l-1,N)+1$ such that \[
\barbelow{W}_l=W_{t_l, i_l}.
\] 
Observe that $l = \flo{u_{i_l}T}-\flo{N/2}+t_l$, and that  $\barbelow{W}_l$ is a measurable function of $\{\ldots, \epsilon_{l-1}, \epsilon_{l}\}$. Recalling the notation in
 \eqref{eqnr106}, 
we additionally introduce  the quantities
 \[
 \tilde{X}^{u_{i_l}}_{t}=\tilde{X}^{(u_{i_l}; t)}_{u_{i_l};t}, \quad 
\barbelow{X}_{l}=\tilde{X}^{u_{i_l}}_{t_l}
 \quad \barbelow{N}_{l}=\sum_{s=1}^{t_l-1} \tilde{w}^{(\omega)}_{\bf,s,t_l}\tilde{X}^{u_{i_l}}_{s}, \quad Y_{MN}=\sum_{l=1}^{MN}\mathrm{1}_{2\le t_l \le k} \barbelow{W}_l,
\tageq\label{eq:short2}
 \]
and we equip $m$-dependent versions with an $m$ in the superscripts. 
With this in place, \autoref{lem:Burkh} yields
\begin{align}
\nonumber
&\bignorm{\sum_{i=1}^{M}
\sum_{t=2}^k  W_{t,i}-\E W_{t,i}-(W^{(m)}_{t,i}-\E W^{(m)}_{t,i})}^2_{S_1,\gamma}
=
\bignorm{\sum_{l=1}^{MN}\mathrm{1}_{2 \le t_l \le k} \Big(\barbelow{W}_l -\E \barbelow{W}_l-(\barbelow{W}^{(m)}_l -\E \barbelow{W}^{(m)}_l)\Big)}^2_{S_1,\gamma}
\\ &
\le 
2 \bignorm{\sum_{j=-\infty}^{MN}P_j(Y_{MN} -\tilde{Y}^{(m)}_{MN})}^2_{S_1,\gamma}+
2 \bignorm{\sum_{j=-\infty}^{MN}P_j( \tilde{Y}^{(m)}_{MN}-Y^{(m)}_{MN})}^2_{S_1,\gamma}\\ &
\le 
2 K^2_\gamma \Big(\sum_{r \in \nnum} \sqrt{\sum_{j=-\infty}^{MN} \bignorm{  P_j(Y_{MN} -\tilde{Y}^{(m)}_{MN})(e_r)}^2_{\Hi,\gamma}}\Big)^2+
2 K^2_\gamma \Big(\sum_{r \in \nnum} \sqrt{\sum_{j=-\infty}^{MN} \bignorm{  P_j(\tilde{Y}^{(m)}_{MN}-Y^{(m)}_{MN})(e_r)}^2_{\Hi,\gamma}}\Big)^2
\label{eqnr103}
\end{align}
Observe that \[
\E[Y_{MN} -\tilde{Y}^{(m)}_{MN}|\G_{j-1}](e_r)=
\E[Y_{MN,\{j\}}-\tilde{Y}^{(m)}_{MN,\{j\}}|\G_{j}](e_r)~,
\]
where $Y_{MN,\{j\}}$ is shorthand notation for the sum  $Y_{MN}$ with elements 
$\mathrm{1}_{2\le t_l \le k} \barbelow{W}_{l,\{j\}}=\mathrm{1}_{2\le t_l \le k}\E[\barbelow{W}_{l}|\G_{l,\{j\}}]$. The contraction property of the conditional expectation and Minkowski's inequality yield
\begin{align*}
\bignorm{P_{j}(Y_{MN} -\tilde{Y}^{(m)}_{MN})(e_r)}^2_{\Hi,\gamma}& \le  \bignorm{\Big(Y_{MN} -\tilde{Y}^{(m)}_{MN}-(Y_{MN,\{j\}} -\tilde{Y}^{(m)}_{MN,\{j\}})\Big)(e_r)}^2_{\Hi,\gamma} \leq  2( J_{1,j}^2+J_{2,j}^2)  ~, \tageq \label{eq:J1J2}
\end{align*}
where
\begin{align*}
J_{1,j} & = \bignorm{\sum_{l=1}^{NM}\mathrm{1}_{2 \le t_l\le k}\Upsilon_{u_{i_l},\omega} \Big( (\barbelow{{X}}_{l}-\barbelow{{X}}_{l,\{j\}}) \otimes (\barbelow{N}_{l}-\barbelow{N}^{(m)}_{l}\Big)(e_r)}_{\Hi,\gamma}~,
\\
J_{2,j}& = \bignorm{\sum_{l=1}^{NM}\mathrm{1}_{2 \le t_l\le k}\Upsilon_{u_{i_l},\omega} \Big(\barbelow{{X}}_{l,\{j\}} \otimes (\barbelow{N}_{l}-\barbelow{N}^{(m)}_{l}-(\barbelow{N}_{l,\{j\}}-\barbelow{N}^{(m))}_{l,\{j\}})\Big)(e_r)}_{\Hi,\gamma}.
\end{align*}
Furthermore, 
\begin{align*}
&  \sup_{u \in [0,1]} \| ( \barbelow{\ddot{X}}^{(u)}_{l}-\barbelow{\ddot{X}}^{(u)}_{l,\{j\}}- \barbelow{\ddot{X}}^{(u)}_{m,l}+\barbelow{\ddot{X}}^{(u)}_{m,l,\{j\}})(e_r) \|_{\mathbb{C},2\gamma} 
\\
& \le  \sup_{u \in [0,1]} \min \Big \{ \|(\barbelow{\ddot{X}}^{(u)}_{l}- \barbelow{\ddot{X}}^{(u)}_{m,l})
(e_r)\|_{\mathbb{C},2\gamma}+
\|( \barbelow{\ddot{X}}^{(u)}_{l,\{j\}} -\barbelow{\ddot{X}}^{(u)}_{m,l,\{j\}}) (e_r)\|_{\mathbb{C},2\gamma}, \\
&  ~~~~~~~~~
+ \|( \barbelow{\ddot{X}}^{(u)}_{l}-\barbelow{\ddot{X}}^{(u)}_{l,\{j\}}) (e_r) \|_{\mathbb{C},2\gamma}
 +\|(\barbelow{\ddot{X}}^{(u)}_{m,l}-\barbelow{\ddot{X}}^{(u)}_{m,l,\{j\}})(e_r)\|_{\mathbb{C},2\gamma} \Big \} 
 \\
 &\le 2  \sup_{u \in [0,1]}\min\Big( (\sum_{i=m+1}^{\infty} \|D^{(u)}_{l,i}\|^2_{\mathbb{C},2\gamma})^{1/2}, \nu^{X^\cdot_\cdot (e_r)}_{\mathbb{C},2\gamma}(l-j)\Big),
\end{align*}
where $D^{(u)}_{l,j}=\E[\barbelow{\ddot{X}}^{(u)}_{l}|\G_{l-j}]-\E[\barbelow{\ddot{X}}^{(u)}_{l}|\G_{l-j+1}]$, which satisfies $\sup_{u \in [0,1]}\|D^{(u)}_{l,j}\|^2_{\mathbb{C},2\gamma} \le (\nu^{X_\cdot^\cdot}_{\mathbb{C},2\gamma}(j))^2$. Then, \autoref{lem:Burklin}, \autoref{as:mappings}(i) and tedious calculations yield 
\begin{align*}
 &
 \sum_{j=-\infty}^{NM} J^2_{2,j}
\\
&\le
 \sum_{j=-\infty}^{NM}\Big ( \sum_{i=1}^{M}\sum_{s=1}^{k-1}\norm{\Upsilon_{u_{i},\omega}}_\infty \bignorm{\sum_{t=s+1}^{k} \tilde{w}^{(-\omega)}_{\bf,s,t} \tilde{X}^{(u_{i})}_{t,\{j\}}}_{\Hi,2\gamma} \bignorm{\Big( \tilde{X}^{(u_{i})}_{s}- \tilde{X}^{(u_{i})}_{m,s}+\tilde{X}^{(u_{i})}_{m,s,\{j\}} -\tilde{X}^{(u_{i})}_{s,\{j\}}\Big)(e_r) }_{\cnum,2\gamma}\Big)^2
\\&\le \sup_u \norm{\Upsilon_{u,\omega}}^2_\infty \sum_{j=-\infty}^{NM}
\Big ( \sum_{\ell=1}^{MN}\mathrm{1}_{1\le s_\ell \le k-1} \bignorm{\sum_{t=s_\ell+1}^{k} \tilde{w}^{(-\omega)}_{\bf,s_\ell,t} \tilde{X}^{(u_{i_\ell})}_{t,\{j\}}}_{\Hi,2\gamma}\sup_{u\in[0,1]} \bignorm{\big( \barbelow{\ddot{X}}^{(u)}_{\ell}- \barbelow{\ddot{X}}^{(u)}_{m,\ell}+\barbelow{\ddot{X}}^{(u)}_{m,\ell,\{j\}} - \barbelow{\ddot{X}}^{(u)}_{\ell,\{j\}} \big)(e_r)}_{\cnum,2\gamma}\Big)^2
\\& \le \sup_u \norm{\Upsilon_{u,\omega}}^2_\infty K^2_{2\gamma}\big( \sum_{j=0}^\infty \nu^{{X}^{\cdot}_{\cdot}}_{\hi,2\gamma}(j)\big)^2   \sum_{j=-\infty}^{NM} \max_{1\le s \le k-1} \sum_{t=s+1}^{k}
\big\vert \tilde{w}_{\bf,s,t}^{(\omega)} \big\vert^2 \\
& \times \Big( \sum_{\ell=1}^{NM}\mathrm{1}_{1\le s_\ell \le k-1}  2 \sup_{u \in [0,1]}\min\big(\sum_{m^\pr=m+1}^{\infty} \|D^{(u)}_{\ell,m^\pr}(e_r)\|_{\cnum,2\gamma},\nu^{X^\cdot_\cdot(e_r)}_{\cnum,2\gamma}(\ell-j)\big)\Big)^2 \\
&= k\sum_{h=1-k}^{-1} |w(\bf h)|^2 O\Big( M  \sup_{u}\|\Upsilon_{u,\omega}\|^2_\infty \big( \sum_{j=0}^\infty \nu^{{X}^{\cdot}_{\cdot}}_{\hi,2\gamma}(j)\big)^2  \Lambda^2_{2\gamma,m}(e_r)\Big)~.
 \tageq \label{eqnr104}
\end{align*}
For the first term in \eqref{eq:J1J2}, mimicking the previous argument yields
\begin{align*}
\sum_{j=-\infty}^{NM} & J^2_{1,j}=\sum_{j=-\infty}^{NM} \bignorm{\sum_{l=1}^{NM} \mathrm{1}_{2 \le t_l \le  k} \Upsilon_{u_{i_l},\omega} \Big((\barbelow{{X}}_{l}-\barbelow{{X}}_{l,\{j\}}) \otimes (\barbelow{N}_{l}-\barbelow{N}^{(m)}_{l})\Big)(e_r)}^2_{\hi,\gamma}
\\ & \le \sup_u \norm{\Upsilon_{u,\omega}}^2_\infty \sum_{j=-\infty}^{NM} \Big(\sum_{l=1}^{NM} \mathrm{1}_{2 \le t_l \le  k}  \norm{\barbelow{{X}}_{l}-\barbelow{{X}}_{l,\{j\}}}_{\Hi,2\gamma} \norm{(\barbelow{N}_{l}-\barbelow{N}^{(m)}_{l})(e_r)}_{\cnum,2\gamma}\Big)^2
\\& \lesssim  \sup_u \norm{\Upsilon_{u,\omega}}^2_\infty \max_{2 \le t\le k}\sum_{s=1}^{t-1}\big\vert {w}_{\bf,s,t} \big\vert^2 \big( \sum_{m^\pr=m+1}^\infty \nu^{{X}^{\cdot}_{\cdot}(e_r)}_{\cnum,2\gamma}(m^\pr)\big)^2 \sum_{l=1}^{NM} \mathrm{1}_{2 \le t_l \le  k}    \sum_{j=-\infty}^{NM} \nu^{X^\cdot_\cdot}_{\hi,2\gamma}(l-j) \sum_{k=1}^{NM} \nu^{X^\cdot_\cdot}_{\hi,2\gamma}(k-j) 
\\& \lesssim\max_{2 \le t \le k}\sum_{s=1}^{t-1}\big\vert {w}_{\bf}(s-t) \big\vert^2 O\Big(Mk \sup_u \norm{\Upsilon_{u,\omega}}^2_\infty \big( \sum_{m^\pr=m+1}^\infty \nu^{{X}^{\cdot}_{\cdot}(e_r)}_{\cnum,2\gamma}(m^\pr)\big)^2  \big(\sum_{j=0}^\infty \nu^{{X}^{\cdot}_{\cdot}}_{\hi,2\gamma}(j)\big)^2 \Big)
\\& =k \sum_{h=1}^{k-1}\big\vert w(\bf h) \big\vert^2 
O \Big (M 
\sup_{u}\|\Upsilon_{u,\omega}\|^2_\infty \Lambda^2_{2\gamma,m}(e_r)\big(\sum_{j=0}^\infty \nu^{{X}^{\cdot}_{\cdot}}_{\hi,2\gamma}(j)\big)^2 \Big )~.
 \tageq \label{eqnr105}
\end{align*}
Combining  \eqref{eqnr104} and \eqref{eqnr105}, we obtain
\begin{align*}
& \Big(\sum_{r \in \nnum} \sqrt{\sum_{j=-\infty}^{MN} \bignorm{  P_j(Y_{MN} -\tilde{Y}^{(m)}_{MN})(e_r)}^2_{\Hi,\gamma}}\Big)^2
\\& \lesssim \Bigg(\sum_{r \in \nnum} \sqrt{k \sum_{h=1}^{k-1}\big\vert w(\bf h) \big\vert^2
M \sup_{u}\|\Upsilon_{u,\omega}\|^2_\infty \Lambda^2_{2\gamma,m}(e_r)\big(\sum_{j=0}^\infty \nu^{{X}^{\cdot}_{\cdot}}_{\hi,2\gamma}(j)\big)^2 }~~\Bigg)^2
\\&= k \sum_{h=1}^{k-1}\big\vert w(\bf h) \big\vert^2 
O \Big (M \sup_{u}\|\Upsilon_{u,\omega}\|^2_\infty \big(\sum_{j=0}^\infty \nu^{{X}^{\cdot}_{\cdot}}_{\hi,2\gamma}(j)\big)^2 \Big(\sum_{r \in \nnum} \sqrt{\Lambda^2_{2\gamma,m}(e_r)} \Big)^2~. 
\end{align*}
Then, if we let $g(b,k)= k \sum_{h=b+1}^{b+k}\big\vert w(\bf h) \big\vert^2 $ we observe that for all $b \ge 0$ and $1\le k < k+l$,
\[ 
g(b,k) + g(b+k,l) \leq g(b,k+l).
 \tageq \label{eq:mor7}
\]
It then follows from Theorem 1 in \cite{Mor76}, and from the fact that the second term in \eqref{eqnr103} can be treated similarly, that
\begin{align*}
& \E\Big(\max_{1\le k \le N}\bignorm{\sum_{i=1}^{M}  k\Upsilon_{u_i,\omega} \Big(\big(\tilde{\F}^{(u_i,\omega)}_{k,t>s}-\E \tilde{\F}^{(u_i,\omega)}_{k,t>s}\big)-k\Upsilon_{u_i,\omega} \Big(\big(\tilde{\F}^{(u_i,\omega)}_{m,k,t>s}-\E \tilde{\F}^{(u_i,\omega)}_{m, k,t>s}\big)}_{S_1}\Big)^\gamma
\\ & = O\Bigg( \Big(\sup_{u}\|\Upsilon_{u,\omega}\|_\infty N \sum_{h=1}^{N-1}\big\vert w(\bf h) \big\vert^2 M \big(\sum_{r \in \nnum} \Lambda_{2\gamma,m}(e_r) \big)^2\big(\sum_{j=0}^\infty \nu^{{X}^{\cdot}_{\cdot}}_{\hi,2\gamma}(j)\big)^2\Big)^{\gamma/2}\Bigg)
\\&= O\Big((N\bf^{-1})^{\gamma/2} M^{\gamma/2} \big(\sum_{r \in \nnum} \Lambda_{2\gamma,m}(e_r) \big)^{\gamma}\Big).
\end{align*}

\end{proof}

\begin{lemma}[martingale approximation to the $m$-dependent process]\label{lem:marappr}
 Under the conditions of \autoref{thm:approx} we have, for some positive constant $C_\gamma$,
\begin{align*}
 \E\Big(\max_{1\le k \le N}&\bignorm{\sum_{i=1}^{M} k \Upsilon_{u_{i},\omega} \big(\tilde{\F}^{(u_i,\omega)}_{m,k,t>s}-\E \tilde{\F}^{(u_i,\omega)}_{m, k,t>s}\big)- \sum_{i=1}^{M}\sum_{t=2}^{k} \Upsilon_{u_{i},\omega} \big(V^{(u_i,\omega)}_{m,N,t}\big)}_{S_1}\Big)^\gamma
\\& \le C^\gamma_{\gamma} m^{\frac{3}{2}\gamma} {(M N)}^{\gamma/2} (\sup_{u\in [0,1]} \|X^{(u)}_0\|_{\hi,2\gamma} \sup_{u\in [0,1]}\sum_{r\in \nnum}\|X^{(u)}_0(e_r)\|_{\cnum,2\gamma})^{\gamma} 
\\& \times( \sup_{u \in [0,1]}\|\Upsilon_{u,\omega}\|_{\infty})^\gamma
 (2 \max_{t}|w( t)|^2+ 2m\sum_{s=1}^{N} |w(\bf( s))-w(\bf( s-1))|^2 )^{\gamma/2}.
\end{align*}
\end{lemma}
\begin{proof}[Proof of \autoref{lem:marappr}]
 Define
\begin{align*}
\ddot{D}^{(v,\omega)}_{m,u,t}  = D_{m,\tu{N}{u}+t}^{(v,\omega)}=\sum_{p=0}^{m} P_{\tu{N}{u}+t}\Big({X}^{(v)}_{m,\tu{N}{u}+t+p}  \Big) e^{-\im \omega p}
\end{align*} 
and set 
$
\ddot{X}^{(v)}_{m,u,t}= {X}^{(v)}_{m,\tu{N}{u}+t}$ and  $\tilde{X}^{(v,t)}_{m,u,t}={X}^{((\tu{N}{v}+t)/T)}_{m,\tu{N}{u}+t}$.
 Using the notation introduced
in  the proof of  \autoref{lem:mdepap}, we have
\[
\bignorm{\sum_{i=1}^{M}\sum_{t=2}^{k}(W^{(m)}_{t,i}-\E W^{(m)}_{t,i})- \sum_{i=1}^{M}\sum_{t=1}^{k} 
\Upsilon_{u_{i},\omega} \big (  V^{(u_i,\omega)}_{m,N,t}\big)}_{S_1,\gamma}
=\bignorm{\sum_{l=1}^{M N}\mathrm{1}_{2\le t_l \le k}(\barbelow{W}^{(m)}_l-\E\barbelow{W}^{(m)}_l-   \Upsilon_{u_{i_l},\omega}  ( V^{(u_{i_l},\omega)}_{m,N,t_l} )\big)}_{S_1,\gamma}
~, 
\]
and we write
$$
\barbelow{W}^{(m)}_l-\E\barbelow{W}^{(m)}_l-  
 \Upsilon_{u_{i_l},\omega} (V^{(u_{i_l},\omega)}_{m,N,t_l}  ) = 
\Upsilon_{u_{i_l},\omega}(A_l  +B_l  -\E B_l +Z_l  -\E Z_l ),
$$
where
\begin{align*}
A_l &=  \tilde{X}^{(u_{i_l},t_l)}_{m,u_{i_l},t_l} \otimes \sum_{s=1}^{t_l-4m}  \tilde{w}^{(\omega)}_{\bf,s,t_l}  ( \tilde{X}^{(u_{i_l},s)}_{m,u_{i_l},s}- \dmpi{u_{i_l}}{\omega}{N}{s}) 
~,
\tageq \label{eq:Al}
\\ 
B_l &=  \tilde{X}^{(u_{i_l},t_l)}_{m,u_{i_l},t_l} \otimes \sum_{s=t_l-4m+1 \vee 1}^{t_l-1}  \tilde{w}^{(\omega)}_{\bf,s,t_l}  (\tilde{X}^{(u_{i_l},s)}_{m,u_{i_l},s}-\dmpi{u_{i_l}}{\omega}{N}{s})
~,\tageq \label{eq:Bl}
\\
Z_l &= (\tilde{X}^{(u_{i_l},t_l)}_{m,u_{i_l},t_l}-\dmpi{u_{i_l}}{\omega}{N}{t_l}) \otimes \sum_{s=1}^{t_l-1}  \tilde{w}^{(\omega)}_{\bf,s,t_l}  \dmpi{u_{i_l}}{\omega}{N}{s}
~. \tageq \label{eq:Zl}
\end{align*}
We treat the terms \eqref{eq:Al}, \eqref{eq:Bl}, and \eqref{eq:Zl} separately  making  use of counterparts with the re-scaled time parameter of the auxiliary process independent of $l$, that is,
\begin{align*}
A^{(u)}_l &= \ddot{X}^{(u)}_{m,u_{i_l},t_l} \otimes \sum_{s=1}^{t_l-4m}   \tilde{w}^{(\omega)}_{\bf,s,t_l}(
\ddot{X}^{(u)}_{m,u_{i_l},s}  
- \ddot{D}^{(u,\omega)}_{m,u_{i_l},s})~,
\\ 
B^{(u)}_l &= \ddot{X}^{(u)}_{m,u_{i_l},t_l} \otimes \sum_{s=t_l-4m+1 \vee 1}^{t_l-1} \tilde{w}^{(\omega)}_{\bf,s,t_l}  (
\ddot{X}^{(u)}_{m,u_{i_l},s} 
-\ddot{D}^{(u,\omega)}_{m,u_{i_l},s})~,
\\
Z^{(u)}_l &= \ddot{X}^{(u)}_{m,u_{i_l},t_l}-\ddot{D}^{(u,\omega)}_{m,u_{i_l},t_l}\otimes \sum_{s=1}^{t_l-1}  \tilde{w}^{(\omega)}_{\bf,s,t_l} \ddot{D}^{(u,\omega)}_{m,u_{i_l},s} ~.
\end{align*}
 We start with $A_l $
 noting that we can write $\ddot{X}^{(u)}_{m,u_{i_l},s}= Q_{l,s}- \E[Q_{l,s+1}|\G_{\tu{N}{u_{i_l}},s}]e^{-i\omega}$,  where $$
 Q_{l,s}= \sum_{r=0}^{m} \E[ \ddot{X}^{(u)}_{m,u_{i_l},s+r} |\G_{\tu{N}{u_{i_l}},s}] e^{-\im r\omega}.
 $$
Observe that $\ddot{D}^{(u,\omega)}_{m,u_{i_l},s}= Q_{l,s}-\E[Q_{l,s}|\G_{\tu{N}{u_{i_l}},s-1}]$, then

\begin{align*}
& \bignorm{ \sum_{s=1}^{t_l-4m}  \tilde{w}^{(\omega)}_{\bf,s,t_l}(\ddot{X}^{(u)}_{m,u_{i_l},s}- \ddot{D}^{(u,\omega)}_{m,u_{i_l},s}) }_{\Hi,\gamma}\\
& ~~~~~~~
=  \bignorm{   \sum_{s=1}^{t_l-4m} (2\pi)^{-1} w(\bf(s-t_l))  \Big(e^{-\im (s-t_l)\omega} \E[Q_{l,s}| \G_{\tu{N}{u_{i_l}},s-1}] -e^{\im (t_l-(s+1))\omega} \E[Q_{l,s+1}| \G_{\tu{N}{u_{i_l}},s}]\Big)}_{\hi,\gamma}. 
\end{align*}
Set $\tilde{Q}_{l,s} = e^{\im (t_l-s)\omega} \E\big[Q_{l,s}| \G_{\tu{N}{u_{i_l}},s-1}\big] $. Summation by parts, H\"older's inequality and \autoref{lem:Burkh} yield
\begin{align*}
&\bignorm{\sum_{s=1}^{t_l-4m}  w(\bf(s-t_l)) \Big(\tilde{Q}_{l,s} -\tilde{Q}_{l,s+1} \Big)}_{\hi,\gamma}
\\& \le \bignorm{w(\bf(4m)) \tilde{Q}_{l,t-4m} }_{\hi,\gamma}+\bignorm{\sum_{s=1}^{t_l-4m} (w(\bf(s-t_l))-w(\bf(s-t_l-1)))\tilde{Q}_{l,s}}_{\hi,\gamma} 
\\& \le \max_{t}|w(t)| \|\tilde{Q}_{l,t-4m})\|_{\hi,\gamma}  + \sqrt{\bignorm{\sum_{q=1}^{m}\sum_{s=1}^{t_l-4m} (
w(\bf(s-t_l))-w(\bf(s-t_l-1)) ) P_{\tu{N}{u_{i_l}},s-q}\big(\tilde{Q}_{l,s}\big)}^2_{\hi,\gamma}}
\\& \le  2m \sup_{u \in [0,1]}\|X^{(u)}_0\|_{\hi,\gamma}\Big(  \max_{t}|w(t)|+C \big(\sum_{s=1}^{t_l-4m}|w(\bf(s-t_l))-w(\bf(s-t_l-1))|^2 m \big)^{1/2} \Big)~,
\end{align*}
where it was used that $ (  \sum_{q=1}^{m}\|P_{\tu{N}{u_{i_l}},s-q}(\tilde{Q}_{l,s})\|^2_{\hi,\gamma}  )^{1/2} \le 2{m^{3/2}} \sup_{u \in [0,1]}\|{X}^{(u)}_0\|_{\hi,\gamma} $
since $\|\tilde{Q}_{l,s}\|_{\hi,\gamma}  \le 2m \sup_{u \in [0,1]}\|{X}^{(u)}_0\|_{\hi,\gamma} $.
Thus we obtain, for some finite constant $C_\gamma$, 
\begin{align*}
\|\Upsilon_{u_{i_l},\omega} A^{(u)}_l\|_{\hi,\gamma} &\le  C_\gamma m
\sup_{u \in [0,1]}\|{X}^{(u)}_0\|^2_{\hi,\gamma} \sup_{u \in [0,1]}\|\Upsilon_{u,\omega}\|_{\infty} \\& \times \Bigg(  \max_{t}|w(t)| +
 m^{1/2} \Big (\sum_{s=1}^{t_l-4m} |w(\bf(s-t_l))-w(\bf(s-t_l-1))|^2 \Big )^{1/2}\Bigg)
\end{align*}
  It may verified that $\big\{A^{(u)}_{l+4mq}\big\}_{q \in \mathbb{N}}$ is a martingale difference sequence in $\op^2_{S_1}$ with respect to $
\{\G_{l+4mq}\}_{q \in \nnum}$. Hence, the above and  \autoref{lem:Burkh} yield
\begin{align*}
&\bignorm{\sum_{l=1}^{T}\mathrm{1}_{2 \le t_l \le k}\Upsilon_{u_{i_l},\omega} (A_l) }_{S_1,\gamma} 
\\& \le \sup_{u \in [0,1]} \sum_{s=1}^{4m} \bignorm{\sum_{q=0}^{\lfloor (T-s)/(4m) \rfloor} \mathrm{1}_{2\le t_{s+4mq}\le k} \Upsilon_{u_{i_{s+4m q}},\omega} (A^{(u)}_{s+4m q}) }_{S_1,\gamma} 
\\
&  \lesssim \sup_{u \in [0,1]} \sum_{s=1}^{4m}
\sum_{r \in \nnum} \sqrt{ \sum_{q=0}^{\lfloor (T-s)/(4m) \rfloor} \mathrm{1}_{2\le t_{s+4mq}\le k} \norm{\Upsilon_{u_{i_{s+4mq}},\omega} }^2_\infty  \norm{A^{(u)}_{s+4mq}(e_r)}^2_{\Hi,\gamma} } 
\\
\\& = O\Big(m^{3/2}  \sup_{u \in [0,1]}\|\Upsilon_{u,\omega}\|_{\infty}\sup_{u \in [0,1]}\|{X}^{(u)}_0\|_{\hi,\gamma} \sum_{r \in \nnum}\|{X}^{(u)}_0(e_r)\|_{\cnum,\gamma}\big)
\\&~~\times {(Mk)}^{1/2}  \big(  \max_{t}|w(t)|+m^{1/2}(\sum_{s=1}^{k} |w(\bf(s))-w(\bf(s-1))|^2 )^{1/2}\big).
\end{align*}
For $B_{l}$, we note that
$$\sum_{s=t_l-4m+1 \vee 1}^{t_l-1}  \tilde{w}^{(\omega)}_{\bf,s,t_l}  (\tilde{X}^{(u_{i_l},s)}_{m,u_{i_l},s}-\dmpi{u_{i_l}}{\omega}{N}{s})  $$ is $\G_{\tu{N}{u_{i_l}},t_l-1}$ measurable and that it is $5m$-dependent. Therefore, Minkowski's inequality, the Cauchy-Schwarz inequality and a similar argument as above yields
\begin{align*}
\bignorm{\sum_l \mathrm{1}_{2 \le t_l \le k} \Upsilon_{u_{i_l},\omega} (B_{l} - \E B_l) }_{S_1,\gamma}
& \lesssim   m^{3/2} {(Mk)}^{1/2}  \sup_{u \in [0,1]}\|{X}^{(u)}_0\|_{\hi,2\gamma}\sum_{r \in \nnum}\|{X}^{(u)}_0(e_r)\|_{\cnum,2\gamma} \sup_{u \in [0,1]}\|\Upsilon_{u,\omega}\|_{\infty}
\\&  ~~~\times \big(  \max_{t}|w(t)|+m^{1/2}(\sum_{s=1}^{k} |w(\bf(s))-w(\bf(s-1))|^2 )^{1/2}\big).\end{align*}
A similar derivation yields a similar upper bound for $\|\sum_l \mathrm{1}_{2\le t_l\le k} \Upsilon_{u_{i_l},\omega} (Z_{l} - \E Z_l) \|_{S_1,\gamma}$. Finally, let
 \[
 g(b,k) = k\Big( \max_{t}|w(t)|^2+m\sum_{s=b+1}^{b+k} |w(\bf(s))-w(\bf(s-1))|^2\Big)
 \] Then observe that for any bounded function $f$, and all $b \ge 0$ and $1\le k < k+l$,  \[k \sum_{s=b+1}^{b+k}(1+|f(s)|^2) +l\sum_{s=b+k+1}^{b+k+l}(1+|f(s)|^2) \le (k+l)\sum_{s=b+1}^{b+k+l}(1+|f(s)|^2) \tageq \label{eq:morsum}\] from which we find $g(b,k) + g(b+k, l) \le  g(b, k+l)$. The result now follows from
 Theorem 1 in \cite{Mor76}.
\end{proof}
 
\begin{lemma}\label{lem:diag}
Suppose that Assumptions \ref{as:depstruc} (with $p\ge 6$), \ref{as:Weights}, \ref{as:mappings} hold, and let $\gamma=p/2$. Then for any  fixed  $\omega \in \rnum $
\begin{align*}
\E\Big(\max_{1\le k \le N}\bignorm{\sum_{i=1}^M \Upsilon_{u_i,\omega}\big(\tilde{\F}_{k,t=s}^{(u_i,\omega)}-\E\tilde{\F}_{k,t=s}^{(u_i,\omega)}\big)}_{S_1}\big)^{\gamma} =O\big((MN)^{\gamma/2}\Big).
\end{align*}
\end{lemma}

\begin{proof}[Proof of \autoref{lem:diag}]
 For $1\le t_l \le N$, let $\mathcal{V}_l=\mathrm{1}_{1 \le t_l \le k} 
\Upsilon_{u_{i_l},\omega} \big(\barbelow{X}_l \otimes  \barbelow{X}_l \big)$. 
\autoref{lem:Burkh} yields
\[
\bignorm{\sum_{i=1}^M \Upsilon_{u_i,\omega}\big(\tilde{\F}_{k,t=s}^{(u_i,\omega)}-\E\tilde{\F}_{k,t=s}^{(u_i,\omega)}\big)}^2_{S_1,\gamma} =\bignorm{\sum_{l=1}^{NM}\mathcal{V}_l-\E\mathcal{V}_l}^2_{S_1,\gamma}
\lesssim \Big(\sum_{r \in \nnum} \sqrt{\sum_{j=-\infty}^{NM}\bignorm{P_j\big(\sum_{l=1}^{NM}\mathcal{V}_l(e_r)\big)}^2_{\Hi,\gamma}} \Big)^2~, 
\]
where 
\begin{align*}
\sum_{j=-\infty}^{NM}
\bignorm{P_j\big(\sum_{l=1}^{NM}\mathcal{V}_l(e_r)\big)}^2_{\Hi,\gamma}
&\le  2 \sum_{j=-\infty}^{NM}\bignorm{\sum_{l=1}^{NM} \mathrm{1}_{1 \le t_l \le k}  \Upsilon_{u_{i_l},\omega}\big(\barbelow{X}_l \otimes (\barbelow{X}_l-\barbelow{X}_{l,\{j\}})\big)(e_r)}^2_{\Hi,\gamma}
\\& + 2 \sum_{j=-\infty}^{NM}  \bignorm{\sum_{l=1}^{NM} \mathrm{1}_{1 \le t_l \le k}\Upsilon_{u_{i_l},\omega}\big((\barbelow{X}_l-\barbelow{X}_{l,\{j\}}) \otimes \barbelow{X}_{l,\{j\}}\big)(e_r)}^2_{\Hi,\gamma}. \tageq\label{diag2}
\end{align*}
For the first term in \eqref{diag2}, we obtain
\begin{align*}
& \sum_{j=-\infty}^{NM}\bignorm{\sum_{l=1}^{NM} \mathrm{1}_{1 \le t_l \le k}
\Upsilon_{u_{i_l},\omega}\big(\barbelow{X}_l \otimes (\barbelow{X}_l-\barbelow{X}_{l,\{j\}})\big)(e_r)}^2_{\Hi,\gamma}
\\
& \le \sup_{u}\norm{ \Upsilon_{u,\omega}}^2_{\infty}\sup_{u \in [0,1]} \norm{X^{(u)}_0}^2_{\Hi,2\gamma}\sum_{j=-\infty}^{NM}\Big(\sum_{l=1}^{NM}\mathrm{1}_{1 \le t_l \le k} 
 \nu^{{X}^{\cdot}_{\cdot}(e_r)}_{\hi,2\gamma}(l-j)\Big)^2
\\
&=O\Big(k M \sup_{u}\norm{ \Upsilon_{u,\omega}}^2_{\infty} \sup_{u \in [0,1]} \norm{X^{(u)}_0}^2_{\Hi,2\gamma} \big(\sum_{j=0}^\infty \nu^{{X}^{\cdot}_{\cdot}(e_r)}_{\hi,2\gamma}(j)\big)^2\Big)  ~. 
\end{align*}
A similar argument can be applied to the second term in \eqref{diag2}, so that we find 
\begin{align*}
&\max_{1\le k \le N}\Big(\sum_{r \in \nnum} \sqrt{\sum_{j=-\infty}^{NM}\bignorm{P_j\big(\sum_{l=1}^{NM}\mathcal{V}_l(e_r)\big)}^2_{\Hi,\gamma}}\Big)^2 \\& \le
O\Big(\max_{1\le k\le N}k M \sup_{u}\norm{ \Upsilon_{u,\omega}}^2_{\infty} \sup_{u \in [0,1]} \norm{X^{(u)}_0}^2_{\Hi,2\gamma}\Big( \sum_{r \in \nnum}\sum_{j=0}^\infty \nu^{{X}^{\cdot}_{\cdot}(e_r)}_{\hi,2\gamma}(j)\Big)^2\Big)  
\\& +  O\Big(\max_{1\le k\le N}k M \sup_{u}\norm{ \Upsilon_{u,\omega}}^2_{\infty} \sup_{u \in [0,1]}(\sum_{r \in \nnum} \norm{X^{(u)}_0(e_r)}_{\Hi,2\gamma})^2  \big(\sum_{j=0}^\infty \nu^{{X}^{\cdot}_{\cdot}}_{\hi,2\gamma}(j)\big)^2\Big).  
\end{align*}
It is straightforward (see e.g., \eqref{eq:morsum}) that condition (1.1) in Theorem 1 of \cite{Mor76} is satisfied for the function $g(b,k) = \sum_{i=b+1}^{b+k} 1 = k$, which proves the statement. 
\end{proof}

\subsection{Auxiliary  lemmas  for the  proof of  \autoref{thm:maxPS1}}

\begin{Remark}\label{Rem:fixu}
{\rm 
We remark that all statements in this section hold for fixed $u \in (0,1)$ as well, in which case the respective right-hand sides do not depend on $M$.
}
\end{Remark}

\begin{lemma}\label{lem:approxloc}
Assume the conditions of \autoref{thm:Conv}  with $p> 4$ and let $\gamma=p/2$. Then, for  fixed  $\omega \in \rnum$,
\begin{align*}
\E \Big(\sup_{\eta \in I}\bignorm{\sum_{i=1}^M \flo{\eta N}\Upsilon_{u_i,\omega}\big(\hat{\F}_{u_i,\omega}(\eta)
 -\tilde{\F}_{u_i,\omega}(\eta))}_{S_1}\Big)^\gamma =O((MN)^{\gamma}T^{-\zeta\gamma } \bf^{-\gamma/2}).
\end{align*}
\end{lemma}

\begin{proof}[Proof of \autoref{lem:approxloc}.]
 As in the proof of \autoref{thm:approx}, we again focus on the non-interpolated component as the interpolation error term can be treated similarly, and can be shown to be of lower order. Elementary calculations give
\begin{align}
& \Bignorm{\sum_{i=1}^{M}k \Upsilon_{u_i,\omega} ( {\hat{\F}}_{k}^{(u_i,\omega)}-\tilde{\F}_{k}^{(u_i,\omega)})}^2_{S_1,\gamma} 
 \label{eq:FNk} \\
  \nonumber 
&  ~~~~~~~~~~~ ~~~~~~~~~~~
\le  2 
\sup_{u}\norm{ \Upsilon_{u,\omega}}^2_{\infty}\Big( \sum_{i=1}^{M}\sum_{s=1}^{k} \Bignorm{\sum_{t=1}^{k} \tilde{w}_{\bf,s,t}^{(\omega)} {X}_{u_i;t}}_{\hi,2\gamma} \Bignorm{  X_{u_i;s}-  \tilde{X}^{(u_i,s)}_{u_i;s}}_{\hi,2\gamma}\Big)^2 \\
 & \nonumber  
  ~~~~~~~~~~~ ~~~~~~~~~~~
 +2 \sup_{u}\norm{ \Upsilon_{u,\omega}}^2_{\infty}\Big(\sum_{i=1}^{M} \sum_{t=1}^{k}  \Bignorm{  X_{u_i;t}-  \tilde{X}^{(u_i,t)}_{u_i;t} }_{\hi,2\gamma} \Bignorm{\sum_{s=1}^{k} \tilde{w}_{b,t,s}^{(\omega)}  \tilde{X}^{(u_i,s)}_{u_i;s}}_{\hi,2\gamma}\Big)^2~.
\end{align}
\autoref{lem:Burklin} and \autoref{as:depstrucnonstat} yield
\begin{align*}
\Bignorm{\sum_{t=1}^{k} \tilde{w}_{\bf,s,t}^{(\omega)}  X_{u_i;t}}^2_{\hi,2\gamma} 
\lesssim  \sum_{t=1}^{k} \big\vert \tilde{w}_{\bf,s,t}^{(\omega)} \big\vert^2  \Big(\sum_{j=0}^\infty  \sup_{T \in \nnum} \nu^{X_{\cdot,\cdot}}_{\hi,2\gamma}(j,T)\Big)^{2} ~, 
\end{align*}
and similarly, using \autoref{as:depstruc}, we obtain 
\begin{align*}
\Bignorm{\sum_{s=1}^{k} \tilde{w}_{b,t,s}^{(\omega)}  \tilde{X}^{(u_i,s)}_{u_i;s}}_{\hi,2\gamma}^2 \lesssim \sum_{t=1}^{k} \big\vert \tilde{w}_{\bf,s,t}^{(\omega)} \big\vert^2 \Big( \sum_{j=0}^\infty \nu^{{X}^{\cdot}_{\cdot}}_{\hi,2\gamma}(j)\Big)^{2} ~.
\end{align*}
Observing that, by  \autoref{def:locstat},  we get 
$ \norm{ X_{u;t}-\tilde{X}^{(u,t)}_{u,t}}_{\hi,2\gamma}  = O(T^{-\zeta}).
$  We thus find that  \eqref{eq:FNk} satisfies
\begin{align*}
& \lesssim  T^{-2 \zeta } M (Mk)^{2-1}  \sum_{s=1}^{k}\sum_{t=1}^{k} \big\vert \tilde{w}_{\bf,s,t}^{(\omega)} \big\vert^2 \Big(  \big( \sum_{l=0}^\infty\nu^{{X}^{\cdot}_{\cdot}}_{\hi,2\gamma}(l)\big)^{2} + \big(  \sum_{j=0}^\infty\sup_{T \in \nnum} \nu^{X^{\cdot}_{\cdot,\cdot}}_{\hi,2\gamma}(j,T)\big)^{2} \Big)
\end{align*}
Note that the  symmetry of the window function  implies $k \sum_{t=1}^{k} \sum_{s=1}^{k}|w(\bf(t-s))|^2 
 \le 4k^2 \sum_{h=0}^{k-1} |w(\bf(h))|^2 $.  Therefore, setting 
 $g(b,k)= 4k^2 \sum_{h=b-1}^{k-1} |w(\bf h )|^2 $ and using a similar argument as in the proof of \autoref{lem:mdepap} together with
bandwidth conditions and Theorem 1 in \cite{Mor76}
yields
\[
\E \big(\max_{1\le k \le N}\bignorm{\sum_{i=1}^M k\Upsilon_{u_i,\omega}\big(\hat{\F}_{k}^{(u_i,\omega )}
 -\tilde{\F}_{k}^{(u_i,\omega)})}_{S_1}\big)^\gamma
= \Big[O\Big(\frac{(MN)^{2}}{\bf T^{2\zeta}}\Big)\Big]^{\gamma/2} = O\Big(\frac{(MN)^{\gamma}}{\bf^{\gamma/2}T^{\gamma\zeta}}\Big),
\]
which completes the proof.
\end{proof}

\begin{lemma}\label{lem:approxvarint}
Suppose that Assumptions \ref{as:depstruc} (with $p\ge 6$), \ref{as:Weights}, \ref{as:mappings} hold, and let $\gamma=p/2$. Then for any  fixed  $\omega \in \rnum $
\begin{align*}
\E\Big(\sup_{\eta} \bignorm{\sum_{i=1}^M \flo{\eta N}\Upsilon_{u_i,\omega}\big(\tilde{\F}_{u_i,\omega}(\eta)
 -\E \tilde{\F}_{u_i,\omega}(\eta)\big)}_{S_1}\Big)^\gamma  =O((MN)^{\gamma/2} \bf^{-\gamma/2}).
 \end{align*}
\end{lemma}
\begin{proof}[Proof of \autoref{lem:approxvarint}]
Write $\tilde{\F}_{u_i,\omega}(\eta)=\tilde{\F}^{u_i,\omega}_k+\big(\tilde{\F}_{u_i,\omega}(\eta)-\tilde{\F}^{u_i,\omega}_k\big)$. Then the bound for the 
non-interpolated term follows directly from \autoref{lem:approxvar} and \autoref{lem:diag}. The reader may verify that the latter two lemmas then can be readily applied to show that  the interpolation term is of lower order. Details are omitted for the sake of brevity.
\end{proof}

\begin{lemma}\label{lem:approxvar}
Suppose that Assumptions \ref{as:depstruc} (with $p\ge 6$), \ref{as:Weights}, \ref{as:mappings} hold, and let $\gamma=p/2$. Then for any  fixed  $\omega \in \rnum $
\begin{align*}
\E\Big(\max_{1\le  k \le N} \bignorm{\sum_{i=1}^M k\Upsilon_{u_i,\omega}\big(\tilde{\F}_{k, t> s}^{u_i,\omega}
 -\E \tilde{\F}_{k, t>  s}^{(u_i,\omega)})}_{S_1} \Big)^\gamma =O((MN)^{\gamma/2} \bf^{-\gamma/2}).
\end{align*}
Furthermore, for fixed $u \in (0,1) $
\begin{align*}
\E\Big(\max_{1\le  k \le N} \bignorm{ k\Upsilon_{u_i,\omega}\big(\tilde{\F}_{k, t> s}^{u_i,\omega}
 -\E \tilde{\F}_{k, t>  s}^{(u_i,\omega)})}_{S_1} \Big)^\gamma =O(N^{\gamma/2} \bf^{-\gamma/2}).
\end{align*}
\end{lemma}
\begin{proof}[Proof of \autoref{lem:approxvar}]
 Define for $2 \le t\le N$ 
\[
{G}_{t,i}=\Upsilon_{u_i,\omega}\big( \tilde{X}^{(u_i,t)}_{u_i;t} \otimes \sum_{s=1}^{t-1} \tilde{w}^{(\omega)}_{b_f,s,t} \tilde{X}^{(u_i,s)}_{u_i;s}\big)
\]
and, for $t=1$, set ${G}_{t,i}=0$. Denote by $\barbelow{G} = \text{vec}\big(\{G_{t,i}\}_{t,i}\big)$ the vectorization of the matrix $\{G_{t,i}\}_{t,i}$ by stacking its columns on top of each other. Then using the same argument as in \autoref{lem:mdepap}, we can write $\barbelow{G}_l=G_{t_l, i_l}$ and  $\barbelow{G}_l$ is a measurable function of $\{\ldots, \epsilon_{l-1}, \epsilon_{l}\}$. 
We have $k \Upsilon_{u}^{\lambda} \tilde{\F}_{k,t>s }^{(u_i,\lambda)}=\sum_{i=1}^M\sum_{t=1}^k \tilde{G}_{t,i}$, and thus by the second part of \autoref{lem:Burkh}
\begin{align*}
\bignorm{\sum_{i=1}^M\sum_{t=1}^k{G}_{t,i} -\E{G}_{t,i}}^2_{S_1,\gamma} 
\lesssim \Big(\sum_{r \in \nnum}\sqrt{ \sum_{j=-\infty}^{NM} \bignorm{
\sum_{l=0}^{MN}
P_{j}\big(\mathrm{1}_{2\le t_l \le k} {G}_{t_l,i_l}(e_r)\big)}_{\Hi,\gamma}^2}\Big)^2~.
\end{align*}
Then,
\begin{align*}
&\sum_{j=-\infty}^{NM} \bignorm{ \sum_{l=0}^{MN}\mathrm{1}_{2\le t_l \le k}
P_{j}\big({G}_{t_l,i_l}(e_r)\big)}_{\Hi,\gamma}^2\\
&   \le\sum_{j=-\infty}^{NM} \Big \{  \sum_{i=1}^M \sum_{s=1}^{k-1} \norm{\Upsilon_{u_i,\omega}}_\infty
 \big( \bignorm{ \sum_{t=s+1}^{k} \tilde{w}^{(-\omega)}_{\bf,t,s} X^{(u_{i},t)}_{u_{i},t}}^2_{\hi, 2\gamma} \bignorm{\big(X^{(u_{i},s)}_{u_{i},s} - X^{(u_{i},s)}_{u_{i},s,\{j\}}\big)(e_r)}^2_{\cnum,2\gamma} \big)^{1/2}
 \\
&\phantom{\sum_{j=-\infty}^{NM}~~~~}
+\sum_{i=1}^M\sum_{t=2}^k\norm{\Upsilon_{u_i,\omega}}_\infty \big(\bignorm{\big(X^{(u_{i},t)}_{u_{i},t}- X^{(u_{i},t)}_{u_{i},t,\{j\}}\big)(e_r) }^2_{\cnum,2\gamma} \bignorm{ \sum_{s=1}^{t-1}
\tilde{w}^{(\omega)}_{\bf,t,s} 
X^{(u_{i},s)}_{u_{i},s,\{j\}}}^2_{\hi,2\gamma}\big)^{1/2}\Big \} ^2\\
&
 \lesssim    \sum_{j=-\infty}^{NM}\Big \{ \sum_{\ell=1}^{MN}\mathrm{1}_{1\le s_\ell \le k-1}\norm{\Upsilon_{u_{i_\ell},\omega}}_{\infty} 
 \nu^{{X}^{\cdot}_{\cdot}(e_r)}_{\cnum,2\gamma}(\ell-j)\Big(\sum_{t=s_\ell+1}^{k}  \vert \tilde{w}_{\bf,s_\ell,t}^{(-\omega)} \vert^2 \big( \sum_{t=0}^{\infty}\nu^{{X}^{\cdot}_{\cdot}}_{\hi,2\gamma}(t)\big)^2 \Big)^{1/2}\Big \} ^{2}\\
 & + \sum_{j=-\infty}^{NM}\Big \{ \sum_{l=1}^{MN}\mathrm{1}_{2\le t_l \le k} \norm{\Upsilon_{u_{i_l},\omega}}_{\infty} 
\nu^{{X}^{\cdot}_{\cdot}(e_r)}_{\mathbb{C},2\gamma}(l-j)
 \Big( \sum_{s=1}^{t_l-1}  \vert \tilde{w}_{\bf,s,t_l}^{(\omega)} \vert^2  \big(\sum_{t=0}^{\infty} \nu^{{X}^{\cdot}_{\cdot}}_{\hi,2\gamma}(t)\big)^2\Big)^{1/2} \Big \}^2\\
&=   k\sum_{h=1}^{k-1} \big\vert {w}(\bf h) \big\vert^2 O\Big(\sup_{u \in [0,1]}\norm{\Upsilon_{u,\omega}}_\infty^{2}
  M \big(\sum_{j=0}^{\infty}\nu^{{X}^{\cdot}_{\cdot}}_{\hi,2\gamma }(j)\big)^{2} \big(\sum_{t=0}^{\infty} \nu^{{X}^{\cdot}_{\cdot}}_{\hi,2\gamma }(t)\big)^{2}\Big)~.
\end{align*}
A similar argument as given in the  proof of \autoref{lem:mdepap}  
completes the proof.
\end{proof}

\begin{lemma}\label{lem:approxint}
Suppose the conditions of \autoref{thm:Conv} hold with $p\ge 2 $ and let $\gamma=p/2$. Then, for any fixed $\omega$
\begin{align*}
  \sup_{u \in [0,1]}
\E\Big( \sup_{\eta \in [0,1]}  \bignorm{\sum_{i=1}^M \flo{\eta N}\Upsilon_{u_i,\omega}\big(\hat{\F}_{u_i,\omega}(\eta) - \hat{\F}_{\flo{\eta N}}^{u_i,\omega}}_{S_1} \Big)^\gamma  =O(M^{\gamma} \bf^{-\gamma/2})
.
\end{align*}
\end{lemma}

\begin{proof}[Proof of \autoref{lem:approxint}]
We have
\begin{align*}
& \bignorm{\sum_{i=1}^M \flo{\eta N}\Upsilon_{u_i,\omega}\big(\hat{\F}_{u_i,\omega}(\eta) - \hat{\F}_{\flo{\eta N}}^{u_i,\omega}
 \big)}^2_{S_1,\gamma}  
\\& ~~~~~~~~~~~~~~~~~~~~ \le
M^2\sup_{u \in [0,1]}\norm{\Upsilon_{u,\omega}}^2_\infty
 \sup_{u \in [0,1]} \Big \|   f(\eta,N)X_{u;\flo{\eta N}+1} \otimes  \sum_{s=1}^{\flo{\eta N}}\big(\tilde{w}^{(\omega)}_{\bf,\flo{\eta N}+1,s}\ X_{u;s} \big)\Big\|_{S_1,\gamma}^2 \tageq\label{eq:C71}
\\& ~~~~~~~~~~~~~~~~~~~~+
 M^2 \sup_{u \in [0,1]}\norm{\Upsilon_{u,\omega}}^2_\infty  \sup_{u \in [0,1]}  \Big \|  f(\eta,N)X_{u;\flo{\eta N}+1} \otimes X_{u;\flo{\eta N}+1}\Big\|^2_{S_1,\gamma},\tageq\label{eq:C72}
 \end{align*}
where we recall that $\sup_{\eta}|f(\eta,N)| \le 1$. For \eqref{eq:C71}, Cauchy Schwarz's inequality, \autoref{lem:Burklin} and \autoref{as:depstrucnonstat} yield
\begin{align*}
\sup_{u \in [0,1]} \Bignorm{ X_{u;k+1} \otimes \sum_{s=1}^{k} \tilde{w}^{(\omega)}_{\bf,k+1,s} X_{u;s}}^2_{S_1,\gamma}
&\le \sup_{u \in [0,1]} \Bignorm{\sum_{s=1}^{k} \tilde{w}^{(\omega)}_{\bf,k+1,s} X_{u;s}}^2_{\hi,2\gamma}\sup_{T\in \mathbb{N}}\sup_{t\in \znum}\bignorm{ X_{t,T} }^2_{\hi,2\gamma}
\\& \lesssim \sum_{s=1}^{k} |w(\bf(k+1-s))|^2   \big( \sum_{j=0}^\infty\sup_{T \in \nnum} \nu_{\hi,2\gamma}(j,T)\big)^2\sup_{T \in \nnum, t\in \znum}\bignorm{ X_{t,T} }^2_{\hi,2\gamma}. 
\end{align*}
Let  
$g(b,k)=\sum_{s=b+1}^{b+k} |w(\bf(k+1-s))|^2 = \sum_{h=b+1}^{b+k} |w(\bf(h))|^2$
for which it is clear that $
g(b,k)+g(b+k,l) \le g(b,k+l)$.
Therefore, by Theorem 1 of \cite{Mor76}
it follows that 
\begin{align*}
\E \Big(\max_{1\le k \le N} \Bignorm{\sum_{i=1}^M \Upsilon_{u_i,\omega}\big(X_{u_i,k+1} \otimes \sum_{s=1}^{k} \tilde{w}^{(\omega)}_{\bf,s,k+1} X_{u_i;s}\big)}_{S_1}\Big)^{\gamma} &\le  C \Big(M^2 \sup_{u \in [0,1]}\norm{\Upsilon_{u,\omega}}_\infty \sum_{s=1}^{N} |w(\bf(N+1-s))|^2 \Big)^{\gamma/2}
\\&=O(M^{\gamma }\bf^{-\gamma/2}).
\end{align*}
For \eqref{eq:C72}, we note that $
\norm{ X_{u;k+1}\otimes X_{u;k+1}}^2_{S_1,\gamma} 
\le  \sup_{T\in \mathbb{N}}\sup_{t\in \znum}\norm{ X_{t,T} }^2_{\hi,2\gamma} = O(1)$.
The result now follows.
\end{proof}
\begin{lemma}\label{lem:bias}
Assume Assumptions \ref{as:depstruc}-\ref{as:smooth} hold  with $p\ge 2$.  Then
\[
\max_{u \in U_M} \sup_{\omega} \sup_{\eta \in [0,1]}\bignorm{ \flo{\eta N}\big(\E \tilde{\F}^{(u,\omega)}(\eta) -{\F}_{u,\omega}\big)}_{S_1}  =O\Big(\frac{N^{2}}{T^{\zeta}\bf^{1/2}}\Big)+O(1)+O(N \bf^2).
\]
\end{lemma}

\begin{proof}
We first consider the 
non-interpolated part and  show that
\[
\max_{u \in U_M} \sup_{\omega} \max_{1\le k\le N}\bignorm{ k\big(\E \tilde{\F}_{k}^{(u,\omega)} -{\F}_{u,\omega}\big)}_{S_1}  =O\Big(\frac{N^{2}}{T^\zeta\bf^{1/2}}\Big)+O(1)+O(N \bf^2).\tageq \label{eq:bidif}
\]
Using a change of variables 
and the fact that $w(\cdot)$ has compact support, we have
 \begin{align*} 
& \sum_{s=1}^{k} \sum_{t=1}^{k} 
\tilde{w}^{(\omega)}_{\bf,s,t}
\E (\Xu{\tu{N}{u}+s}\otimes \Xu{\tu{N}{u}+t}) -k \F_{u,\omega}
\\&=\frac{1}{2\pi} \sum_{|r| < \min(k,1/\bf)}  w(\bf r) \sum_{t=\tu{N}{u}+\max(1,1-r)}^{\tu{N}{u}+\min(k,k-r)} e^{-\im r \omega} \big(\E(\Xu{t}\otimes \Xu{t+r}\big) -k\F_{u,\omega}
\intertext{set $n =\min(k,1/\bf)$}
& =\frac{1}{2\pi} \sum_{|r| < n}  w(\bf r ) e^{-\im r \omega}\sum_{t=\tu{N}{u}+\max(1,1-r)}^{\tu{N}{u}+\min(k,k-r)}  \big(\E(\Xu{t}\otimes \Xu{t+r})-\E({X}^{(u)}_t \otimes {X}^{(u)}_{t+r})\big)+\frac{1}{2\pi} \sum_{|r| < n} |r|e^{-\im r \omega} \C^{(u)}_{r}
\\
&  +\frac{1}{2\pi}\sum_{|r|<n}
(k - |r|)  [w(\bf r)-1] e^{-\im r \omega} \C^{(u)}_{r} -\frac{1}{2\pi}k \sum_{|r|\ge n}e^{-\im r \omega} \C^{(u)}_{r}
\\
&= I+II+III+IV~, 
\tageq \label{de1}
 \end{align*}
 where the last line defines the terms
 $ I, II, III, IV $ in an obvious manner.
For $I$, we write
\begin{align*}
&\sum_{|r| < k}  w(\bf  r ) e^{-\im r \omega}\sum_{t=\tu{N}{u}+\max(1,1-r)}^{\tu{N}{u}+\min(k,k-r)}  \big(\E(\Xu{t}\otimes \Xu{t+r})-\E(X^{(u)}_t \otimes X^{(u)}_{t+r})\big)
\\&= \sum_{|r| < k}  w(\bf  r ) e^{-\im r \omega}\sum_{t=\tu{N}{u}+\max(1,1-r)}^{\tu{N}{u}+\min(k,k-r)}  \big(\E( (\Xu{t}-X^{(u)}_t) \otimes \Xu{t+r})\big)
 +  \big(\E( X^{(u)}_t \otimes (\Xu{t+r}-X^{(u)}_{t+r})\big).\tageq \label{eq:C81}
\end{align*}
Jensen's inequality, Cauchy Schwarz's inequality and  \autoref{lem:Burklin} then yield
\begin{align*}
& \max_{1\le k \le N}\bignorm{\sum_{r=0}^{k-1}  w(\bf  r ) e^{-\im r \omega}\sum_{t=\tu{N}{u}+1}^{\tu{N}{u}+{k-r}}  \E\big( (\Xu{t}-X^{(u)}_t) \otimes \Xu{t+r}\big)}_{S_1}
\\&\le\max_{1\le k \le N} \big(\E\bignorm{ \sum_{t=\flo{uT}-\flo{N/2}+1}^{\flo{uT}-\flo{N/2}+k}  (\Xu{t}-X^{(u)}_t) \otimes \sum_{s=t}^{\flo{uT}-\flo{N/2}+k} w(\bf(s-t)) e^{-\im (t-s) \omega} \Xu{s}}_{S_1} \big)
\\
&\lesssim  \max_{1\le k \le N}  \sum_{t=\flo{uT}-\flo{N/2}+1}^{\flo{uT}-\flo{N/2}+k}  \bignorm{ \Xu{t}-X^{(u)}_t}_{\Hi,2} \sqrt{\sum_{s=t}^{\flo{uT}-\flo{N/2}+k}|{w}(\bf(s-t))|^2\ 
\Big(\sum_{j=0}^\infty\sup_{T \in \nnum}\nu^{X^{\cdot}_{\cdot}}_{2}(j,T)\Big)^2}
\\& \lesssim  \max_{1\le k \le N} N/T^\zeta \sqrt{\big  (\sum_{j=0}^\infty \sup_{T \in \nnum} \nu^{X^{\cdot}_{\cdot}}_{2}(j,T) \big )^2} \sum_{t=\flo{uT}-\flo{N/2}+1}^{\flo{uT}-\flo{N/2}+k}  \sqrt{\sum_{s=t}^{\flo{uT}-\flo{N/2}+k}|{w}(\bf(s-t))|^2} 
.
\end{align*}
Hence, as the second sum inside the square root has at most $1/\bf$ terms, we have 
\begin{align*}
& \max_{1\le k \le N}\bignorm{\sum_{|r|<\min(1/\bf,k)}  w(\bf  r ) e^{-\im r \omega}\sum_{t=\tu{N}{u}+\max(1,1-r)}^{\tu{N}{u}+\min(k,k-r)}  \E\big( (\Xu{t}-X^{(u)}_t) \otimes \Xu{t+r}\big)}_{S_1}
\\& \le 2 \max_{1\le k \le N}\bignorm{\sum_{r=0}^{\min(1/\bf,k)-1}  w(\bf  r ) e^{-\im r \omega}\sum_{t=\tu{N}{u}+\max(1,1-r)}^{\tu{N}{u}+\min(k,k-r)}  \E\big( (\Xu{t}-X^{(u)}_t) \otimes \Xu{t+r}\big)}_{S_1} =O\Big(\frac{N^{2}}{T^\zeta} \frac{1}{\bf^{1/2}}\Big).
\end{align*}
The same order is obtained for the second term in \eqref{eq:C81}
For $II$ in \eqref{de1},  \autoref{as:smooth} yields
\begin{align*} \max_{1\le k \le N}\bignorm{\sum_{|r|<\min(k,1/\bf)} |r| e^{-\im r \omega} \C^{(u)}_{r}}_{S_1}
\le \max_{1\le k \le N} \sum_{|r|<\min(k,1/\bf)}|r| \norm{\C^{(u)}_{r}}_{S_1} 
=O(1).
\end{align*}
 Finally, for  $III$ and $IV$, we
 use the assumptions imply that $\sup_{u \in [0,1]} \sum_{r \in \znum} |r|^{\iota} \norm{ \mathcal{C}^{(u)}_{r}}_{S_1}<\infty$ and  that $w(x)-1 =O(|x|^\iota)$ as $x\to 0$, from which we find 
\begin{align*}
& \max_{1\le k \le N}\bignorm{\frac{1}{2\pi}\sum_{|r| <n } (k - |r|)
[w(\bf r)-1] e^{-\im r \omega} \C^{(u)}_{r}
- \frac{1}{2\pi}k \sum_{|r|\ge n} e^{-\im r \omega} \C^{(u)}_{r}}_{S_1}
\\&\le \max_{1\le k \le N} \Big(k \bf^\iota \sum_{|r| < \min(k,1/\bf)} O(r^\iota)   \norm{\C^{(u)}_{r}}_{S_1}+k/(\min(k,1/\bf))^{\iota} \sum_{|r| \ge \min(k,1/\bf)} |r|^\iota  \norm{ \C^{(u)}_{r} }_{S_1} \Big)
\\&=O(N \bf^\iota).
\end{align*}
This proves \eqref{eq:bidif}. 
The result then follows from the triangle  inequality and the fact that
\begin{align}
\label{de2}
\max_{u \in U_M}\sup_{\omega}\sup_{\eta}\norm{\flo{\eta N}\E\tilde{\F}^{(u,\omega)}(\eta)-\E\tilde{\F}_{\flo{\eta N}}^{(u,\omega)}}_{S_1}=O(1). 
\end{align}
To prove the latter, note that under \autoref{as:depstruc} and \ref{as:Weights}, orthogonality of the projections, Jensen's inequality and stationarity, and a change of variables yield
\begin{align*}
&  \max_{1\le k \le N}  \bignorm{\sum_{t=1}^{k} 
\tilde{w}^{(\omega)}_{\bf,t,k+1}
\E (\Xu{\tu{N}{u}+k+1}\otimes \Xu{\tu{N}{u}+t}) }_{S_1}
\\
&\le  \sup_x|w(x)| \max_{1\le k \le N} \E \sum_{t=1}^{k}\sum_{j=-\infty}^{\tu{N}{u}+t} \sup_{v\in[0,1]}\bignorm{
P_{0}(X^{(v)}_{\tu{N}{u}+k+1-j})}_{\Hi} \bignorm{P_{0}({X}^{(v)}_{\tu{N}{u}+t-j}) }_{\Hi}
\\& \le \sup_x|w(x)|\max_{1\le k \le N}  \sum_{j=-\infty}^{\tu{N}{u}+k}  \sum_{t=1\vee j-\tu{N}{u}}^{k}\nu^{X_{\cdot}^\cdot}_{\Hi,2}(\tu{N}{u}+t-j) \nu^{X_{\cdot}^\cdot}_{\Hi,2}(\tu{N}{u}+k+1-j)
\\&\le\sup_x|w(x)|\Big(\sum_{j=0}^\infty \nu^{X_{\cdot}^\cdot}_{\Hi,2}(j)\Big)^2 <\infty.
\end{align*}
Since the other terms that make up the difference $\E\tilde{\F}^{(u,\omega)}(\eta) -\E\tilde{\F}_{\flo{\eta N}}^{(u,\omega)}$
in \eqref{de2} can be estimated similarly, the result follows from the triangle inequality and from the fact that $\sup_\eta |f(\eta,N)|\le 1$. 
\end{proof} 

\subsection{Auxiliary results that complete the proof of  
\autoref{thm:maxdev}}

\begin{thm}\label{thm:quadtail}
Assume the conditions of \autoref{thm:maxdev} hold, and let \begin{align*}
Q_{i} =\sum_{t=(i-1)\bw+1}^{i \bw \wedge T} \sum_{s=t-\bw \vee 1}^{t}\phi(\bf(t-s)) \Xu{t} \otimes \Xu{s}, \quad i =1, \ldots, K.
\end{align*}
where $K=\flo{\tilde{T}/\bw}$ and $\bw \le \tilde{T} \le T$. Then for  $0<\beta<1 $, there exist constants $C_{\gamma,\rho}, C_X, C$
\begin{align*}
&\Pr\Big( \bignorm{\sum_{i=1}^K(Q_{i}-\E Q_{i})}_{S_1} \ge \epsilon\Big) 
 \le \begin{cases}
C\exp\Big(-\frac{\epsilon^2  \bf}{(2+\beta) \tilde{T} C_X  }\Big)+C_{\gamma,\rho} \frac{\tilde{T}  \bw^{\gamma-1}}{\epsilon^\gamma} & \text{ if } \rho >1/2-1/\gamma\\
C\exp\Big(-\frac{\epsilon^2  \bf}{(2+\beta) \tilde{T}  C_X  }\Big)+C_{\gamma,\rho}\Big( \frac{\tilde{T}  \bw^{\gamma-1}}{\epsilon^\gamma}+\frac{\tilde{T}^{\gamma/2-\rho \gamma} \bw^{\gamma/2}}{\epsilon^\gamma}\Big)
& \text{ if } \rho <1/2-1/\gamma
\end{cases}
~. 
\end{align*}
\end{thm}

\begin{proof}
Consider $l$-block dependent versions of $Q_i$;
$$Q_{l,i}=\E[Q_{i} |\G_{(i-l-1)\bw+1}^{i\bw}]\quad l \ge 0.$$
In the following, denote $d_l=2^l$ for $1 \le l \le L-1$ where $L =\flo{log_2(K)}$. We use the following decomposition.
\begin{align*} 
\sum_{i=1}^K Q_{i}-\E Q_{i} &=  \sum_{i=1}^{K} (Q_{i}-Q_{K,i} + Q_{K,i}-Q_{2,i}+Q_{2,i})-\sum_{i=1}^{K} \E(Q_{i}-Q_{K,i} + Q_{K,i}-Q_{2,i}+Q_{2,i})
\\& =   \sum_{i=1}^{K}( Q_{i}-Q_{K,i})-\sum_{i=1}^{K}\E( Q_{i}-Q_{K,i}) +  \sum_{l=2}^{\flo{log_2(K)}} \sum_{i=1}^{K} (\bar{Q}_{d_l,i}-\bar{Q}_{d_{l-1},i})+ \sum_{i=1}^{K} (\bar{Q}_{2,i}) \tageq \label{eq:Q123}
\end{align*}
where $\bar{\cdot}=(\cdot)-\E[\cdot]$. To treat the first term and second term in \eqref{eq:Q123}, we write 
\begin{align*}
\sum_{i=1}^{K} (Q_{i}-Q_{K,i})= \sum_{i=1}^{K} \sum_{l=K+1}^{\infty} (Q_{l,i}-Q_{l-1,i})~. 
\end{align*}
Then it follows from \autoref{lem:Qldepblocks} that
\begin{align*}
\bignorm{\sum_{l=K+1}^{\infty} \sum_{i=1}^K (Q_{l,i}-Q_{l-1,i})}_{S_1,\gamma} 
& \le \sum_{l=K+1}^{\infty} \bignorm{ \sum_{i =1}^K (Q_{l,i}-Q_{l-1,i})}_{S_1,\gamma}  
\\&
 \le \sum_{l=K+1}^{\infty} K^{1/2} \bw O\Bigg(\sum_{\ell \in \nnum}\sum_{j=(l-2)\bw+1}^{(l+1)\bw}  \nu^{X^{\cdot}_{\cdot}(e_\ell)}_{\cnum,2\gamma}(j)+ \sum_{j=(l-1)\bw+1}^{(l+1)\bw} \nu^{X^{\cdot}_{\cdot}}_{\Hi,2\gamma}(j)\Big)
 \\&
 \le  O\Bigg(K^{1/2} \bw\Big[\sum_{\ell \in \nnum}\sum_{l=K+1}^{\infty} \sum_{j=(l-2)\bw+1}^{(l+1)\bw}  \nu^{X^{\cdot}_{\cdot}(e_\ell)}_{\cnum,2\gamma}(j)+\sum_{l=K+1}^{\infty}  \sum_{j=(l-1)\bw+1}^{(l+1)\bw} \nu^{X^{\cdot}_{\cdot}}_{\Hi,2\gamma}(j)\Big]\Big).
\end{align*}
Therefore if the process has a polynomial moment decay, then the sums will be of order $O((K-1)\bw)^{-\rho})=O(\tilde{T}^{-\rho})$ and we find that the term is of order
$O\Big(\tilde{T}^{1/2-\rho} \bw^{1/2}\Big)$.
The same order holds for the second term in 
\eqref{eq:Q123}. For the third term, for fixed $l$, we recall that $\bar{Q}_{d_l,i}$ is $d_l$-block dependent. Hence, it is independent of $\bar{Q}_{d_l, i^\prime}$ for $|i-i^\prime|>d_l $. We then decompose the sum over $i$ further by considering
\begin{align*}
H_{l, h} = \sum_{i=d_l(h-1)+1}^{d_l h \wedge K}  (\bar{Q}_{d_l,i}-\bar{Q}_{d_{l-1},i})
\end{align*}
where $1\le h \le \lceil K/d_l \rceil$ so that for $|h-h^\prime|>1$, $H_{l, h}$ and $H_{l, h^\prime}$ are independent. We now use this to construct two sums of independent zero-mean random elements, namely 
\begin{align*}
S^e_{K,l}=\sum_{\substack{ 1\le h \le \lceil K/d_l \rceil\\ h  \in 2 \mathbb{N}}}H_{l, h}  \quad \text{ and }
\quad  
S^o_{K,l}=\sum_{\substack{ 1\le h \le \lceil K/d_l \rceil\\ h  \in 2 \mathbb{N}+1}}  H_{l, h}.
\end{align*}
From  \autoref{lem:expin} we obtain 
\[
\mathbb{P}\Bigg( \norm{S^e_{K,l}}_{S_1} \ge x \Bigg) \le \exp\Big(-\frac{x^2}{(2+\beta)\Lambda_{K,l}} \Big)+ C \E \bignorm{\sum_{\substack{ 1\le h \le \lceil K/d_l \rceil\\ h  \in 2 \mathbb{N}}}H_{l, h}}^q_{S_1} x^{-q},
\]
where $\Lambda_{K,l}=\sup\{ \sum_{h \text{ even}}\E |v^\prime(H_{l,h})|^2: v^\prime \in V^\prime_1\}= C_X \tilde{T}  d_l^{-2\rho } \bw^{1-2\rho } $, and where $C_X$ denotes a constant that depends on the process $X$. This follows similarly to the following argument. From  \autoref{lem:Qldepblocks}  we have 
\[
 \bignorm{\sum_{i=d_l(h-1)+1}^{d_l h \wedge K}  (\bar{Q}_{d_l,i}-\bar{Q}_{d_{l-1},i})(e_r)}^q_{\Hi,q} =O\Bigg(({d_l}\bw^2)^{q/2}\Big( \sum_{s=(d_{l-1}-2)\bw+1}^{(d_l+1)\bw}  \nu^{X^{\cdot}_{\cdot}(e_r)}_{\cnum,2\gamma}(s)+ \sum_{s=(d_{l-1}-1)\bw+1}^{(d_l+1)\bw} \nu^{X^{\cdot}_{\cdot}}_{\Hi,2\gamma}(s)\Big)^q\Bigg)~.
\]
Therefore, using \eqref{eq:upbtr2} and independence 
\begin{align*}
 \E \bignorm{\sum_{\substack{ 1\le h \le \lceil K/d_l \rceil\\ h  \in 2 \mathbb{N}}}H_{l, h}}^q_{S_1}
 &=   \bignorm{\sum_{\substack{ 1\le h \le \lceil K/d_l \rceil\\ h  \in 2 \mathbb{N}}} \sum_{i=d_l(h-1)+1}^{d_l h \wedge K}  (\bar{Q}_{d_l,i}-\bar{Q}_{d_{l-1},i})}^q_{S_1,q}
\\  & \le \Big( \sum_{r \in \nnum} \sqrt[q]{\sum_{\substack{ 1\le h \le \lceil K/d_l \rceil\\ h  \in 2 \mathbb{N}}}   \bignorm{\sum_{i=d_l(h-1)+1}^{d_l h \wedge K}  (\bar{Q}_{d_l,i}-\bar{Q}_{d_{l-1},i})(e_r)}^q_{\Hi,q}}\Big)^q
 \\&=O\Big(\tilde{T} d_l^{q/2-\rho q-1} \bw^{q-\rho q-1} \Big).
\end{align*}
Define
\[
\ddot{\epsilon}_l=
\begin{cases} 
\frac{3}{\pi^2} \frac{1}{(l-1)^2} &\text{ if } 2\le l\le L/2\\
\frac{3}{\pi^2} \frac{1}{(L+1-l)^2} &\text{ if } L/2< l\le L,\\
\end{cases}
\]
then
\begin{align*}
\mathbb{P}\Big(\bignorm{\sum_{l=2}^{L} \sum_{i=1}^{K} (\bar{Q}_{d_l,i}-\bar{Q}_{d_{l-1},i})}_{S_1} \ge x\Big) &\le \sum_{l=2}^{L} \mathbb{P}\Big(\norm{S^e_{K,l}}_{S_1} \ge \ddot{\epsilon}_l x\Big)+\sum_{l=2}^{L} \mathbb{P}\Big(\norm{S^o_{K,l}}_{S_1} \ge \ddot{\epsilon}_l x\Big)
\\&
\le C\sum_{l=2}^{L}\exp\Big(-\frac{\ddot{\epsilon}^2_l x^2  d_l^{2\rho } \bw^{2\rho -1}}{(2+\beta) \tilde{T}  C_X  } \Big)+C_{q,\rho} \frac{\tilde{T}  \bw^{q-\rho q-1}}{x^q}\sum_{l=2}^{L} \frac{ d_l^{q/2-1-\rho q}}{\ddot{\epsilon}^q_l} . 
\end{align*}
Similarly to \eqref{epsisum},
\begin{align*}
\sum_{l=2}^{L} \frac{ d_l^{q/2-1-\rho q}}{\ddot{\epsilon}^q_l}& 
= O\Big( (\log_2(K)^{2q+1} K^{(q/2-1-\rho q)/2}\Big) \\&\le R_{K,q,\rho}= \begin{cases}
 c & \text{ if } q/2-1-\rho q<0\\
 c K^{q/2-1-\rho q} & \text{ if } q/2-1-\rho q >0
\end{cases}~.
\end{align*}
Furthermore, it is easily seen that there exists a constant $C>0$ such that
\[
 \sum_{l=2}^{L}\exp\Big(-\frac{\ddot{\epsilon}^2_l d_l^{2\rho } x^2  \bw^{\rho 2-1}}{(2+\beta) \tilde{T}  C_X  } \Big) \le C \exp\Big(-\frac{x^2 \bw^{\rho 2-1}}{(2+\beta) \tilde{T}  C_X  } \Big).
\]
Therefore, we find
\[
\mathbb{P}\Big(\bignorm{\sum_{l=2}^{L} \sum_{i=1}^{K} (\bar{Q}_{d_l,i}-\bar{Q}_{d_{l-1},i})}_{S_1} \ge x\Big) \le C \exp\Big(-\frac{x^2}{(2+\beta) \tilde{T}  C_X  \bw }  \Big)+C_{q,\rho}R_{K,q,\rho}\frac{\tilde{T}  \bw^{q-\rho q-1}}{x^q} .
\]
The final term in \eqref{eq:Q123} can be dealt with similarly after noticing that 
$\bar{Q}_{2,i}$ is independent of $\bar{Q}_{2,i^\prime}$ for $|i-i^\prime|>2$. Indeed we obtain
\begin{align*}
\mathbb{P}\Big(\bignorm{\sum_{i=1}^{K} (\bar{Q}_{2,i})}_{S_1} \ge x\Big) &\le 3 \exp\Big(-\frac{x^2}{(2+\beta)\tilde{\Lambda}_{K}} \Big)+ 3 C \E \bignorm{\sum_{i=1}^{K} (\bar{Q}_{2,i})}^q_{S_1} x^{-q},
\\&
\le 3 \exp\Big(-\frac{x^2}{(2+\beta)\bw \tilde{T}  C_X  } \Big)+ O\Big(\frac{\tilde{T}  \bw^{q-1}}{x^q} \Big)~, 
\end{align*}
where $\tilde{\Lambda}_{K}=\sup\{\sum_{i=1}^{K}\E |v^\prime(\bar{Q}_{2,i})|^2: v^\prime\in V^\prime_1\}=C_X\tilde{T}  \bw$. All in all, we obtain
\begin{align*}
&\Pr\Big( \bignorm{\sum_{i=1}^K(Q_{i}-\E Q_{i})}_{S_1} \ge \epsilon\Big) 
\\& =\epsilon^{-q} O\Big( \tilde{T} ^{q/2-\rho q} \bw^{q/2}+  \tilde{T}  \bw^{q-1} +\tilde{T}  \bw^{q-\rho q-1} R_{K,q,\rho} \Big)+ C\exp\Big(-\frac{\epsilon^2  }{(2+\beta)\bw \tilde{T}  C_X  }\Big)
\end{align*}
if $\rho >1/2-1/q$ this reduces to 
\[
\epsilon^{-q} O\Big( \tilde{T}  \bw^{q-1} \Big)+ C\exp\Big(-\frac{\epsilon^2  }{(2+\beta)\bw \tilde{T}  C_X  }\Big)
\]
whereas if $\rho <1/2-1/q$ then we find
\[
\epsilon^{-q} O\Big(\tilde{T}\bw^{q-1}+\tilde{T}^{q/2-\rho q} \bw^{q/2} \Big)+C\exp\Big(-\frac{\epsilon^2 }{(2+\beta) \tilde{T} \bw C_X  }\Big).
\]
\end{proof}

\begin{lemma} \label{lem:Qldepblocks}
Under the conditions of \autoref{thm:quadtail} we have 
\begin{align*} 
\bignorm{ \sum_{i =1}^K (Q_{l,i}-Q_{l-1,i})}^2_{S_1,\gamma}& =O\Big( K \bw^2 (\sum_{j=0}^\infty  \nu^{X_{\cdot}^{\cdot}}_{\hi,2\gamma}(j,T) \big )^2 \Bigg(\sum_{j \in \nnum}\sqrt{ \Big(\sum_{q=(l-2)\bw+1}^{(l+1)\bw}  \nu^{X^{\cdot}_{\cdot}(e_j)}_{\cnum,2\gamma}(q)\Big)^2}\Bigg)^2\Big)\\&+O\Big( K \bw^2 \Big(\sum_{q=(l-1)\bw+1}^{(l+1)\bw} \nu^{X^{\cdot}_{\cdot}}_{\Hi,2\gamma}(q) \Big)^2  \Big(\sum_{\ell \in \nnum}\sum_{j=0}^\infty  \nu^{X_{\cdot}^{\cdot}(e_\ell)}_{\cnum,2\gamma}(j,T) \Big )^2\Big)
.
\end{align*}
\end{lemma}

\begin{proof}
We make the following two observations:
\begin{enumerate}
\item[(1)] $\E[Q_{l,i}-Q_{l-1,i}|\G^{i\bw}_{(i-l)\bw+1}] = \E[Q_{i} |\G_{(i-l)\bw+1}^{i\bw}]-\E[Q_{i} |\G_{(i-l)\bw+1}^{i\bw}]=0$
and thus, for fixed $l$, the differences $Q_{l,i}-Q_{l-1,i}$ form a backward martingale difference with respect to $(\G_{(i-l)\bw+1}^{i\bw})_{1\le i\le K}$.
\item[(2)]  Furthermore, we can write
\begin{align*}
Q_{l,i}-Q_{l-1,i}&=\sum_{k=1}^{\bw} \E[Q_{i} |\G_{(i-l-1)\bw+k}^{i\bw}]-\E[Q_{i} |\G_{(i-l-1)\bw+k+1}^{i\bw}]
=\sum_{k=1}^{\bw} P^{(i-l-1)\bw+k}(Q_{i})
\end{align*}
where, for fixed $i$, $P^{(i-l-1)\bw+k}(Q_{i})$, $1\le k \le \bw$ are backward martingale differences with respect to $\{\G_{(i-l-1)\bw+k}^{i\bw}\}_{1\le k\le \bw}$.
\end{enumerate}
By \autoref{lem:Burkh} it follows that  
\begin{align*}
\bignorm{\sum_{i=1}^K[Q_{l,i}-Q_{l-1,i}]}^2_{S_1,\gamma}
\le
K^2_\gamma \Big(\sum_{\ell \in \nnum} \sqrt[2]{\sum_{i=1}^K \sum_{r=1}^{\bw} \norm{P^{(i-l-1)\bw+r}(Q_{i})(e_\ell)}^2_{\Hi,\gamma}}\Big)^2~. 
\end{align*}
The proof is now  similar to the proof of \autoref{lem:mdepap}. We can write 
\begin{align*}
& \norm{P^{(i-l-1)\bw+r}(Q_{i})(e_\ell)}^2_{\Hi,\gamma} 
\\& \le 2 \bignorm{\Big(\sum_{t=(i-1)\bw+1}^{i \bw \wedge T} (\Xu{t} -X^{(\frac{t}{T})}_{t, \{(i-l-1)\bw+r\}}) \otimes  \sum_{s=(t-\bw) \vee 1}^{t}\phi(\bf(t-s)) \Xu{s}\Big)(e_\ell)}^2_{\Hi,\gamma}
\\& +2\bignorm{\Big(\sum_{s=((i-1)\bw+1-\bw) \vee 1}^{i\bw \wedge T} \sum_{t=s}^{(s+\bw) \vee  T}\overline{\phi}(\bf(t-s))X^{(\frac{t}{T})}_{t, \{(i-l-1)\bw+r\}} \otimes \big(\Xu{s} - X^{(\frac{s}{T})}_{s,\{(i-l-1)\bw+r\}}\big)(e_\ell)\Big)}^2_{\Hi,\gamma}=:2(J_1^2+J_2^2)~,
\end{align*}
where the last equality  defines $J_1$ and $J_2$ in an obvious manner.
Then
\begin{align*}
J^2_1& \lesssim \Bigg(\sum_{t=(i-1)\bw+1}^{i\bw \wedge T} \nu^{X^{\cdot}_{\cdot}}_{\Hi,2\gamma}(t-(i-l-1)\bw-r) \sqrt{\sum_{s=t-\bw \vee 1}^{t}|\phi(\bf(t-s))|^2  (\sum_{j=0}^\infty  \nu^{X_{\cdot}^{\cdot}(e_{\ell})}_{\cnum,2\gamma}(j,T) \big )^2}~~\Bigg)^2
\end{align*}
and thus
\begin{align*}
\Big(\sum_{\ell \in \nnum}\sqrt{\sum_{i=1}^K\sum_{r=1}^{\bw} J^2_1}\Big)^2& \lesssim
\Bigg(\sum_{\ell \in \nnum}\sqrt{\sum_{i=1}^K\sum_{r=1}^{\bw}  \Big(\sum_{q=l \bw-r+1}^{(l+1)\bw-r} \nu^{X^{\cdot}_{\cdot}}_{\Hi,2\gamma}(q)\Big)^2  
 \max_{t} \sum_{s=t-\bw \vee 1}^{t}|\phi(\bf(t-s))|^2   (\sum_{j=0}^\infty  \nu^{X_{\cdot}^{\cdot}(e_{\ell})}_{\cnum,2\gamma}(j,T) \big )^2}~~\Bigg)^2 
  \\& \lesssim K \bw^2 \Big(\sum_{q=(l-1)\bw+1}^{(l+1)\bw} \nu^{X^{\cdot}_{\cdot}}_{\Hi,2\gamma}(q) \Big)^2  \Big(\sum_{\ell \in \nnum}\sum_{j=0}^\infty  \nu^{X_{\cdot}^{\cdot}(e_{\ell})}_{\cnum,2\gamma}(j,T)\Big)^2.
\end{align*}
Similarly,
\begin{align*}
J^2_2 &\lesssim \Big(\sum_{s=(i-2)\bw+1 \vee 1}^{i\bw \wedge T}  \sqrt{\sum_{t=s}^{(s+\bw) \vee T}|\phi(\bf(t-s))|^2 (\sum_{j=0}^\infty  \nu^{X_{\cdot}^{\cdot}}_{\hi,2\gamma}(j,T) \big )^2}\nu^{X^{\cdot}_{\cdot}(e_\ell)}_{\cnum,2\gamma}(s-(i-l-1)\bw-r) \Big)^2~, 
\end{align*}
which in turn yields
\begin{align*}
\Big(\sum_{\ell \in \nnum}\sqrt{\sum_{i=1}^K\sum_{r=1}^{\bw} J^2_2}\Big)^2& \lesssim
\Bigg(\sum_{\ell \in \nnum}\sqrt{\sum_{i=1}^K \sum_{r=1}^{\bw} \Big(  \sum_{q=l \bw-\bw+r+1}^{(l+1)\bw-r}   \nu^{X^{\cdot}_{\cdot}(e_\ell)}_{\cnum,2\gamma}(q)\Big)^2
 \max_{s}\sum_{t=s}^{(s+\bw) \vee T}|\phi(\bf(t-s))|^2(\sum_{j=0}^\infty  \nu^{X_{\cdot}^{\cdot}}_{\hi,2\gamma}(j,T) \big )^2}~~\Bigg)^2 
 \\& \lesssim 1/\bf (\sum_{j=0}^\infty  \nu^{X_{\cdot}^{\cdot}}_{\hi,2\gamma}(j,T) \big )^2 \Bigg(\sum_{\ell \in \nnum}\sqrt{ \sum_{i=1}^K\sum_{r=1}^{\bw}\Big(\sum_{q=l \bw-\bw-r+1}^{(l+1)\bw-r}  \nu^{X^{\cdot}_{\cdot}(e_\ell)}_{\cnum,2\gamma}(q)\Big)^2}~~\Bigg)^2
    \\& \lesssim K\bw^2 (\sum_{j=0}^\infty  \nu^{X_{\cdot}^{\cdot}}_{\hi,2\gamma}(j,T) \big )^2\Big(\sum_{\ell\in \nnum} \sum_{q=(l-2)\bw+1}^{(l+1)\bw}  \nu^{X^{\cdot}_{\cdot}(e_\ell)}_{\cnum,2\gamma}(q)\Big)^2.
\end{align*}

\end{proof}

\subsection{Auxiliary lemmas for the statements in \autoref{sec:thm31proof}}
\begin{lemma}\label{lem:Vmax}
Let $V^{(u,\lambda)}_{m,N,t}$ be as in \eqref{eq:VNTfix}. Then for $\gamma> 2$
\[
\E\Big(\max_{1\le k \le N}  \bignorm{\sum_{i=1}^{M} \sum_{t=1}^{k}\Upsilon_{u_i,\omega}(V^{(u_i,\omega)}_{m,L,t}  )}_{S_1}\Big)^\gamma =O(M^{\gamma/2}N^{\gamma/2} b_L^{-\gamma/2}), \quad 1\le L \le N.
\]
\end{lemma}
\begin{proof}
We use a similar argument as in the proof of  \autoref{lem:mdepap}. Recall the notation \eqref{eq:dmpi}.  Then using \eqref{eq:upbtr2}, and \autoref{lem:Burkh} 
yields 
\begin{align*}
&\bignorm{\sum_{i=1}^{M} \sum_{t=1}^{k}\Upsilon_{u_i,\omega}(V^{(u_i,\omega)}_{m,L,t}  )}_{S_1,\gamma} 
\\
&\le  \sum_{r \in \nnum}\sqrt{ \bignorm{\sum_{l=1}^{M N} \mathrm{1}_{2 \le t_l\le k} \Upsilon_{u_{i_l},\omega}\Big(\dmpi{u_{i_l}}{\omega}{L}{t_l} \otimes \sum_{s=1}^{t_l-1} \tilde{w}_{b_L,t,s}^{(\omega)} \dmpi{u_{i_l}}{\omega}{L}{s} \Big)(e_r)}^2_{\Hi, \gamma}} \\
&
\le  \sum_{r \in \nnum}\sqrt{ \sum_{l=1}^{M N} \mathrm{1}_{2 \le t_l\le k} \norm{\Upsilon_{u_{i_l},\omega}}_\infty^2
\bignorm{\Big(\dmpi{u_{i_l}}{\omega}{L}{t_l} \otimes \sum_{s=1}^{t_l-1} \tilde{w}_{b_L,t,s}^{(\omega)} \dmpi{u_{i_l}}{\omega}{L}{s} \Big)(e_r)}^2_{\Hi, \gamma}}.
\end{align*}
Now
\begin{align*}
&
\sum_{l=1}^{M N} \mathrm{1}_{2 \le t_l\le k} \norm{\Upsilon_{u_{i_l},\omega}}_\infty^2
\bignorm{\Big(\dmpi{u_{i_l}}{\omega}{L}{t_l} \otimes \sum_{s=1}^{t_l-1} \tilde{w}_{b_L,t,s}^{(\omega)} \dmpi{u_{i_l}}{\omega}{L}{s} \Big)(e_r)}^2_{\Hi, \gamma}
\\
& \le C\sup_u \norm{\Upsilon_{u,\omega}}_{\infty}^{2} \sup_{\omega} \sup_{u \in [0,1]} \|{D}^{(u,\omega)}_{0}\|^2_{\hi,2\gamma }\sum_{l=1}^{M N} \mathrm{1}_{2 \le t_l\le k} \bignorm{\sum_{s=1}^{t_l-1} \tilde{w}_{b_L,t,s}^{(\omega)}{D}^{(u,\omega)}_{m,\flo{u_{i_l}T}-\flo{L/2}+s}(e_r)}^2_{\cnum,2\gamma}
\\& \le C\sup_u \norm{\Upsilon_{u,\omega} }_{\infty}^{2} \sup_{\omega } \sup_{u \in [0,1]} \|{D}^{(u,\omega)}_{0}\|^2_{\hi,2\gamma} \|{D}^{(u,\omega)}_{0}(e_r)\|^2_{\cnum,2\gamma}M k  \max_{2\le t\le k} \sum_{s=1}^{t_l-1} |\tilde{w}_{b_L,t,s}^{(\omega)} |^2~,
\end{align*}
where we applied the Cauchy Schwarz's inequality and made use of \autoref{lem:Burkh} again. An argument similar to the proof of \autoref{lem:mdepap} then shows that Theorem 1 of \cite{Mor76} is satisfied  and therefore, the statement follows.
\end{proof}
\begin{lemma} \label{lem:tedious}
Suppose the conditions of \autoref{thm:conv_an} hold with $p\ge 6$ and let  $L = N-m$. Then, for any fixed $\omega \in \rnum$, 
\begin{align*}
&\E \Big(\sup_{\eta \in [0,1]}  \bignorm{\sum_{i=1}^M   \sum_{t=1}^{\flo{\eta N}} \Upsilon_{u_i,\omega} ({V}^{(u_i,\omega)}_{m,N,t})- \sum_{i=1}^M   \sum_{t=1}^{\flo{\eta L}} \Upsilon_{u_i,\omega} ({V}^{(u_i,\omega)}_{m,L,t})}_{S_1} \Big)^\gamma
= O(m^{\gamma/2} M^{\gamma/2} \bf^{-\gamma/2})+o( M^{\gamma/2} L^{\gamma/2} \bf^{-\gamma/2}) ~,
\end{align*}
where $\gamma=p/2$.
\end{lemma}
\begin{proof}[Proof of \autoref{lem:tedious}]
 In the following we make use of the time-shifted version 
\begin{align*}\grave{V}^{(u_i,\omega)}_{L,t}={D}^{(t/T,\omega)}_{m,t} \otimes \sum_{s=\flo{u_{i}T}-\flo{L/2}+1}^{t-1}\tilde{w}^{(\omega)}_{b_L,s,t}
{D}^{(s/T,\omega)}_{m,s} 
\end{align*}
 of $V^{(u_i,\omega)}_{m,N,t}$,
where we recall the notation ${D}^{(t/T,\omega)}_{m,t}  = \sum_{p=0}^{m} P_{t}\Big({X}^{(t/T)}_{m,t+p}  \Big) e^{-\im \omega p} $. Then, using that 
\begin{equation} \label{eq106}
\flo{\eta N} =\flo{\eta L}+\flo{\eta m}+c
\end{equation}
for some  $c \in \{0,1\}, $
we can write 
\begin{align*}
& \sum_{t=1}^{\flo{\eta N}} \Upsilon_{u_i,\omega} ({V}^{(u_i,\omega)}_{m,N,t})-  \sum_{t=1}^{\flo{\eta L}} \Upsilon_{u_i,\omega} ({V}^{(u_i,\omega)}_{m,L,t})
\\&= \sum_{t=1}^{\flo{\eta N}}\Upsilon_{u_i,\omega}(\grave{V}^{(u_i,\omega)}_{N,\flo{u_i T}-\flo{N/2}+t}) -\sum_{t=1}^{\flo{\eta L}}\Upsilon_{u_i,\omega}(\grave{V}^{(u_i,\omega)}_{L,\flo{u_i T}-\flo{L/2}+t}) 
\\&= \sum_{t=-\flo{m/2}-c_2+1}^{\flo{\eta L}+\flo{\eta m}-\flo{m/2}+c_1-c_2}\Upsilon_{u_i,\omega}(\grave{V}^{(u_i,\omega)}_{N,\flo{u_i T}-\flo{L/2}+t}) -\sum_{t=1}^{\flo{\eta L}}\Upsilon_{u_i,\omega}(\grave{V}^{(u_i,\omega)}_{L,\flo{u_i T}-\flo{L/2}+t}) 
\\&
=\sum_{t=-\flo{m/2}-c_2+1}^{0} \Upsilon_{u_i,\omega}(\grave{V}^{(u_i,\omega)}_{N,\flo{u_i T}-\flo{L/2}+t}) +\mathrm{1}_{\eta >1/2}\Big(\sum_{t=1+\flo{\eta L}}^{\flo{\eta L}+\flo{\eta m}-\flo{m/2}+c_1-c_2}\Upsilon_{u_i,\omega}(\grave{V}^{(u_i,\omega)}_{N,\flo{u_i T}-\flo{L/2}+t})  \Big)
\\& +\sum_{t=1}^{\flo{\eta L}}\Upsilon_{u_i,\omega}(\grave{V}^{(u_i,\omega)}_{N,\flo{u_i T}-\flo{L/2}+t}-\grave{V}^{(u_i,\omega)}_{L,\flo{u_i T}-\flo{L/2}+t}) 
\\& 
- \mathrm{1}_{\eta \le 1/2} \Big ( 
 \sum_{t = \flo{\eta L}+\flo{\eta m}-\flo{m/2}+c_1-c_2 +1
 }^{\flo{\eta L}}
 \Upsilon_{u_i,\omega}(\grave{V}^{(u_i,\omega)}_{N,\flo{u_i T}-\flo{L/2}+t})
 \Big )~.
\tageq \label{eq:ted1}
\end{align*}
Using similar  arguments  as in the proof of  \autoref{lem:mdepap} we 
 first show for the third term that
\begin{align*}
\E\Big(\max_{1 \le k \le L}\bignorm{\sum_{i=1}^M \sum_{t=1}^{k} \Upsilon_{u_i,\omega}(\grave{V}^{(u_i,\omega)}_{N,\flo{u_i T}-\flo{L/2}+t}-\grave{V}^{(u_i,\omega)}_{L,\flo{u_i T}-\flo{L/2}+t})}_{S_1}\Big)^\gamma= 
o( M^{\gamma/2} L^{\gamma/2} \bf^{-\gamma/2}).
\tageq \label{eq:ted2}
\end{align*}
 Let $A_l=\{t_l: \flo{m/2}+ c_2+2 \le t_l \le \flo{m/2}+ c_2+ k \}$, then \eqref{eq106}  implies  
\begin{align*}
&\bignorm{\sum_{i=1}^M \sum_{t=1}^{k} \Upsilon_{u_i,\omega}(\grave{V}^{(u_i,\omega)}_{N,\flo{u_i T}-\flo{L/2}+t}-\grave{V}^{(u_i,\omega)}_{L,\flo{u_i T}-\flo{L/2}+t})}^2_{S_1,\gamma} \\~~~~~~~~&=\bignorm{\sum_{l=1}^{MN} \mathrm{1}_{A_l} \Upsilon_{u_{i_l},\omega}(\grave{V}^{(u_{i_l},\omega)}_{N,\flo{u_{i_l} T}-\flo{N/2}+t_l}-\grave{V}^{(u_{i_l},\omega)}_{L,\flo{u_{i_l} T}-\flo{N/2}+t_l})}^2_{S_1,\gamma}. 
\end{align*}
Similar to the proof of \autoref{lem:mdepap},
we consider
 \begin{align*}
&\bignorm{\sum_{l=1}^{MN} \mathrm{1}_{A_l} \Upsilon_{u_{i_l},\omega}(\grave{V}^{(u_{i_l},\omega)}_{N,\flo{u_{i_l} T}-\flo{N/2}+t_l}-\grave{V}^{(u_{i_l},\omega)}_{L,\flo{u_{i_l} T}-\flo{N/2}+t_l})(e_r)}^2_{\hi,\gamma} 
\\ &\le \sup_{u}\norm{\Upsilon_{u,\omega}}_\infty \sum_{l=1}^{NM} \mathrm{1}_{A_l}  \bignorm{ \grave{V}^{(u_{i_l},\omega)}_{N,\flo{u_{i_l} T}-\flo{N/2}+t_l}-\grave{V}^{(u_{i_l},\omega)}_{L,\flo{u_{i_l} T}-\flo{N/2}+t_l}(e_r)}^2_{\hi,\gamma}.
\end{align*}
We have $\grave{V}^{(u_{i_l},\omega)}_{N,\flo{u_{i_l} T}-\flo{N/2}+t_l}-\grave{V}^{(u_{i_l},\omega)}_{L,\flo{u_{i_l} T}-\flo{N/2}+t_l}
 = \dmpi{u_{i_l}}{\omega}{N}{t_l}\otimes  \sum_{s=1}^{t_l-1}\theta^\omega_{t_l,s}\dmpi{u_{i_l}}{\omega}{N}{s}$
where we denoted 
\[
\theta^\omega_{t_l,s}=\tilde{w}^{(\omega)}_{b_N,s,t_l} -\tilde{w}^{(\omega)}_{b_L,s,t_l} 
\]
to ease notation. Then using the Cauchy Schwarz inequality and \autoref{lem:Burkh}
\begin{align*}
 \sum_{l=1}^{NM} &\mathrm{1}_{A_l}  \bignorm{\Big( \grave{V}^{(u_{i_l},\omega)}_{N,\flo{u_{i_l} T}-\flo{N/2}+t_l}-\grave{V}^{(u_{i_l},\omega)}_{L,\flo{u_{i_l} T}-\flo{N/2}+t_l}\Big)(e_r)}^2_{\hi,\gamma}
\\& \lesssim \sum_{l=1}^{NM} \mathrm{1}_{A_l} \norm{\dmpi{u_{i_l}}{\omega}{N}{t_l}}^2_{\Hi,2\gamma} \sum_{s=1}^{t_l-1}\Big\vert\theta^\omega_{t_l,s}\Big\vert^2 \norm{\dmpi{u_{i_l}}{\omega}{N}{s}(e_r)}^2_{\cnum,2\gamma}.
 \end{align*}
A change of variables and using that the window is even in zero gives
\begin{align*}\max_{l} \max_{t_l \in A_l} \sum_{s=1}^{t_l-1}\Big \vert \tilde{w}_{b_N,t_l,s}^{(\omega)}-\tilde{w}_{b_L,t_l,s}^{(\omega)} \Big \vert^2 &\le   \max_{ \flo{m/2}+c +2 \le t \le \flo{m/2}+c+k} \sum_{h=1}^{t-1} \Big \vert {w}(b_N h)-{w}(b_{L}h)\Big \vert^2
\\&=  \sum_{h=1}^{\flo{m/2}+c+k-1} \Big \vert {w}(b_N h)-{w}(b_{L}h)\Big \vert^2
\end{align*}
Therefore, we obtain the following upper bound
\begin{align}
\nonumber 
& \sum_{l=1}^{NM} \mathrm{1}_{A_l}  \bignorm{ \big(\grave{V}^{(u_{i_l},\omega)}_{N,\flo{u_{i_l} T}-\flo{N/2}+t_l}-\grave{V}^{(u_{i_l},\omega)}_{L,\flo{u_{i_l} T}-\flo{N/2}+t_l}\big)(e_r)}^2_{\hi,\gamma}
\\& 
= O\Big (M  \sup_u \sup_\omega \norm{\tilde{D}^{(u,\omega)}_{m,0}}^2_{\Hi,2\gamma}\norm{\tilde{D}^{(u,\omega)}_{m,0}(e_r)}^2_{\cnum,2\gamma}k  \sum_{h=1}^{\flo{m/2}+c+k-1} \Big \vert {w}(b_N h)-{w}(b_{L}h)\Big \vert^2\Big)
\label{hd100}
\end{align}
An argument similar as in the proof of \autoref{lem:mdepap} shows that Theorem 1 of \cite{Mor76} implies 
\begin{align*}
&\E\Big(\max_{1 \le k \le L}\bignorm{\sum_{i=1}^M \sum_{t=1}^{k} \Upsilon_{u_i,\omega}(\grave{V}^{(u_i,\omega)}_{N,\flo{u_i T}-\flo{L/2}+t}-\grave{V}^{(u_i,\omega)}_{L,\flo{u_i T}-\flo{L/2}+t})}_{S_1}\Big)^\gamma
= o( M^{\gamma/2}  L^{\gamma/2} b_N^{-\gamma/2}  )
\end{align*}
Here we use the following general argument to estimate the sum  in \eqref{hd100}. \autoref{as:Weights} implies that, for any given $\epsilon>0$, $\exists \delta>0$ such that $|{w}(b_{N}h)-{w}(b_{L}h)| < \epsilon$ with  $|b_{L}-b_{N}|  < \delta$, uniformly for $h \le Q/b_{N}$ (except for a finite number of points). Since $b_{L} \to 0, b_{{N}} \to 0$ as $T \to \infty$, we find $|{w}(b_{N}h)-{w}(b_{L}h)| =o(1)$ for $h\le Q/b_{N}$. Because  $w(\cdot)$ is bounded, we obtain the  order of $o(1/b_{N})$ for the sum over $h \le Q/b_{N}$. For $h >Q/b_{N}$, we denote the length of the interval over which the sum is taken, say $I_{Q}$, converges to zero for fixed $b_{N}, T$ as $Q \to \infty$. Since  \autoref{as:Weights} implies that the summand is at most of order $b_{N}^{-1}$ this tail sum is therefore of order $O(I_Q/b_{N}) =o(1/b_{N})$ as well. This concludes the proof of \eqref{eq:ted2}. \\
Next, we focus on the first  term of \eqref{eq:ted1} observing that
\begin{align*}
\bignorm{\sum_{i=1}^M\sum_{t=-\flo{m/2}-c+1}^{0} \Upsilon_{u_i,\omega}(\grave{V}^{(u_i,\omega)}_{N,\flo{u_i T}-\flo{L/2}+t}) }^2_{S_1,\gamma}
&\le \bignorm{\sum_{l=1}^{NM}  \mathrm{1}_{1\le t_l \le \flo{m/2}} \Upsilon_{u_{i_l},\omega}(\grave{V}^{(u_{i_l},\omega)}_{N,\flo{u_{i_l} T}-\flo{N/2}+t_l}) }^{2}_{S_1,\gamma}
\\& =O(M m b_{N}^{-1}), 
\end{align*}
which follows from \eqref{eq106},  the Cauchy Schwarz inequality and \autoref{lem:Burkh}. Finally, we have for the second term of \eqref{eq:ted1}
(using again \eqref{eq106})
\begin{align*}
&\mathrm{1}_{\eta >1/2}\bignorm{\sum_{i=1}^M \sum_{t=1+\flo{\eta L}}^{\flo{\eta L}+\flo{\eta m}-\flo{m/2}+c_1-c_2}\Upsilon_{u_i,\omega}(\grave{V}^{(u_i,\omega)}_{N,\flo{u_i T}-\flo{L/2}+t})}_{S_1,\gamma}
\\&= \mathrm{1}_{\eta >1/2}\bignorm{\sum_{i=1}^M \sum_{t=1+\flo{\eta L}}^{\flo{\eta N}-\flo{m/2}-c_2}\Upsilon_{u_i,\omega}(\grave{V}^{(u_i,\omega)}_{N,\flo{u_i T}-\flo{N/2}+t+\flo{m/2}+c_2})}_{S_1,\gamma}
\end{align*}
 Let $B_l=\{t_l: k- \flo{mk/N}+c_3+ \flo{m/2}+2 \le t_l \le k \}$ with $c_3\in \{0,\pm 1\}$. Then
\begin{align*}
&\E\Big(\sup_{\eta \in (1/2,1]}\bignorm{\sum_{i=1}^M \sum_{t=1+\flo{\eta L}}^{\flo{\eta N}-\flo{m/2}-c_2}\Upsilon_{u_i,\omega}(\grave{V}^{(u_i,\omega)}_{N,\flo{u_i T}-\flo{N/2}+t+\flo{m/2}+c_2})}_{S_1}\Big)^\gamma 
\\& =\E\Big(\max_{\flo{N/2} < k \le N}\bignorm{\sum_{l=1}^{NM}\mathrm{1}_{B_l} \Upsilon_{u_{i_l},\omega}(\grave{V}^{(u_{i_l},\omega)}_{N,\flo{u_{i_l} T}-\flo{N/2}+t_l})}_{S_1}\Big)^\gamma~.
\end{align*}
Now, it follows similarly to the proof of \autoref{lem:Vmax} that 
\begin{align*}
\bignorm{\sum_{l=1}^{NM}\mathrm{1}_{B_l} \Upsilon_{u_{i_l},\omega}(\grave{V}^{(u_{i_l},\omega)}_{N,\flo{u_{i_l} T}-\flo{N/2}+t_l})}^2_{S_1,\gamma} 
&\le  O\Bigg(M \sup_{u} \norm{\Upsilon_{u,\omega}}^2_{\infty}  \sup_u \sup_\omega \norm{\tilde{D}^{(u,\omega)}_{m,0}}^2_{\Hi,2\gamma} \sum_{r \in \nnum}\norm{\tilde{D}^{(u,\omega)}_{m,0}(e_r)}_{\cnum,2\gamma}\Bigg)
\\&\phantom{O} 
\times  \big(\flo{mk/N}+ c_3+\flo{m/2}\big) \max_{l} \max_{t_l \in B_l}
\sum_{h=1}^{t_l-1} |w(\bf h)|^2
\end{align*}
and thus, from a similar argument as in proofs of the previous lemmas, we find
\begin{align*}
\E\Big(\max_{\flo{N/2} \le k \le N}\bignorm{\sum_{i=1}^{M}\sum_{t=\flo{k-mk/N}+1}^{k-\flo{m/2}-c_2} \Upsilon_{u_i,\omega}(\grave{V}^{(u_i,\omega)}_{N,\flo{u_i T}-\flo{N/2}+t+\flo{m/2}+c_2})}_{S_1}\Big)^\gamma = O( M^{\gamma/2} m^{\gamma/2} \bf^{-\gamma/2}).
\end{align*}
The final term in \eqref{eq:ted1} can be treated similarly. Details are omitted for the sake of brevity.
\end{proof}

\begin{lemma} \label{lem:tedious2} 
Suppose the conditions of \autoref{thm:conv_an} hold with $p\ge 6$. Then, for fixed  $\omega \in \rnum$
\begin{align*}
 \E \Big(\max_{1\le k \le N}& \bignorm{\sum_{i=1}^M f(\eta,N)\Upsilon_{u_i,\omega}(V^{(u_i,\omega)}_{m,N,\flo{\eta N}+1} ) }_{S_1}\Big)^\gamma =O(M^{\gamma/2} \bf^{-\gamma/2})~, 
\end{align*}
where $\gamma=p/2$.
\end{lemma}
\begin{proof}
The proof is similar to the proofs in \autoref{sec:Cbounds} and follows from \autoref{lem:Burkh} and  \autoref{as:depstruc} along the lines of \autoref{lem:mdepap}.
 \end{proof}

\section{Proofs of statements Section  \ref{sec:sec35} and Section \ref{sec4}} 
\label{sec5}
\def\theequation{D.\arabic{equation}}
\setcounter{equation}{0}
\subsection{Proof of the results in Section \ref{sec35}} \label{sec5}

\subsubsection{Preliminary results} 

\begin{proposition}\label{prop:bexplam}
For $r \in \{1,2\}$, let 
$\hat{B} \in L^1_{S_r(\Hi)^\dagger}( I\times [0,1]\times [0,\pi])$ be a perturbed version of 
$B \in  L^1_{S_r(\Hi)^\dagger}(I\times[0,1]\times [0,\pi])$ defined by $\eta \mapsto B_u^\omega(\eta)= \eta B_u^\omega$ continuous with respect to $\norm{\cdot}_{S_r}$. 
For  $(\eta,u, \omega) \in I \times [0,1]\times [0,\pi]$, denote the  eigendecomposition of $\eta B_u^{\omega}$ by $\eta B_u^{\omega}=\sum_{k=1}^{\infty}\eta\beta^{(u,\omega)}_k {P}^{(u,\omega)}_k $ and of $\eta \hat{B}_u^{\omega}(\eta)$ by $\eta\hat{B}_u^{\omega}(\eta)=\sum_{k=1}^{\infty}\eta\hat{\beta}^{(u,\omega)}_k(\eta) \hat{P}^{(u,\omega)}_k(\eta)$.  Then
\begin{align*}
\eta\hat{\beta}^{(u,\omega)}_{k}(\eta)-\eta\beta^{(u,\omega)}_{k}= \Tr\Big(P^{(u,\omega)}_{k}(\eta\hat{B}_{u}^{\omega}(\eta)-\eta {B}_u^{\omega})\Big) + \Tr\Big(P^{(u,\omega)}_{k} R^{(u,\omega)}_{1,k}(\eta)\Big),  \tageq \label{eq:eigd}
\end{align*}
almost everywhere on $[0,1]\times [0,\pi]$, 
where 
$$
R^{(u,\omega)}_{1,k}(\eta) = \eta\big(\hat{B}_{u}^{\omega}(\eta)-\eta{B}_u^{\omega}\big) \big( \hat{P}^{(u,\omega)}_{k}(\eta) -P^{(u,\omega)}_k\big)- \eta\big(\hat{\beta}^{(u,\omega)}_{k}(\eta)-\beta^{(u,\omega)}_{k}\big) \big( \hat{P}^{(u,\omega)}_{k}(\eta) -P^{(u,\omega)}_k\big).
$$
If $B \in C(I \times [0,1]\times [0,\pi], S_r(\Hi)^\dagger)$ and the first $k$ eigenvalues of $B$ are distinct, uniformly in $u,\omega$ and the conditions of \autoref{thm:maxdevcon} hold, then \eqref{eq:eigd} holds everywhere on $[0,1]\times [0,\pi]$ and the process
$\eta \mapsto R^{(u,\omega)}_{1,k}(\eta)$ satisfies
\begin{align*}
\norm{R^{(u,\omega)}_{1,k}(\cdot)}_{C_{S_r}} = O_p(\norm{\cdot^{1/3}(\hat{B}_{u}^{\omega}(\cdot)-{B}_u^{\omega})}^3_{C_{S_r}}) +O_p(\norm{\cdot^{1/2}(\hat{B}_{u}^{\omega}(\cdot)-{B}_u^{\omega})}^2_{C_{S_r}})\tageq \label{eq:RUk}
\end{align*}
\end{proposition}
\begin{proof}
The proof of \eqref{eq:eigd} follows similarly to \citep[][Proposition 3.3]{vdd21}, whereas the order of the error term \eqref{eq:RUk} 
is a consequence of \autoref{thm:eigp2} 
.
\end{proof}


\begin{lemma}\label{lem:invtra}
The map
\[\mathcal{T}:  C\big((0,1], S_1(\Hi)\big) \to C((0,1], \rnum\big), A(\eta) \mapsto \frac{1}{\Tr(A(\eta))}\]
is Fr{\'e}chet differentiable and the derivative map of $\mathcal{T}$ in $A(\eta)$ is given by
\[
\mathcal{T}^\prime_{A(\eta)}: h(\eta) \mapsto  -\frac{\Tr(h(\eta))}{\Tr^2(A(\eta))}. \tageq \label{eq:Dpsi}
\]
\end{lemma}
\begin{proof}
Observe that for $h \in  C\big((0,1], S_1(\Hi)\big)$, we have
\begin{align*}
\mathcal{T}(A(\eta)+h(\eta))-\mathcal{T}(A(\eta)) = 
- \frac{\Tr(h(\eta)}{\Tr(A(\eta)+h(\eta))\Tr(A(\eta)},
\end{align*}
and thus 
\begin{align*}
\mathcal{T}(A(\eta)+h(\eta))-\mathcal{T}(A(\eta)) -\mathcal{T}^\prime_{A(\eta)}(h(\eta))&=\frac{\Tr(h(\eta))}{\Tr^2(A(\eta))} - \frac{\Tr(h(\eta)}{\Tr(A(\eta)+h(\eta))\Tr(A(\eta)}
\\&
=\frac{\Tr^2(h(\eta))}{\Tr(A(\eta)+h(\eta))\Tr^2(A(\eta))},
\end{align*}
{where we note that
\begin{align*}
\sup_{\eta \in (0,1]}\Big|\frac{\Tr^2(h(\eta))}{\Tr(A(\eta)+h(\eta))\Tr^2(A(\eta))}\Big|\le \sup_{\eta \in (0,1]}\Big|\Tr^2(h(\eta))\Big| \sup_{\eta \in (0,1]}\Big|\frac{1}{\Tr(A(\eta)+h(\eta))\Tr^2(A(\eta))}\Big|.
\end{align*}
The second term on the right-hand side is uniformly bounded in $\eta \in (0,1]$, whereas for the first term one notes that the trace map is Lipschitz continuous w.r.t. to $\norm{\cdot}_{S_1}$ and hence, 
\begin{align*}
\sup_{\eta \in (0,1]}\big|\Tr^2(h(\eta))| \le \big(\sup_{\norm{A}_{S_1}=1}|\Tr(A)| \big)^2\sup_{\eta}\norm{h(\eta)}_{S_1}^2=O(\|h\|^2_{C_{S_1}}).
\end{align*} 
}
\end{proof}

\subsubsection{Proof of \autoref{thm:ldFPCA} and  \autoref{thm:lPSCA}}
\label{secd2} 
As both results  follow by  similar arguments, we 
restrict ourselves to the proof of   \autoref{thm:lPSCA}, which is technically more demanding.

In order to define a sequential estimator for the measure of total variation explained by its optimal degree $d$-separable approximation, note that $\big(\delta^{(u,\omega)}_{j}\big)^2$ is the $j$th largest eigenvalue of the operator 
\[
G_{u,\omega}= {F}_{u,\omega}{F}^{\dagger}_{u,\omega} =\sum_{j=1}^{\infty} (\delta^{(u,\omega)}_{j})^2 A^{(u,\omega)}_j {\otimes} A^{(u,\omega)}_j
\]
where we recall that ${F}_{u,\omega}=\sum_{j=1}^{\infty} \delta^{(u,\omega)}_{j} A^{(u,\omega)}_{j} \otimes B^{(u,\omega)}_{j} \in S_2(\Hi_1) {\otimes} S_2(\Hi_2)$ (see Section \ref{sec223}), which implies $G_{u,\omega} \in S_1( \Hi_1 \otimes \Hi_1)^+$. In the following, we use the notation 
$$\eta \tilde{\hat{G}} \eta)=\mathrm{L}^{a,b}_{U_M}(\eta \hat{G}_{\cdot,\cdot}(\eta)),~~
\eta\tilde{G}=\mathrm{L}^{a,b}(\eta{G}_{\cdot,\cdot}),
$$
and $A_{kk}^{(u,\omega)}=A^{(u,\omega)}_k {\otimes} A^{(u,\omega)}_k$.
 Then \autoref{prop:bexplam} implies
\begin{align*}
\frac{\eta (\hat \delta^{(u,\omega)}_{k} (\eta )
)^2
 }{\Tr\big(\eta\tilde{\hat{G}}(\eta)\big)}-
\frac{\eta( \delta^{(u,\omega)}_{k})^2}{\Tr\big(\eta{\tilde{G}}\big)} = \biginprod{\frac{\eta\hat{G}_{u,\omega}(\eta)}{\Tr\big(\eta\tilde{\hat{G}}(\eta)\big)}-\frac{\eta{G}_{u,\omega}}{\Tr\big(\eta{\tilde{G}}\big)} }{A^{(u,\omega)}_{kk}}+ \biginprod{R^{(u,\omega)}_{1,k}(\eta)}{A^{(u,\omega)}_{kk}}, \tageq \label{eq:error1}
\end{align*}
where  
\begin{align*}
R^{(u,\omega)}_{1,k}(\eta) &= \Big(\frac{\eta\hat{G}_{u,\omega}(\eta)}{\Tr\big(\eta\tilde{\hat{G}}(\eta)\big)}-\frac{\eta{G}_{u,\omega}}{\Tr\big(\eta \tilde{G}\big)} \Big) \big( \hat{A}^{(u,\omega)}_{kk}-A^{(u,\omega)}_{kk}\big)
 - \Big(\frac{\eta 
 ( \hat \delta^{(u,\omega)}_{k} (\eta ))^2  }{\Tr\big(\eta \tilde{\hat{G}}(\eta)\big)}-
\frac{\eta   ( \delta^{(u,\omega)}_{k})^2}{\Tr\big(\eta \tilde{G} \big)} \Big) \big(\hat{A}^{(u,\omega)}_{kk}-A^{(u,\omega)}_{kk}\big).
\end{align*}
Elementary calculations now yield
for \eqref{zeq31}
\begin{align*}
\eta^3(\hat{s}_d(\eta) - s_d) = 
& \eta^3
\frac{\sum_{k=1}^d {1 \over M} \sum_{u \in U_M} \int_a^b \eta(\hat \delta^{(u,\omega)}_{k} (\eta ) )^2 d\omega}{\Tr\big(\eta \tilde{\hat{G}}(\eta)\big)}- \eta^3
\frac{\sum_{k=1}^d {1 \over M} \sum_{u \in U_M} \int_a^b \eta ( \delta^{(u,\omega)}_{k})^2  d\omega}{\Tr\big({\eta \tilde{G}}\big)}
\\& +
\eta^3 \frac{\sum_{k=1}^d  {1 \over M} \sum_{u \in U_M} \int_a^b ( \delta^{(u,\omega)}_{k})^2 d\omega}{\Tr\big( \tilde{G}\big)}-
\eta^3
\frac{\sum_{k=1}^d   \int_0^1 \int_a^b ( \delta^{(u,\omega)}_{k})^2  d\omega du}{\Tr\big(\tilde{G}\big)} 
\\&=
\eta^2
\frac{1}{\Tr(\eta \tilde{G})}  {1 \over M} \sum_{u \in U_M} \int_a^b \Tr\Big( \sum_{k=1}^d A^{(u,\omega)}_{kk}\big(\eta^2\hat{G}_{u,\omega}(\eta)-\eta^2 {G}_{u,\omega}\big)\Big)d\omega 
\tageq \label{dynpcas1}
\\& +
\eta^2 \big(\frac{\eta}{\Tr\big(\eta^2 \tilde{\hat{G}}(\eta)\big)}-\frac{\eta}{\Tr\big(\eta^2 \tilde{G}\big)}\big)   {1 \over M} \sum_{u \in U_M} \int_a^b \Tr\Big(\sum_{k=1}^dA^{(u,\omega)}_{kk}\big(\eta^2 {G}_{u,\omega} \big)\Big)d\omega \tageq \label{dynpcas2}
\\&+\eta^3  {1 \over M} \sum_{u \in U_M} \int_a^b\Tr\Big( \sum_{k=1}^d A^{(u,\omega)}_{kk}\big(R^{u,\omega}_{1,k}(\eta)+R^{u,\omega}_{2, k}(\eta)\big)\Big) d\omega +\eta^3  R_{d}(\eta)\tageq \label{dynpcas}
\end{align*}
where \[
R^{(u,\omega)}_{2,k}(\eta)=\big(\frac{1}{\Tr\big(\eta^2\tilde{\hat{G}}(\eta)\big)}-\frac{1}{\Tr\big(\eta^2 \tilde{G}\big)}\big)  
\big(\eta^2\hat{G}_{u,\omega}(\eta)-\eta^2{G}_{u,\omega}\big)
\]
and where 
\[
 R_{d}(\eta)= \frac{\sum_{k=1}^d {1 \over M} \sum_{u \in U_M} \int_a^b(\eta( \delta^{(u,\omega)}_{k})^2d\omega- \int_0^1 \int_a^b(\eta( \delta^{(u,\omega)}_{k})^2 )d\omega du}{\Tr\big(\eta \tilde{G} \big)}\tageq \label{dynpcas3}
\]
In the following, we show that the terms \eqref{dynpcas1} and \eqref{dynpcas2}
are dominating, while \eqref{dynpcas}  and 
\eqref{dynpcas3}
are asymptotically negligible. 
We first focus on the term  \eqref{dynpcas1}, for which we note that
\begin{align*}
\eta^2(\hat{G}_{u,\omega}(\eta)-{G}_{u,\omega })
& =\eta\big( \hat{F}_{u,\omega}(\eta)-{F}_{u,\omega}\big) \eta\big(\hat{F}^{\dagger}_{u,\omega}(\eta)-{F}^{\dagger}_{u,\omega}\big)
\\&+\eta\big( \hat{F}_{u,\omega}(\eta)-{F}_{u,\omega}\big) \eta{F}^{\dagger}_{u,\omega}+\eta{F}_{u,\omega}\eta\big(\hat{F}^{\dagger}_{u,\omega}(\eta) -{F}^{\dagger}_{u,\omega} \big). \tageq \label{eq:decG}
\end{align*}
In order to treat these three terms, note that we have an isometric isomorphism 
\begin{align*}
\mathfrak{T}:
\begin{cases}
 C([0,1]\times [0,\pi]\times I, S_2(\Hi_1 \otimes \Hi_2))
  \to 
 C([0,1]\times [0,\pi]\times I,  S_2(\Hi_1) \otimes S_2(\Hi_2) )
   \\ 
  \sum_{i,j,k,l \ge 1}\alpha_{ijkl}^{u,\omega}(\eta) h_i \otimes g_j \otimes h_k \otimes g_l
   \mapsto 
      \sum_{i,j,k,l \ge 1}\alpha_{ijkl}^{u,\omega}(\eta)h_i \otimes h_k\otimes g_j \otimes g_l 
\end{cases} \tageq \label{eq:isom}
\end{align*}
where $\{ h_i \otimes g_j \otimes h_k \otimes g_l\}_{i,j,k,l \ge 1}$ is a tensor basis of 
$\Hi_1 \otimes \Hi_2 \otimes \Hi_1 \otimes \Hi_2$ and $\alpha_{ijkl}^{u,\omega}(\eta) \in \cnum$. 
Therefore, by definition of the the Hilbert-Schmidt norm, the cyclic property of the trace functional and the fact that $A^{(u,\omega)}_{kk}$ is a self-adjoint  operator, we obtain for \eqref{dynpcas1}
\begin{align*}
&\frac{1}{\Tr( \tilde{G})} \frac{1}{M} \sum_{u \in U_M} \int_a^b\sum_{k=1}^d \Tr\Big(A^{(u,\omega)}_{kk}\big( \eta[\hat{F}_{u,\omega}(\eta)-{F}_{u,\omega}] \eta^2{F}^{\dagger}_{u,\omega}\big)\Big) d\omega
\\& =\frac{\eta^2}{\Tr(\tilde{G})}\frac{1}{M} \sum_{u \in U_M} \int_a^b \biginprod{\eta( \hat{F}_{u,\omega}(\eta)-{F}_{u,\omega})}{ \sum_{k=1}^d A^{(u,\omega)}_{kk} {F}_{u,\omega}}_{S_2(\Hi_1) \otimes S_2(\Hi_2)} d\omega
\\& =\frac{\eta^2}{\Tr(\tilde{G})}\frac{1}{M} \sum_{u \in U_M} \int_a^b\biginprod{\eta(\hat{\F}_{u,\omega}(\eta)-{{\F}}_{u,\omega})}{ \mathfrak{T}
\Big(\sum_{k=1}^d A^{(u,\omega)}_{kk}F_{u,\omega}\Big)}_{S_2(\Hi_1 \otimes \Hi_2)}d\omega
\\& =\frac{\eta^2}{\Tr(\tilde{G} )}\frac{1}{M} \sum_{u \in U_M} \int_a^b\Tr\Big(\big(\mathfrak{T}
( \sum_{k=1}^d A^{(u,\omega)}_{kk}F_{u,\omega})\big)^\dagger \eta( \hat{\F}_{u,\omega}(\eta)-{{\F}}_{u,\omega})\Big) d\omega. \tageq \label{eq:mickeymouse}
\end{align*}
We define therefore the mapping  in \autoref{as:mappings} by  
\[
\Upsilon^{u,\omega}_{1}=\frac{1}{\Tr(\tilde{G} )}\mathfrak{T}
\big(\sum_{k=1}^d A^{(u,\omega)}_{kk}F_{u,\omega}\big)^\dagger
=\frac{1}{\Tr(\tilde{G} )}\sum_{k=1}^d {\delta^{(u,\omega )}_{k}} A^{(u,\omega)}_{k} \widetilde{\otimes} B^{(u,\omega)}_{k} 
\]
and the mapping $   \G_{\Upsilon_1}$ with $x=3$
 and $\phi (z) =z$  in \eqref{eq7} by 
$$
\G_{\Upsilon_1, u,\omega}(\eta) = \Upsilon^{u,\omega}_1( \eta^2\mathrm{id}(\eta \F_{u,\omega}))~. 
$$
To show that $\Upsilon_1$  satisfies \autoref{as:mappings}, we note that the first and second are obvious since  $\Upsilon^{u,\omega}_1$ is a (scaled) rank d approximation of $\F_{u,\omega}$. To verify the third property, we first note that $\mathfrak{T}$ preserves the smoothness properties of an operator in the direction of $u \in [0,1]$. Indeed the properties of the $\mathfrak{T}$ imply for example that we have
\begin{align*}
F_{u,\omega} = \mathfrak{T}(\F_{u,\omega})=\mathfrak{T}(\F_{v,\omega})+\sum_{i=1}^2 (v-u)^i \mathfrak{T}\Big(\frac{\partial^i \F_{x,\omega}}{\partial x^i}\vert_{x=v}\Big)+\mathfrak{T}(R)~, 
\end{align*}
where the involved operators are bounded uniformly in $u$ for almost every $\omega$ with respect to $\norm{\cdot}_{S_2(\Hi_1)\otimes S_2(\Hi_2)}$ under \autoref{as:smooth}. The composition $G_{u,\omega}=F_{u,\omega}F^\dagger_{u,\omega}$ is then also twice-differentiable in $u$. Finally, we observe that $A_k^{u,\omega}$ is the eigenprojector belonging to the $k$th largest eigenvalue of $G_{u,\omega}$ and thus it follows from \autoref{thm:eigp2} that on the set $\mathcal{O}_T=\cap_{k=1}^d \mathcal{O}_{k,T}$, with $\mathcal{O}_{k,T}$ as defined in \eqref{eq:Okt}, $1_{\mathcal{O}_T}=1$ for sufficiently large $T$.  Consequently, on this set $ \Upsilon^{u,\omega}_{1}$ satisfies \autoref{as:mappings}(iii) when evaluated at the points $u_{i,T}, i=1,\ldots, M$.  

Then, by \autoref{thm:Conv} (with $x=3$), 
we have
\[\Big\{
 \rho_{T} M^{1/2}\Big(
 \big (\mathrm{L}^{a,b}_{U_M}\circ  {\cal G}_{\Upsilon_1} (\hat \F)  \big ) (\eta) - \big (
 \mathrm{L}^{a,b}\circ {\cal G}_{\Upsilon_1} ( \F)  \big ) (\eta)  \Big\}_{\eta \in I } \stackrel{\mathcal{D}}{\Longrightarrow} \Big\{\eta^{2}\mathbb{W}_{\mu_{\Upsilon_1}}(\eta)\Big\}_{\eta \in I}~.
 \tageq \label{de50}
 \]
We note that the third term in \eqref{eq:decG} is the conjugate of the second term, so that their sum simply corresponds to taking the real part of the left- and right-hand side in \eqref{de50}. For the first term, observe  that $\sum_{k=1}^d A^{(u,\omega)}_{kk}$ is a positive operator and its square root is therefore well-defined. The  isometric isomorphism \eqref{eq:isom}, \autoref{thm:maxPS1} and  \eqref{eq:etaNapp} then imply for the summand as given in the first term of  \eqref{eq:decG} that for fixed $u,\omega$
\begin{align*}
&\sup_\eta \Big\vert\Tr\Big( \sum_{k=1}^d \eta A^{(u,\omega)}_{kk}\big(\eta (\hat{F}_{u,\omega}(\eta)-{F}_{u,\omega})) (\eta \hat{F}^{\dagger}_{u,\omega}(\eta)-{F}^{\dagger}_{u,\omega})\big)\Big) \Big\vert 
\\&  \le  \bignorm{\big(\sum_{k=1}^d A^{(u,\omega)}_{kk}\big)^{1/2}}^2_{S_\infty} \bignorm{\eta( \hat{\F}_{u,\omega}(\eta)- {\F}_{u,\omega})}^2_{C_{S_1}} = o_p(\rho^{-1}_T M^{-1/2}).\tageq \label{eq:errb1}
\end{align*}
For  the term \eqref{dynpcas2},  \eqref{eq:isom} implies
\[
\Tr(\eta^2 \tilde{G} ) =  \Tr\big(\int_0^1 \int_a^b\eta^2 (\F_{u,\omega})^2 d\omega du\big)=\Tr\big( \mathrm{L}^{a,b} \circ \G_{\Upsilon_2}(\F)(\eta)\big)
\]
where $\G_{\Upsilon_2, u,\omega} = I_{\Hi_1 \otimes \Hi_2} \widetilde{\otimes} I_{\Hi_1 \otimes \Hi_2}\big( \phi(\eta \F_{u,\omega})\big)$ with $\phi: D \to \cnum, z \mapsto z^2$. 
Note that \autoref{thm:Conv}(b) applies with $h(\eta) = \eta$ and $x=1$, that is
\[
\Big\{M^{1/2}\rho_T \Big( L_{M}^{a,b} \circ \G_{\Upsilon_2}(\hat{\F})(\eta)-\mathrm{L}^{a,b} \circ \G_{\Upsilon_2}({\F})(\eta)\Big\}_{\eta \in I} \stackrel{\mathcal{D}}{\Longrightarrow} \Big\{\eta \mathbb{W}_{\mu_{\Upsilon_2}}(\eta)\Big\}_{\eta \in I}~.
\]
Observe that for this component \autoref{cor:wcFdif2} 
with $\Psi_2 (\cdot) = 1/\Tr(\cdot)$ and \autoref{lem:invtra} imply that
\begin{align*}
\Big\{\rho_T M^{1/2} \eta^5\Big(  \Psi_2\Big(\mathrm{L}^{a,b}_{U_M} \circ \G_{\Upsilon_2}(\hat{\F})(\eta)\Big)- 
 \Psi_2 \Big( \mathrm{L}^{a,b} \circ \G_{\Upsilon_2}({\F})(\eta)\Big)\Big)\Big\}_{\eta \in I}  &\stackrel{\mathcal{D}}{\Longrightarrow} 
 \Bigg\{\eta^5 \frac{\Tr\big(\eta \mathbb{W}_{\mu_{\Upsilon_2}}(\eta)\big)}{\Tr^2\big(\mathrm{L}^{a,b}(\eta^2 \F^2)\big)}\Bigg\}_{\eta \in I} 
 \\
&
 \tageq \label{eq:psi3}
~~~~=  
\Bigg\{\eta^2 \frac{\Tr\big( \mathbb{W}_{\mu_{\Upsilon_2}}(\eta)\big)}{\Tr^2\big(\mathrm{L}^{a,b}( \F^2)\big)}\Bigg\}_{\eta \in I} ~.
\end{align*}
Furthermore, a similar argument yields that the first and second term of 
\eqref{dynpcas} are of the same order as \eqref{eq:errb1}. Finally, the term defined by 
\eqref{dynpcas3} is also of order $o_p(\rho^{-1}_T M^{-1/2})$, which follows from the observation that 
$(\delta^{u,\omega}_{j})^2=\Tr(A_j^{u,\omega} G_{u,\omega})$, that the trace functional is linear, and the previously verified smoothness properties of the eigenprojectors and $G_{u,\omega}$ on the set $\mathcal{O}_T$. 
and mimicking the first few lines of the proof of \autoref{thm:fullapproximation}.
Altogether, setting $\Psi_1= \Tr(\cdot)$ and using linearity of the trace functional it follows  that \autoref{cor:wcFdif2} is applicable
with $k=2$.
Moreover, 
\eqref{eq14} holds  with $g_{ij}(\eta)= \eta^2, i,j =1,2$,
and the continuous mapping theorem yields
\begin{align*}
& \begin{Bmatrix}
M^{1/2} \rho_T \eta^3( \hat{s}_d(\eta) -s_d)
 \end{Bmatrix}_{\eta \in I }\\&=
 \begin{Bmatrix}M^{1/2}\rho_T 
\begin{pmatrix}
2   & c  \eta^5 \\
\end{pmatrix} \begin{pmatrix} 
\Psi_1\Big(\Re\big(\mathrm{L}^{a,b}_{U_M} \circ  \G_{\Upsilon_1}(\hat{\F})\big)\Big)(\eta)-\Psi_1\Big(\Re\big( \mathrm{L}^{a,b} \circ \G_{\Upsilon_1}({\F})\big)\Big)(\eta) 
 \\
  \Psi_2\Big(\mathrm{L}^{a,b}_{U_M} \circ \G_{\Upsilon_2}(\hat{\F})\Big)(\eta)-  \Psi_2\Big( \mathrm{L}^{a,b} \circ \G_{\Upsilon_2}({\F})\Big) (\eta)
 \end{pmatrix}
 \end{Bmatrix}_{\eta \in I } 
  + o_P(1)  
 \\& \stackrel{\mathcal{D}}{\Longrightarrow}  \begin{Bmatrix} 2\eta^2 \Tr\big(\Re(\mathbb{W}_{\mu_{\Upsilon_1}})(\eta) \big) + \eta^2 
 c \eta^2 \Tr\big(\mathbb{W}_{\mu_{\Upsilon_2}}(\eta) \big) \end{Bmatrix}_{\eta \in I }~, \tageq \label{eq:SPCwc1}
\end{align*}
where $c=-\Tr\big(\mathrm{L}^{a,b}\big( \sum_{k=1}^d A_{kk}({G})\big)\big)/(\Tr^2\big( \mathrm{L}^{a,b}( \F^2)\big))$, and where $\Re(\cdot)$ denotes the real part of $(\cdot)$. 
Note that $\mathbb{W}=(\mathbb{W}_{\mu_{\Upsilon_1}}, \mathbb{W}_{\mu_{\Upsilon_2}})^\top$ is
a bivariate vector of $S_1(\Hi)$-valued Brownian motions corresponding to a real-valued Gaussian measure with positive definite covariance structure. 
Recall that for any $B \in S_{\infty}(\Hi)$, functionals $f_B: S_{1} \to \cnum, A \mapsto \Tr(A B)$ are elements of the topological dual space of $S_1(\Hi)$, which proves that
\begin{align*}
\Bigg\{
M^{1/2}\rho_T \eta^3(\hat{s}_d(\eta) - s_d)\Bigg\}_{\eta \in I } 
\stackrel{\mathcal{D}}{\Longrightarrow} \begin{Bmatrix}\eta^2 \sigma\mathbb{B}(\eta) \end{Bmatrix}_{\eta \in I } 
\end{align*}
for some $\sigma \geq 0$ and a standard real-valued Brownian motion $\mathbb{B}$.

\subsubsection{Proof of \autoref{thm:lFC}}
\label{secd3}

Recall from  Section \ref{sec224} that $P^{\Hi}_i, i=1,2$ denotes the orthogonal projection of $\Hi=\Hi_1 \oplus \Hi_2$ onto $\Hi_i$ such that
\[
 \F^{ij}=  P^{\Hi}_i {\F}P^{\Hi}_j \text{ and } \hat{\F}^{ij}=  P^{\Hi}_i  \hat{\F}P^{\Hi}_j 
\]
denotes the restriction of $ {\F}$, respectively $\hat{\F}$, to $\Hi_i $ and $\Hi_j$, $j=1,2$. By \autoref{thm:Conv} and the continuous mapping theorem 
\[\Big\{
 \rho_{T} \Big(
 \big (\mathrm{L}^{a,b}_{U_M}\circ  {\cal G}_\Upsilon (\hat{\F}^{ij})  \big ) (\eta) - \big (
 \mathrm{L}^{a,b}\circ {\cal G}_\Upsilon  ( \F^{ij})  \big ) (\eta)  \Big\}_{\eta \in I } \stackrel{\mathcal{D}}{\Longrightarrow} \Big\{\mathbb{W}^{ij}_{\mu_\Upsilon}(\eta)\Big\}_{\eta \in I}
 \]
 where 
 $ \mathbb{W}^{ij}_{\mu_{\Upsilon}}(\eta)=P_i^\Hi  \mathbb{W}_{\mu_{\Upsilon}}(\eta) P_j^\Hi$. 
Consider then the composite operators
 \[
 \mathfrak{F}^{ij}_{u,\omega}=\F^{ij}_{u,\omega}  \F^{ji}_{u,\omega}~~ \text{ and } ~ \hat{\mathfrak{F}}^{ij}_{u,\omega}(\eta)=\hat{\F}^{ij}_{u,\omega} (\eta)\hat{\F}^{ji}_{u,\omega}(\eta)~, 
 \]
which have respective eigendecompositions given by
 \[
  \mathfrak{F}^{ij}_{u,\omega}=\sum_{k=1}^{\infty} \big({\nu}_k^{u,\omega}\big)^2 \Pi^{u,\omega}_{\mathfrak{F}_{ij}, k} ~~\text{ and }  ~~ \hat{\mathfrak{F}}^{ij}_{u,\omega}(\eta)=\sum_{k=1}^{\infty} \big(\hat{\nu}_k^{u,\omega}(\eta)\big)^2 \hat{\Pi}^{u,\omega}_{\mathfrak{F}_{ij}, k}
  ~.
  \tageq \label{eq:mathfrakF}
 \] 
 Furthermore, recall that  
  $$
{\mathcal{R}}^{u,\omega}_{d}= \frac{{\nu}_d^{u,\omega}}{\sqrt{{\lambda}^{(u,\omega)}_{11,d} {\lambda}^{(u,\omega)}_{22,d}}}  
 \text{~~ and ~~}  
\hat{\mathcal{R}}^{u,\omega}_{d}(\eta)= \frac{\hat{\nu}_d^{u,\omega}(\eta)}{\sqrt{\hat{\lambda}^{(u,\omega)}_{11,d}(\eta) \hat{\lambda}^{(u,\omega)}_{22,d}(\eta)}}  
~. 
$$
Then, letting $x={\nu}_d^{u,\omega}$, $y=\sqrt{{\lambda}^{(u,\omega)}_{11,d}}$ and $z=\sqrt{ {\lambda}^{(u,\omega)}_{22,d}}$ and 
using the notation $\hat x, \hat y, \hat  z$ for the empirical counterparts, we can write
\begin{align*}
(\hat{\mathcal{R}}^{u,\omega}_{d}(\eta)-{\mathcal{R}}^{u,\omega}_{d})=
\frac{\hat{x}}{\hat{y} \hat{z}}-\frac{x}{yz}&= \frac{x}{z}\Big(\frac{1}{\hat{y}}-\frac{1}{y}\Big)+ \frac{x}{y}\Big(\frac{1}{\hat{z}}-\frac{1}{z}\Big)+ \frac{1}{yz}\big(\hat{x}-x\big)
\\& + \hat{x} \Big(\frac{1}{\hat{y}}-\frac{1}{y}\Big) \Big(\frac{1}{\hat{z}}-\frac{1}{z}\Big)+ \frac{1}{z}\Big(\frac{1}{\hat{y}}-\frac{1}{y}\Big)\big(\hat{x}-x\big)+ \frac{1}{y}\Big(\frac{1}{\hat{z}}-\frac{1}{z}\Big)\big(\hat{x}-x\big). \tageq \label{eq:fracxyz}
\end{align*}
 Observe that after scaling with $\eta^x$ for the first two terms, a Taylor expansion  yields for $i, j \in \{1,2\}$,
\begin{align*}
&  \frac{\eta^x{\nu}_d^{u,\omega}}{\sqrt{ {\lambda}^{(u,\omega)}_{ii,d}}}\Big(\big(
\hat{\lambda}^{(u,\omega)}_{jj,d}(\eta)\big)^{-1/2}-\big({\lambda}^{(u,\omega)}_{jj,d}\big)^{-1/2}\Big)
 \\&=-\frac{\eta^x{\nu}_d^{u,\omega}}{\sqrt{{\lambda}^{(u,\omega)}_{ii,d}}}\Big[\Big(\hat{\lambda}^{(u,\omega)}_{jj,d}(\eta)-{\lambda}^{(u,\omega)}_{jj,d}\Big) \frac{1}{2}\big({\lambda}^{(u,\omega)}_{jj,d}\big)^{-3/2}\Big]
 +O_p\Big(\Big\vert\eta^x \Big(\hat{\lambda}^{(u,\omega)}_{jj,d}(\eta)-{\lambda}^{(u,\omega)}_{jj,d}\Big)^2 \Big\vert\Big)~.
\end{align*}
Hence, by \autoref{prop:bexplam},
\begin{align*} 
&=-\frac{\eta^x{\nu}_d^{u,\omega}}{\sqrt{{\lambda}^{(u,\omega)}_{ii,d}}}\Big(\hat{\lambda}^{(u,\omega)}_{jj,d}(\eta)-{\lambda}^{(u,\omega)}_{jj,d}\Big) \frac{1}{2}\big({\lambda}^{(u,\omega)}_{jj,d}\big)^{-3/2}\Big]
+O_p\Big(\bignorm{\eta^{x/2}(\hat{\F}_{u,\omega}(\eta)-\F_{u,\omega} )}^2_{C_{S_1}}\Big)
\\& 
=-\frac{{\nu}_d^{u,\omega}}{2 \big({\lambda}^{(u,\omega)}_{ii,d}({\lambda}^{(u,\omega)}_{jj,d})^3\big)^{1/2}} \Big[ \Tr \Big(\Pi^{jj}_{u,\omega,d}\eta^x\big( \hat{\F}^{jj}_{u,\omega}(\eta)-\F^{jj}_{u,\omega}(\eta)\big)\Big) 
+o_p(M^{-1/2}\rho^{-1}_T)+o_p(M^{-1/2}\rho^{-1}_T)~,
\end{align*}
where the error term holds for all $x\ge 3$, which follows from \autoref{prop:bexplam},  \autoref{thm:maxPS1}, 
and \eqref{eq:etaNapp}.
Therefore, we define 
\[
\Upsilon^{u,\omega}_1 = -\frac{{\nu}_d^{u,\omega}}{2 \big({\lambda}^{(u,\omega)}_{22,d}({\lambda}^{(u,\omega)}_{11,d})^3\big)^{1/2}}\Pi^{11}_{u,\omega,d},  \quad \Upsilon^{u,\omega}_2 = -\frac{{\nu}_d^{u,\omega}}{2 \big({\lambda}^{(u,\omega)}_{11,d}({\lambda}^{(u,\omega)}_{22,d})^3\big)^{1/2}}\Pi^{22}_{u,\omega,d}~. 
\]
Note that for sufficiently small perturbations in the direction of $u$ of $\F$, both mappings  are products of analytic functions of $\F$, and hence a similar argument as in the proof of \autoref{thm:lPSCA} allows to verify that \autoref{as:mappings} is satisfied. Hence, letting $\phi=\mathrm{id}$ and $\Psi: C_{S_1} \to C_\rnum, A(\eta) \mapsto \Tr(A(\eta))$, $x =3$, we obtain
\[
 \frac{\eta^x{\nu}_d^{u,\omega}}{\sqrt{ {\lambda}^{(u,\omega)}_{ii,d}}}\Big(\big(
 \hat{\lambda}^{(u,\omega)}_{jj,d}(\eta)\big)^{-1/2}-\big({\lambda}^{(u,\omega)}_{jj,d}\big)^{-1/2}\Big)
=
\big( \Psi  \big ( 
\mathrm{L}^{a,b}_{U_M}\circ \G_{\Upsilon_j} (\hat \F^{jj} \big ) (\eta)
-  \Psi  \big ( 
\mathrm{L}^{a,b}\circ  \G_{\Upsilon_j} (\F^{jj} \big ) (\eta)
\big) +o_p(M^{-1/2}\rho^{-1}_T).
\]
For the third term of \eqref{eq:fracxyz}, a Taylor expansion and \autoref{thm:maxPS1} yield
\begin{align*}
& \eta^x\Big({\lambda}^{(u,\omega)}_{11,d}{\lambda}^{(u,\omega)}_{22,d}\Big)^{-1/2}\Big(
\big( \hat{\nu}_d^{u,\omega}(\eta))^{2}\big)^{1/2}-\big( ( {\nu}_d^{u,\omega})^{2}\big)^{1/2}
\Big) 
\\&=  \eta^x\Big({\lambda}^{(u,\omega)}_{11,d}{\lambda}^{(u,\omega)}_{22,d}\Big)^{-1/2}\Big[\Big(
(\hat{\nu}_d^{u,\omega}(\eta))^{2}- ( {\nu}_d^{u,\omega})^{2}
\Big) \frac{1}{2}\big( ({\nu}_d^{u,\omega})^{2}\big)^{-1/2}+ O_p\big(\sup_{\eta}\bignorm{\eta^{\ell/2}\big( \hat{\mathfrak{F}}^{12}_{u,\omega}(\eta)- {\mathfrak{F}}^{12}_{u,\omega}\big)}^2_{{S_1}}\big)\Big].
\end{align*}
Then, using again \autoref{prop:bexplam},  \autoref{thm:maxPS1} and \eqref{eq:etaNapp}, and recalling the definition of ${\Pi}^{u,\omega}_{\mathfrak{F}_{ij}, k}$ and ${\Pi}^{u,\omega}_{\hat{\mathfrak{F}}_{ij}, k}$ in  \eqref{eq:mathfrakF}, we obtain for $x \ge 4$
\begin{align*}
& =  \frac{1}{2}\big(  {\nu}_d^{u,\omega}\big)^{-1} \Big({\lambda}^{(u,\omega)}_{11,d}{\lambda}^{(u,\omega)}_{22,d}\Big)^{-1/2}\Tr\Big({\Pi}^{u,\omega}_{\mathfrak{F}_{12}, k}\eta^x\big( \hat{\mathfrak{F}}^{12}_{u,\omega}(\eta)- \mathfrak{F}^{12}_{u,\omega}\big)\Big) +o_p(M^{-1/2}\rho^{-1}_T)
\\&= \frac{1}{2 }\Big(\big( {\nu}_d^{u,\omega}\big)^2{\lambda}^{(u,\omega)}_{11,d}{\lambda}^{(u,\omega)}_{22,d}\Big)^{-1/2}
\Big(\Tr\Big({\Pi}^{u,\omega}_{\mathfrak{F}_{12}, k}\eta^x\big( \hat{\F}^{12}_{u,\omega}(\eta)- {\F}^{12}_{u,\omega}\big) {\F}^{21}_{u,\omega}\Big)+\Tr\Big({\Pi}^{u,\omega}_{\mathfrak{F}_{12}, k}{\F}^{12}_{u,\omega}\eta^x \big( \hat{\F}^{21}_{u,\omega}(\eta)- {\F}^{21}_{u,\omega}\big)  \Big)\\& +o_p(M^{-1/2}\rho^{-1}_T).
\end{align*}

Define
$\Upsilon^{u,\omega}_3 =  \frac{1}{2}\Big(\big( {\nu}_d^{u,\omega}\big)^2{\lambda}^{(u,\omega)}_{11,d}{\lambda}^{(u,\omega)}_{22,d}\Big)^{-1/2}
{\Pi}^{u,\omega}_{\mathfrak{F}_{12}, k}{\F}^{12}_{u,\omega}$. 
A similar argument as in the case of $\Upsilon_{1,u,\omega}$ and $\Upsilon_{2,u,\omega}$ shows that this map satisfies  \autoref{as:mappings}. A similar line of proof yields that all three terms in the second line of \eqref{eq:fracxyz} are of order $o_p(M^{-1/2}\rho^{-1}_T)$ for $x \ge 4$. Altogether, the continuous mapping theorem and \autoref{cor:wcFdif2} therefore yield 
\begin{align*}
& \Big\{M^{1/2}\rho_T \Big(\int_0^1 \int_a^b \eta^4\hat{\mathcal{R}}^{u,\omega}_{d}(\eta) d\omega du-\int_0^1 \int_a^b\eta^4 {\mathcal{R}}^{u,\omega}_{d} d\omega du\Big)\Big\}_{\eta \in I} =
\\&
 \begin{pmatrix}
1& 1& 2 \\
\end{pmatrix} M^{1/2}\rho_{T}
\begin{Bmatrix}
 \Psi  \Big ( 
\mathrm{L}^{a,b}_{U_M}\circ \G_{\Upsilon_1}(\hat{\F}^{11})\Big) (\eta)
-  \Psi  \Big ( 
\mathrm{L}^{a,b} \circ  \G_{\Upsilon_1} (\F^{11} ) \Big)(\eta)
\\
\Psi  \Big ( 
\mathrm{L}^{a,b}_{U_M}\circ \G_{\Upsilon_2}(\hat{\F}^{22})\Big) (\eta)
-  \Psi  \Big ( 
\mathrm{L}^{a,b}\circ  \G_{\Upsilon_2} (\F^{22} ) \Big)(\eta)\\
 \Psi  \Big (\Re\big( 
\mathrm{L}^{a,b}_{U_M}\circ \G_{\Upsilon_3}(\hat{\F}^{21})\big)\Big) (\eta)
-  \Psi  \Big (\Re\big( 
\mathrm{L}^{a,b}\circ  \G_{\Upsilon_3} (\F^{21} )\big)\Big) (\eta)
\end{Bmatrix}_{\eta \in I}+ o_p(1) \stackrel{\mathcal{D}}{\Rightarrow} \Big\{ \eta^3 \sigma \mathbb{B}(\eta) \Big\}_{\eta \in I}~.
\end{align*}

\subsubsection{Proof of \autoref{optimalprocr}
and \autoref{thm:procconv}}\label{secsqrmet}

\begin{proof}[Proof of \autoref{optimalprocr}]
The argument follows straightforwardly from the minimum distance principle \citep[see][section 2.3]{vDCD18} and the fact that the involved operators are elements of $L^1_{S_1(\Hi)^+}\big([0,1]\times [0,\pi]\big)$. 
\end{proof}

\begin{proof}[Proof of \autoref{thm:procconv}]
Recalling the definition of the operators $\hat{\ddot{\F}}_{\omega,d}$  and ${\ddot{\F}}^{1/2}_{\omega,d}$ in \eqref{de32}  we  have
\begin{align*}
&M^{1/2} \rho_T {\eta} (\hat{r}_{d}(\eta)-r_{d}(\eta) )
\\&=M^{1/2} {\eta} \rho_T  \Big\{\frac{1}{M}\sum_{u \in U_M}\int_{0}^\pi d^2_{R}(\hat{\ddot{\F}}_{\omega,d}(\eta), \eta\hat{\F}_{u,\omega, d}(\eta))  d\omega 
-\int_{0}^\pi \int_{0}^1 d^2_{R}(\ddot{\F}_{\omega,d}, \eta\F_{u,\omega, d}) du d\omega\Big\} ~.
\tageq \label{eq:Pr2}
\end{align*}
Using \autoref{optimalprocr} we can rewrite this as
\begin{align*}
 M^{1/2}\rho_T  \Big\{ &\int_0^\pi \frac{1}{M}\sum_{u \in U_M} \Tr\big(\eta^2\hat{\F}_{u,\omega, d}(\eta)\big)d\omega- \int_0^\pi \int_0^1\Tr\big(\eta^2{\F}_{u,\omega, d}\big) du d\omega \tageq \label{eq:er1}
\\& -\Big(  \int_0^\pi\Tr\big({\eta}\hat{\ddot{\F}}_{\omega,d}(\eta)\big)d\omega-  \int_0^\pi\Tr\big({\eta}{\ddot{\F}}_{\omega,d}\big) d\omega \Big)\Big\}_{\eta \in I }.
\tageq \label{eq:er2}
\end{align*}
We treat these two terms in turn. For \eqref{eq:er1}, we observe that this fits 
\autoref{cor:wcFdif2} with $\Psi_1:C(I, S_1(\Hi)) \to C(I, \rnum), A(\eta) \mapsto  \Tr(A(\eta))$, and where $\G_{\Upsilon_1}$ is defined by \eqref{eq4a} with $\Upsilon^{u,\omega}_1=I \widetilde{\otimes} I$ for all $u \in [0,1]$, $\omega \in [0,\pi]$, $\phi = \mathrm{id}$ and $x=2$. More precisely, we have 
\begin{align*}
& M^{1/2}\rho_T  \Big\{  \int_0^\pi \frac{1}{M}\sum_{u \in U_M} \Tr\big(\eta^2\hat{\F}_{u,\omega, d}(\eta)\big)d\omega- \int_0^\pi \int_0^1\Tr\big(\eta^2{\F}_{u,\omega, d}\big) du d\omega\Big\}_{\eta \in I }  
\\& =M^{1/2}\rho_{T}\Big\{ \big( \Psi_1  \big ( 
\mathrm{L}^{a,b}_{U_M}\circ \G_{\Upsilon_1} (\hat \F) \big ) (\eta)
-  \Psi_1  \big ( 
\mathrm{L}^{a,b}\circ  \G_{\Upsilon_1} (\F )\big ) (\eta)
  \Big\}_{\eta \in I }   
\end{align*}
The term in \eqref{eq:er2} may be written as
\begin{align*}
&={\eta}  \int_0^\pi \Tr\Big( \Big[\frac{1}{M}\sum_{u \in U_M}  (\eta \hat{\F}_{u,\omega, d}(\eta))^{1/2}\Big]^2\Big)d\omega- \eta
\int_0^\pi \Tr\Big( \Big[\int_0^1 (\eta\F_{u,\omega, d})^{1/2}du\Big]^2 \Big)d\omega
\\
& ={\eta} \int_0^\pi \bignorm{ \frac{1}{M}\sum_{u \in U_M} \phi(\eta \hat{{\F}}_{u,\omega, d}(\eta))- \int_0^1  \phi( \eta\F_{u,\omega, d})du  }^2_{S_2} d\omega \tageq \label{eq:Prterm1}
\\
&+
\int_0^\pi {\eta} \biginprod{\frac{1}{M}\sum_{u \in U_M} \phi(\eta \hat{{\F}}_{u,\omega,d}(\eta))- \int_0^1 \phi( \eta\F_{u,\omega, d})du }{{\ddot{\F}}^{1/2}_{u,\omega,d}(\eta)}  d\omega
\tageq  \label{eq:Prterm2a}
\\
&+
\int_0^\pi {\eta} \biginprod{{\ddot{\F}}^{1/2}_{u,\omega,d}(\eta)}{\frac{1}{M}\sum_{u \in U_M} \phi(\eta \hat{{\F}}_{u,\omega,d}(\eta))- \int_0^1 \phi( \eta\F_{u,\omega, d})du }  d\omega
~,
\tageq \label{eq:Prterm2}
\end{align*}
where $\phi: D \to \cnum, z \mapsto z^{1/2}$, and where $D \subset \cnum$ is an open bounded region excluding elements of the negative ray $\rnum_{-} = (-\infty,0]$. Note that $\phi$ is analytic on $D$ for $d<\infty$. 
\eqref{eq:Prterm2} can be written as
\begin{align*}
& \int_0^\pi \biginprod{\frac{1}{M}\sum_{u \in U_M} \eta\phi(\eta \hat{{\F}}_{u,\omega,d}(\eta))- \int_0^1 \eta\phi( \eta\F_{u,\omega, d})du }{\ddot{\F}^{1/2}_{\omega,d}(\eta)}  d\omega 
\\&=\int_0^\pi \Tr\Bigg(\ddot{\F}^{1/2}_{\omega,d}(\eta) \widetilde{\otimes} I \Big(\frac{1}{M}\sum_{u \in U_M} \eta\phi(\eta \hat{{\F}}_{u,\omega,d}(\eta))- \int_0^1 \eta\phi( \eta\F_{u,\omega, d})du \Big)\Bigg)  d\omega 
\\&=\Tr\Bigg(L^{0,\pi}_{U_M} \Big(\Upsilon_{2,\omega} \eta^{3/2}\phi(\eta \hat{{\F}}_{u,\omega,d}(\eta))\Big)- L^{0,\pi}\Big(\Upsilon_{2,\omega} \eta^{3/2}\phi( \eta\F_{u,\omega, d})\Big)\Bigg)  
\end{align*}
since $\ddot{\F}^{1/2}_{\omega,d}(\eta)=\eta^{1/2}\ddot{\F}^{1/2}_{\omega,d} $, and 
where we used the notation $\Upsilon_{2,\omega} =
\ddot{\F}^{1/2}_{\omega,d} \widetilde{\otimes} I$. Hence, set $\Psi_2:C(I, S_1(\Hi)) \to C(I, \rnum), A(\eta) \mapsto \Tr(A(\eta))$ and $x=5/2$.  Then we can rewrite  \eqref{eq:Prterm2}  as 
\begin{align*}
 \Big\{M^{1/2}\rho_{T}\big( \Psi_2  \big ( 
\mathrm{L}^{0,\pi}_{U_M}\circ \G_{\Upsilon_2} (\hat \F \big ) (\eta)
-  \Psi_2  \big ( 
\mathrm{L}^{0,\pi}\circ  \G_{\Upsilon_2} (\F \big ) (\eta)
 \Big\}~,
\end{align*}
Using \autoref{thm:Conv}(b) with $h(\eta) = \eta^{-1/2}$  we observe that this fits as a component term in the sense of
\autoref{cor:wcFdif2}  and 
the term \eqref{eq:Prterm2a} corresponds to its adjoint.
\\
On the other hand,  
the   term  in  
\eqref{eq:Prterm1} is of order $o_p(\rho^{-1}_T M^{-1/2})$ uniformly in $\eta \in I$. To see this, we write
\begin{align*}
&{\eta} \int_0^\pi \bignorm{ \frac{1}{M}\sum_{u \in U_M} \phi(\eta \hat{{\F}}_{u,\omega, d}(\eta))- \int_0^1  \phi( \eta\F_{u,\omega, d})du  }^2_{S_2} d\omega
\\
& 
{\eta} \int_0^\pi \bignorm{ \frac{1}{M}\sum_{u \in U_M}
\phi(\eta \hat{{\F}}_{u,\omega,d}(\eta))-\frac{1}{M}\sum_{u \in U_M} \phi(\eta \F_{u,\omega, d})+\frac{1}{M}\sum_{u \in U_M} \phi(\eta \F_{u,\omega, d})- \int_0^1  \phi( \eta\F_{u,\omega, d})du  }^2_{S_2} d\omega
\\& \le 2\pi
 \sup_{\eta,\omega} \bignorm{\frac{1}{M} \sum_{u\in U_M} \Big[\phi(\eta \hat{\F}_{u,\omega, d}(\eta))-\phi( \eta\F_{u,\omega, d})   \Big] }^2_{S_2} \tageq \label{eq:sqrtderror1}
\\& +2\pi
\sup_{\eta,\omega} \bignorm{\sum_{i = 1}^{M} \int_{x_{i-1,T}}^{x_{i,T}}\big(\phi(\eta {\F}_{u_{i,T},\omega, d}(\eta))-\phi(\eta \F_{u,\omega, d})\big) du}^2_{S_2}~, \tageq \label{eq:sqrtderror2}
\end{align*}
where we used the same notation as in the proof of \autoref{thm:fullapproximation}. Now, the terms  \eqref{eq:sqrtderror1} and \eqref{eq:sqrtderror2} can be shown to be of order $o_p(\rho^{-1}_T M^{-1/2})$. 
For this purpose, denote $\hat{\Delta}_T(\eta)=\eta (\hat{\F}_{u\omega,d}(\eta)-\F_{u,\omega,d})$ and define the sets 
$$\Omega_T= \big \{x \in \Omega: \sup_{\omega,u,\eta}  \norm{\hat{\Delta}_T(\eta) }_{S_1}\le \varepsilon \big \} \quad \text{ and } E_T =\bigcap_{i=1}^M \Big\{\sup_{u \in [x_{i-1,T}, x_{i,T}]}\sup_{\omega}  \norm{\F_{u_{i,T},\omega,d}-\F_{u,\omega,d} }_{S_1} \le \epsilon \Big \}
$$
for some $\epsilon, \varepsilon>0$. 
By \autoref{thm:PerEx} and \autoref{thm:maxdevcon},  for sufficiently large $T$, 
\begin{align*}
M^{1/2}\rho_{T} & \sup_{\eta}\bignorm{\big(\phi(\eta \hat{\F}_{u,\omega,d}(\eta))-\phi( \eta\F_{u,\omega,d})\Big)}^{2}_{S_1} 
\\ &= M^{1/2}\rho_{T} \sup_{\eta}
\bignorm{\phi^\prime_{\eta{\F}_{u,\omega,d}} (\hat{\Delta}_T(\eta) +R_{u,\omega,d}(\eta))}^2_{S_1}\mathrm{1}_{\Omega_T}  +o_p(1)
\\&  =M^{1/2}\rho_{T} \sup_{\eta}
\bignorm{\phi^\prime_{{\F}_{u,\omega,d}} (\hat{\Delta}_T(\eta) +R_{u,\omega,d}(\eta))}^2_{S_1}\mathrm{1}_{\Omega_T}  +o_p(1)
\\&  \le M^{1/2}\rho_{T} \sup_{\eta}
\bignorm{\phi^\prime_{{\F}_{u,\omega,d}}}_{\mathfrak{L}(S_1)} \Big(\norm{\hat{\Delta}_T(\eta) }^2_{S_1}+\norm{\hat{\Delta}_T(\eta) }^4_{S_1}\Big)\mathrm{1}_{\Omega_T}  +o_p(1)=o_p(1), \end{align*}
for all $u \in [0,1]$ and $\omega \in [a,b]$, where the error follows from \autoref{thm:maxPS1}. Finally,  a similar argument as in first few lines the proof of \autoref{thm:fullapproximation} provides the result for \eqref{eq:sqrtderror2}. 
Hence, we obtain  by  \autoref{cor:wcFdif2} for \eqref{eq:Pr2}
\begin{align*}
&\big \{ M^{1/2}\eta\rho_T(\hat{r}_{d}(\eta)-r_{d}(\eta) )
\big \}_{\eta \in I}
\\&~~~~~~~~~ =
\begin{Bmatrix}M^{1/2}\rho_T 
\begin{pmatrix}
1& 2 \\
\end{pmatrix}
 \begin{pmatrix}  \Psi_1  \Big ( 
\mathrm{L}^{0,\pi}_{U_M}\circ \G_{\Upsilon_1} (\hat \F) \big ) (\eta)
-  \Psi_1  \big ( 
\mathrm{L}^{0,\pi}\circ  \G_{\Upsilon_1} (\F) \Big ) (\eta)
 \\
\Psi_2  \Big( \Re\big( 
\mathrm{L}^{0,\pi}_{U_M}\circ \G_{\Upsilon_2} (\hat \F )\big)\Big ) (\eta)
-  \Psi_2  \Big ( \Re\big( 
\mathrm{L}^{0,\pi}\circ  \G_{\Upsilon_2} (\F )\big) \Big) (\eta)
 \end{pmatrix} \end{Bmatrix}_{\eta \in I} +o_p(1)
\\&~~~~~~~~~
 \stackrel{\mathcal{D}}{\Longrightarrow}  \begin{Bmatrix} \Tr\big(\eta\mathbb{W}_{\mu_{\Upsilon_1}}(\eta) \big) +2\eta^{3/2-1/2}\Tr\big(\Re(\mathbb{W}_{\mu_{\Upsilon_2}})(\eta) \big)
 \end{Bmatrix}_{\eta \in I}~.
\end{align*}
The result now follows by similar arguments as given in the proof of  \autoref{thm:lPSCA}.
\end{proof}

\subsection{Proof of the results in Section \ref{sec4} and \ref{sec5} }\label{proofsec4}

\begin{proof}[Proof of Theorem \ref{thm1}]
Recalling the definition of $\hat I_T$ in \eqref{deq13}, we obtain by an elementary calculation that
(note that the distribution of the random variable $\mathbb{T}$ is symmetric, which implies $q_\alpha/2 = -  q_{1-\alpha/2}$)
	$$
	\lim_{T \to \infty}	\mathbb{P}  \big  \{ r \in \hat I_T  \big \}
	=
			\lim_{T \to \infty} \mathbb{P}    \big \{ q_{\alpha/2} { { V}}   \le  \hat r
			-
			r \le q_{1-\alpha/2} { { V}}  \big  \}
			=
		 \mathbb{P}   \big  \{ q_{\alpha/2} \le \mathbb{T}  \le q_{1-\alpha/2}   \big  \}=1 - \alpha .
	$$
	\end{proof}

\begin{proof}[Proof of Theorem \ref{thm3}] 
 Note that by \eqref{hx2}
$$
\mathbb{P} \big (\hat{d} < d^* \big  )  = \mathbb{P} \Big( \bigcup^{d^*-1}_{d=1}  \big \{ \hat{\mathbb{T}}_d > q_\alpha  \big  \}\Big  ) \leq  \sum^{d^*-1}_{d=1} \mathbb{P} \big(   \big \{ \hat{\mathbb{T}}_d > q_\alpha  \big  \}\big  )
 \longrightarrow 0.
$$
Similarly, if $d>d^*$ and $s_{d^*} >  \nu$, we have
$$
\mathbb{P}\big (\hat{d} > d^* \big ) =  \mathbb{P} \Big( \bigcap^{d^*}_{d=1}  \big  \{ \hat{\mathbb{T}}_d \leq q_\alpha   \big  \}\Big  )\leq \mathbb{P} \big  (T_{d^*} \leq q_\alpha \big )~,
$$
where by \eqref{hx1} and \eqref{hx2} the right-hand side converges to
$0$ or  $\alpha$   if  $s_{d^*}> \nu$ or $s^*_{d} = \nu$, respectively. Finally,
 the  remaining  assertion follows from the fact that the limiting distribution of $\mathbb{T}_{d^*}$
 is supported on the real line.
\end{proof}

\begin{proof}[Proof of Theorem \ref{thm4}]  If the null hypothesis holds, we have
$$
\mathbb{P}_{H_0} \big (\hat{d} > d_0 \big ) = \mathbb{P}_{d^* \leq d_0} \big( \hat{\mathbb{T}}_d \leq q_\alpha \ ; \ d=1,\ldots,d_0 \big) \leq \mathbb{P}_{d^* \leq d_0} \big( T_{d^*} \leq q_\alpha \big )~,
$$
where by \eqref{hx1} and \eqref{hx2} the probability on the right-hand side converges to $\alpha$ if $s_{d^*}=\nu$ and to $0$ if $s^*_d > \nu$. 
Consequently, the test which rejects $H_0$ if  $\hat{d} > d_0$,
defines an asymptotic level $\alpha $ test.
Similarly, under the alternative 
it follows that
$$
\mathbb{P}_{H_1} (\hat{d} > d_0) = \mathbb{P}_{d^* > d_0} \big( \hat{\mathbb{T}}_d \leq q_\alpha \ ; \ d=1,\ldots,d_0 \big) \rightarrow 1
$$
because, by \eqref{hx2}, $\hat{\mathbb{T}}_d {\stackrel{\mathbb{P}}{\longrightarrow}}- \infty$ for $d=1,\ldots,d_0 < d^*.$  This proves consistency.
\end{proof}

 \begin{proof}[Proof of Theorem \ref{thm5}]  
For the probability of rejection we have
$$
\mathbb{P} \big (  \hat{s}_{d_0} > \nu  + q_{1-\alpha,d_0}  V_{{d_0},{d_0}} \big ) =
\mathbb{P} \Big (  \frac{\hat{s}_{d_0}  -s_{d_0}}{V_{{d_0},{d_0}}}
>  \frac{ \nu -s_{d_0}}{V_{{d_0},{d_0}}} + q_{1-\alpha,d_0}   \Big )~.
$$
By \eqref{deq12}, the statistic $\frac{\hat{s}_{d_0}  -s_{d_0}}{V_{{d_0},{d_0}}} $ converges weakly to $\mathbb{T}_{d_0}$. Note that \eqref{deq14} is equivalent to \eqref{1.5}.
Therefore, under the null hypothesis we have $s_{d_0} \leq \nu$, and, observing \eqref{deq19}, it follows that 
$\mathbb{P} \big (  \hat{s}_{d_0} > \nu  + q_{1-\alpha,d_0}  V_{{d_0},{d_0}} \big ) $ converges to $ 0 $ if $s_{d_0} < \nu$
and to $\alpha$ if $s_{d_0}= \nu$. By the same argument we obtain that the probability of rejection converges to $1$ under the alternative, i.e.,  $s_{d_0} > \nu$.

      \end{proof}

\purp{\begin{proof}[Proof of Proposition \ref{propbias}] 
We make use of results derived in \cite{vde16} for estimation of time-varying spectral density operators based on localized tapered periodogram operators. With a bit of effort, we can indeed represent the (non-interpolated) local estimator \eqref{eq:Fint} as
\[
\hat \F_{u,\omega}(\eta)=\frac{1}{ b_f}\int  K_f \Big (\frac{\omega-\alpha}{\bf} \Big) I_{u,\alpha}(\eta) d\alpha ~,
\]
where $K_f (\cdot )$ is a symmetric  kernel,
$I_{u,\alpha}(\eta)$ is a tapered local periodogram operator and where the taper relates to a smoothing kernel in time direction by $
K_{t,\eta}(x) = \frac{1}{H_N}  h(x+\frac{1}{2})^2$, $x \in [-1/2,\eta-1/2]$,  
with $H_{N}= \sum_{s=0}^{N-1} h(\frac{s}{N})^2 e^{-\im \omega s}$. With this representation, a Taylor expansion about the point $x=(u_o,\omega_o)$ yields
\begin{align*}
\E &(\hat \F_{u_o,\omega_o}(\eta))=\F_{u_o,\omega_o}(\eta) + \sum_{i=1}^{2} \frac{1}{i! }b_t^i \int v^i K_{t,\eta} (v) dv \, \int_{\Pi}  K_f(\alpha) d\alpha \frac{\partial^i\F_{u,\omega}(\eta)}{\partial u^i} \Big|_{(u,\omega)=x}  \\&+ 
\sum_{i=1}^{2} \frac{1}{i!} b_f^{i} \int_{\Pi}\alpha^i K_f(\alpha) d\alpha  \int K_{t,\eta}(v)dv  \frac{\partial^i\F_{u\,\omega}(\eta)}{\partial \omega^i}\Big|_{(u,\omega)=x}\\&+
\frac{1}{2} b_t b_f \int v K_{t,\eta} (v) dv \, \int_{\Pi} \alpha K_f(\alpha) d\alpha  \,\Big[\frac{\partial^2\F_{u,\omega}(\eta)}{\partial u \partial{\omega}} + \frac{\partial^2\F_{u,\omega}(\eta)}{\partial {\omega} \partial{u}}\Big]_{(u,\omega)=x}
+ R_{T,p}. \tageq \label{eq:taylbias}
\end{align*}
If $\eta=1$ the symmetry of $K_{t,1}(\cdot)$ and $K_f(\cdot)$ kernels about $0$ implies the bias reduces to 
\begin{align*}
\E \big [ \hat\F_{u_o,\omega_o} \HDB{(1)} \big  ] &=\F_{u_o,\omega_o} +\frac{1}{2} b_t^2  \int v^2 K_{t,\eta}(v)dv \frac{\partial^2}{\partial u^2} \F_{u,\omega}\Big|_{(u,\omega)=x} \\
& + 
\frac{1}{2} b_f^{2} \int_{\Pi}\alpha^2 K_f(\alpha) d\alpha  \frac{\partial^2}{\partial\omega^2}\F_{u\,,\omega}\Big|_{(u,\omega)=x}\\
& +o(b_t^2) + o(b_f^2)+O\big(\frac{\log(b_t\,T)}{b_t \,T}\big). 
\end{align*}
Therefore, 
$$
\hat{\F}^{\text{Jack}}_{u_0,\omega_0}(1)=\F_{u_o,\omega_o}(1)+o(b_t^2) + o(b_f^2)+O\Big ( \frac{\log(b_t\,T)}{b_t \,T}\Big )~,
$$ 
which shows the statement is true for $\eta=1$. If $\eta \in (0,1)$ the kernel in time direction is truncated which results in a slightly larger bias. More specifically, note that in our case the taper is simply $h(x)=1$ for all $x \in [0,1]$ and therefore \eqref{eq:taylbias} becomes 
\begin{align*}
\E  \big [ \hat\F_{u_o,\omega_o}^{(N,\bf)}(\eta)\big  ] &=\F_{u_o,\omega_o}(\eta) +\sum_{i=1}^2 \Big (\delta_{i,1}\frac{N}{T}\frac{(\eta-1/2)^2-\frac{1}{4}}{2 N}+\delta_{i,2}\frac{1}{2} \frac{N^2}{T^2}\frac{(\eta-1/2)^3-\frac{1}{8}}{3N} \Big)\\
& ~~~~~~~~~~~~~~~~~~~~~~~~~~~~~~~~~~~~~~~~~~~~~~~~~~~~~~~~~~~~~~~~~~~~~~~~~~
\times 
\int_{\Pi}  K_f(\alpha) d\alpha \frac{\partial^i\F_{u,\omega}(\eta)}{\partial u^i} \Big|_{(u,\omega)=x} \\&+ 
\frac{1}{2} b_f^{2} \int_{\Pi}  \alpha^2 K_f(\alpha) d\alpha \frac{\partial^2}{\partial\omega^2}\F_{u\,,\omega}\Big|_{(u,\omega)=x}+o(b_t^2) + o(b_f^2)+O\Big (\frac{\log(b_t\,T)}{b_t \,T} \Big ).
\end{align*}
Thus, for $\eta \in (0,1)$
\begin{align*}
\E \big  [ \hat{\F}^{\text{Jack}}_{u_0,\omega_0}(\eta) 
\big ] 
&=\F_{u_0,\omega_o}(\eta)+\sum_{i=1}^2
\Big (\delta_{i,1}\frac{N}{T}\frac{(\eta-\frac{1}{2})^2-\frac{1}{4}}{2 N} + \delta_{i,2}(\frac{1}{\sqrt{2}}-\frac{1}{2})\frac{N^2}{T^2}\frac{(\eta-\frac{1}{2})^3-\frac{1}{8}}{3N}\Big )
\\
& ~~~~~~~~~~~~~~~~~~~~~~~~~~~~~~~~~~~~~~~~~~~~~~~~~~~~~~~~~~~~~~~~~~~~~~~~~~
\times \int_{\Pi}  K_f(\alpha) d\alpha \frac{\partial^i\F_{u,\omega}(\eta)}{\partial u^i} \Big|_{(u,\omega)=x} \\
&+o(b_t^2) + o(b_f^2)+O\big(\frac{\log(b_t\,T)}{b_t \,T}\big).
\end{align*}
Since $1/{\sqrt{2}}-1/2 < 1/2$, the result follows.
\end{proof}}

\section{Holomorphic functional Calculus} 
\label{HFC} 
\def\theequation{E.\arabic{equation}}
\setcounter{equation}{0}

Let $\mathcal{L}(V)$ denote the set of bounded linear operators on a separable Banach space $V$. For $A \in \mathcal{L}(V)$, we recall that the spectrum is given by
\[
\sigma(A) =\{ \lambda \in \mathbb{C}: \lambda I -A \text{ is not invertible}\},
\]
where $I$ denotes the identity element of $\mathcal{L}(V)$. The mapping 
\[
R(\cdot, A): \rho(A) \to \mathcal{L}(V), \quad
 R(z,A) = (z I- A)^{-1}, \quad z \in \rho(A)
\]
is the resolvent of $A$, which is an analytic operator-valued function of $\lambda$ on the resolvent set $\rho(A)=(\sigma(A))^\complement$, which is an open subset of $\mathbb{C}$  \citep[see e.g.,][thm 1.2]{HisSig}. We can therefore define the following analogue to the Cauchy's integral formula.
\begin{definition}[Dunford-Riesz/Dunford-Taylor integral] \label{def:DTintegral}
Let $\phi(z)$ be holomorphic in an open neighborhood $D \in \mathbb{C}$ containing $\sigma(A)$. Let $\gamma \subset D$ be a closed curve counter clockwise enclosing $\sigma(A)$ in its interior. Then the bounded linear operator 
\begin{align}\label{eq:DTint}
\phi(A) = \frac{1}{2\pi \im} \oint_\gamma \phi(\zeta) R(\zeta, A) d\zeta 
\end{align}
is well-defined.  
\end{definition}
The correspondence $\phi(\zeta) \to \phi(A)$ is a homomorphism of the algebra of holomorphic functions on $D$ into $\mathcal{L}(\mathcal{H})$ (see \citep[][p. 44-45]{Kato66} or \citep[][thm 5, section 17.2]{Lax02}. 

\begin{definition}\label{adcon}
Let $A \in \mathcal{L}(V)$ and $\lambda \in \sigma(A)$ be an isolated point of the spectrum of $A$. A simple closed contour $\gamma_\lambda$ around $\lambda$ is an \textit{admissible} contour for $\lambda$ and $A$ if the closure of the region bounded by $\gamma_\lambda$ and containing $\lambda$ intersects $\sigma(A)$ only at $\lambda$.
\end{definition}

\begin{definition}\label{RProj}
Let $A \in \mathcal{L}(V)$ be self-adjoint and let $\lambda_0$ be an isolated point of $\sigma(A)$. For an admissible contour $\gamma_{\lambda_0}$, the Riesz integral for $A$ and $\lambda_0$, given by
\[
\Pi_{\lambda_0}=\frac{1}{2\pi \mathrm{i}} \oint_{\gamma_{\lambda_0}}  R(\zeta, A) d\zeta 
\]
is an orthogonal projection onto ker$(A-\lambda_0)$.
\end{definition} 
\citep[see e.g.,][Proposition 6.3]{HisSig}. Next we obtain the following result, which generalizes the result of \cite{GHJR09,RMS16}. 

\begin{thm}\label{thm:PerEx}
Let $A \in \mathcal{L}(V)$, where $V$ is a complex separable Banach space. Assume that $\phi: D \to \cnum$ is holomorphic and that $\sigma(A) \subset \bar{\Omega} \subset D$ for $\Omega$  a bounded open region in $\cnum$ with smooth boundary $\partial \Omega$. Consider the  perturbed operator $A_{\Delta}=A+\Delta$ with $\norm{\Delta}_{\mathcal{L}(V)} \le c\text{dist}(\partial \Omega, \sigma(A) )/K$ for constants $K=|\partial \Omega| \sup_{z \in \partial \Omega}\norm{R(z,A) }_{\mathcal{L}(V)}<\infty$ and $0 < c< 1$ and where $|\partial \Omega|$ denotes the length of $\partial \Omega$. Then $\phi$ is twice Fr{\'e}chet differentiable at $A$, tangentially to $\mathcal{L}(V)$. More specifically we have 
\[
\phi(A+\Delta) = \phi(A) +\phi^\prime_{A}(\Delta)+\phi^{\prime\prime}_{A}(\Delta,\Delta)+R_{\phi,A}(\Delta) 
\]
where $\norm{R_{\phi,A}(\Delta)}_{\mathcal{L}(V)}=O(\norm{\Delta}^3_{\mathcal{L}(V)})$, and where the mappings $\phi^\prime_{A}: \mathcal{L}(V)\to \mathcal{L}(V)$ and $\phi^{\prime\prime}_{A}: \mathcal{L}(V) \to \mathcal{L}(V)$, given by
\begin{align*}
\phi^\prime_{A}(\Delta) &=\frac{1}{2\pi \mathrm{i}}  \oint_{\partial \Omega} \phi(z)  R(z,A) \Delta R(z,A)  dz, \tageq\label{eq:PerEx}
\\ \phi^{\prime\prime}_{A}(\Delta, \Delta) &=\frac{1}{2\pi \mathrm{i}} \oint_{\partial \Omega} \phi(z)   R(z,A) \big(\Delta R(z,A)\big)^2  dz, \tageq\label{eq:PerEx2}
\end{align*}
are the first and second derivatives, respectively, 
 which satisfy $\norm{\phi^\prime_{A}(\Delta)}_{ \mathcal{L}(V)}<\infty$ and $\norm{\phi^{\prime\prime}_{A}(\Delta,\Delta)}_{ \mathcal{L}(V)}<\infty$. Furthermore, if $A \in  \mathbb{G}=C(\mathcal{I},S_p(\Hi))$ where $\mathcal{I}$ is a compact set in $\rnum^d$ and $\norm{\Delta}_{\mathbb{G}} \le c\text{dist}(\partial \Omega, \sigma(A) )/K$, then  $\norm{\phi^\prime_{A}(\Delta)}_{\mathbb{G}} <\infty$ and $\norm{\phi^{\prime\prime}_{A}(\Delta, \Delta)}_{\mathbb{G}} <\infty$. Finally, if $A \in S_p(\Hi)^\dagger$, $p <\infty$, we may write 
\[
\phi^\prime_{A}(\Delta) = \sum_{j=1}^\infty \phi^\prime(\lambda_j) P_j \Delta P_j + \sum_{i \ne j=1}^\infty \frac{\phi(\lambda_i)-\phi(\lambda_j)}{\lambda_i -\lambda_j} P_i \Delta P_j  \tageq \label{eq:Frderivherm}
\]
where $\lambda_1 > \lambda_2 > \ldots  \downarrow 0$ are the distinct eigenvalues in descending order and $P_i$ are the orthogonal projections onto the corresponding finite-dimensional eigenspaces.
\end{thm}
We remark that \eqref{eq:PerEx} is a symmetric bilinear form and hence is characterized by the corresponding diagonal form. 
\begin{proof}
We refer to \cite{GHJR09,RMS16} for the detailed derivations of the expression of the first and second derivatives, respectively, which follow both from the inequality in \cite{DS}[thm VII.6.10], which yields $\sup_{x \in \mathcal{I}}\norm{R(z,A(x))}_{\infty}\le K/\mathrm{dist}(\partial\Omega,\sigma(A))$, and a Neumann series expansion to obtain an  expression for the resolvent of $A_\Delta$ in terms of $\Delta$ and the resolvent of $A$.
if $A \in S_p(\Hi)$ is compact hermitian, the resolvent is given by
\[
R(z,A) =\sum_{i=1}^{\infty}  \frac{1}{z-\lambda_i} P_i
\]
and we may write 
$\phi(A) =\sum_{i=1}^{\infty}\phi( \lambda_i) P_i $ where $\phi: D \subset \sigma(A) \to \rnum$. Then  \eqref{eq:PerEx} can be rewritten  using partial fractions and Cauchy's integral to obtain \eqref{eq:Frderivherm}.
Now, if $A \in  \mathbb{G}=C(\mathcal{I},S_p(\Hi))$, we remark specifically that we obtain
\[
\norm{\phi^\prime_{A}(\Delta)}_{\mathbb{G}}=\sup_{x\in \mathcal{I}}\norm{\phi^\prime_{A(x)}(\Delta(x))}_{S_p} \le \frac{1}{2\pi}|\partial\Omega| \sup_{z \in \partial \Omega}|\phi(z)| \frac{K^2}{\mathrm{dist}^2(\partial\Omega,\sigma(A))} \sup_{x\in \mathcal{I}}\norm{\Delta_{(x)}}_{S_p} <\infty,
\]
and similarly $\norm{\phi^{\prime\prime}_{A}(\Delta,\Delta)}_{\mathbb{G}} =O(\norm{\Delta}^2_{\mathbb{G}})<\infty $.
\end{proof}

\begin{thm}\label{thm:eigp}
Let $\lambda_k$ be an eigenvalue of an operator $A \in S_{p}(\Hi)^\dagger$ of algebraic multiplicity 1, and denote the  corresponding  eigenprojector by $P_k$.
 Furthermore, let $A_{\Delta} = A+\Delta \in S_{p}(\Hi)^\dagger$ be a perturbed version of $A$. 
Then for $\Delta$ sufficiently small (i.e., $\|\Delta\|_{S_p}\le c \text{dist}(\Omega_k, \Omega \setminus \Omega_k)>0$), $A_{\Delta}$ has an eigenvalue $\lambda_{\Delta,k}$ of algebraic multiplicity 1, and the eigenprojector satisfies 
\[
 P_{\Delta,k}= P_{k}+ P_k\Delta Q_k+Q_k \Delta P_k + O(\norm{\Delta}^2_{S_p})
\] 
where $Q_k \in S_p(\Hi)$ defined in \eqref{eq:resid}.
\end{thm}
\begin{proof}
We can choose a region $\Omega \supset \sigma(A)$ s.t. it has a connected component $\Omega_k$ that satisfies $\Omega_k \cap \sigma(A) = \lambda_k$ and $\text{dist}(\Omega_k, \Omega \setminus \Omega_k)>0$ where $\Omega_o=\Omega \setminus \Omega_k$
Then, define the analytic function 
\[
\phi_k(z)=1, z \in \Omega_k, \quad \text{ and } \quad \phi_k(z), z \in \Omega_o
\]
and observe that $\phi_k(A) =P_k$. In addition, we can write
$A=\lambda_k P_k + A_o$ 
where $A_o \in S_{p}(\Hi)^\dagger$ with spectrum $\sigma(A_o) \subset \Omega_o$. 
The spectral theorem (e.g., \cite{Hall13,DS})) implies that there exists a resolution of the identity $E(\lambda)$ such that $
A_o=\int_{\sigma(A_o)} \lambda dE(\lambda)$.
Furthermore, $P_k E(\lambda) = E(\lambda) P_k =O$ for all $\lambda \in \sigma(A_o)$, and the resolvent of $A$ is given by
\[
R(z) = \frac{1}{z-\lambda_k}P_k +\underbrace{\int_{\sigma(A_o)} \frac{1}{z-
\lambda} d E(\lambda)}_{=Q_k} \tageq \label{eq:resid}
\]
The result now easily follows from \autoref{thm:PerEx} and properties of Cauchy integrals (see for example, \cite{GHJR09} for details).
\end{proof}
\begin{thm}\label{thm:eigp2}
 Assume that, uniformly in $u,\omega$, the $k$th eigenvalue $\lambda^{(u,\omega)}_k$ of $\F_{u,\omega}$ has algebraic multiplicity 1. Then
\begin{enumerate}[leftmargin=*,label=\roman*.]
\item Let $\Delta = \eta(\hat{\F}_{u,\omega}(\eta)- \F_{u,\omega})$. Under the conditions of  \autoref{thm:maxdevcon}
\[
\sup_{u,\omega} \norm{\hat{\Pi}^{(u,\omega)}_k-\Pi^{(u,\omega)}_k}_{S_p}  =O_p( \sup_{\eta,u,\omega}\norm{\Delta}_{S_p}+\sup_{\eta,u,\omega}\norm{\hat{R}^{u,\omega}(\eta)}^2_{S_p}) 
\]
where $\sup_{\eta,u,\omega}\norm{\hat{R}^{u,\omega}(\eta)}^2_{S_p} \mathrm{1}_{E_{k,T}}=O_p(\sup_{\eta,u,\omega}\norm{\Delta}^2_{S_p}) $, and  where the set $E_{k,T}$ is defined in \eqref{eq:Ekt} and satisfies  $\mathbb{P}(E_{k,T})=1$ for sufficiently large $T$.
\item 
Let $ x_{T,i} = i \frac{1}{M}$ such that 
the set $U_M$ in \eqref{eq:midset} defines the midpoints of the intervals $[  x_{i-1,T} ,x_{i,T} ] $,
that is  $u_{i,T} = ( x_{i-1,T} + x_{i,T})/2 $.
for $i=1, \ldots, M$. Then, for $u \in [  x_{i-1,T} ,x_{i,T} ]$,
$$\Pi^{(u_{i,T},\omega)}_k-\Pi^{(u,\omega)}_k= \phi^\prime_{\F_{u_{i,T},\omega}}(\F_{u,\omega}-\F_{u_{i,T},\omega})+\phi^{\prime\prime}_{\F_{u_{i,T},\omega}}(\F_{u,\omega}-\F_{u_{i,T},\omega},\F_{u,\omega}-\F_{u_{i,T},\omega}) +R_{\phi,\F_{u_{i,T},\omega}}$$ where $ \norm{R_{\phi,\F_{u_{i,T},\omega}}}^2_{S_p} \mathrm{1}_{\mathcal{O}_{k,T}}=O(\norm{\F^{u_{i,T},\omega}-\F^{u,\omega}}^2_{S_p}) $, and where the set  ${\mathcal{O}_{k,T}}$ is defined in equation \eqref{eq:Okt} below 
  and satisfies $\mathrm{1}_{\mathcal{O}_{k,T}}=1 $ for sufficiently large $T$.
\end{enumerate}
\end{thm}
\begin{proof}
 Under the stated assumption, we can, uniformly in $u,\omega$, choose a region $\Omega \supset \sigma(\F_{u,\omega})$ s.t. it has a connected component $\Omega^{u,\omega}_k$ that satisfies $\Omega^{u,\omega}_k \cap \sigma(\F_{u,\omega}) = \lambda^{u,\omega}_k$ and $\text{dist}(\Omega^{u,\omega}_k, \Omega^{u,\omega}\setminus \Omega^{u,\omega}_k)>0$. Then, under the conditions of \autoref{thm:maxdevcon}, the set $E_{k,T}$ defined by
\[E_{k,T}=\Big\{\sup_{\eta,u,\omega}\bignorm{\eta(\hat{\F}_{u,\omega}(\eta)-\F_{u,\omega})}_{S_p} \le \inf_{u,\omega}\text{dist}(\Omega^{u,\omega}_k, \Omega^{u,\omega}\setminus \Omega^{u,\omega}_k)\Big\} \tageq \label{eq:Ekt}\]
satisfies $\mathbb{P}(E_{k,T})= 1$ for sufficiently large $T$. Similarly, 
define  the set 
\[
\mathcal{O}_{k,T}=  \bigcap _{i=1}^M  \Big\{\sup_{u \in [{x_{i-1,T}},{x_{i,T}] } }\sup_{\lambda}\norm{\cdot({\F}_{u,\lambda}-{\F}_{u_{i,T},\lambda})}_{C_{S_p}}\le \inf_{\omega}\text{dist}(\Omega^{u,\omega}_k, \Omega^{u,\omega}\setminus \Omega^{u,\omega}_k)\Big\}  \tageq \label{eq:Okt}~.
\]
Then,
Lipschitz continuity of the map $u \mapsto \eta\F_{u,\lambda}$  uniformly in $\omega, \eta$ w.r.t. $\norm{\cdot}_{S_1}$ implies $\mathrm{1}_{\mathcal{O}^\complement_{k,T}}=0$ for sufficiently large $T$. The statement therefore follows from the expansion in \autoref{thm:eigp}.
\end{proof}


\begin{thebibliography}{}
\newcommand{\enquote}[1]{``#1''}
\expandafter\ifx\csname natexlab\endcsname\relax\def\natexlab#1{#1}\fi
\expandafter\ifx\csname url\endcsname\relax
 \def\url#1{{\tt #1}}\fi
\expandafter\ifx\csname urlprefix\endcsname\relax\def\urlprefix{URL }\fi


\bibitem[DeAcosta, 1970]{Ac70}
de Acosta, A. (1970). { Existence and convergence of probability measures in Banach spaces.}
{\em Trans. Amer. Math. Soc.\/}, 152: 273--298.

\bibitem[Anderson, 1963]{Anderson1963}
Anderson, T.W., (1963). { Asymptotic theory for principal component analysis.}
{\em Annals of Mathematical Statistics.\/}, 34: 122--148.
  
\bibitem[\protect\citeauthoryear{Aston, Pigoli, and Tavakoli}{Aston
  et~al.}{2017}]{AstPigTav2017}
Aston, J. A.~D., D.~Pigoli, and S.~Tavakoli (2017).  Tests for separability in nonparametric covariance operators of
  random surfaces.
 {\em The Annals of Statistics\/}~{\em 45\/}(4), 1431--1461.


\bibitem[\protect\citeauthoryear{Bagchi and Dette}{Bagchi and
  Dette}{2020}]{BagDette2017}
Bagchi, P. and H.~Dette (2020).
 A test for separability in covariance operators of random surfaces.
 {\em Annals of Statistics\/}, 
{\em 48\/}(4), 2303–2322.




\bibitem[\protect\astroncite{Brillinger}{1981}]{b81}
Brillinger, D.\ (1981).
 {\it Time Series: Data Analysis and Theory.}
 McGraw Hill, New York.

\bibitem[Constantinou et~al., 2017]{conkokrei2017}
Constantinou, P., Kokoszka, P., and Reimherr, M. (2017).
 Testing separability of space-time functional processes.
 {\em Biometrika}, 104(2):425--437.



\bibitem[Cupidon et~al., 2007]{cup07}
Cupidon, J., Gilliam, D.S., Eubank, R., and Ruymgaart, F. (2007).
 The delta method for analytic functions of random operators with application to
functional data.
 {\em Bernoulli\/}, 13(4):1179--1194.

\bibitem[{Dahlhaus(1997)}]{dahlhaus1997}
Dahlhaus, R. (1997).
  {Fitting time series models to nonstationary processes.}
 {\em Annals of Statistics\/}  25(1): 1--37.


\bibitem[{Dahlhaus et al. (2019)}]{drw2019}
Dahlhaus, R., Richter, S., and Wu, W.B. (2019).
  {Towards a general theory for nonlinear
locally stationary processes.}
 {\em Bernoulli\/}  25(2): 1013--1044.




\bibitem[{van Delft (2020)}]{vD19}
{van~Delft}, A. (2020).
{A note on quadratic forms of functional time series under mild conditions.}
{\em Stochastic Processes and their Applications\/}, 130(7): 4206--4251.

\bibitem[{van Delft and Dette(2022)}]{vdd21}
van Delft, A. and \ Dette, H.
(2022). 
 Pivotal tests for relevant differences in the second order dynamics
  of functional time series.
 {\em Bernoulli}, 28: 2260--2293.

\bibitem[{van Delft and Eichler(2020)}]{vDE19}
van Delft, A. \&\ Eichler, M (2020).
{A note on Herglotz's theorem for time series on function spaces.}
{\em Stochastic Processes and their Applications\/},  130(6):3687--3710.


\bibitem[{van Delft et al. (2012)}]{vDCD18}
van Delft, A., Characiejus, V., and Dette, H. (2021). 
 A nonparametric test for stationarity in functional time series.
 {\em Statistica Sinica}, 31: 1375--1395.



\bibitem[{van Delft and Eichler(2018)}]{vde16}
van Delft, A. and Eichler, M. (2018).
{Locally stationary functional time series.}
 {\em Electronic Journal of Statistics\/}, 12:107--170.

\bibitem[Dette et~al., 2020]{detkokvol2020}
Dette, H., Kokot, K., and Volgushev, S. (2020b).
 Testing relevant hypotheses in functional time series via
  self-normalization.
 {\em Journal of the Royal Statistical Society: Series B (Statistical
  Methodology)}, 82(3):629--660.
  
\bibitem[Dryden et~al., 2009]{Drydenetal2009}
Dryden, I.~L., Koloydenko, A., and Zhou, D. (2009).
 {Non-Euclidean statistics for covariance matrices, with applications
  to diffusion tensor imaging}.
 {\em The Annals of Applied Statistics}, 3(3):1102 -- 1123.

\bibitem[Fremdt et~al., 2013]{fremdt2013}
Fremdt, S., Steinebach, J.~G., Horv\'ath, L., and Kokoszka, P. (2013).
 Testing the equality of covariance operators in functional samples.
 {\em Scandinavian Journal of Statistics}, 40(1):138--152.









\bibitem[\protect\citeauthoryear{Genton}{Genton}{2007}]{MarcGenton2007}
Genton, M. (2007).
 Separable approximations of space-time covariance matrices.
 {\em Environmetrics\/}~{\em 18\/}(7), 681--695.


\bibitem[{H{\"o}rmann et~al.(2015)}]{Hormann2015}
H{\"o}rmann S., Kidzi{\'n}ski, L. \&\ Hallin, M. (2015). 
 Dynamic functional principal components.
 {\em The Royal Statistical Society: Series B\/} 77:319--348.


\bibitem[{Jentsch and Subba~Rao(2015)}]{jensub2015}
Jentsch, C. and Subba~Rao, S. (2015).
  {A test for second order stationarity of a multivariate time
  series.}
 {\em Journal of Econometrics\/}  185:124--161.

\bibitem[{Jolliffe (2002)}]{Jolliffe}
Jolliffe, I.T. (2002).
{\it Principal Component Analysis}.  Springer.

\bibitem[{Kreiss and Paparoditis (2015)}]{kp15}
Kreiss, J.P., and Paparoditis, E. (2015).
 Bootstrapping locally stationary processes. 
 {\em The Royal Statistical Society: Series B\/} 77: 267--290.


\bibitem[{Liu and Wu(2010)}]{LiuWu10}
Liu, W. and \ Wu, W. B. (2010).
  {Asymptotics of spectral density estimates.}
  {\em Econometric Theory\/}, 26:1218--1245.
 


\bibitem[{Mas(2006)}]{Mas2006}
Mas, A. (2006)
 {\em A sufficient condition for the CLT in the space of nuclear operators—Application to covariance of random functions}.
{\em Statistics and Probability letters\/}, 76(14):1503--1509.


\bibitem[\protect\citeauthoryear{Masak, Sarkar, and Panaretos}{Masak
  et~al.}{2023}]{Masak2020}
Masak, T., S.~Sarkar, and V.~M. Panaretos (2023).
 {Principal separable component analysis via the partial inner product.}
{\em Biometrika \/}, 110(1): 225–247.

\bibitem[{Nason et al. (2000)}]{nvsk2000}
Nason, G.P., von Sachs, R., and Kroisandt, G. (2000).
 {Wavelet processes and adaptive estimation
of the evolutionary wavelet spectrum.}
 {\em Journal of the Royal Statistical Society. Series B.\/}
 62(2):271--292.

\bibitem[{Ledoux and Talagrand (1991)}]{LT91}
Ledoux, M., and Talagrand, M.  (1991).
 {\em Probability in Banach Spaces} 
 Springer Berlin, Heidelberg.

  
\bibitem[Panaretos et~al., 2010]{Panaretos2010}
Panaretos, V.~M., Kraus, D., and Maddocks, J.~H. (2010).
Second-order comparison of {G}aussian random functions and the
  geometry of {DNA} minicircles.
{\em Journal of the American Statistical Association},
  105(490):670--682.
  
\bibitem[{Panaretos \harvardand\ Tavakoli(2013)}]{PanTav2013}
Panaretos, V. \harvardand\ Tavakoli, S.   (2013).
 {Cram{\'e}r--Karhunen--Lo{\`e}ve representation and harmonic principal component analysis of functional time series.}
 {\em Stochastic Processes and their Applications\/} 123:2779--2807.



\bibitem[Pigoli et~al., 2014]{Pigolietal2014}
Pigoli, D., Aston, J. A.~D., Dryden, I.~L., and Secchi, P. (2014).
 {Distances and inference for covariance operators}.
 {\em Biometrika}, 101(2):409--422.


\bibitem[{Politis et. al, (1999)}]{polromwol}
Politis, D. N., Romano, J. P., and Wolf, M. (1999).
{\it Subsampling.} 
Springer, New York. 

\bibitem[{Shao(2015)}]{shao2015}
Shao, X. (2015).
  {Self-Normalization for Time Series: A Review of Recent Developments.}
 {\em Journal of the American Statistical Association \/}  110(512): 1797--1817.

\bibitem[Tomczak-Jaegerman(1974)]{tj74}
Tomczak-Jaegerman, N (1974).
{The moduli of smoothness and convexity of trace class Sp.}
{\em Studia Math\/},  50:163--182.







\bibitem[{Wu(2005)}]{Wu05}
Wu, W. B. (2005).
 {Nonlinear system theory: Another look at dependence.}
 {\em Proceedings of the National Academy of Sciences.\/}, 102:14150--14154.



\bibitem[\protect\astroncite{Yang et al.}{2011}]{Yang2011}
Yang, W.,\ M{\"u}ller, H-G.\ \&\ Stadtm{\"u}ller, U. (2011).
 Functional Singular Component Analysis.
 {\it Journal of the Royal Statistical Society: Series B\/} 73:303--324.

\bibitem[{Yang and Zhou (2022+)}]{yz2021}
Yang, J, and Zhou, Z. (2021+).
{Spectral Inference under Complex Temporal Dynamics}
 To appear in: {\it Journal of the American Statistical Association\/} arXiv:1812.07706.

\bibitem[Zhang and Shao(2015)]{Zhang2015}
Zhang, X. and Shao, X. (2015).
 Two sample inference for the second-order property of temporally
  dependent functional data.
 {\em Bernoulli} 21:90--929.
  
\bibitem[{Zhang and Wu(2021)}]{ZW21}
Zhang, D and \ Wu, W. B. (2021).
{
Convergence of covariance and spectral density estimates for high-dimensional locally stationary processes.
}
 {\em The Annals of Statistics\/}, 49(1):233--254.


\end{thebibliography}

\begin{thebibliography}{}
\newcommand{\enquote}[1]{``#1''}
\expandafter\ifx\csname natexlab\endcsname\relax\def\natexlab#1{#1}\fi
\expandafter\ifx\csname url\endcsname\relax
 \def\url#1{{\tt #1}}\fi
\expandafter\ifx\csname urlprefix\endcsname\relax\def\urlprefix{URL }\fi

\bibitem[DeAcosta, 1970]{Ac70}
de Acosta, A. (1970). { Existence and convergence of probability measures in Banach spaces.}
{\em Trans. Amer. Math. Soc.\/}, 152: 273--298.


\bibitem[{Conway(1990)}]{Conway}
Conway, J.B. (1990).
 {\em A Course in Functional Analysis.} 
 Springer-Verlag, New York.



\bibitem[{Day(1973)}]{Day}
Day, M. (1973). 
 {\em Normed Linear Spaces.} 
 3rd ed., Ergebnisse der Mathematik und ihrer Grenzgebiete, Band 21, Springer-Verlag, New York.



\bibitem[{van Delft (2020)}]{vD19}
van Delft, A. (2020).
{A note on quadratic forms of functional time series under mild conditions.}
{\em Stochastic Processes and their Applications\/}, 130(7): 4206--4251.



\bibitem[{van Delft and Dette(2022)}]{vdd21}
van Delft, A. and \ Dette, H.
(2022). 
 Pivotal tests for relevant differences in the second order dynamics
  of functional time series.
 {\em Bernoulli}, 28: 2260--2293.

\bibitem[\protect\citeauthoryear{Dette, Dierickx, and Kutta}{Dette
  et~al.}{2022}]{DetDieKut21}
Dette, H., G.~Dierickx, and T.~Kutta (2022).
\newblock {Quantifying deviations from separability in space-time functional
  processes}.
\newblock {\em Bernoulli\/}~{\em 28\/}(4), 2909--2940.

\bibitem[\protect\citeauthoryear{Dunford and Schwartz}{Dunford and Schwartz}{1988}]{DS}
N. Dunford and J.T. Schwartz, (1988).
 {\em Linear Operators. Part I: general theory.\/}
 Wiley Classics Library, Wiley-
Interscience, New York.

\bibitem[Einmahl and Li (2008)]{el08} 
Einmahl, U., and Li, D. (2008).
{Characterization of Lil Behavior in Banach Space}.
 {\em Transactions of the American Mathematical Society}, 360(12):6677 -- 6693.


\bibitem[{Gilliam et al.(2009)}]{GHJR09}
Gilliam, D.S, Hohage, T., Ji, X. and Ruymgaart, F.
 {The Fr{\'e}chet Derivative of an Analytic Function of a Bounded Operator with Some Applications}
 {\em International Journal of Mathematics and Mathematical Sciences\/}, Article ID 239025 (2009).



\bibitem[\protect\citeauthoryear{Hislop and Sigal}{Hislop and Sigal}{1996}]{HisSig}
Hislop, P.~D. and I.~M. Sigal (1996).
 {\em Introduction to Spectral Theory\/}  {\em Applied Mathematical Sciences, Volume 113}.
 Springer.



\bibitem[{Hall(2013)}]{Hall13}
Hall, B.C. (2013).
{\em Quantum Theory for Mathematicians}
{Graduate Texts in Mathematics}, vol. 267, Springer. 

\bibitem[{Heil(2011)}]{Heil}
Heil, C. (2011)
 {\em A Basis Theory Primer.} 
 Birkh{\"a}user, Boston.

\bibitem[{Kato(1966)}]{Kato66}]
Kato, T. (1966).
 {\em Perturbation Theory for Linear Operators\/}.
 Springer, Berlin, Germany.

\bibitem[{Kuelbs (1973)}]{Kuelbs73}
Kuelbs, J. (1973).
 The invariance principle for Banach space valued random variables.  
 {\em Journal of Multivariate Analysis\/} 3(2): 161--172.

\bibitem[{Lax(2002)}]{Lax02}
Lax, P. D.
 {\em Functional Analysis, Pure and Applied Mathematics\/}.
 John Wiley \& Sons, New York, USA, 2002.

\bibitem[{M{\`o}ricz(1976)}]{Mor76}
M{\`o}ricz, F. (1976).
 {Moment inequalities and the strong law of large numbers.}
 {\em Z. Wahrscheinlichkeitstheorie verw. Gebiete \/}, 35:299--314.



\bibitem[{McLeish(1974)}]{McL74}
McLeish, D. L.  (1974).
 {Dependent central limit theorems and invariance principles.}
 {\em The Annals of Probability\/}, 2:620--628.




\bibitem[{Parthasarathy(1967)}]{Part67}
Parthasarathy, K.R. (1967).
 {\em Probability Measures on Metric Spaces.} 
 New York: Academic Press.

\bibitem[{Pedersen(1989)}]{Pedersen}
Pedersen, G.K. (1989).
 {\em Analysis Now.} 
 Springer, New York.




\bibitem[{Ripple et al. (2016)}]{RMS16}
Ripple, A., Munk, A. and \ Sturm, A (2016).
{Limit laws of the empirical Wasserstein distance: Gaussian distributions.}
{\em Journal of Multivariate Analysis \/},  151:90--109.


\bibitem[Tomczak-Jaegerman(1974)]{tj74}
Tomczak-Jaegerman, N (1974).
{The moduli of smoothness and convexity of trace class Sp.}
{\em Studia Math\/},  50:163--182.




\bibitem[{Walk(1977)}]{Walk77}
Walk, H. (1977).
 {An Invariance Principle for the Robbins-Monro Process
in a Hilbert Space}
 {\em Proceedings of the National Academy of Sciences.\/}, 102:14150--14154.

\bibitem[{Zhang and Wu(2021)}]{ZW21}
Zhang, D and \ Wu, W. B. (2021).
{
Convergence of covariance and spectral density estimates for high-dimensional locally stationary processes.
}
 {\em The Annals of Statistics\/}, 49(1):233--254.


\bibitem[{Zygmund (2002)}]{Zygmund}
Zygmund, A. (2002).
{\em Trigonometric Series. Vol I, II. Third Ed.}
 Cambridge University Press, Cambridge.

\end{thebibliography}
\end{document}